\documentclass[11pt]{article}
\usepackage[margin=.85in]{geometry}
\usepackage{amsmath, amsthm, amssymb} 
\usepackage{mathtools, thmtools} 
\usepackage{graphicx, subcaption}
\usepackage[hidelinks]{hyperref}
\usepackage{cleveref}

\usepackage[mathscr]{euscript} 
\usepackage{xcolor} 
\usepackage{natbib} 
\usepackage{setspace} 
\usepackage[shortlabels]{enumitem}
\usepackage{algorithm, algpseudocode}
\usepackage{thm-restate}
\usepackage{soul}
\usepackage{esvect}

\usepackage[title, titletoc]{appendix}
\renewcommand\appendixname{Supplement}

\newtheorem{thm}{Theorem}[section]

\newtheorem{lemma}[thm]{Lemma}
\newtheorem{cor}[thm]{Corollary}

\newcommand{\iidsim}{\stackrel{\mathclap{\rm{iid}}}{\sim}}

\makeatletter
\def\namedlabel#1#2{\begingroup
    #2%
    \def\@currentlabel{#2}%
    \phantomsection\label{#1}\endgroup
}
\makeatother

\newcommand{\inoutprob}{\xrightarrow{P_0^*}}
\newcommand{\condinoutprob}{\underset{W}{\xrightarrow{P_0^*}}}

\newcommand{\condinoutdist}{\underset{W}{\overset{P_0^*}{\leadsto}}}

\newcommand{\fasterthan}{o_{\mathrm{P}_0^*}}

\newcommand{\E}{E}
\newcommand{\prob}{P}

\newcommand{\cov}{\mathrm{Cov}}
\newcommand{\BL}{\mathrm{BL}}

\renewcommand{\qed}{$\hfill\blacksquare$}
\newcommand{\sd}{\, \mathrm{d}} 

\newcommand{\s}[1]{\mathscr{#1}}
\renewcommand{\d}[1]{\mathbb{#1}} 

\renewcommand{\rm}[1]{\mathrm{#1}} 
\renewcommand{\d}[1]{\mathbb{#1}} 
\def\norm#1{\left\Vert{#1}\right\Vert}
\def\abs#1{\left\vert{#1}\right\vert}

\allowdisplaybreaks

\title{Consistency of the bootstrap for asymptotically linear estimators based on machine learning}
\author{Zhou Tang \hspace{2em} Ted Westling \\
Department of Mathematics and Statistics \\
University of Massachusetts Amherst}

\begin{document}
\maketitle

\begin{abstract}
The bootstrap is a popular method of constructing confidence intervals due to its ease of use and broad applicability. Theoretical properties of bootstrap procedures have been established in a variety of settings.  However, there is limited theoretical research on the use of the bootstrap in the context of estimation of a differentiable functional in a nonparametric or semiparametric model when nuisance functions are estimated using machine learning. In this article, we provide general conditions for consistency of the bootstrap in such scenarios. Our results cover a range of estimator constructions, nuisance estimation methods, bootstrap sampling distributions, and bootstrap confidence interval types. We provide refined results for the empirical bootstrap and smoothed bootstraps, and for one-step estimators, plug-in estimators, empirical mean plug-in estimators, and estimating equations-based estimators. We illustrate the use of our general results by demonstrating the asymptotic validity of bootstrap confidence intervals for the average density value and G-computed conditional mean parameters, and  compare their performance in finite samples using numerical studies.  Throughout, we emphasize whether and how the bootstrap can produce asymptotically valid confidence intervals when standard methods fail to do so.
\end{abstract}

\clearpage

\tableofcontents

\clearpage 

\doublespacing

\section{Introduction}

In many problems in statistics, interest focuses on a  parameter that depends on one or more unknown infinite-dimensional nuisance functions. For example, treatment effects in causal inference often depend on conditional mean functions \citep{robins1986new, holland1986statistics, gill2001causal}, average derivative parameters depend on derivatives of conditional mean functions \citep{powell1989semiparametric}, and survival functions with informative censoring depend on conditional distribution functions \citep{dabrowska1989uniform}. Estimating such a parameter typically involves estimating one or more infinite-dimensional nuisance parameters. In order to reduce potential bias due to model misspecification, researchers may choose to use a data-adaptive estimator for the nuisance, such as a tree-based estimator, a neural network, splines, or an ensemble of many such estimators. However, directly plugging a data-adaptive nuisance estimator into the parameter mapping often results in asymptotic bias that hinders valid statistical inference for the parameter of interest. If the parameter is a smooth function of the underlying data-generating distribution, it is possible to mitigate this bias enough for the estimator to be asymptotically linear under a sufficient rate of convergence of the nuisance estimator \citep{pfanzagl1982contributions, van1991differentiable}. There are several general approaches for constructing asymptotically linear estimators in the presence of data-adaptive nuisance estimators, including the one-step construction \citep{bickel1982adaptive, pfanzagl1982contributions}, sieve and series approximations \citep{geman1982nonparametric},  under-smoothing \citep{newey1998undersmoothing}, estimating equations \citep{liang1986longitudinal, hardin2002generalized}, and targeted minimum loss-based estimators \citep{van2006targeted, van2011targeted}.

If an asymptotically linear estimator can be constructed, it can often be used to conduct asymptotically valid inference. The most popular method of constructing confidence intervals based on an asymptotically linear estimator is the Wald interval using a normal approximation to the sampling distribution and the so-called \emph{influence function-based} variance estimator.  The bootstrap \citep{efron1982jackknife, efron1994introduction} is an alternative method of constructing confidence intervals that has several potential advantages over Wald intervals. First, the bootstrap has been shown in some settings to have higher-order accuracy, and hence better finite-sample coverage, than Wald intervals \citep{hall1988theoretical, hall1992bootstrap, diciccio1988review}. Second, bootstrap confidence intervals can sometimes automatically correct bias, and can thus be asymptotically valid under  weaker conditions \citep{cattaneo2018kernel, cattaneo2022average}. Third, in some cases, an estimator of the asymptotic variance is not readily available because, for instance, the influence function of the estimator does not have a closed form (e.g., \citealp{geskus1999asymptotically, quale2006locally}). Finally, Wald intervals are not guaranteed to respect constraints on the parameter space, whereas some types of bootstrap intervals are. 

The theoretical properties of the bootstrap have been studied for many problems. For example, \cite{hall1988theoretical, hall1992bootstrap} established higher-order properties of bootstrap confidence intervals; \cite{gine1990bootstrapping} and \cite{van1996weak} established uniform bootstrap central limit theorems;  \cite{chernozhukov2017central, chernozhukov2022improved} established properties of the bootstrap for high-dimensional data; and \cite{han2018gaussian}, \cite{austern2020asymptotics}, and \cite{cattaneo2020bootstrap} studied properties of the bootstrap for non-Gaussian limits.  Several authors have studied bootstrap methods for problems involving nuisance parameters. \cite{kosorok2004robust} and \cite{dixon2005functional} proved consistency of bootstrap procedures in semiparametric models when the nuisance parameter can be estimated at the $n^{-1/2}$ rate. \cite{ma2005robust} and \cite{cheng2010bootstrap} demonstrated consistency of the weighted empirical bootstrap for $M$-estimators in semiparametric models, permitting the nuisance to converge at a rate slower than $n^{-1/2}$. \cite{van2018targeted} proposed methods of bootstrapping targeted minimum loss-based estimators. \cite{cai2020highly} used the nonparametric bootstrap to achieve better finite sample coverage  using the highly adaptive lasso targeted minimum loss-based estimator. However, to the best of our knowledge, a comprehensive study of bootstrap procedures for asymptotically linear estimators with data-adaptive nuisance estimators does not yet exist.

In this article, we provide results for consistency of bootstrap methods for asymptotically linear estimators involving data-adaptive nuisance estimation. Several notable contributions of our work include: (1) we propose a general framework that allows us to study a variety of asymptotically linear estimator constructions, and we study several specific constructions in depth; (2) we provide conditions under which the bootstrap estimator is conditionally asymptotically linear; (3) we use our framework to provide conditions for consistency of bootstrap confidence intervals, highlighting in particular settings in which automatic bias correction is and is not possible; and (4) we cover a variety of bootstrap sampling distributions, including both the empirical bootstrap and smooth bootstraps, and a variety of methods of bootstrap nuisance estimation.

This last contribution is especially important in the context of data-adaptive nuisance estimators. As noted in \cite{bickel1997resampling} and \cite{van2018targeted}, the empirical bootstrap can fail if the estimator is sensitive to ties in the data. For example, a nuisance estimator that uses cross-validation to select tuning parameters or as part of an ensemble learning strategy may not behave as expected when applied to an empirical bootstrap sample because duplicate observations can appear in both the training and testing folds. Smooth bootstraps---i.e., bootstrap distributions that are dominated by Lebesgue measure---can resolve this issue by producing bootstrap samples without duplicates. Although smooth bootstraps have received considerably less theoretical attention than the empirical bootstrap, asymptotic properties of the smooth bootstrap have been established in, e.g., \cite{hall1989smoothing}, \cite{cuevas1997differentiable}, and \cite{gaenssler2003smoothed}. Our results for smooth bootstraps in Section~\ref{section: type of bootstrap} build on these works. Another way to avoid issues with duplicate observations is to alter the way that the nuisance is estimated for the bootstrap sample. For example, rather than using the bootstrap sample to choose tuning parameters, some authors have proposed fixing tuning parameters such as bandwidths at the values selected by the original data when constructing the bootstrap nuisance estimator \citep{hall2001bootstrapping}. Alternatively, the entire bootstrap nuisance estimator could be fixed at the value estimated by the original data. Our framework permits these types of approaches.

The remainder of the article is organized as follows. In Section~\ref{sec:classical}, we define the statistical setting we work in and outline our estimation and bootstrap frameworks. In Section~\ref{sec:theory}, we provide general theoretical results, including conditional asymptotic linearity and weak convergence of the bootstrap estimator and consistency of bootstrap confidence intervals. In Section~\ref{sec: T examples}, we provide refined conditions implying a key condition of our main results for four estimator constructions. In Section~\ref{sec: applications}, we illustrate the use of our theoretical results by studying various candidate bootstrap procedures for two parameters, deriving new results for both parameters.. In Section~\ref{sec: simulation}, we present a simulation study for the methods studied in  Section~\ref{sec: applications}. Section~\ref{sec: conclusion} presents a brief discussion. Proofs of all theorems and additional technical details are provided in Supplementary Material.

\section{Estimation and bootstrap framework}\label{sec:classical}
\subsection{Statistical setup}

We suppose that $X_1, \dotsc, X_n \in \d{R}^d$ is an IID sample from a probability measure $P_0$ on a measurable space $(\s{X}, \s{B})$. We assume that $P_0$ is known to lie in a statistical model $\s{M}$, which is a nonparametric or semiparametric model in our motivating applications. We will use subscript $0$ to indicate that an object depends on $P_0$. We let $\d{P}_n$ be the empirical distribution of $X_1, \dotsc, X_n$. For any measure $P$ and $P$-integrable function $f$, we define $Pf := \int f \sd P$. We define the \textit{empirical process} evaluated at a $P_0$-integrable function $f$ as $\d{G}_n f := n^{1/2} (\d{P}_n - P_0)f$.

For a set of functions $\s{F}$, we define $\ell^\infty(\s{F})$ as the Banach space of real-valued bounded functions $z: \s{F} \mapsto \d{R}$ equipped with the supremum norm $\norm{z}_{\s{F}} := \sup_{f \in \s{F}} \abs{z(f)}$. To characterize weak convergence in the space $\ell^\infty(\s{F})$, we utilize the bounded dual Lipschitz distance based on outer expectations. For an arbitrary metric space $(\d{D}, d)$ (frequently, $\d{D} = \ell^\infty(\s{F})$ and $d(z_1, z_2) = \sup_{f \in \s{F}}\abs{z_1(f) -z_2(f)}$), we denote $C_b(\d{D})$ as the set of bounded and continuous functions from $\d{D}$ to $\d{R}$, and we denote $\BL_1(\d{D})$ as $h: \d{D} \mapsto [-1, 1]$ such that $h$ is 1-Lipschitz; i.e., $\sup_{z_1, z_2 \in \d{D}, z_1 \neq z_2} \abs{h(z_1) - h(z_2)}/d(z_1, z_2) \leq 1$. We then say that a possibly non-measurable sequence of stochastic processes $G_n$ on $\d{D}$ converges weakly to a Borel measurable limit $G$ in $\d{D}$, denoted as $G_n \leadsto G$, if $\E_0^* h(G_n) \to \E_0 h(G)$ for every $h \in C_b(\d{D})$. Here $\E_0^*$ is the \emph{outer expectation}, which is used to accommodate non-measurable $G_n$. If $G$ is a  separable process (i.e., there exists $S \subseteq \d{D}$ such that $P(G \in S) = 1$ and $S$ has a countable dense subset), then $G_n \leadsto G$ if and only if $\sup_{h \in \BL_1(\d{D})} \abs{ \E_0^* h(G_n) - \E_0 h(G) } \to 0$. We refer the reader to \cite{van1996weak} for a review of outer expectation and weak convergence. 

A class $\s{F}$ of measurable functions $f: \s{X} \mapsto \d{R}$ is called \textit{$P_0$-Donsker} if the sequence of empirical processes $\{ \d{G}_n f: f\in \s{F}\}$ converges weakly in $\ell^\infty(\s{F})$ to a tight, Borel measurable limit process $\d{G}_0$, or $\sup_{h \in \BL_1(\ell^\infty(\s{F}))} \abs{\E_0^* h(\d{G}_n) - \E_0 h(\d{G}_0)} \to 0$ \cite[Chapter 1.12]{van1996weak}. The limit process $\d{G}_0$ is necessarily a Gaussian process with zero mean and covariance function $\cov(\d{G}_{0} f, \d{G}_{0} g) = P_0(fg) - P_0f P_0g$ for $f, g \in \s{F}$, which is known as \textit{$P_0$-Brownian bridge process}. Implicitly, the Donsker property requires that the sample paths $f \mapsto  \d{G}_nf$ are almost surely uniformly bounded for every $n$, so that $\d{G}_n$ may be regarded as a  map from the underlying product measurable space $(\s{X}^\infty, \s{B}^\infty)$ to $\ell^\infty(\s{F})$. This is the case if $\sup_{f \in \s{F}} | f(x) - P_0 f|< \infty$ for all $x \in \s{X}$.

\subsection{Asymptotically linear estimator framework}
\label{sec: targeted estimators of differentiable functional}

We are interested in inference for a real-valued target parameter $\psi: \s{M} \to \d{R}$. All the results in this article extend to Euclidean parameters $\psi: \s{M} \to \d{R}^p$ for $p < \infty$ fixed, but we assume $p = 1$ for simplicity of exposition. We assume that $\psi$ is a smooth enough mapping on the model $\s{M}$ to permit construction of an asymptotically linear estimator $\psi_n$ with influence function $\phi_0 := \phi_{P_0}$, meaning that $\psi_n = \psi_0 + \d{P}_n \phi_0 + \fasterthan(n^{-1/2})$, where $\psi_0 := \psi(P_0)$. For any $P \in \s{M}$, $\phi_P: \s{X} \to \d{R}$ is a function satisfying  $P\phi_P = 0$ and $P \phi_P^2 < \infty$. Conditions under which this is possible and derivation of influence functions is not our main focus here, but we refer the interested reader to \cite{pfanzagl1982contributions},  \cite{klaassen1987consistent}, \cite{pfanzagl1990estimation}, and \cite{van1991differentiable}. Here, we will not require that $\phi_0$ is the efficient influence function, but rather allow it to be any influence function. In addition, we do not explicitly require that $\psi_n$ is a regular estimator. 
 
The premise of this article is that the parameter $\psi$ 
depends on $P$ through an infinite-dimensional nuisance parameter $\eta : \s{M} \to \s{H}$. For instance, $\eta$ may be a (conditional) density function, a (conditional) cumulative distribution function, a regression function, or some combination of these. A nuisance estimator $\eta_n$ may then be used in the construction of an estimator $\psi_n$ of $\psi_0$,  However, if $\eta_n$ is constructed without consideration of $\psi$ or its influence function, then the resulting $\psi_n$ may have excess bias inherited from $\eta_n$ that precludes asymptotic linearity. There are several existing remedies. One approach is to construct the nuisance estimator to reduce the bias of the plug-in estimator. Undersmoothing \citep{newey1998undersmoothing}, twicing kernels \citep{newey2004twicing}, sieves \citep{geman1982nonparametric, shen1997methods}, and TMLE \citep{van2006targeted, van2011targeted} are examples of this approach. Another approach is to abandon plug-in estimation and instead use the influence function to target the parameter of interest. One-step estimators \citep{bickel1982adaptive, pfanzagl1982contributions} and estimating equations-based estimators \citep{liang1986longitudinal} are examples of this approach. 

We now introduce a general framework that encompasses many approaches to constructing an asymptotically linear estimator. This will allow us to study these approaches in a unified manner. We assume there exists a function $T: \s{H} \times \s{M}^+$, where $\s{M}^+$ is the union of $\s{M}$ and the set of finite discrete probability measures on $(\s{X}, \s{B})$, such that $\psi_0 = T(\eta_0, P_0)$ and $\psi_n = T(\eta_n, \d{P}_n)$ for a nuisance estimator $\eta_n \in \s{H}$. We note that for a given estimator $\psi_n$,  there may be multiple representations of $\psi_n$ in terms of different functions $T$ because $\eta_n$ also typically depends on $\d{P}_n$.  To illustrate this general framework, we provide two brief examples of estimator constructions $T$. These examples and others will be discussed further in Section~\ref{sec: T examples}. The \emph{one-step} construction is defined as $T(\eta, P) = \psi(\eta) + P \phi_{\eta}$, where it is assumed that the parameter $\psi(P)$ and its influence function $\phi_P$ depend on $P$ only through $\eta_P$. The mean-zero property of influence functions implies that $T(\eta_0, P_0) = \psi_0$. A one-step estimator is then given by $T(\eta_n, \d{P}_n) = \psi(\eta_n) + \d{P}_n \phi_{\eta_n}$. Alternatively, the \emph{empirical mean plug-in} construction is defined as $T(\eta, P) := \int g(x, \eta) \sd P(x)$ for a function $g : \s{X} \times \s{H} \to \d{R}$, which can be used when $\psi(P) = \int g(x, \eta_P) \sd P(x)$. The resulting estimator is then given by $T(\eta_n, \d{P}_n) = \int g(x, \eta_n) \sd \d{P}_n(x) = \frac{1}{n} \sum_{i=1}^n g(X_i, \eta_n)$.

\subsection{Bootstrap framework}\label{sec:bootstrap_framework}

We now introduce the class of bootstrap schemes that we will consider. At a high level, the bootstrap schemes we will consider involve three steps. First, bootstrap samples $X_1^*, \dotsc, X_n^*$ are generated in some manner based on the data $X_1, \dotsc, X_n$. Second, a version of the estimation procedure is applied to the bootstrap data to produce a bootstrap estimate $\psi_n^*$. Third, this process is repeated $B$ times to approximate the sampling distribution of $\psi_n^* - \psi_n$ given the data. Throughout, we will ignore the effect of approximating the distribution of $\psi_n^* - \psi_n$ using a finite number of repetitions. The distribution of $\psi_n^* - \psi_n$ given the data is used to approximate the sampling distribution of $\psi_n - \psi_0$, and ultimately to construct confidence intervals for $\psi_0$. There are many specific approaches to each of these three steps, and as discussed in the introduction, our goal  is to provide results that cover a broad set of these approaches. In this section, we precisely define the approaches to the first two steps that we will consider; procedures for constructing confidence intervals are discussed in Section~\ref{sec:conf_int}.

We assume that, given the data $X_1, \dotsc, X_n$, the bootstrap sample $X_1^*, \dotsc, X_n^*$ is drawn IID from an estimate $\hat{P}_n$ of $P_0$ based on $X_1, \dotsc, X_n$. We refer to $\hat{P}_n$ as the \emph{bootstrap sampling distribution}. Taking $\hat{P}_n = \d{P}_n$ corresponds to the empirical bootstrap, in which case the bootstrap sample consists of $n$ samples drawn IID (i.e., with replacement) from $X_1, \dotsc, X_n$. Another common approach to defining $\hat{P}_n$ is through smoothing methods such as kernel density estimation. This will be discussed at more length in Section~\ref{section: type of bootstrap}. Notably, since we assume that $n$ samples are drawn independently from $\hat{P}_n$, we exclude the exchangeable, weighted, and $m$-out-of-$n$ bootstraps. Our results could be generalized to the weighted bootstrap if the nuisance estimator utilizes sample weights, and to the exchangeable bootstrap if the nuisance estimator does not depend on the independence of the data. We define the \textit{bootstrap empirical distribution} $\d{P}_n^*$ as the empirical distribution of the bootstrap data $X_1^*, \dotsc, X_n^*$ and the \textit{bootstrap empirical process} $\d{G}_n^*$ as  $\d{G}_n^* := n^{1/2} (\d{P}_n^* - \hat{P}_n)$. 

Once the process for generating bootstrap data has been defined, the next step is to define the process for constructing the estimator using the bootstrap data. Since our definition of the original estimator is $\psi_n = T(\eta_n, \d{P}_n)$, we will define $\psi_n^*:= T(\eta_n^*, \d{P}_n^*)$ as the estimator using the bootstrap data, where $\eta_n^*$ is an estimator of the nuisance parameter $\eta_0$ based on the bootstrap data and original data, and $\d{P}_n^*$ is the bootstrap empirical distribution as previously defined. The spirit of bootstrap estimation would suggest that $\eta_n^*$ be constructed using the bootstrap data in the exact same manner as $\eta_n$ was constructed using the original data. However, we will not require this --- instead, we will be agnostic about the way $\eta_n^*$ is constructed. There are several reasons motivating this increase in generality. First, in many of our intended applications, $\eta_n$ is estimated using machine learning, which may be very computationally intensive. Repeating a computationally intensive procedure for every bootstrap sample may be infeasible since the number of bootstrap samples $B$ is typically in the hundreds or thousands. Second, as discussed in the introduction, there are certain cases where it is not advisable to exactly mirror the estimation of $\eta_n$ when constructing $\eta_n^*$. For instance, many machine learning algorithms involve cross-validation. However, if the bootstrap sampling process produces replicates in the bootstrap data, which is the case for the empirical bootstrap, cross-validation may not perform as expected (see, e.g., \citealp[Page~51]{silverman1986density} and \citealp[Chapter~28]{van2018targeted}). In this case, modifications of the estimation procedure of $\eta_n^*$ have been proposed to avoid these issues. In addition, some authors have proposed fixing tuning parameters such as bandwidths at the values selected by the original data when constructing $\eta_n^*$ using the bootstrap data \citep{hall2001bootstrapping}.

Our asymptotic results presented in Section~\ref{sec:theory} will require high-level conditions about $\eta_n^*$: consistency, a Donsker condition, and negligibility of a remainder term. This mirrors the high-level conditions required of $\eta_n$ for asymptotic linearity of $\psi_n$. Our conditions can be satisfied if the construction of $\eta_n^*$  mirrors that of $\eta_n$ completely or partially. In some cases, our conditions will also permit $\eta_n^* = \eta_n$. That is, we will permit that the nuisance is not re-estimated at all using the bootstrap data, but rather the estimator computed on the original data is used when constructing the bootstrap estimator. This is discussed more in Section~\ref{sec: T examples}.

\subsection{Additional bootstrap notation} \label{section: type of bootstrap}

Crucial to bootstrap theory is taking conditional expectations of the bootstrap data given the original data. To do so precisely, we make the following definitions, which are common in the bootstrap literature (see, e.g., \citealp{van1996weak, van2000asymptotic, kosorok2008introduction}). We suppose that $W_n = (W_{n1}, \dotsc, W_{nn})$ is an IID sample independent of $X_1, \dotsc, X_n$, where each $W_{ni}$ is a random vector with distribution $Q_n$ on a measurable space $(\s{W}_n, \s{C}_n)$. We then assume that for each $i \in \{1, \dotsc, n\}$, $X_i^* = \gamma_n(X_1, \dotsc, X_n, W_{ni})$, where $\gamma_n$ is a fixed measurable function. Hence, $W_{ni}$ represents the additional source of randomness used in generating the bootstrap observation $X_i^*$. Since $W_{n1}, \dotsc, W_{nn}$ are assumed to be IID and independent of $X_1, \dotsc, X_n$, $X_1^*, \dotsc, X_n^*$ are IID conditional on $X_1, \dotsc, X_n$. With this setup, $\hat{P}_n$ is defined as the conditional distribution of $X_i^*$ given $X_1, \dotsc, X_n$. The bootstrap sample $X_1^*, \dotsc, X_n^*$ lies in the product probability space $(\s{X}, \s{B}, P_0)^n \times (\s{W}_n, \s{C}_n, Q_n)^n$.

We provide two concrete examples to illustrate the above definitions.  For the empirical bootstrap, we can let $W_{ni}$ have a categorical distribution $Q_n$ on $\{1, \dotsc, n\}$ with event probabilities $(1/n, \dotsc, 1/n)$, and set $X_i^* = \gamma_n(X_1, \dotsc, X_n, W_{ni}) = X_{W_{ni}}$. For one-dimensional observations $X_i$ and any bootstrap, we can let $W_{ni}$ be IID Uniform$(0, 1)$, and set $X_i^* := F_n^{-1}(W_{ni})$ for $F_n^{-1}$ the quantile function corresponding to $\hat{P}_n$. Here, $F_n^{-1}$ is assumed to be a measurable function of $X_1, \dotsc, X_n$. 

We now define conditional expectations given the data $X_1, \dotsc, X_n$. Since we will be dealing with processes that may not be conditionally measurable, we will use outer expectations. For any real-valued function $h: \s{X}^n \times \s{W}_n^n \to \bar{\d{R}}$,  we define the \textit{conditional outer expectation} of $h$ given $X_1, \dotsc, X_n$ with respect to $Q_n$ as 
\[
    (\E_{W}^* h)(X_1, \dotsc, X_n) := \inf_U\left\{ \int_{\s{W}_n^n} U(X_1, \dotsc, X_n, w_1, \dotsc, w_n) \sd Q_n^n(w_1, \dotsc, w_n)  \right\},
\]
where the infimum is taken over all  functions $U$ such that $w_1, \dotsc, w_n \mapsto U(X_1, \dotsc, X_n, w_1, \dotsc, w_n)$ is measurable,  $U\geq h$ almost surely, and $\int U\sd Q_n^n$ exists. Since each $X_i^*$ is a measurable function of $X_1, \dotsc, X_n$ and $W_{ni}$, this can also be used to define conditional expectations of functions of $X_1, \dotsc, X_n$ and $X_1^*, \dotsc, X_n^*$. 
We then define the \textit{conditional outer probability} given $X_1, \dotsc, X_n$ as $\prob_W^*(A) := \E_W^* 1_A$ for any $A \in \s{B}^n\times \s{C}_n^n$. 

We say that a (possibly non-measurable) sequence of random elements $G_n: \s{X}^n\times \s{W}_n^n \to \d{D}$ for a metric space  $\d{D}$ conditionally weakly converges to a tight, Borel measurable limit $G$ in $\d{D}$ given $X_1, \dotsc, X_n$ in outer probability if $\sup_{h \in \BL_1(\d{D})} \abs{\E_W^* h(G_n) - \E_0 h(G)} \inoutprob 0$, and we denote this as $G_n \condinoutdist G$. In addition, we say $\s{F}$ is \textit{$P_0$-Donsker in probability} if it holds that $\sup_{h \in \BL_1(\ell^\infty(\s{F}))} \abs{ \E_W^* h(\d{G}_n^*) - \E_0 h(\d{G}) } = o_{\prob_0^*}(1)$ \cite[Section~3]{gine1990bootstrapping}. We say a sequence of random variables $Y_n: \s{X}^n \times \s{W}_n^n \to \d{R}$ conditionally converge to $0$ in probability if for any $\varepsilon > 0$, $\prob_W^*(|Y_n| \geq \varepsilon)  \inoutprob 0$, and we denote this as $Y_n = o_{\prob_W^*}(1)$. We say $Y_n$ is conditionally stochastically bounded if for any $\varepsilon > 0$, there exist $M \in (0, \infty)$ such that $\prob_0^*(\prob_W^*(|Y_n| \geq M) \geq \epsilon) \to 0$, and we denote this as $Y_n = O_{\prob_W^*}(1)$.

\section{General results}\label{sec:theory}

\subsection{Asymptotic linearity of the estimator}\label{sec:asym linearity}

Before moving to our general bootstrap consistency results, we provide general conditions under which $T(\eta_n, \d{P}_n)$ is an asymptotically linear estimator of $\psi_0$ with influence function $\phi_0$. While the result is simple and based on well-known ideas, it is important background because our theoretical study of the bootstrap will focus on consistency of the sampling distribution of the bootstrap estimator, which requires knowing the distribution it should be consistent for --- i.e., the asymptotic distribution of the original estimator. In addition, the conditions for conditional asymptotic linearity of the bootstrap estimator will mirror the conditions for asymptotic linearity of $T(\eta_n, \d{P}_n)$. Hence, these conditions will shed more light on the bootstrap conditions presented below. 

For each $n$, we define $\phi_n : \s{X} \to \d{R}$ as an estimator of the influence function $\phi_0$. For the one-step and estimating equations-based estimators, $\phi_n$ will be the influence function with the estimated nuisance parameter $\eta_n$.  For estimators such as the plug-in estimator where the influence function is not explicitly estimated as part of the estimation procedure, the role of $\phi_n$ will be discussed more below. We also define $R_n := T(\eta_n, \d{P}_n) - T(\eta_0, P_0) - (\d{P}_n - P_0) \phi_n$.  We then introduce the following conditions.

\begin{description}[style=multiline, labelindent=.9cm, leftmargin=2cm]
    \item[\namedlabel{cond: limited complexity}{(A1)}] There exists a class $\s{F}$ of measurable functions from $(\s{X}, \s{B})$  to $\d{R}$ such that:
    \begin{enumerate}[(a)]
        \item $\phi_0 \in \s{F}$ and $\prob_0^*(\phi_n \in \s{F}) \to 1$, and 
        \item $\d{G}_n \leadsto \d{G}_0$ in $\ell^\infty(\s{F})$, where $\d{G}_0$ is the $P_0$-Brownian bridge process.
    \end{enumerate}

    \item[\namedlabel{cond: weak consistency}{(A2)}] It holds that 
$\| \phi_n - \phi_0 \|_{L_2(P_0)} = o_{\prob_0^*}(1)$.

    \item[\namedlabel{cond: second order}{(A3)}] The remainder satisfies $R_n = o_{\prob_0^*}(n^{-1/2})$.
\end{description}

Under these conditions, we have the following result regarding asymptotic linearity of $T(\eta_n, \d{P}_n)$.
\begin{restatable}{thm}{thmclassical}\label{thm: classical}
    If conditions~\ref{cond: limited complexity}--\ref{cond: second order} hold, then $T(\eta_n, \d{P}_n)$ is asymptotically linear in the sense that $T(\eta_n, \d{P}_n) = T(\eta_0, P_0) + \d{P}_n \phi_0 + o_{\prob_0^*}(n^{-1/2})$, which implies that $n^{1/2} \left[T(\eta_n, \d{P}_n) - T(\eta_0, P_0)\right] \leadsto \d{G}_0 \phi_0.$
\end{restatable}
Theorem~\ref{thm: classical} provides conditions under which $\psi_n = T(\eta_n, \d{P}_n)$ is an asymptotically linear estimator of $\psi_0 = T(\eta_0, P_0)$. This implies convergence in distribution of $n^{1/2} \left[T(\eta_n, \d{P}_n) - T(\eta_0, P_0)\right]$ to $ \d{G}_0 \phi_0$, which follows a normal distribution with mean zero and variance $P_0 \phi_0^2$. In particular, this gives a method for constructing asymptotically valid confidence intervals. If $\sigma_n$ is a consistent estimator of $(P_0 \phi_0^2)^{1/2}$, then the two-sided Wald confidence interval 
\begin{equation}\label{eq: classical ci}
    [T(\eta_n, \d{P}_n) + z_{\beta}\sigma_n / n^{1/2}, T(\eta_n, \d{P}_n) + z_{1-\alpha}\sigma_n / n^{1/2}]
\end{equation}
has asymptotic level $1 - \alpha - \beta$ by Slutsky's lemma.  The simplest variance estimator is the \emph{influence function-based variance} given by $\sigma_n^2:= \mathbb{P}_n \phi_n^{2}$. The next proposition provides general conditions under which this variance estimator is consistent and demonstrates that  conditions~\ref{cond: limited complexity}--\ref{cond: weak consistency} in particular imply that it is consistent.
\begin{restatable}{prop}{propvar} \label{prop: variance}
    If (i) $\prob_0^*(\phi_n^2 \in \s{G}) \to 1$ for a $P_0$-Glivenko Cantelli class of measurable functions $\s{G}$, and (ii) $P_0(\phi_n^2 - \phi_0^2)\inoutprob 0$, then $\mathbb{P}_n \phi_n^{2}  \inoutprob \sigma_0^2$. Furthermore, if $\sup_{f \in \s{F}}|P_0 f| < \infty$, condition~\ref{cond: limited complexity} implies (i) and condition~\ref{cond: weak consistency} implies (ii).
\end{restatable}

Theorem~\ref{thm: classical} and Proposition~\ref{prop: variance} yield the following result for consistency of the Wald-type confidence interval with influence function-based variance estimator.
\begin{restatable}{cor}{corclassicalci}\label{cor: classical CI}
    If $\sup_{f \in \s{F}}|P_0 f| < \infty$, conditions~\ref{cond: limited complexity}--\ref{cond: second order} hold, and $\sigma_0^2 > 0$, then the Wald-type confidence interval defined in~\eqref{eq: classical ci} with $\sigma_n^2 = \mathbb{P}_n \phi_n^{2}$ has asymptotic confidence level $1-\alpha - \beta$.
\end{restatable}

We now discuss conditions~\ref{cond: limited complexity}--\ref{cond: second order}. Theorem~\ref{thm: classical} is based on the first-order expansion
\begin{align*}
    T(\eta_n, \d{P}_n) = T(\eta_0, P_0) + \d{P}_n \phi_0 & +  R_n + S_n,
\end{align*}
where $S_n := n^{-1/2} \d{G}_n(\phi_n - \phi_0)$ is an \emph{empirical process term}. Conditions~\ref{cond: limited complexity} and~\ref{cond: weak consistency} together imply that $S_n = o_{\prob_0^*}(n^{-1/2})$. Condition~\ref{cond: limited complexity} requires that the estimated influence function $\phi_n$ falls in a $P_0$-Donsker class with probability tending to one. Satisfying this condition typically requires restricting the complexity of the function class that the nuisance estimator $\eta_n$, and hence $\phi_n$, falls in. A main way this is accomplished is by using bracketing entropy or uniform entropy (Chapters~2.6 and~2.7 of \citealp{van1996weak}). Condition~\ref{cond: weak consistency} requires that $\phi_n$ is a consistent estimator of $\phi_0$ in the $L_2(P_0)$ norm. When $\phi_n$ and $\phi_0$ depend on $n$ and $P_0$ through $\eta_n$ and $\eta_0$, respectively, this is typically implied by consistency of $\eta_n$ for $\eta_0$ in an appropriate sense coupled with continuity of $\phi$ as a function of $\eta$. 

Condition~\ref{cond: second order} controls the remainder term $R_n$ in the above expansion. Other authors have used analogous conditions in related work (see, e.g., the smoothness property discussed in Section~4.1 of \citealp{shen1997methods} and the quadratic functional in Section~3.2 of \citealp{cattaneo2018kernel}). 
This remainder term can typically be decomposed into further remainders, including the so-called \emph{second-order remainder} and the bias term $-\d{P}_n \phi_n$. The exact way that this remainder decomposes depends on the form of $T$. Estimators such as one-step and estimating equations-based estimators that explicitly use the influence function as part of the construction control the asymptotic bias as part of the estimation procedure. On the other hand, plug-in estimators typically control the bias term $-\d{P}_n \phi_n$ through construction of the nuisance estimator $\eta_n$. Several approaches to constructing nuisance estimators that yield $-\d{P}_n \phi_n = \fasterthan(n^{-1/2})$ are sieve estimators \citep{shen1997methods}, under-smoothing \citep{newey1998undersmoothing},  twicing kernels \citep{newey2004twicing}, and TMLE. Simpler sufficient conditions for~\ref{cond: second order} for several different estimator constructions will be discussed at more length in Section~\ref{sec: T examples}.

\subsection{Conditional asymptotic linearity of the bootstrap estimator}\label{sec:boot asym linearity}

We now present our general result regarding conditional asymptotic linearity of the bootstrap procedure defined in Section~\ref{sec:bootstrap_framework}. As mentioned above, $\phi_n^*$ is a bootstrap influence function estimator. We begin by introducing conditions we will rely upon. We define $R_n^* := T(\eta_n^*, \d{P}_n^*) - T(\eta_n, \hat{P}_n) - (\d{P}_n^* - \hat{P}_n) \phi_n^*$.

\begin{description}[style=multiline, labelindent=.9cm, leftmargin=2cm]
    \item[\namedlabel{cond: bootstrap limited complexity}{(B1)}] There exists a class $\s{F}$ of measurable  functions from $(\s{X}, \s{B})$ to $\d{R}$ such that
    \begin{enumerate}[(a)]
        \item $\phi_0 \in \s{F}$ and $\prob_W^*(\phi_n^* \in \s{F}) \inoutprob 1$, and 
        \item $\d{G}_n^* \condinoutdist \d{G}_0$ in $\ell^\infty(\s{F})$, where $\d{G}_0$ is the $P_0$-Brownian bridge process.
    \end{enumerate}

    \item[\namedlabel{cond: bootstrap weak consistency}{(B2)}] It holds that $ \| \phi_n^* - \phi_0 \|_{L_2(P_0)} = o_{\prob_W^*}(1)$.

    \item[\namedlabel{cond: bootstrap second order}{(B3)}] The remainder term satisfies $R_n^* = o_{\prob_W^*}(n^{-1/2})$.
\end{description}

Under these conditions, we have the following result.
\begin{restatable}{thm}{thmbootstrap}\label{thm: bootstrap}
    If conditions~\ref{cond: bootstrap limited complexity}--\ref{cond: bootstrap second order} hold, then $T(\eta_n^*, \d{P}_n^*)$ is conditionally asymptotically linear in the sense that $T(\eta_n^*, \d{P}_n^*) = T(\eta_n, \hat{P}_n) + (\d{P}_n^* - \hat{P}_n)\phi_0 + o_{\prob_W^*}(n^{-1/2})$, which implies that
    \[
        n^{1/2} \left[T\left(\eta_n^*, \d{P}_n^*\right)- T(\eta_n, \hat{P}_n) \right] \condinoutdist \d{G}_0(\phi_0).
    \]
\end{restatable}

Theorem~\ref{thm: bootstrap} establishes general conditions under which the bootstrap estimator $T(\eta_n^*, \d{P}_n^*)$ is conditionally asymptotically linear. Theorem~\ref{thm: bootstrap} is notable for its generality: conditions~\ref{cond: bootstrap limited complexity}--\ref{cond: bootstrap second order} cover a variety of estimator constructions $T$, bootstrap nuisance estimators $\eta_n^*$, and bootstrap sampling distributions $\hat{P}_n$. We expect this generality to increase the range of potential applications of the result. 

The most important implication of conditional asymptotic linearity is that the bootstrap provides a consistent approximation to the sampling distribution of $n^{1/2}(\psi_n - \psi_0)$.  Here, consistency means that the conditional distribution of the centered and scaled bootstrap estimator $n^{1/2} [T\left(\eta_n^*, \d{P}_n^*\right)- T(\eta_n, \hat{P}_n)]$ converges weakly in outer probability to the same limit as $n^{1/2}(\psi_n - \psi_0)$. In Section~\ref{sec:conf_int}, we will discuss how this can be used to demonstrate asymptotic validity of bootstrap confidence intervals. However, as with asymptotic linearity, conditional asymptotic linearity offers additional utility beyond conditional weak convergence. In particular, conditional asymptotic linearity of multiple bootstrap estimators implies \emph{joint} conditional asymptotic normality of the estimators, which is useful for constructing simultaneous confidence regions and confidence regions for functions of two or more parameters. In contrast, such joint asymptotic behavior cannot be determined by marginal weak convergence results.

Theorem~\ref{thm: bootstrap} centers the bootstrap estimator $T(\eta_n^*, \d{P}_n^*)$ around $T(\eta_n, \hat{P}_n)$ rather than $T\left(\eta_n, \d{P}_n\right) = \psi_n$. There is no difference between these two possibilities when using the empirical bootstrap, so that $\hat{P}_n = \d{P}_n$. However, when $\hat{P}_n \neq \d{P}_n$, such as when utilizing a smooth bootstrap procedure, the two are not necessarily the same. Intuitively, we center around $T(\eta_n, \hat{P}_n)$ because $\hat{P}_n$ is the distribution used to generate the bootstrap data upon which the bootstrap empirical distribution $\d{P}_n^*$ is based. Hence, $\hat{P}_n$ plays the role of $P_0$ in the bootstrap. The potential consequences of this for the construction of confidence intervals will be further explored in Section~\ref{sec:conf_int}.

Conditions~\ref{cond: bootstrap limited complexity}--\ref{cond: bootstrap second order} mirror conditions~\ref{cond: limited complexity}--\ref{cond: second order} used to demonstrate asymptotic linearity of the estimator $T(\eta_n,\d{P}_n)$ in Theorem~\ref{thm: classical}. However, since \ref{cond: bootstrap limited complexity}--\ref{cond: bootstrap second order} concern the bootstrap estimator, they require convergence conditional on the original observations. In some cases, conditional asymptotic linearity of the bootstrap estimator is actually implied by the conditions of Theorem~\ref{thm: classical}. We will discuss this more in Section~\ref{sec: T examples}. This is related to \cite{beran1997diagnosing}, who showed that for locally asymptotically normal parametric models, conditional weak convergence of the parametric and nonparametric bootstraps is equivalent to regularity of the estimator.  For pathwise differentiable parameters, an asymptotically linear estimator is regular if and only if its influence function is a gradient of $\psi$ relative to $\s{M}$ at $P_0$ \citep{pfanzagl1982contributions, pfanzagl1990estimation, van1991differentiable}. We typically construct the estimator so that $\phi_0$ is indeed a gradient. Furthermore, negligibility of the second-order remainder term for condition~\ref{cond: second order} is often established using conditions that also imply that $\phi_0$ is a gradient, as we will discuss in Section~\ref{sec: T examples}. Hence, while we do not explicitly require regularity, the estimator is regular in most of our intended applications. However, it is not entirely clear if regularity of the estimator plays as strong a role in our setting as it does in that of \cite{beran1997diagnosing} for parametric models.

Theorem \ref{thm: bootstrap} is based on the bootstrap first-order expansion 
\begin{align*}
    T(\eta_n^*, \d{P}_n^*) = T(\eta_n, \hat{P}_n) + (\d{P}_n^* - \hat{P}_n)\phi_0 & + R_n^* + S_n^*,
\end{align*}
where $S_n^* = n^{-1/2} \d{G}_n^*(\phi_n^* - \phi_0)$ is a \emph{bootstrap empirical process remainder term}. Conditions~\ref{cond: bootstrap limited complexity}--\ref{cond: bootstrap weak consistency} are used to control $S_n^*$. This is analogous to how  conditions~\ref{cond: limited complexity}--\ref{cond: weak consistency} were used to control the ordinary empirical process remainder $S_n$. In particular, condition~\ref{cond: bootstrap limited complexity}(b) requires conditional weak convergence in outer probability of the bootstrap empirical process $\d{G}_n^* = n^{1/2}(\d{P}_n^* - \hat{P}_n)$ in the space $\ell^\infty(\s{F})$ to a $P_0$-Brownian bridge process.  For the empirical bootstrap where $\hat{P}_n = \d{P}_n$, this holds as long as $\s{F}$ is $P_0$-Donsker \citep{gine1990bootstrapping}. However, for other types of bootstraps, this condition is more difficult to verify. This will be discussed in depth and further sufficient conditions for smooth bootstrap sampling distributions will be provided in Section~\ref{sec:bootstrap_dist_condition}.

Condition~\ref{cond: bootstrap weak consistency} requires conditional weak consistency of the bootstrap nuisance estimator $\eta_n^*$ to the true nuisance $\eta_0$. If $\eta_n^*$ is constructed in an exactly analogous manner using the bootstrap data as $\eta_n$ is constructed using the original data, the bootstrap data has replicated observations, and the method of constructing $\eta_n$ is sensitive to ties in the data, \ref{cond: bootstrap weak consistency} may not be satisfied. As discussed in Section~\ref{sec:bootstrap_framework}, for this reason and others we do not require that $\eta_n^*$ be constructed in an exactly analogous manner to $\eta_n$, so these issues can be avoided. In particular, the simplest approach for constructing $\eta_n^*$ is to define $\eta_n^* = \eta_n$. This approach is appealing in its computational simplicity because it does not require re-fitting the nuisance estimator using the bootstrap data, which can be computationally intensive when machine learning estimators are used to construct $\eta_n$. Furthermore, when $\eta_n^* = \eta_n$ and $\phi_P$ only depends on $P$ through $\eta_P$,  condition~\ref{cond: bootstrap weak consistency} reduces to the requirement that $\|\phi_n - \phi_0\|_{L_2(P_0)} \inoutprob 0$, which was already required for asymptotic linearity of $\psi_n$ in Theorem~\ref{thm: classical}. Intuitively, the precise behavior of the nuisance estimator $\eta_n$ does not play a role in the first-order asymptotic behavior of $\psi_n$ as long as the high-level conditions~\ref{cond: limited complexity}--\ref{cond: second order} hold, and similarly the precise behavior of $\eta_n^*$ does not play a role in the first-order asymptotic behavior of the bootstrap estimator as long as the high-level conditions~\ref{cond: bootstrap limited complexity}--\ref{cond: bootstrap second order} hold. However, setting $\eta_n^* = \eta_n$ may yield worse finite-sample coverage, and does not yield valid bootstrap confidence intervals when conditions~\ref{cond: second order} and/or~\ref{cond: bootstrap second order} do not hold, as we will discuss in Section~\ref{sec:conf_int}.

Our proof technique for Theorem~\ref{thm: bootstrap} could be adapted to other tight and Borel-measurable limit processes in condition~\ref{cond: bootstrap limited complexity}(b). This is not relevant to demonstrating consistency of the bootstrap when the estimator is asymptotically Gaussian, as is the case here, but it may be of interest in other situations where the limit distribution is not Gaussian. However, if $\d{G}_0$ were a non-Gaussian process in condition~\ref{cond: bootstrap limited complexity}(b) and the sample paths of $\d{G}_0$ were not almost surely uniformly continuous in the $L_2(P_0)$ metric, it would be necessary to replace condition~\ref{cond: bootstrap weak consistency} with the requirement that $\rho(\phi_n^*, \phi_0) = o_{\prob_W^*}(1)$ for a semimetric $\rho$ on $\ell^\infty(\s{F})$ such that the sample paths of $\d{G}_0$ are almost surely uniformly $\rho$-continuous. 

Condition~\ref{cond: bootstrap second order} requires that the bootstrap analogue of the remainder term in~\ref{cond: second order} be sufficiently negligible. As discussed after Theorem~\ref{thm: classical}, this remainder term is again a combination of two remainders: the bootstrap bias term $-\d{P}_n^*\phi_n^*$ and the bootstrap second-order remainder term. Further sufficient conditions for~\ref{cond: bootstrap second order} for specific estimators $T$ will be provided in Section~\ref{sec: T examples}. 

Theorem~\ref{thm: bootstrap} also relies on the following bootstrap version of Lemma~19.24 of \cite{van2000asymptotic}, which is useful in its own right.

\begin{restatable}{lemma}{lemmabootran}\label{lemma: boot_ran}
    Suppose that $\s{F}$ is a class of measurable functions such that $\d{G}_n^* \condinoutdist \d{G}_0$ in $\ell^\infty(\s{F})$. Let $\phi_n^*$ be a sequence of random functions possibly depending on the bootstrap sample such that $\prob^*_W \left( \phi_n^*  \in \s{F} \right) \inoutprob 1$. If $\rho(\phi_n^*, \phi_\infty) = o_{\prob_W^*}(1)$ for some $\phi_\infty \in \s{F}$ and a semimetric $\rho$ on $\s{F}$ for which the sample paths of $\d{G}_0$ are almost surely uniformly $\rho$-continuous, then $\d{G}_n^*(\phi_n^* - \phi_\infty) = o_{\prob_W^*}(1)$.
\end{restatable}

\subsection{Conditional weak convergence of the bootstrap empirical process}\label{sec:bootstrap_dist_condition}

We now provide further sufficient conditions for conditional weak convergence of the bootstrap empirical process $\d{G}_n^*$ required by condition~\ref{cond: bootstrap limited complexity}(b). We first discuss the case of the empirical bootstrap, where $\hat{P}_n = \d{P}_n$. The properties of the empirical bootstrap have been extensively studied by \cite{efron1982jackknife}, \cite{gine1990bootstrapping,gine1991gaussian}, \cite{praestgaard1993exchangeably} and \cite{van1996weak}, among others. In particular, Theorem~3.1 of \cite{gine1990bootstrapping} and Theorem~3.6.1 of \cite{van1996weak} provided the following necessary and sufficient condition for~\ref{cond: bootstrap limited complexity}(b) in the case of the empirical bootstrap. 
\begin{lemma}[Theorem~3.6.1 of \citealp{van1996weak}]\label{lemma: empirical bootstrap}
   If $\hat{P}_n = \d{P}_n$ and $\s{F}$ is a class of measurable function with finite envelope function, then $\s{F}$ is  $P_0$-Donsker if and only if condition~\ref{cond: bootstrap limited complexity}(b) holds and $\d{G}_n^*$ is asymptotically measurable.
\end{lemma}

We now turn to the case where $\hat{P}_n$ is not the empirical distribution. We first provide a general set of sufficient conditions for~\ref{cond: bootstrap limited complexity}(b) based on the notion of \textit{uniform Donsker classes}. 
For a probability measure $P$ on $(\s{X}, \s{B})$, we denote $\d{G}_{n, P} := n^{1/2}(\d{P}_n - P)$ as the empirical process centered at $P$ and $\d{G}_P$ as the $P$-Brownian bridge process. Following \cite{gine1991gaussian} and \cite{anne1992uniform}, for a set $\s{P}$ of probability measures on $(\s{X}, \s{B})$, we then say $\s{F}$ is \textit{Donsker uniformly in $P \in \s{P}$} if
\[    
    \sup_{P \in \s{P}} \sup_{h \in \BL_1(\ell^\infty(\s{F}))} \abs{ \E^*h(\d{G}_{n, P}) - \E h(\d{G}_P) } \to 0,
\]
and $\d{G}_P$ satisfies $\sup_{P \in \s{M}}\E\norm{\d{G}_P}_{\s{F}} < \infty$ and 
\[
    \lim_{\delta \to 0}\sup_{P \in \s{P}} \E\sup_{\rho_P(f, g) < \delta}\abs{\d{G}_P(f) - \d{G}_P(g)} = 0,
\]
where $\rho_P : (f,g) \in \s{F} \times \s{F} \to [ P(f - g)^2 ]^{1/2}$ is the $P$-standard deviation semi-metric on $\s{F} \times \s{F}$. Theorem~4.5 of \cite{anne1992uniform} provides general sufficient conditions for an almost sure convergence version of~\ref{cond: bootstrap limited complexity}(b). Below, we restate their result relaxed to convergence in outer probability.

\begin{restatable}[Theorem~4.5 of \citealp{anne1992uniform}]{lemma}{lemmasmooth}\label{prop: uniform donsker}
    If $\s{F}$ is a class of measurable functions with envelope function $F$ such that:  (i) $\s{F}$ is square integrable uniformly in $P \in \s{P}$ in the sense that $\lim_{M \to \infty} \sup_{P \in \s{P}} P F^2 1\{F > M \} = 0$; (ii) $\s{F}$ is Donkser uniformly in $P \in \s{P}$ where $\s{P}$ is such that $\prob_0^* ( \hat{P}_n\in \s{P}) \to 1$; and (iii) the semi-metric $L_2(\hat{P}_n)$ converges uniformly to $L_2(P_0)$ in the sense that
    \begin{equation}\label{eq: convergence of semimetric}
        \sup_{f,g \in \s{F}} \abs{ \norm{f-g}_{L_2(\hat{P}_n)} - \norm{f-g}_{L_2(P_0)}} \inoutprob 0,
    \end{equation}
   then $\d{G}_n^* \condinoutdist \d{G}_0$ in $\ell^\infty(\s{F})$.
\end{restatable}

\Cref{prop: uniform donsker} is a bootstrap version of Lemma~2.8.7 of \cite{van1996weak}. Theorems~2.8.9 and~2.8.10 of \cite{van1996weak} provide sufficient conditions for a class $\s{F}$ to be uniform Donkser using uniform entropy and bracketing entropy conditions, respectively. We also note that if $F$ is constant, then uniform square integrability holds, and, by Theorem~2.8.3 of \cite{van1996weak}, if some measurability conditions are satisfied, then $\s{F}$ is Donsker uniformly in $P \in \s{P}$ provided the uniform entropy integral is finite:
\begin{equation}
    \int_0^\infty \sup_{Q} \left\{ \log N(\varepsilon \norm{F}_{L_2(Q)}, \s{F}, L_2(Q)) \right\}^{1/2} \sd \varepsilon < \infty, \label{eq: uniform entropy integral}
\end{equation}
where the supremum is taken over all finitely discrete probability measures $Q$ on $\s{X}$ with $\int F^2 \sd Q > 0$.  Finally, we note that  
\[\sup_{f,g \in \s{F}} \abs{ \norm{f-g}_{L_2(\hat{P}_n)} - \norm{f-g}_{L_2(P_0)}} \leq  \sup_{f,g \in \s{F}} \left\{ \abs{(\hat{P}_n- P_0)(f-g)^2}\right\}^{1/2} \leq  2 \sup_{f,g \in \s{F}}\left\{ \abs{(\hat{P}_n - P_0)(fg) }\right\}^{1/2},\]
so $\sup_{f,g \in \s{F}} \abs{ (\hat{P}_n - P_0)(fg) } \inoutprob 0$ implies~\eqref{eq: convergence of semimetric}. Hence, as noted in \cite{van2018targeted}, in many cases it is not necessary for the bootstrap sampling distribution $\hat{P}_n$ to be globally consistent for $P_0$; rather, it is sufficient that means of products or squared differences of influence functions under $\hat{P}_n$ be consistent for means of the same under $P_0$ uniformly over the class $\s{F}$ induced by the nuisance estimators $\eta_n$ and $\eta_n^*$.

We now use \Cref{prop: uniform donsker} to show that bootstrap distributions obtained via \emph{smoothing through convolution} satisfy \ref{cond: bootstrap limited complexity}(b). Specifically, for a sequence of probability measures $L_n$, which we will require converges weakly to 0,  we say $\hat{P}_n$ is obtained by convolution of $\d{P}_n$ with  $L_n$, and we write $\hat{P}_n = \d{P}_n * L_n$, if for any $B \in \s{B}$, 
\begin{equation}
    \hat{P}_n(B) := \int_{\s{X}} \int_{\s{X}}  1_B(x+y) \sd \d{P}_n(x) \sd  L_n(y) = \frac{1}{n} \sum_{i=1}^n \int_{\s{X}} 1_B(X_i+y) \sd  L_n(y) = \frac{1}{n}\sum_{i=1}^n L_n(B-X_i).
    \label{eq: convolution}
\end{equation}
The most well-known example of smoothing through convolution is the kernel density estimator. Let $K: \s{X} \mapsto \d{R}$ be a fixed kernel function and $h_n >0$  a possibly random sequence of bandwidths. If $L_n(B) := \int_{B}  h_n^{-d} K\left(h_n^{-1}x \right) \sd x$ for any $B \in \s{B}$, then $\hat{P}_n := \d{P}_n * L_n$ defines a kernel density estimator with kernel $K$ and bandwidth $h_n$. 

Properties of smoothing through convolution estimators, including weak convergence of $n^{1/2}(\hat{P}_n - P_0)$, were studied in \cite{yukich1992weak}, \cite{van1994weak}, \cite{rost2000limit}, \cite{radulovic2000weak, gaenssler2003smoothed}, and \cite{eric20223donsker}, among others.  In their Theorem~2.1, \cite{gaenssler2003smoothed} demonstrated that the bootstrap empirical process $n^{1/2}(\d{P}_n^* - \hat{P}_n)$ converges weakly to $\d{G}_0$ in $\ell^\infty(\s{F})$ if $\s{F}$ is equicontinuous and other conditions hold. However, equicontinuity is a strong assumption, and may not hold in some of our applications of interest. For instance, in some cases, the influence function involves indicator functions, which are not continuous. Equicontinuity is used in their result to show that $\sup_{f \in \s{F}} \abs{(P_0 * L_n)f^2 - P_0 f^2} \to 0$. Weak convergence of $L_n$ does not generally imply this result, as shown in Example~2.3 of \cite{gaenssler2000uniform}. However, if $P_0$ is absolutely continuous with respect to Lebesgue measure $\lambda$ and the corresponding density function is Lipschitz continuous, then weak convergence of $L_n$ does imply that $\sup_{f \in \s{F}} \abs{(P_0 * L_n)f^2 - P_0 f^2} \to 0$. While the condition that $P_0$ is dominated by Lebesgue measure is strong, it is typically assumed when using kernel density estimation. Proposition~\ref{prop: weak convergence of convolution smooth bootstrap} below formalizes this to provide general conditions under which smoothing through convolution estimators $\hat{P}_n$ satisfy condition~\ref{cond: bootstrap limited complexity}(b) without assuming equicontinuity of $\s{F}$.

\begin{restatable}{prop}{propconvolution}\label{prop: weak convergence of convolution smooth bootstrap}
    Suppose $\s{F}$ is a class of Borel measurable functions with uniformly bounded envelope function $F$ and finite uniform entropy integral as in~\eqref{eq: uniform entropy integral} such that $\s{F}_{\delta, P} = \{f - g: f, g \in \s{F}, \norm{f -g}_{L_2(P)} < \delta \}$ and $\s{F}_{\infty}^2 = \{ (f-g)^2: f, g \in \s{F}\}$ are $P$-measurable for every $\delta >0$ and $P\in \s{M}$. If $P_0$ is absolutely continuous with respect to Lebesgue measure $\lambda$ with uniformly bounded and Lipschitz continuous density function $p_0$, and $\hat{P}_n = \d{P}_n * L_n$ for a sequence of random measures $L_n$ converging weakly to Dirac measure at 0,  then the conditions of Proposition~\ref{prop: uniform donsker} hold, so that $\d{G}_n^* \condinoutdist \d{G}_0$ in $\ell^\infty(\s{F})$.
\end{restatable}

As discussed in \cite{anne1992uniform}, if $L_n(B) = \int_B h_n^{-d} K\left( h_n^{-1}x\right) \sd x$ is a kernel density estimator with bandwidth $h_n$ satisfying $h_n \to 0$, $nh_n^d/\log(n) \to \infty$ and $ \|\sd \hat{P}_n/ \sd \lambda - \sd P_0/ \sd \lambda\|_\infty \xrightarrow{\mathrm{a.s.^*}} 0$, then $L_n$ converges weakly to Dirac measure at 0.  Furthermore, $ \|\sd \hat{P}_n/ \sd \lambda - \sd P_0/ \sd \lambda\|_\infty \xrightarrow{\mathrm{a.s.^*}} 0$ follows if $\sd P_0/ \sd \lambda$ is uniformly continuous, $nh_n^d/\log(h_n^{-1}) \to \infty$, $|\log h_n|/\log\log(n) \to \infty$, and $h_n^d \leq c h_{2n}^d$ for some constant $c>0$ by Theorem~2.3 of \cite{gine2002rate}.

\subsection{Consistency of bootstrap confidence intervals}\label{sec:conf_int}

We now discuss general conditions for asymptotic validity of bootstrap confidence intervals. Conditional asymptotic linearity of the bootstrap is sufficient for asymptotic validity of many bootstrap confidence intervals. Hence, the conditions of Theorem~\ref{thm: bootstrap} in many cases imply that associated bootstrap confidence intervals are asymptotically valid. However, conditional asymptotic linearity or conditional weak convergence of the bootstrap estimator are not \emph{necessary} for asymptotic validity of bootstrap confidence intervals. In some cases, bootstrap confidence intervals are asymptotically valid even when Theorem~\ref{thm: bootstrap} fails. We illustrate this phenomenon in detail for several types of bootstrap confidence intervals.

\subsubsection{Percentile and percentile \texorpdfstring{$t$}{Lg}-methods}

We first consider the percentile and percentile $t$-methods. We note that we are using the terminology of \cite{van2000asymptotic}, but that in other literature, what we are calling the percentile method is called the ``basic" or ``reverse percentile" method. We suppose that $\sigma_n^{*2}$ is an estimator of $\sigma_0^2$ based on the bootstrap data. 
We then define $\xi_{n, p}^*$ as the $p$th quantile of the conditional distribution of $[T(\eta_n^*, \d{P}_n^*) - T(\eta_n, \hat{P}_n)]/\sigma_n^*$ given the data, i.e.,
\[
    \xi_{n, p}^* := \inf\left\{ \xi \in \d{R} : \prob_W^* \left( \frac{T(\eta_n^*, \d{P}_n^*) - T(\eta_n, \hat{P}_n)}{\sigma_n^*} \leq \xi \right) \geq p \right\}.
\]
We emphasize that $[T(\eta_n^*, \d{P}_n^*) - T(\eta_n, \hat{P}_n)]/\sigma_n^*$ is centered around $T(\eta_n, \hat{P}_n)$ rather than $T(\eta_n, \d{P}_n) = \psi_n$ for reasons discussed following~\Cref{thm: bootstrap}. A two-sided $(1-\alpha - \beta)$-level \textit{bootstrap percentile $t$-method} confidence interval is then given by
\begin{align}
    &\left\{ \psi: \xi_{n, \beta}^* \leq \frac{T(\eta_n, \d{P}_n) - \psi}{\sigma_n} \leq \xi_{n, 1-\alpha}^*\right\} = \left[T(\eta_n, \d{P}_n) - \xi_{n, 1-\alpha}^* \sigma_n, \, T(\eta_n, \d{P}_n) - \xi_{n, \beta}^* \sigma_n\right].\label{eq: percentile t method}
\end{align}
This interval is based on the $t$-statistic $n^{1/2}[T(\eta_n, \d{P}_n)- T(\eta_0, P_0) ] / \sigma_n$. Typically, $\alpha$ and $\beta$ are chosen to be equal, resulting in an equi-tailed confidence interval. Setting $\sigma_n = \sigma_n^* = 1$ yields the \textit{bootstrap percentile method} confidence interval
\begin{equation}\label{eq: percentile method}
   \left[T(\eta_n, \d{P}_n) - \xi_{n, 1-\alpha}^*, \, T(\eta_n, \d{P}_n) - \xi_{n, \beta}^*\right].
\end{equation}
The percentile $t$-method has been shown to be more accurate than the percentile method in many cases because the studentized statistic is asymptotically pivotal \citep{hall1992bootstrap}. The next result provides conditions under which the bootstrap percentile and percentile $t$-intervals are asymptotically valid.
\begin{restatable}{thm}{thmprecinterval}\label{thm: perc bootstrap CI}
    Suppose that $\hat{P}_n \phi_0^2 \inoutprob P_0 \phi_0^2$ and $(\hat{P}_n - P_0) [\phi_0^2 1\{|\phi_0| > M\}] \inoutprob 0$ for every $M>0$.
    If $S_n^* - S_n = o_{\prob_W^*}(n^{-1/2})$ and $R_n^* - R_n = o_{\prob_W^*}(n^{-1/2})$, then 
    \[\sup_{t \in \d{R}} \left| \prob_W^*\left(T(\eta_n^*, \d{P}_n^*) - T(\eta_n, \hat{P}_n) \leq t \right) - \prob_0^*\left(T(\eta_n, \d{P}_n) - T(\eta_0, P_0) \leq t \right)\right|\inoutprob 0, \]
    and the bootstrap percentile confidence interval defined in~\eqref{eq: percentile method} has asymptotic confidence level $1-\alpha - \beta$. If in addition $\sigma_n^{2} \inoutprob \sigma_0^{2}$, $\sigma_n^{*2} \condinoutprob \sigma_0^{2}$ and $(S_n+R_n)(\sigma_n^* - \sigma_n) = o_{\prob_W^*}(n^{-1/2})$, then 
    \[\sup_{t \in \d{R}} \left| \prob_W^*\left([T(\eta_n^*, \d{P}_n^*) - T(\eta_n, \hat{P}_n)] / \sigma_n^* \leq t \right) - \prob_0^*\left([T(\eta_n, \d{P}_n) - T(\eta_0, P_0)] / \sigma_n \leq t \right)\right|\inoutprob 0, \]
    and the percentile $t$-confidence interval defined in~\eqref{eq: percentile t method} has asymptotic confidence level $1-\alpha - \beta$.
\end{restatable}
Conditions~\ref{cond: limited complexity}--\ref{cond: second order} imply that $S_n$ and $R_n$ are $o_{\prob_0^*}(n^{-1/2})$, and conditions~\ref{cond: bootstrap limited complexity}--\ref{cond: bootstrap second order} imply that $o_{\prob_W^*}(n^{-1/2})$. This yields the following Corollary.
\begin{restatable}{cor}{corprecinterval}\label{cor: perc bootstrap CI}
    If conditions~\ref{cond: limited complexity}--\ref{cond: second order} and \ref{cond: bootstrap limited complexity}--\ref{cond: bootstrap second order} hold, then the bootstrap percentile confidence interval has asymptotic confidence level $1-\alpha - \beta$. If in addition $\sigma_n^{2} \inoutprob \sigma_0^{2}$ and $\sigma_n^{*2} \condinoutprob \sigma_0^{2}$, then the percentile $t$-confidence interval defined in~\eqref{eq: percentile t method} has asymptotic confidence level $1-\alpha - \beta$.
\end{restatable}

\Cref{cor: perc bootstrap CI} demonstrates that if the conditions of Theorems~\ref{thm: classical} and~\ref{thm: bootstrap} hold, so that the estimator is asymptotically linear and the bootstrap estimator is conditionally asymptotically linear, then bootstrap confidence intervals using the percentile and percentile $t$-methods are asymptotically valid. However, \Cref{thm: perc bootstrap CI} demonstrates that the percentile and percentile $t$-methods can yield valid confidence intervals even if~\ref{cond: limited complexity}--\ref{cond: second order} and~\ref{cond: bootstrap limited complexity}--\ref{cond: bootstrap second order} do not hold. Specifically, even if $S_n$ and $R_n$ are not $o_{\prob_0^*}(n^{-1/2})$, and $S_n^*$ and $R_n^*$ are not $o_{\prob_W^*}(n^{-1/2})$, as long as $S_n^*$ and $R_n^*$ are sufficiently good approximations of $S_n$ and $R_n$, respectively, the percentile methods can yield asymptotically valid confidence intervals. This phenomenon was studied in \cite{cattaneo2018kernel, cattaneo2022average} for kernel-based nuisance estimators. In particular, $R_n^* - R_n$ can be $o_{\prob_W^*}(n^{-1/2})$ under slower rates of convergence of nuisance estimators than those used to demonstrate that~\ref{cond: second order} and~\ref{cond: bootstrap second order} hold, or if the estimator is not targeted toward $\psi$. Similarly, $S_n^* - S_n$ can be $o_{\prob_W^*}(n^{-1/2})$ under weaker entropy conditions than those used to demonstrate that~\ref{cond: limited complexity} and~\ref{cond: bootstrap limited complexity} hold. Hence, in some cases bootstrap confidence intervals are asymptotically valid even when the estimator is not asymptotically linear because bootstrap intervals can automatically correct excess bias in the estimator. We will discuss this in more depth in Sections~\ref{sec: T examples} and~\ref{sec: applications}. 

\Cref{thm: perc bootstrap CI} also requires consistency of certain moments of the bootstrap sampling distribution, which is used to establish that $\d{G}_n^*\phi_0 \condinoutdist \d{G}_0 \phi_0$. For the empirical bootstrap, these conditions are implied by the law of large numbers. For smooth bootstraps, these conditions are satisfied if $\norm{\sd \hat{P}_n / \sd \lambda - \sd P_0 / \sd \lambda}_\infty \xrightarrow{\mathrm{a.s.^*}} 0$. 

\Cref{thm: perc bootstrap CI} and \Cref{cor: perc bootstrap CI} both require that $\sigma_n^*$ is a conditionally consistent estimator of $\sigma_0$ for validity of the percentile $t$-method. A bootstrap analogue of the influence function-based variance estimator defined in Section~\ref{sec:asym linearity} is $\sigma_n^{*2}:= \d{P}_n^* \phi_n^{*2}$. In the following lemma, we show that for the empirical bootstrap, conditions~\ref{cond: bootstrap limited complexity} and~\ref{cond: bootstrap weak consistency} imply that the bootstrap influence function-based variance estimator $\sigma_n^{*2}$ is conditionally consistent. The situation is not quite as straightforward for other types of bootstraps, but it is still the case that conditions~\ref{cond: bootstrap limited complexity}(a),~\ref{cond: bootstrap weak consistency}, and the sufficient conditions for~\ref{cond: bootstrap limited complexity}(b) established in Proposition~\ref{prop: uniform donsker} together imply conditional consistency of $\sigma_n^{*2}$.

\begin{restatable}{lemma}{lemmabootstrapvar}\label{lemma: bootstrap variance}
    If there exists a class of measurable functions $\s{G}$ such that (i) $\prob_W^*(\phi_n^{*2} \in \s{G}) \inoutprob 1$, (ii) $\sup_{g \in \s{G}} |(\hat{P}_n - P_0) g| = o_{\prob_0^*}(1)$, (iii) $\sup_{g \in \s{G}}|(\d{P}_n^* - \hat{P}_n)g| = o_{\prob_W^*}(1)$, and (iv) $P_0(\phi_n^{*2} - \phi_0^2) = o_{\prob_W^*}(1)$, then $\d{P}_n^* \phi_n^{*2} - \sigma_0^2 = o_{\prob_W^*}(1)$. Furthermore, condition~\ref{cond: bootstrap weak consistency} implies (iv). For the empirical bootstrap where $\hat{P}_n = \d{P}_n$, condition~\ref{cond: bootstrap limited complexity} implies condition~(i)--(iii). For any bootstrap $\hat{P}_n$, condition~\ref{cond: bootstrap limited complexity}(a) and the conditions of Proposition~\ref{prop: uniform donsker} imply conditions~(i)--(iii).
\end{restatable}

Lastly, \Cref{thm: perc bootstrap CI} requires  that $(S_n + R_n)(\sigma_n^* - \sigma_n) = o_{\prob_W^*}(n^{-1/2})$ for validity of the percentile $t$-method. This is satisfied if conditions~\ref{cond: limited complexity}--\ref{cond: second order} hold and $\sigma_n^* \condinoutprob \sigma_0$ and $\sigma_n \inoutprob \sigma_0$. However, if condition~\ref{cond: limited complexity} or \ref{cond: second order} do not hold, then a faster rate of convergence of $\sigma_n^* - \sigma_n$ may be required. For example, \cite{van2018targeted} suggested using targeted estimators for $\sigma_n^*$ and $\sigma_n$, which can yield $\sigma_n^* - \sigma_0 = O_{\prob_W^*}(n^{-1/2}) $ and $\sigma_n - \sigma_0 = O_{\prob_0^*}(n^{-1/2})$ under suitable conditions, so that $(S_n + R_n)(\sigma_n^* - \sigma_n) = o_{\prob_W^*}(n^{-1/2})$ as long as $S_n$ and $R_n$ are $o_{\prob_0^*}(1)$.

\subsubsection{Efron's percentile method}

A third method of constructing bootstrap confidence intervals is \textit{Efron's percentile method}, which is sometimes called the percentile method. In this case, the confidence interval is given by $[\zeta_{n, \beta}^*, \zeta_{n, 1-\alpha}^*]$ for $\zeta_{n, p}^*$ equal to the lower $p$th quantile of the distribution of $T(\eta_n^*, \d{P}_n^*)$ given the data; i.e., $\zeta_{n, p}^* := \inf\{ \zeta \in \d{R} : \prob_W^* \left( T(\eta_n^*, \d{P}_n^*) \leq \zeta \right) \geq p\}$.  The next result provides conditions under which Efron's percentile method yields an asymptotically valid confidence interval.  

\begin{restatable}{thm}{thmefroninterval}\label{thm: efron bootstrap CI}
    Suppose that $\hat{P}_n \phi_0^2 \inoutprob P_0 \phi_0^2$ and $(\hat{P}_n - P_0) [\phi_0^2 1\{|\phi_0| > M\}] \inoutprob 0$ for every $M>0$. If $\left[ R_n + S_n \right] + \left[ R_n^* + S_n^*\right]+ [T(\eta_n, \hat{P}_n) - T(\eta_n, \d{P}_n)] = o_{\prob_W^*}(n^{-1/2})$, then 
    \[\sup_{t \in \d{R}} \left|\prob_W^* \left( T(\eta_n^*, \d{P}_n^*) - T(\eta_n, \d{P}_n) \leq t \right) -  \prob_0^* \left( -\left[T(\eta_n, \d{P}_n) - T(\eta_0, P_0)\right] \leq t \right) \right| \inoutprob 0,\]
    and Efron's percentile confidence interval $[\zeta_{n, \beta}^*, \zeta_{n, 1-\alpha}^*]$ has asymptotic confidence level $1-\alpha - \beta$. 
\end{restatable}
As with the percentile and percentile $t$-intervals, conditions~\ref{cond: limited complexity}--\ref{cond: second order} and~\ref{cond: bootstrap limited complexity}--\ref{cond: bootstrap second order} imply that $S_n$, $R_n$, $S_n^*$, and $R_n^*$ are all $o_{\prob_W^*}(n^{-1/2})$. This yields the following Corollary.
\begin{restatable}{cor}{corefroninterval}\label{cor: efron bootstrap CI}
    If~\ref{cond: limited complexity}--\ref{cond: second order} and~\ref{cond: bootstrap limited complexity}--\ref{cond: bootstrap second order} hold, and $T(\eta_n, \hat{P}_n) - T(\eta_n, \d{P}_n) = o_{\prob_0^*}(n^{-1/2})$, then Efron's percentile confidence interval $[\zeta_{n, \beta}^*, \zeta_{n, 1-\alpha}^*]$ has asymptotic confidence level $1-\alpha - \beta$. 
\end{restatable}

\Cref{cor: efron bootstrap CI} demonstrates that as with the percentile and percentile $t$-methods, if the conditions of Theorems~\ref{thm: classical} and~\ref{thm: bootstrap} hold, then Efron's percentile intervals are asymptotically valid. However, the conditions of \Cref{thm: efron bootstrap CI} differ substantially from those of \Cref{thm: perc bootstrap CI}. First, the sums $S_n^* + S_n$ and $R_n^* + R_n$ appear in the condition for Efron's percentile method, in contrast with the the differences $S_n^* - S_n$ and $R_n^* - R_n$ in the condition for the other percentile methods. Typically, $S_n^* + S_n + R_n^* + R_n = o_{\prob_W^*}(n^{-1/2})$ will be established by showing that each summand is $o_{\prob_W^*}(n^{-1/2})$. Hence, Efron's percentile method generally does not have the potential for automatic bias correction. This was also noted in  \cite{cattaneo2022average}. Furthermore, even if Efron's percentile method is asymptotically valid, the appearance of the sums rather than differences means that it can have worse finite sample behavior than the percentile and percentile $t$-methods. This will be investigated further in Sections~\ref{sec: T examples} and~\ref{sec: simulation}.  Second, the term $T(\eta_n, \hat{P}_n) - T(\eta_n, \d{P}_n)$ appears in the conditions for Efron's percentile method, but not for the other percentile methods. This is because Efron's percentile method uses the bootstrap distribution of $T(\eta_n^*, \d{P}_n^*)$ directly without centering.  The distribution of $T(\eta_n^*, \d{P}_n^*)$ is asymptotically symmetric around $T(\eta_n, \hat{P}_n)$, which does not equal the original estimator $\psi_n = T(\eta_n, \d{P}_n)$ for non-empirical bootstraps. If the term $T(\eta_n, \hat{P}_n) - T(\eta_n, \d{P}_n)$ is not $\fasterthan(n^{-1/2})$ then Efron's percentile method may not have asymptotically valid coverage. This may be the case, for instance, if $\hat{P}_n$ is a distribution based on a kernel density estimator that is not targeted toward the parameter of interest.

\subsubsection{Bootstrap Wald method}

The final method of constructing bootstrap confidence intervals that we will discuss is the \textit{bootstrap Wald method}, which is based on a normal approximation. We define $\bar\sigma_{n}^2:= \E_W^* \{n^{1/2}[T(\eta_n^*, \d{P}_n^*) - T(\eta_n, \hat{P}_n)]\}^2$ as the variance of the centered and scaled bootstrap estimator distribution given the data (not to be confused with $\sigma_n^{*2}$). The two-sided $(1-\alpha - \beta)$-level bootstrap Wald confidence interval is then given by 
\begin{equation}\label{eq: bootstrap wald}
    \left[  T(\eta_n, \d{P}_n) + z_{\beta}\bar\sigma_{n} n^{-1/2}, T(\eta_n, \d{P}_n) + z_{1-\alpha}\bar\sigma_{n} n^{-1/2}\right],
\end{equation}
where $z_p$ is the lower-$p$ quantile of the standard normal distribution. In practice, the bootstrap quantiles $\xi_{n,p}^*$ and $\zeta_{n,p}^*$ and the bootstrap variance $\bar\sigma_{n}^2$ are approximated using empirical analogues based on a large number of bootstrap samples. As mentioned in Section~\ref{sec:bootstrap_framework}, we ignore the effect of this approximation. The final result of this section provides conditions under which the bootstrap Wald interval is asymptotically valid.
\begin{restatable}{thm}{thmwaldbootinterval}\label{thm: wald bootstrap CI}
    Denote $T_n^*:= n^{1/2}[T(\eta_n^*, \d{P}_n^*) - T(\eta_n, \hat{P}_n)]$. If conditions~\ref{cond: limited complexity}--\ref{cond: second order} and \ref{cond: bootstrap limited complexity}--\ref{cond: bootstrap second order} hold and $T_n^{*}$ is asymptotically uniformly square-integrable in the sense that 
    \[\lim_{m\to \infty} \limsup_{n\to \infty} \E_0^*\E_W^* \left[T_n^{*2}1\{T_n^{*2} \geq m\}\right] = 0,\]  then $\bar\sigma_{n} \inoutprob \sigma_0$, so the bootstrap Wald confidence interval defined in~\eqref{eq: bootstrap wald} has asymptotic confidence level $1-\alpha - \beta$. 
\end{restatable}
Since the bootstrap Wald method is based on a normal approximation, its asymptotic validity requires conditional weak convergence of the bootstrap to the same normal limit as the original estimator. Hence, this method does not have the same possibility of automatic bias correction as the percentile and percentile $t$ methods. Furthermore, since weak convergence does not imply convergence of moments, \Cref{thm: wald bootstrap CI} also requires asymptotically uniform integrability of the centered and scaled bootstrap estimator.

\section{Remainder calculations for specific estimator constructions}\label{sec: T examples}
\subsection{One-step estimator}\label{sec: one-step}

In this section, we explain how~\ref{cond: second order},~\ref{cond: bootstrap second order}, and $R_n^* - R_n = o_{\prob_W^*}(n^{-1/2})$ can be verified for several specific estimator constructions $T$. Throughout this section, if $\hat{P}_n$ is in the domain of $\eta$, then we denote $\hat{\eta}_n:= \eta(\hat{P}_n)$. 

We first consider the one-step construction. Suppose that the parameter of interest $\psi(P)$ and its influence function $\phi_P$ depend on $P$ only through $\eta_P$, so that with some abuse of notation we can write $\psi(P) = \psi(\eta_P)$ and $\phi_P = \phi_{\eta_P}$. A one-step estimator of $\psi_0$ based on nuisance estimator $\eta_n$ is then defined as $\psi_{n} := \psi(\eta_n) + \d{P}_n \phi_{\eta_n}$, which can be represented as $\psi_n = T(\eta_n, \d{P}_n)$ for $T(\eta, P) := \psi(\eta) + P \phi_{\eta}$. The mean-zero property of influence functions implies that $T(\eta_0, P_0) = \psi_0$.

For the one-step estimator, $R_n = \psi(\eta_n) - \psi(\eta_0) + P_0 \phi_{\eta_n}$, which is known as the \emph{second-order remainder}. 
For so-called \emph{strongly differentiable parameters}, the second-order remainder term is $o_{\prob_0^*}(n^{-1/2})$ under conditions on the rate of convergence of $\eta_n - \eta_0$  (see, e.g., Chapter 4 of \citealp{pfanzagl1982contributions}). Hence, a benefit of the one-step construction is that $\eta_n$ does not typically need to be tailored to $\psi$ for the one-step estimator to be asymptotically linear, though $\eta_n$ does usually need to satisfy rate and entropy requirements. 

We define the bootstrap one-step estimator based on bootstrap nuisance estimator $\eta_n^*$ as $\psi_n^*:= T(\eta_n^*, \d{P}_n^*) = \psi(\eta_n^*) + \d{P}_n^* \phi_{\eta_n^*}$. We then have $R_n^* = \psi(\eta_n^*) - \psi(\eta_n) + \hat{P}_n(\phi_{\eta_n^*} - \phi_{\eta_n})$. To demonstrate that~\ref{cond: bootstrap second order} holds, we can decompose $R_n^*$ in two ways:
\begin{align}
    \begin{split}\label{eq: one step Rnstar}
        R_n^* &= \big[\psi(\eta_n^*) - \psi(\eta_0) + P_0 \phi_{\eta_n^*}\big] -\big[\psi(\eta_n) - \psi(\eta_0) + P_0 \phi_{\eta_n}\big] + (\hat{P}_n - P_0)(\phi_{\eta_n^*} - \phi_{\eta_n})  \\
        &= \big[\psi(\eta_n^*) - \psi(\hat{\eta}_n) + \hat{P}_n \phi_{\eta_n^*}\big] -\big[\psi(\eta_n) - \psi(\hat{\eta}_n) + \hat{P}_n \phi_{\eta_n}\big].
    \end{split}
\end{align}
For the first decomposition, we note that $\psi(\eta_n) - \psi(\eta_0) + P_0 \phi_{\eta_n}$ is the second-order remainder term defined above, and $\psi(\eta_n^*) - \psi(\eta_0) + P_0 \phi_{\eta_n^*}$ is a second-order remainder term with $\eta_n^*$ playing the role of $\eta_n$. Hence, this term will typically be $o_{\prob_W^*}(n^{-1/2})$ under conditions on the conditional rate of convergence of $\eta_n^* - \eta_0$. Similarly, the terms $\psi(\eta_n^*) - \psi(\hat{\eta}_n) + \hat{P}_n \phi_{\eta_n^*}$ and $\psi(\eta_n) - \psi(\hat{\eta}_n) + \hat{P}_n \phi_{\eta_n}$ in the second decomposition are second-order remainder terms that will typically be $o_{\prob_W^*}(n^{-1/2})$ under conditional rates of convergence of $\eta_n^* - \hat{\eta}_n$ and $\eta_n - \hat{\eta}_n$, respectively. 

To demonstrate that bootstrap percentile intervals are asymptotically valid, we can decompose $R_n^* - R_n$ in two analogous ways:
\begin{align}
    \begin{split}\label{eq: one step RnstarRn}
        R_n^* - R_n &= \big[\psi(\eta_n^*) - \psi(\eta_0) + P_0 \phi_{\eta_n^*}\big] - 2\big[\psi(\eta_n) - \psi(\eta_0) + P_0 \phi_{\eta_n}\big] + (\hat{P}_n - P_0)(\phi_{\eta_n^*} - \phi_{\eta_n}) \\
        &=\big[\psi(\eta_n^*) - \psi(\hat{\eta}_n) + \hat{P}_n \phi_{\eta_n^*}\big] - \big[\psi(\eta_n) - \psi(\eta_0) + P_0 \phi_{\eta_n}\big] - \big[\psi(\eta_n) - \psi(\hat{\eta}_n) + \hat{P}_n \phi_{\eta_n}\big]. 
    \end{split}
\end{align}
As discussed in Section~\ref{sec:conf_int}, these two expressions can be $o_{\prob_W^*}(n^{-1/2})$ even if $R_n$ and/or $R_n^*$ are not. For instance, in the first decomposition, $\psi(\eta_n^*) - \psi(\eta_0) + P_0 \phi_{\eta_n^*}$ may be within $o_{\prob_W^*}(n^{-1/2})$ of $2[\psi(\eta_n) - \psi(\eta_0) + P_0 \phi_{\eta_n}]$ even if each of these terms individually is not $o_{\prob_W^*}(n^{-1/2})$. Similarly, in the second decomposition, $\psi(\eta_n^*) - \psi(\hat{\eta}_n) + \hat{P}_n \phi_{\eta_n^*}$ may be within $o_{\prob_W^*}(n^{-1/2})$ of $\psi(\eta_n) - \psi(\eta_0) + P_0 \phi_{\eta_n}$ even if each of these terms individually is not $o_{\prob_W^*}(n^{-1/2})$. 

The bootstrap sampling distribution is a primary consideration when determining which of the two decompositions in~\eqref{eq: one step Rnstar} and~\eqref{eq: one step RnstarRn} should be used. The first decomposition is more suitable for the empirical bootstrap $\hat{P}_n = \d{P}_n$. This is because for most of our applications of interest, empirical distributions are not in the domain of the nuisance parameter $\eta$, so that $\hat\eta_n := \eta(\d{P}_n)$ does not exist. For example, this is the case if $\eta(P)$ is the Lebesgue density of $P$. Hence, for the empirical bootstrap, the first set of conditions, which do not involve $\hat\eta_n$, should typically be used. In this case, $n^{1/2}(\hat{P}_n - P_0)(\phi_{\eta_n^*} - \phi_{\eta_n}) = \d{G}_n(\phi_{\eta_n^*} - \phi_{\eta_n}) = o_{\prob_W^*}(1)$ under~\ref{cond: limited complexity}, \ref{cond: weak consistency}, \ref{cond: bootstrap limited complexity}(a), and~\ref{cond: bootstrap weak consistency}. Hence, the empirical bootstrap one-step estimator is conditionally asymptotically linear if~\ref{cond: limited complexity}--\ref{cond: second order}, \ref{cond: bootstrap limited complexity}(a), and~\ref{cond: bootstrap weak consistency} hold and $\psi(\eta_n^*) - \psi(\eta_0) + P_0 \phi_{\eta_n^*} = o_{\prob_W^*}(n^{-1/2})$.  Alternatively, bootstrap percentile confidence intervals based on the empirical bootstrap one-step estimator are consistent if~\ref{cond: limited complexity}--\ref{cond: weak consistency}, \ref{cond: bootstrap limited complexity}(a), and~\ref{cond: bootstrap weak consistency} hold and $2[\psi(\eta_n) - \psi(\eta_0) + P_0 \phi_{\eta_n}] - [\psi(\eta_n^*) - \psi(\eta_0) + P_0 \phi_{\eta_n^*}] = o_{\prob_W^*}(n^{-1/2})$.

For non-empirical bootstraps where $\hat\eta_n$ is well-defined, we expect the second decomposition in~\eqref{eq: one step Rnstar} and~\eqref{eq: one step RnstarRn} to be easier to verify. This is because in many cases when $\hat{P}_n \neq \d{P}_n$, it may be difficult to show that $(\hat{P}_n - P_0)(\phi_{\eta_n^*} - \phi_{\eta_n}) = o_{\prob_W^*}(n^{-1/2})$. For the special case of $\hat{P}_n$ equal to the distribution corresponding to a kernel density estimator, Section~3.2 of \cite{gine2008uniform} and Theorem~10 of \cite{radulovic2009uniform} establish asymptotic uniform equicontinuity of $\{n^{1/2}(\hat{P}_n - P_0)f: f \in \s{F}\}$ under conditions on the bandwidth and smoothness of functions $f \in \s{F}$. This implies $(\hat{P}_n - P_0)(\phi_{\eta_n^*} - \phi_{\eta_n}) = o_{\prob_W^*}(n^{-1/2})$ under conditions~\ref{cond: limited complexity}, \ref{cond: weak consistency}, \ref{cond: bootstrap limited complexity}(a), and~\ref{cond: bootstrap weak consistency}. In cases where asymptotic uniform equicontinuity of $\{n^{1/2}(\hat{P}_n - P_0)f : f \in \s{F}\}$ is hard to establish, but $\hat{P}_n$ and $P_0$ are dominated by a fixed measure $\lambda$, a simple but crude approach to showing~\ref{cond: bootstrap weak consistency} is to use the Cauchy-Schwartz inequality:  $|(\hat{P}_n - P_0)(\phi_{\eta_n^*} - \phi_{\eta_n})| \leq \|\sd \hat{P}_n - \sd P_0 \|_{L_2(\lambda)} \|\phi_{\eta_n^*} - \phi_{\eta_n} \|_{L_2(\lambda)}$. Alternatively, a more direct calculation may be employed.

An important special case is $\eta(\hat{P}_n) = \eta_n$; i.e., the bootstrap sampling distribution is based on the original nuisance estimator. For instance, if $\eta_n$ is a Lebesgue density estimator, this would correspond to drawing bootstrap samples from the distribution corresponding to $\eta_n$. If $\eta(\hat{P}_n) = \eta_n$ then $\hat\eta_n = \eta_n$, and hence the second decomposition in \eqref{eq: one step Rnstar} reduces to $\psi(\eta_n^*) - \psi(\hat\eta_n) + \hat{P}_n \phi_{\eta_n^*}$, and the second decomposition in~\eqref{eq: one step RnstarRn} reduces to $[\psi(\eta_n) - \psi(\eta_0) + P_0 \phi_{\eta_n}] - [\psi(\eta_n^*) - \psi(\hat{\eta}_n) + \hat{P}_n \phi_{\eta_n^*}]$.

If we use the original nuisance estimator for the bootstrap estimator, i.e.\ $\eta_n^* = \eta_n$, then $R_n^* = 0$. This leads to the following corollary to Theorem~\ref{thm: bootstrap}.
\begin{restatable}{cor}{coronestepfixed}\label{cor: bootstrap one-step fixed}
    For the bootstrap one-step estimator with $\eta_n^* = \eta_n$, if conditions~\ref{cond: limited complexity}(a), \ref{cond: bootstrap limited complexity}(b) and \ref{cond: weak consistency} hold, then $T(\eta_n, \d{P}_n^*)  = T(\eta_n, \hat{P}_n) + (\d{P}_n^* - \hat{P}_n) \phi_0 + o_{\prob_W^*}(n^{-1/2})$. In particular, for the empirical bootstrap $\hat{P}_n = \d{P}_n$, conditions~\ref{cond: limited complexity} and~\ref{cond: weak consistency} imply the result.
\end{restatable}
\Cref{cor: bootstrap one-step fixed} indicates that a subset of the conditions for asymptotic linearity implies consistency of the empirical bootstrap for the one-step estimator with the original nuisance estimator. This is convenient as it means that in this case, no additional work needs to be done to establish asymptotic validity of bootstrap confidence intervals beyond that for establishing asymptotic linearity of the estimator. However, it also means that consistency of the percentile and percentile $t$-intervals requires that $R_n = o_{\prob_0^*}(n^{-1/2})$, so automatic bias correction of bootstrap confidence intervals does not occur.

\subsection{Plug-in estimator}\label{sec: plug-in}

Our next example of an estimator construction is the plug-in estimator. Suppose that the parameter of interest $\psi(P)$ depends on $P$ only through $\eta(P)$, so that, with some abuse of notation, we can write $\psi(P) = \psi(\eta_P)$. A plug-in estimator is then given by $\psi_n = \psi(\eta_n)$, which can be represented as $\psi_n = T\big(\eta_n, \d{P}_n\big)$ for $T(\eta, P) = \psi(\eta)$. In this case, $T$ is a function of $\eta$ alone, but for consistency of notation, we will continue to write it as a function of $P$ as well.

For the plug-in estimator, $R_n = [\psi(\eta_n) - \psi(\eta_0) + P_0 \phi_n] - \d{P}_n \phi_n$. If the influence function $\phi_P = \phi_{\eta_P, \pi_P}$ depends on $\eta_P$ and an additional nuisance parameter $\pi_P$ and we set $\phi_n = \phi_{\eta_n, \pi_n}$ for $\pi_n$ an estimator of $\pi_P$, then $ \psi(\eta_n) - \psi(\eta_0) + P_0 \phi_{\eta_n, \pi_n}$ is a second-order remainder term that will typically be $o_{\prob_0^*}(n^{-1/2})$ under sufficient rates of convergence of $\eta_n$ to $\eta_0$ and $\pi_n$ to $\pi_0$ in appropriate semi-metrics, as discussed in Section~\ref{sec: one-step}. Plug-in estimators also require that $\d{P}_n \phi_n = o_{\prob_0^*}(n^{-1/2})$, which typically requires careful construction of $\eta_n$, as discussed in Section~\ref{sec:asym linearity}. 

We let $\eta_n^*$ be a bootstrap nuisance estimator, and we define $\psi_n^* = T(\eta_n^*, \d{P}_n^*) = \psi(\eta_n^*)$ as the bootstrap plug-in estimator. We then have $R_n^* = [\psi(\eta_n^*) - \psi(\eta_n) + \hat{P}_n \phi_n^*] - \d{P}_n^* \phi_n^*$. To demonstrate that~\ref{cond: bootstrap second order} holds, we can decompose $R_n^*$ in two ways:
\begin{align}
\begin{split}\label{eq: plugin Rnstar}
        R_n^* &= [\psi(\eta_n^*) - \psi(\eta_0) + P_0 \phi_n^*] - [\psi(\eta_n) - \psi(\eta_0) + P_0 \phi_n] + (\hat{P}_n - P_0)(\phi_n^* - \phi_n) + [\hat{P}_n \phi_n - \d{P}_n^* \phi_n^*] \\
        & = [\psi(\eta_n^*) - \psi(\hat{\eta}_n) + \hat{P}_n \phi_n^*] - [\psi(\eta_n) - \psi(\hat{\eta}_n) + \hat{P}_n \phi_n]  + [\hat{P}_n \phi_n - \d{P}_n^* \phi_n^*]
\end{split}
\end{align}
If $\phi_{n} = \phi_{\eta_n, \pi_n}$ and $\phi_n^* = \phi_{\eta_n^*, \pi_n^*}$, then the first three terms in square braces of the first decomposition and the first two terms in square braces of the second decomposition are second-order remainder terms. Negligibility of these terms was discussed following~\eqref{eq: one step Rnstar}. Compared with \eqref{eq: one step Rnstar}, both decompositions in \eqref{eq: plugin Rnstar} additionally involve the term $\hat{P}_n \phi_n - \d{P}_n^* \phi_n^*$. Ensuring that this term is $o_{\prob_W^*}(n^{-1/2})$ typically requires careful construction of $\eta_n^*$ and $\eta_n$, as discussed in Section~\ref{sec:asym linearity}. If $\hat{P}_n = \d{P}_n$ is the empirical bootstrap, then $\hat{P}_n \phi_n = o_{\prob_0^*}(n^{-1/2})$ is  typically required for~\ref{cond: second order} to hold, as discussed above. If this holds, then it is sufficient that $\d{P}_n^* \phi_n^* = o_{\prob_W^*}(n^{-1/2})$, which is the bootstrap analogue of $\d{P}_n \phi_n = o_{\prob_0^*}(n^{-1/2})$. Alternatively, if $\phi_n = \phi_{\eta_n, \pi_n}$ and $\hat{P}_n$ is based on $(\eta_n, \pi_n)$, then $\hat{P}_n \phi_n = 0$, so it is again sufficient that $\d{P}_n^* \phi_n^* = o_{\prob_W^*}(n^{-1/2})$. 

To demonstrate that bootstrap percentile intervals are asymptotically valid, we can decompose $R_n^* - R_n$ in two analogous ways:
\begin{align}
    \begin{split}\label{eq: plugin RnstarRn}
        R_n^* - R_n &= [\psi(\eta_n^*) - \psi(\eta_0) + P_0 \phi_n^*] - 2[\psi(\eta_n) - \psi(\eta_0) + P_0 \phi_n] + (\hat{P}_n - P_0)(\phi_n^* - \phi_n)   \\
        & \qquad \qquad + [\hat{P}_n \phi_n + \d{P}_n \phi_n - \d{P}_n^* \phi_n^*] \\
        & = [\psi(\eta_n^*) - \psi(\hat{\eta}_n) + \hat{P}_n \phi_n^*] - [\psi(\eta_n) - \psi(\eta_0) + P_0 \phi_n] - [\psi(\eta_n) - \psi(\hat{\eta}_n) + \hat{P}_n \phi_n]    \\
        & \qquad \qquad + [\hat{P}_n \phi_n + \d{P}_n \phi_n - \d{P}_n^* \phi_n^*]
    \end{split}
\end{align}
As mentioned in Section~\ref{sec:conf_int} and Section~\ref{sec: one-step}, these two expressions can be $o_{\prob_W^*}(n^{-1/2})$ even if $R_n$ and/or $R_n^*$ are not. The first three terms in square braces of both decompositions were discussed in Section~\ref{sec: one-step}. The bootstrap sampling distribution is a primary consideration when determining which of the two decompositions in~\eqref{eq: plugin Rnstar} and~\eqref{eq: plugin RnstarRn} should be used. We refer readers to \Cref{sec: one-step} for further discussion.

To show that $R_n^* - R_n = o_{\prob_W^*}(n^{-1/2})$ using either decomposition in \eqref{eq: plugin RnstarRn}, it may be necessary to additionally show that $\hat{P}_n \phi_n + \d{P}_n \phi_n - \d{P}_n^* \phi_n^* = o_{\prob_W^*}(n^{-1/2})$. For the empirical bootstrap $\hat{P}_n = \d{P}_n$, this is the case if  $\d{P}_n^* \phi_n^*$ is within $o_{\prob_W^*}(n^{-1/2})$ of $2\d{P}_n \phi_n$. If instead $\hat{P}_n$ is based on $(\eta_n, \pi_n)$, then it is sufficient that $\d{P}_n^* \phi_n^*$ is within $o_{\prob_W^*}(n^{-1/2})$ of $\d{P}_n \phi_n$. In both cases, the conditions can be met even if $\d{P}_n \phi_n$ and/or $\d{P}_n^* \phi_n^*$ are not $o_{\prob_0^*}(n^{-1/2})$ and $o_{\prob_W^*}(n^{-1/2})$, respectively. Hence, bootstrap percentile confidence intervals based on a plug-in estimator can be asymptotically valid even when the estimator is not targeted toward the functional of interest.

Lastly, we note that neither~\ref{cond: bootstrap second order} nor $R_n^* - R_n = o_{\prob_W^*}(n^{-1/2})$ hold for the bootstrap plug-in estimator if we set $\eta_n^* = \eta_n$ because this would result in $\psi_n^* = \psi(\eta_n^*) = \psi(\eta_n) = \psi_n$ with $\prob_W^*$ probability 1; i.e.,  a bootstrap distribution equal to a point mass at the original estimator.

\subsection{Empirical mean of a nuisance-dependent function}\label{sec: mean plug-in}

Our next example of an estimator construction is the empirical mean of a nuisance-dependent function, or empirical mean plug-in estimator for brevity. Suppose that the parameter of interest can be written as $\psi(P):= P h_{\eta_P}$ for some known function $h_\eta: \s{X} \mapsto \d{R}$ depending on $\eta \in \s{H}$.  A simple estimator is then given by $\psi_n = \d{P}_n h_{\eta_n}$, which can be represented as $\psi_n = T\big(\eta_n, \d{P}_n\big)$ for $T(\eta, P) = P h_\eta$. In some cases, this estimator and the plug-in estimator considered in Section~\ref{sec: plug-in} are the same. However, the representations of the two estimators in terms of $T$ are different, which leads to different conditions for asymptotic linearity and especially for consistency of the bootstrap.

We can write the influence function as $\phi_P(x) = h_{\eta_P}(x) + \gamma_P(x) - \psi(P)$ for $\gamma_P(x) := \phi_P(x) - h_{\eta_P}(x) + \psi(P)$, and we note that $P \gamma_P = 0$. Heuristically, $\gamma_P$ can be viewed as the contribution to the influence function of fluctuating $\eta_P$ within the model. We write $\phi_n$ as $\phi_n(x) = h_{\eta_n}(x) + \gamma_n(x) - \psi_n$ for some estimator $\gamma_n$ of $\gamma_0$. We then have $R_n = P_0 (h_{\eta_n} - h_{\eta_0} + \gamma_n) - \d{P}_n \gamma_n$, which leads to the following conditions under which~\ref{cond: second order} holds.  If $\eta$ and $\gamma$ are compatible in the sense that there exists $P' \in \s{M}$ such that $\eta_{P'} = \eta$ and $\gamma_{P'} = \gamma$, then $\psi(P') - \psi(P) + P \phi_{P'} = P(h_{\eta} - h_{\eta_P} + \gamma)$ is the second-order remainder term discussed in Section~\ref{sec: one-step}. Hence, if there exists $P_n$ such that $\eta_{P_n} = \eta_n$ and $\gamma_{P_n} = \gamma_n$, then $P_0(h_{\eta_n} - h_{\eta_0} + \gamma_n)$ is a second-order remainder and can be expected to be $o_{\prob_0^*}(n^{-1/2})$ if $(\eta_n, \gamma_n)$ converges to $(\eta_0, \gamma_0)$ at a sufficient rate in an appropriate semi-metric. The condition that $\d{P}_n \gamma_n = o_{\prob_0^*}(n^{-1/2})$ typically requires careful construction of $\eta_n$, as discussed in Section~\ref{sec:asym linearity}.

As usual, we let $\psi_n^* = T(\eta_n^*, \d{P}_n^*)$ be the bootstrap empirical mean plug-in estimator for a bootstrap nuisance estimator $\eta_n^*$. We also define the bootstrap influence function estimator as $\phi_n^*(x) = h_{\eta_n^*}(x) + \gamma_n^*(x) - \psi_n^*$, where $\gamma_n^*$ is the bootstrap estimator of $\gamma_0$. We then have $R_n^* = \hat{P}_n (h_{\eta_n^*} - h_{\eta_n} + \gamma_n^*) - \d{P}_n^* \gamma_n^*$. To demonstrate that~\ref{cond: bootstrap second order} holds, we can decompose $R_n^*$ in two ways:
\begin{align}
\begin{split}\label{eq: mean plugin Rnstar}
    R_n^* &= P_0 (h_{\eta_n^*} - h_{\eta_0} + \gamma_n^*) - P_0 (h_{\eta_n} - h_{\eta_0} + \gamma_n) \\ 
    & \qquad + (\hat{P}_n - P_0)(h_{\eta_n^*} - h_{\eta_n} + \gamma_n^* - \gamma_n) + (\hat{P}_n \gamma_n - \d{P}_n^* \gamma_n^*)\\
    & = \hat{P}_n (h_{\eta_n^*} - h_{\hat{\eta}_n} + \gamma_n^*) - \hat{P}_n  (h_{\eta_n} - h_{\hat{\eta}_n} + \gamma_n) + (\hat{P}_n \gamma_n - \d{P}_n^* \gamma_n^*t).
\end{split}
\end{align}
If $\eta_n^*$ and $\gamma_n^*$ are compatible and $\eta_n$ and $\gamma_n$ are compatible, then the first two terms of both decompositions are second-order remainders. We discussed negligibility of these terms and of the third term in the first decomposition in Section~\ref{sec: one-step}. Both decompositions in \eqref{eq: mean plugin Rnstar} also have the term $\hat{P}_n \gamma_n - \d{P}_n^* \gamma_n^*$, and negligibility of this term typically requires careful construction of $\eta_n^*$ and $\eta_n$, as discussed in Section~\ref{sec: plug-in}.

We can decompose $R_n^* - R_n$ in two analogous ways:
\begin{align}
\begin{split}
    R_n^* - R_n &= P_0 (h_{\eta_n^*} - h_{\eta_0} + \gamma_n^*) - 2 P_0 (h_{\eta_n} - h_{\eta_0} + \gamma_n) + (\hat{P}_n - P_0)(h_{\eta_n^*} - h_{\eta_n} + \gamma_n^* - \gamma_n) \\ 
    & \qquad  + (\hat{P}_n \gamma_n + \d{P}_n \gamma_n - \d{P}_n^* \gamma_n^*)\\
    & = \hat{P}_n (h_{\eta_n^*} - h_{\hat{\eta}_n} + \gamma_n^*)  - P_0 (h_{\eta_n} - h_{\eta_0} + \gamma_n)  - \hat{P}_n (h_{\eta_n} - h_{\hat{\eta}_n} + \gamma_n) \\
    & \qquad  + (\hat{P}_n \gamma_n + \d{P}_n \gamma_n - \d{P}_n^* \gamma_n^*).
\end{split}
\end{align}
As with the one-step and plug-in estimators, these two expressions can be $o_{\prob_W^*}(n^{-1/2})$ even if $R_n$ and/or $R_n^*$ are not; we refer readers to the discussions following~\eqref{eq: one step RnstarRn} and~\eqref{eq: plugin RnstarRn}. 

Lastly, we note that whether condition~\ref{cond: bootstrap second order} holds if we set $\eta_n^* = \eta_n$ depends on the particular problem. If $\gamma_P$ is a function of $P$ through $\eta_P$, then condition~\ref{cond: bootstrap second order} usually will not hold for either the empirical or non-empirical bootstraps. This is because in this case, $\eta_n^* = \eta_n$ implies that $\gamma_n^* = \gamma_n$, so that even if $\d{P}_n \gamma_n =  o_{\prob_0^*}(n^{-1/2})$, it does not typically follow that $\d{P}_n^* \gamma_n^* = \d{P}_n^* \gamma_n = o_{\prob_W^*}(n^{-1/2})$. Similarly, $\hat{P}_n \gamma_n + \d{P}_n \gamma_n - \d{P}_n^* \gamma_n^* = \d{P}_n \gamma_n - n^{-1/2} \d{G}_n^* \gamma_n$ will not be typically be $o_{\prob_W^*}(n^{-1/2})$, so that $R_n^* - R_n = o_{\prob_W^*}(n^{-1/2})$ may not hold. However, if $\gamma_P$ involves additional summaries of $P$ beyond $\eta_P$, then it may be possible to construct $\eta_n$ in such a way that both $\d{P}_n \gamma_n = o_{\prob_0^*}(n^{-1/2})$ and $\d{P}_n^* \gamma_n^* = o_{\prob_W^*}(n^{-1/2})$, or such that $R_n^* - R_n = o_{\prob_W^*}(n^{-1/2})$.

\subsection{Estimating equations-based estimator}\label{sec: estimating-equations}
    
We next consider estimating equations-based estimators. Suppose that the influence function $\phi_P$ depends on $P$ through $\psi(P)$ and $\eta_P$, so that we can write $\phi_P = \phi_{\psi(P), \eta_P}$. For each $P \in \s{M}^+$ and $\eta \in \s{H}$, we define the estimating function $G_{P, \eta} : \psi \mapsto P \phi_{\psi, \eta}$, and we assume for simplicity that for each $(P, \eta) \in \s{M}^+ \times \s{H}$, there is a unique $\psi \in \d{R}$ such that $G_{P,\eta}(\psi) = 0$.  For each $\eta \in \s{H}$, we define the true population estimating function as $G_{0,\eta} := G_{P_0, \eta}$, and we define the sample estimating function as $G_{n,\eta_n} := G_{\d{P}_n, \eta_n}$, where $\eta_n$ is a nuisance estimator. If $\psi_0$ is the unique solution to $G_{0,\eta_0}(\psi_0) = 0$, then an estimating equation-based estimator $\psi_n$ of $\psi_0$ is defined as the solution to $G_{n,\eta_n}(\psi_n) = 0$. We can then write $\psi_n = T(\eta_n, \d{P}_n)$ for $T(\eta, P)$ defined as the solution to $G_{P, \eta} =0$. We then have $R_n = \psi_n - \psi_0 + P_0 \phi_{\psi_n, \eta_n}$. We also note that $\psi_n$ does not need to solve the equation exactly; it suffices that $G_{n,\eta_n}(\psi_n) = o_{\prob_0^*}(n^{-1/2})$. The derivations for estimating equations-based estimators are more complicated than those for other estimator constructions we have considered, so here we provide theorems for clarity.

Estimating equations-based estimators have been studied in a variety of contexts. Estimating equations often arise in semiparametric models indexed by the pair $(\psi, \eta) \in \d{R} \times \s{H}$. We do not explicitly require this, but our results can be applied in this setting. It is sometimes assumed that $\psi_n$ approximately solves a ``profile" estimating equation $\psi \mapsto \d{P}_n \phi_{\psi, \eta_n(\psi)} = 0$, where $\eta_n(\psi)$ is a solution to $\eta \mapsto \d{P}_n \phi_{\psi, \eta}$ (see, e.g., \citealp{murphy2000profile} and Chapter~21 of \citealp{kosorok2008introduction}). This is the case when $(\psi_n, \eta_n)$ are defined as the joint optimizers of a criterion function such as a likelihood. In contrast, our goal is to permit $\eta_n$ to be an arbitrary nuisance estimator satisfying certain rate and entropy conditions, and to use the estimating equation provided by the influence function to mitigate the asymptotic bias, as the one-step estimator does. This approach is more in line with \cite{laan2003unified} and Chapter~5 of \cite{van2000asymptotic}, among others. In particular, Theorem~5.31 of \cite{van2000asymptotic} provides conditions for asymptotic linearity of an estimating equations-based estimator $\psi_n$ as we have defined it above. We now provide a slightly reformulated version of this result under which~\ref{cond: second order} holds for an estimating equations-based estimator. Following \cite{van2000asymptotic} and others, we say that $\psi_0$ is a \textit{well-separated} solution of the population estimating equation if $P_0 \phi_{\psi_0, \eta_0} = 0 < \inf_{| \psi - \psi_0| > \delta} |P_0 \phi_{\psi, \eta_0}|$ for every $\delta > 0$.

\begin{restatable}{lemma}{lemmaclassicalestimating}\label{lemma: classical estimating-equations}
    If condition~\ref{cond: limited complexity} holds for $\phi_n:= \phi_{\psi_n, \eta_n}$, $\psi_0$ is a well-separated solution of the population estimating equation, $\psi_n = O_{\prob_0^*}(1)$, there exists a map $G_{0, \eta}': \d{R} \to \d{R}$ depending on $\eta \in \s{H}$ such that $\Gamma_{0,\eta} : \psi \mapsto G_{0, \eta}(\psi) - G_{0, \eta}(\psi_0) - G_{0, \eta}'(\psi_0) (\psi - \psi_0)$ satisfies $\sup_{|\psi|\leq M} \left| \Gamma_{0,\eta_n}(\psi) - \Gamma_{0,\eta_0}(\psi)\right| = o_{\prob_0^*}(1)$ for every $M > 0$ and $\sup_{\eta: \|\eta - \eta_0\|_{\s{H}} < \delta} \left|\Gamma_{0, \eta}(\psi) \right| = o\left(\left| \psi - \psi_0\right|\right)$ for some $\delta > 0$, $G_{0, \eta}'$ satisfies $\lim_{\eta \to \eta_0} G_{0, \eta}'(\psi_0) = G_{0, \eta_0}'(\psi_0) = -1$, $\|\eta_n - \eta_0\|_{\s{H}} = o_{\prob_0^*}(1)$, and $P_0 \phi_{\psi_0, \eta_n}  = o_{\prob_0^*}(n^{-1/2})$, then~\ref{cond: second order} holds for the estimating equations-based estimator.
\end{restatable}
As in Theorem~5.31 of \cite{van2000asymptotic}, the assumption that $\psi_0$ is a well-separated solution of $\psi \mapsto P_0\phi_{\psi,\eta_0}$ is used to establish consistency of $\psi_n$, and the differentiability assumption is used to ensure that $n^{1/2}(\psi_n - \psi_0)$ can be linearized. The requirement that $G_{0, \eta_0}'(\psi_0) = -1$ is not present in Theorem~5.31 of \cite{van2000asymptotic}; this is because we are assuming the estimating function \emph{is} the influence function, which requires proper scaling. Theorem~5.31 also permits that the  ``drift term" $P_0 \phi_{\psi_0, \eta_n}$ contributes to the asymptotic distribution of $n^{1/2} (\psi_n - \psi_0)$, while we assume it is negligible since our goal is to establish asymptotic linearity of $\psi_n$ with influence function $\phi_0$. Negligibility of this term is implied by a sufficient rate of convergence of $\eta_n$ to $\eta_0$ if $\eta\mapsto P_0 \phi_{\psi_0, \eta}$ is differentiable near $\eta_0$ in an appropriate sense with derivative map equal to zero \citep{chernozhukov2018double}. However, this is not always the case, and when it fails, the requirement that $P_0 \phi_{\psi_0, \eta_n} = o_{\prob_0^*}(n^{-1/2})$ requires under-smoothing or otherwise targeting $\eta_n$. Hence, estimating equations do not always sufficiently control the asymptotic bias, unlike the one-step estimator. We will see an example of this in Section~\ref{sec: applications}.

We define the bootstrap estimating equations-based estimator as $\psi_n^* = T(\eta_n^*, \d{P}_n^*)$; i.e., $\psi_n^*$ is the solution to $G_{n,\eta_n^*}^*(\psi) := \d{P}_n^* \phi_{\psi, \eta_n^*} = 0$. As above, it is sufficient that $G_{n,\eta_n^*}^*(\psi_n^*) = o_{\prob_W^*}(n^{-1/2})$.  We provide separate sufficient conditions for~\ref{cond: bootstrap second order} for the empirical and smooth bootstraps. First, for the empirical bootstrap $\hat{P}_n = \d{P}_n$, we have $R_n^* = \psi_n^* - \psi_n + \d{P}_n \phi_{\psi_n^*, \eta_n^*}$, and we have the following result.
\begin{restatable}{lemma}{lemmaempbootstrapestimating}\label{lemma: empirical bootstrap estimating-equations}
    Let $\hat{P}_n = \d{P}_n$ be the empirical bootstrap. Suppose the conditions of \Cref{lemma: classical estimating-equations} hold, condition~\ref{cond: bootstrap limited complexity} holds for $\phi_n^* = \phi_{\psi_n^*, \eta_n^*}$, $\psi_n^*= O_{\prob_W^*}(1)$, $\sup_{|\psi|\leq M} | \Gamma_{0,\eta_n^*}(\psi) - \Gamma_{0,\eta_0}(\psi_0)| = o_{\prob_W^*}(1)$ for every $M > 0$, $\|\eta_n^* - \eta_0\|_{\s{H}} = o_{\prob_W^*}(1)$, and $P_0 \phi_{\psi_0, \eta_n^*}  = o_{\prob_W^*}(1)$. If $P_0 \phi_{\psi_0, \eta_n^*}  = o_{\prob_W^*}(n^{-1/2})$, then~\ref{cond: bootstrap second order} holds for the estimating equations-based estimator. If  $P_0 \phi_{\psi_0, \eta_n^*} - 2P_0 \phi_{\psi_0, \eta_n}  = o_{\prob_W^*}(n^{-1/2})$, then $R_n^* - R_n = o_{\prob_W^*}(n^{-1/2})$ holds for the estimating equations-based estimator.
\end{restatable}

As discussed above, $P_0 \phi_{\psi_0, \eta_n^*}  = o_{\prob_W^*}(n^{-1/2})$ is sometimes implied by a sufficient rate of convergence of $\eta_n^*$ to $\eta_0$. In particular, if $\eta_n^* = \eta_n$, then this condition is implied by $P_0\phi_{\psi_0, \eta_n} = o_{\prob_0^*}(n^{-1/2})$. Hence, similar to \Cref{cor: bootstrap one-step fixed} for the bootstrap one-step estimator with fixed nuisance estimator, the empirical bootstrap estimating equations-based estimator with fixed nuisance parameter is consistent under the same conditions used for asymptotic linearity of the estimating equations-based estimator. \Cref{lemma: empirical bootstrap estimating-equations} also shows that $R_n^* - R_n = o_{\prob_W^*}(n^{-1/2})$ can hold even when $P_0\phi_{\psi_0, \eta_n}$ is not $o_{\prob_0^*}(n^{-1/2})$ or $P_0 \phi_{\psi_0, \eta_n^*}$ is not $o_{\prob_W^*}(n^{-1/2})$, as long as $P_0 \phi_{\psi_0, \eta_n^*}$ is within $o_{\prob_W^*}(n^{-1/2})$ of $2P_0 \phi_{\psi_0, \eta_n}$. This is analogous to discussions about validity of bootstrap percentile intervals above.

Consistency of the empirical bootstrap for estimating equations-based estimators was also studied in \cite{wellner1996bootstrapping} and \cite{cheng2010bootstrap}. \cite{wellner1996bootstrapping} generalized Theorem~3.3.1 of \cite{van1996weak} to weighted bootstraps. \cite{cheng2010bootstrap} focused on the case of a semiparametric model where $\psi_n^*$ solves a profile bootstrap estimating equation. As discussed above, our results can be used in this context, but also permit $\eta_n^*$ to be constructed in a different manner than $\eta_n$, and can be applied in nonparametric models as well. In addition, \Cref{lemma: empirical bootstrap estimating-equations} also addresses consistency of certain bootstrap confidence intervals without conditional weak convergence, which was not studied in these earlier papers.

The next result provides sufficient conditions for~\ref{cond: bootstrap second order} for the bootstrap estimating equations-based estimator that is applicable to non-empirical bootstrap distributions. We denote $\hat\psi_n := T(\hat{\eta}_n, \hat{P}_n)$. By definition, we have $R_n^* = \psi_n^* - \psi_n^\circ + \hat{P}_n \phi_{\psi_n^*, \eta_n^*}$, where $\psi_n^\circ := T(\eta_n, \hat{P}_n)$. Below, we define $\hat{G}_{n, \eta}(\psi) := \hat{P}_n \phi_{\psi, \eta}$.

\begin{restatable}{lemma}{lemmasmobootstrapestimating}\label{lemma: smooth bootstrap estimating-equations}
    Let $\hat{P}_n$ be the smooth bootstrap. Suppose the conditions of \Cref{lemma: classical estimating-equations} hold, condition~\ref{cond: bootstrap limited complexity} holds for $\phi_n^* = \phi_{\psi_n^*, \eta_n^*}$, $\psi_n^*= O_{\prob_W^*}(1)$, $\|\eta_n^* - \eta_0\|_{\s{H}} = o_{\prob_W^*}(1)$, $\sup_{|\psi| \leq M} \left| \Gamma_{0, \hat\eta_n}(\psi) - \Gamma_{0, \eta_0}(\psi) \right| = o_{\prob_0^*}(1)$ for every $M > 0$, there exists a map $\hat{G}_{n, \eta}': \d{R} \to \d{R}$ such that $\hat\Gamma_{n,\eta} : \psi \mapsto \hat{G}_{n, \eta}(\psi) - \hat{G}_{n, \eta}(\hat\psi_n) - \hat{G}_{n, \eta}'(\hat\psi_n) (\psi - \hat\psi_n)$ satisfies $\sup_{|\psi|\leq M} | \hat\Gamma_{n, \eta_n^*}(\psi) - \hat\Gamma_{n, \hat\eta_n}(\hat\psi_n)| = o_{\prob_W^*}(1)$ for every $M > 0$ and $\sup_{\eta: \|\eta - \hat\eta_n\|_{\s{H}} < \delta} |\hat\Gamma_{n, \eta}(\psi) | = o_{\prob_0^*}( \psi - \hat\psi_n)$ for some $\delta > 0$, $\hat{G}_{n, \eta}'$ satisfies $\hat{G}_{n, \eta_n^*}'(\hat\psi_n) + 1 = o_{\prob_W^*}(1)$, there exists a $P_0$-Glivenko Cantelli class $\s{F}$ such that $\phi_{\hat\psi_n, \hat\eta_n}$ is contained in $\s{F}$ with probability tending to one, $\phi_{\psi_n^*, \eta_n^*}$ is contained in $\s{F}$ with conditional probability tending to one, and  $\| \hat{P}_n - P_0\|_{\s{F}} =o_{\prob_W^*}(1)$, $\hat\psi_n= O_{\prob_0^*}(1)$, $P_0 \phi_{\psi_0, \hat\eta_n} = o_{\prob_0^*}(1)$, $\| \hat\eta_n - \eta_0\|_{\s{H}} = o_{\prob_0^*}(1)$,  and $\psi_n^\circ - \hat\psi_n = o_{\prob_0^*}(n^{-1/2})$. If $\hat{P}_n \phi_{\hat\psi_n, \eta_n^*}  = o_{\prob_W^*}(n^{-1/2})$, then~\ref{cond: bootstrap second order} holds for the estimating equations-based estimator. If  $\hat{P}_n \phi_{\hat\psi_n, \eta_n^*} - P_0 \phi_{\psi_0, \eta_n}  = o_{\prob_W^*}(n^{-1/2})$, then $R_n^* - R_n = o_{\prob_W^*}(n^{-1/2})$ holds for the estimating equations-based estimator.
\end{restatable}

\Cref{lemma: smooth bootstrap estimating-equations} requires $\hat{P}_n \phi_{\hat\psi_n, \eta_n^*} = o_{\prob_W^*}(n^{-1/2})$ for~\ref{cond: bootstrap second order}, which as discussed above is sometimes implied by sufficient rates of convergence of $\eta_n^* - \hat\eta_n$. However, as in previous results, $R_n^* - R_n = o_{\prob_W^*}(n^{-1/2})$ can hold even if $R_n^*$ is not $o_{\prob_W^*}(n^{-1/2})$ and/or $R_n$ is not $o_{\prob_0^*}(n^{-1/2})$. The conditions $\|\hat{P}_n -P_0\|_{\s{F}} = o_{\prob_0^*}(1)$ and $\prob_0^* ( \phi_{\hat\psi_n, \hat\eta_n} \in \s{F} ) \to 1$ are used to guarantee conditional consistency of $(\psi_n^*, \eta_n^*)$. As discussed in Section~\ref{sec:bootstrap_dist_condition}, if $\hat{P}_n$ is a smoothing through convolution estimator, then $\|\hat{P}_n - P_0\|_{\s{F}} = o_{\prob_0^*}(1)$ is implied by the conditions of Proposition~\ref{prop: weak convergence of convolution smooth bootstrap}. Lastly, \Cref{lemma: smooth bootstrap estimating-equations} requires $\psi_n^\circ - \hat\psi_n = o_{\prob_0^*}(n^{-1/2})$. If $\hat{P}_n$ is based on the original nuisance estimator $\eta_n$ so that $\hat\eta_n = \eta_n$, then $\psi_n^\circ = \hat\psi_n$. Otherwise, $\psi_n^\circ - \hat\psi_n = o_{\prob_0^*}(n^{-1/2})$ follows if $\prob_0^*(\phi_{\psi_n^\circ, \eta_n} \in \s{F}) \to 1$, $\|\phi_{\psi_n^\circ, \eta_n} - \phi_{\psi_0, \eta_0}\|_{L_2(P_0)} = o_{\prob_0^*}(1)$, and $n^{1/2}(\hat{P}_n - P_0)$ converges weakly to a tight measurable  limit in $\ell^\infty(\s{F})$. For the special case of $\hat{P}_n$ equal to the distribution corresponding to a kernel density estimator, Section~3.2 of \cite{gine2008uniform} and Theorem~10 of \cite{radulovic2009uniform} establish weak convergence of $n^{1/2}(\hat{P}_n - P_0)$ under conditions on the bandwidth and smoothness of functions $f \in \s{F}$.

\section{Applications of the general theory}\label{sec: applications}
\subsection{Average density value}\label{sec: av dens}

We now illustrate the use of our general results for two bootstrap strategies and three estimators of the \emph{average density value} parameter. This example has been used extensively as a test case for semiparametric theory and methods (see, e.g., \citealp{carone2018toward, cai2020highly}, and \citealp{cattaneo2022average}, among others). We suppose that $\s{X} \subseteq \d{R}^d$, and we let $\s{M}$ be the set of probability measures on $\s{X}$ dominated by Lebesgue measure $\lambda$. For $P \in \s{M}$ we let $\eta_P := \sd P / \sd\lambda$ be the Lebesgue density function of $P$. We then define the average density value parameter as $\psi(P) := \int_{\s{X}} \eta_P(x)^2 \sd x$. In this example, the nuisance parameter is the density function $\eta_P \in \s{H}:= \{h \in L_1(\lambda):  h\geq 0, \int_{\s{X}} h(x) \sd x = 1\}$. The nonparametric efficient influence function of $\psi$ is $\phi_P(x) := 2\eta_P(x) - 2\psi(P)$.  For any $\alpha := (\alpha_1, \dotsc, \alpha_d) \in \d{N}^d$ and $x \in \d{R}^d$, we define $|\alpha| := \sum_{i=1}^d \alpha_i$, $\alpha! := \prod_{i=1}^d \alpha_i$, and $x^{\alpha} := \prod_{i=1}^d x_i^{\alpha_i}$. For any suitable function $f: \d{R}^d \to \d{R}$, we denote $D^{\alpha} f := \frac{\partial^{\alpha} f}{\partial x_1^{\alpha_1} \cdots \partial x_d^{\alpha_d} }$. For deterministic sequences $r_n$ and $s_n$, we 
say $r_n \prec s_n$ if $r_n / s_n = o(1)$.

We consider three approaches to constructing an asymptotically linear estimator of $\psi_0$. We let $\eta_n$ be an estimator of the density $\eta_0$. First, we consider the one-step estimator discussed in Section~\ref{sec: one-step}. In this case, $T_1(\eta, P) = \psi(\eta) + P \phi_\eta = 2P \eta - \int \eta^2$, so that the one-step estimator is $\psi_{n, 1} := T_1(\eta_n, \d{P}_n) = 2\d{P}_n \eta_n - \int \eta_n^2$.  Second, we consider the plug-in estimator discussed in Section~\ref{sec: plug-in}. In this case, $T_2(\eta, P) = \int \eta(x)^2 \sd x$, so that the plug-in estimator is given by $\psi_{n, 2} := T_2(\eta_n, \d{P}_n) = \int \eta_n(x)^2 \sd x$. Third, we consider the mean of a nuisance-dependent function discussed in Section~\ref{sec: mean plug-in}. In this case, $g_\eta(x) := \eta(x)$ and $T_3(\eta, P) = P \eta$, so that the  the estimator is given by $\psi_{n, 3} := T_3(\eta_n, \d{P}_n) = \d{P}_n \eta_n$. We note that $\psi_{n, 3}$ can also be viewed as an estimating equations-based estimator (discussed in Section~\ref{sec: estimating-equations}) because it solves $\psi \mapsto \d{P}_n \phi_{\psi, \eta_n} = 2\d{P}_n \eta_n - 2\psi = 0$.

The next result uses Theorem~\ref{thm: classical} and the derivations in Section~\ref{sec: T examples} to provide conditions on $\eta_n$ under which these three estimators are asymptotically linear.

\begin{restatable}{prop}{propclassicalaverage}\label{prop: classical average}
   Suppose that $\eta_n$ falls in a $P_0$-Donsker class with probability tending to one, there exists $M \in (0,\infty)$ such that $ \|\eta_0\|_\infty < M$ and $\|\eta_n\|_\infty < M$ with probability tending to one, and $\norm{\eta_n - \eta_0}_{L_2(\lambda)} = o_{\prob_0^*}(n^{-1/4})$. Then $\psi_{n, 1}$ is asymptotically linear with influence function $\phi_0$.  If in addition $\int \eta_n^2 - \d{P}_n \eta_n = o_{\prob_0^*}(n^{-1/2})$, then $\psi_{n, 2}$ and $\psi_{n, 3}$ are asymptotically linear as well.
\end{restatable}

We now discuss the conditions of Proposition~\ref{prop: classical average} for the specific case where $\eta_n$ is a kernel density estimator (KDE) with bandwidth $h \in \d{R}$ (depending on $n$) and kernel function $K: \d{R}^d \to \d{R}$.  For KDEs, the Donsker condition is satisfied if $\eta_n$ falls in a class of functions with uniformly bounded partial derivatives up to order $\ell > d/2$ and $P_0$ satisfies a tail bound \citep[Example~19.9]{van2000asymptotic}. 
If $\eta_0$ is $m > d/2$ times differentiable,  $\int [D^{\alpha} \eta_0(x)]^2 \sd x < \infty$ for all $\alpha \in \d{N}^d$ with  $|\alpha| = m$ and $m$th order kernels are used, then $\norm{\eta_n - \eta_0}_{L_2(\lambda)} = O_{\prob_0^*}(\{nh^d\}^{-1/2} + h^{m})$. Hence, if $n^{-1/(2d)} \prec h \prec n^{-1/(4m)}$, 
then $\norm{\eta_n - \eta_0}_{L_2(\lambda)} = o_{\prob_0^*}(n^{-1/4})$. These conditions are satisfied if $h$ is selected at the optimal rate $h \propto n^{-1/(2m+d)}$, which yields $\norm{\eta_n - \eta_0}_{L_2(\lambda)} = O_{\prob_0^*}(n^{-m/(2m+d)})$.

Proposition~\ref{prop: classical average} requires the extra condition that $\int \eta_n^2 - \d{P}_n \eta_n = o_{\prob_0^*}(n^{-1/2})$ for asymptotic linearity of the plug-in and empirical mean plug-in estimators. This illustrates the excess bias sometimes incurred by these methods. For $\eta_n$ a KDE, if $\eta_0$ is $m$ times differentiable and $m$th order kernels are used, then $\int \eta_n^2 - \d{P}_n \eta_n = O_{\prob_0^*}(\{nh^d\}^{-1} + h^m)$. Hence, the condition is satisfied if $m > d$ and $n^{-1/(2d)} \prec h \prec n^{-1/(2m)}$. 
This requires that $h$ go to zero faster than the optimal rate $h \propto n^{-1/(2m+d)}$, which is why this method is known as \emph{under-smoothing}. It also requires more smoothness of $\eta_0$. If $\eta_0$ is only assumed to be $m \in (0, d]$ times differentiable, then there is no choice of $h$ that makes $\int \eta_n^2 - \d{P}_n \eta_n = o_{\prob_0^*}(n^{-1/2})$, so in this case it may not be possible to make the bias term for $\psi_{n, 2}$  or $\psi_{n, 3}$ asymptotically negligible using a standard KDE.

We now turn to methods of bootstrapping $\psi_{n,1}$, $\psi_{n,2}$, and $\psi_{n,3}$. We let $\eta_n^*$ be a bootstrap estimator of $\eta_0$, and we consider the bootstrap estimators $\psi_{n, 1}^* = T_1(\eta_n^*, \d{P}_n^*) = 2\d{P}_n^* \eta_n^* - \int \eta_n^{*2}$, $\psi_{n, 2}^* = T_2(\eta_n^*, \d{P}_n^*) = \int \eta_n^{*2}(x) \sd x$, and $\psi_{n, 3}^* = T_3(\eta_n^*, \d{P}_n^*) = \d{P}_n^* \eta_n^*$. The next result uses Theorem~\ref{thm: bootstrap} and the derivations in Section~\ref{sec: T examples} to provide conditions for conditional asymptotic linearity of these three estimators.

\begin{restatable}{prop}{propempbootstrapaverage}\label{prop: empirical bootstrap average}
    Suppose $\hat{P}_n = \d{P}_n$ is the empirical bootstrap and the assumptions of Proposition~\ref{prop: classical average} hold. For $\s{F}$ and $M$ defined in Proposition~\ref{prop: classical average}, suppose $\prob_W^*(\eta_n^* \in \s{F}) \inoutprob 1$, $\prob_W^*(\|\eta_n^*\|_\infty \geq M) = o_{\prob_0^*}(1)$, and $\norm{\eta_n^* - \eta_n}_{L_2(\lambda)} = o_{\prob_W^*}(n^{-1/4})$. Then $\psi_{n, 1}^* = \psi_{n,  1} + (\d{P}_n^* - \d{P}_n)\phi_0 + o_{\prob_W^*}(n^{-1/2})$. If in addition $\int \eta_n^{*2} - \d{P}_n^* \eta_n^* = o_{\prob_W^*}(n^{-1/2})$, then $\psi_{n, 2}^* = \psi_{n,  2} + (\d{P}_n^* - \d{P}_n)\phi_0 + o_{\prob_W^*}(n^{-1/2})$ and $\psi_{n, 3}^* = \psi_{n,  3} + (\d{P}_n^* - \d{P}_n)\phi_0 + o_{\prob_W^*}(n^{-1/2})$. If $\eta_n$ and $\eta_n^*$ are KDEs with the same kernel and bandwidth $h$ such that $n^{-1/(2d)} \prec h$, 
    then $\norm{\eta_n^* - \eta_n}_{L_2(\lambda)} = o_{\prob_W^*}(n^{-1/4})$. 
\end{restatable}

Proposition~\ref{prop: empirical bootstrap average} requires that $\eta_n^*$ converge fast enough to $\eta_n$ conditional on the data. If $\eta_n^* = \eta_n$, then this condition holds automatically. Hence, the empirical bootstrap for the one-step estimator with fixed nuisance is consistent if the conditions of Proposition~\ref{prop: classical average} hold, as guaranteed by \Cref{cor: bootstrap one-step fixed}. If $\eta_n^*$ is estimated using the bootstrap sample, some care must be taken to ensure that $\norm{\eta_n^* - \eta_n}_{L_2(\lambda)} = o_{\prob_W^*}(n^{-1/4})$. If the bootstrap nuisance estimator is sensitive to ties in the data, this condition might not hold. However, Proposition~\ref{prop: empirical bootstrap average} demonstrates that for KDEs, fixing the bandwidth at the value selected by the original data, as suggested by \cite{hall2001bootstrapping},  \cite{cattaneo2022average}, and others, yields $\norm{\eta_n^* - \eta_n}_{L_2(\lambda)} = o_{\prob_W^*}(n^{-1/4})$. 

Proposition~\ref{prop: empirical bootstrap average} requires the extra condition  $\int \eta_n^{*2} - \d{P}_n^* \eta_n^* = o_{\prob_W^*}(n^{-1/2})$ for conditional asymptotic linearity of the empirical bootstrap plug-in and empirical mean estimators, which is analogous to the condition required by Proposition~\ref{prop: classical average} for asymptotic linearity of these estimators. If $\eta_n^* = \eta_n$, then this condition reduces to $\int \eta_n^2 - \d{P}_n^* \eta_n = o_{\prob_W^*}(n^{-1/2})$, which does not hold because we can write $n^{1/2}\left(\int \eta_n^2 - \d{P}_n^* \eta_n\right) = n^{1/2}\left(\int \eta_n^2 - \d{P}_n \eta_n\right) - \d{G}_n^* \eta_n$, and $\d{G}_n^* \eta_n$ converges weakly to a non-degenerate limit conditional on the data. Hence, fixing the bootstrap nuisance does not yield a conditionally asymptotically linear plug-in or empirical mean plug-in estimator for the empirical bootstrap for this parameter. If $\eta_0$ is $m > d/2$ times differentiable with $\int[D^{\alpha} \eta_0 (x)]^2 \sd x < \infty$ for all $|\alpha| = m$, $\prob_0^*(\eta_n \in \s{F}) \to 1$, and $\eta_n^*$ is a KDE based on the bootstrap data with deterministic bandwidth $h$ and $m$th order kernel functions, then $\int \eta_n^{*2} - \d{P}_n^* \eta_n^* = O_{\prob_W^*}(\{nh^d\}^{-1} + h^{m})$. Hence, $\int \eta_n^{*2} - \d{P}_n^* \eta_n^* = o_{\prob_W^*}(n^{-1/2})$ if $n^{-1/(2d)} \prec h \prec n^{-1/(2m)}$.
This again requires $m > d$ and under-smoothing the bootstrap nuisance estimator.

The next result provides conditions under which the empirical bootstrap percentile method is asymptotically valid.

\begin{restatable}{prop}{propautobiasempbootstrapaverage}\label{prop: auto bias correction}
    Suppose $\hat{P}_n = \d{P}_n$ is the empirical bootstrap, $\eta_0$ is uniformly bounded and $m$-times continuously differentiable with $\int[D^{\alpha} \eta_0 (x)]^2 \sd x < \infty$ for all $|\alpha| = m$, both $\eta_n$ and $\eta_n^*$ are KDEs with common symmetric $m$th order kernel function and common bandwidth $h$ such that $n^{-1/d} \prec h$.  
    If  $h \prec n^{-1/(4m)}$, 
    then bootstrap percentile  intervals based on $\psi_{n,1}^*$ are asymptotically valid. If $h \prec n^{-1/(2m)}$, 
    then bootstrap percentile intervals based on $\psi_{n,2}^*$ and $\psi_{n,3}^*$ are asymptotically valid. 
\end{restatable}

The results of \Cref{prop: auto bias correction} agree with those of \cite{cattaneo2022average}, though they consider a different one-step estimator than ours. The conditions for the one-step estimator in \Cref{prop: auto bias correction} can be satisfied if $m > d/4$, and the conditions for the plug-in and empirical mean plug-in estimators can be satisfied if $m > d/2$. Both of these conditions are weaker than the conditions for (conditional) asymptotic linearity from \Cref{prop: empirical bootstrap average}, and the conditions for the one-step estimator are again weaker than the plug-in estimators. The conditions for the plug-in and empirical mean plug-in estimators still require undersmoothing. For the one-step estimator, if $m \in (d/4, d/2]$, then satisfying $h \prec n^{-1/(4m)}$ 
also requires under-smoothing.

In the common case that $\eta_0$ is assumed to be $m = 2$ times differentiable and a second-order kernel is used, the (empirical bootstrap) one-step estimator is (conditionally) asymptotically linear for $d \leq 3$ if $n^{-1/(2d)} \prec h \prec n^{-1/8}$, 
so that the optimal bandwidth may be used. However, bootstrap percentile intervals based on the one-step estimator are asymptotically valid for $d \leq 7$ as long as $n^{-1/d} \prec h \prec n^{-1/8}$, 
which requires under-smoothing for $4 \leq d \leq 7$, but not for $d \leq 3$. The (empirical bootstrap) plug-in and empirical mean plug-in estimators are (conditionally) asymptotically linear only for $d = 1$ if $n^{-1/2} \prec h \prec n^{-1/4}$, 
which requires under-smoothing. Empirical bootstrap percentile intervals based on the plug-in and empirical mean plug-in estimators are asymptotically valid for $d \leq 3$ if $n^{-1/d} \prec h \prec n^{-1/4}$, 
which again requires under-smoothing.

The next result provides conditions for conditional asymptotic linearity of the three bootstrap estimators when $\hat{P}_n$ is a smooth bootstrap distribution. 
\begin{restatable}{prop}{propsmobootstrapaverage}\label{prop: smooth bootstrap average}
    Suppose $\hat{P}_n$ possesses Lebesgue density function $\hat{\eta}_n$ and the assumptions of Proposition~\ref{prop: uniform donsker} and Proposition~\ref{prop: classical average} hold. For $\s{F}$ and $M$ defined in Proposition~\ref{prop: classical average}, assume that $\s{F}$ is a $\s{M}$-uniform Donsker class such that $\prob_W^*(\eta_n^* \in \s{F}) \inoutprob 1$ and  $\prob_W^*(\|\hat{\eta}_n\|_\infty \geq M) = o_{\prob_0^*}(1)$. If  $\norm{\eta_n^* - \hat{\eta}_n}_{L_2(\lambda)} = o_{\prob_W^*}(n^{-1/4})$ and $\norm{\hat{\eta}_n - \eta_n}_{L_2(\lambda)} = o_{\prob_0^*}(n^{-1/4})$, then $\psi_{n, 1}^* = \psi_{n, 1} + (\d{P}_n^* - \hat{P}_n)\phi_0 + o_{\prob_0^*}(n^{-1/2})$. If in addition $\int \eta_n^{*2} - \d{P}_n^* \eta_n^* = o_{\prob_W^*}(n^{-1/2})$ and $\int \eta_n^2 - \hat{P}_n \eta_n = o_{\prob_0^*}(n^{-1/2})$, then $\psi_{n, 2}^* = \psi_{n, 2} + (\d{P}_n^* - \hat{P}_n)\phi_0 + o_{\prob_0^*}(n^{-1/2})$ and $\psi_{n, 3}^* = \psi_{n, 3} + n^{1/2}(\d{P}_n^* - \hat{P}_n)\phi_0 + o_{\prob_0^*}(n^{-1/2})$.  
\end{restatable}

Proposition~\ref{prop: smooth bootstrap average} requires that $\eta_n^*$ converge fast enough given the data to the density $\hat{\eta}_n$ used for generating the bootstrap data. If $\eta_n^*$ is an estimator based on the bootstrap sample, then this can again be achieved by many nonparametric estimators under mild smoothness conditions. For example, if $\eta_0$ is $m > d/2$ times differentiable with $\int[D^{\alpha} \eta_0 (x)]^2 \sd x < \infty$ for all $|\alpha| = m$, $\prob_0^*(\hat{\eta}_n \in \s{F}) \to 1$, and $\eta_n^*$ is a kernel density estimator with bandwidth $h^*$ and $m$th order kernel functions, then $\norm{\eta_n^* - \hat{\eta}_n}_{L_2(\lambda)} = O_{\prob_W^*}(\{nh^{*d}\}^{-1/2} + h^{*m})$. Hence, if $n^{1/(2d)} \prec h^* \prec n^{-1/(4m)}$, 
then $\norm{\eta_n^* - \hat{\eta}_n}_{L_2(\lambda)} = o_{\prob_W^*}(n^{-1/4})$. Proposition~\ref{prop: smooth bootstrap average} also requires that $\eta_n^*$ falls in a $\s{P}$-uniform Donsker class.  
Furthermore, Proposition~\ref{prop: smooth bootstrap average} requires that the conditions of Proposition~\ref{prop: uniform donsker} hold. If $\hat{\eta}_n$ is a kernel density estimator, then  Proposition~\ref{prop: weak convergence of convolution smooth bootstrap} can be used to establish the conditions of Proposition~\ref{prop: uniform donsker} and the uniform Donsker condition.

Proposition~\ref{prop: smooth bootstrap average} requires the extra conditions $\int \eta_n^{*2} - \d{P}_n^* \eta_n^* = o_{\prob_W^*}(n^{-1/2})$ and $\int \eta_n^{2} - \hat{P}_n \eta_n = o_{\prob_0^*}(n^{-1/2})$ for consistency of the smooth bootstrap plug-in and empirical mean plug-in estimators. If $\eta_0$ is $m \geq 2$ times differentiable, $\eta_n^*$ and $\hat\eta_n$ are KDEs with bandwidths $h^*$ and $\hat{h}$, respectively, an $m$th order kernel function is used for $\eta_n^*$, $n^{-1/(2d)} \prec h^* \prec n^{-1/(4m)}$, $\hat{h} \prec n^{-1/(2m)}$, and $\prob_0^*(\eta_n \in \s{F}) \to 1$, then $\int \eta_n^{*2} - \d{P}_n^* \eta_n^* = o_{\prob_W^*}(n^{-1/2})$ and $\int \eta_n^{2} - \hat{P}_n \eta_n = o_{\prob_0^*}(n^{-1/2})$.
 Hence, consistency of the smooth bootstrap for the plug-in or empirical mean plug-in estimators with an under-smoothed nuisance estimator also requires under-smoothing the bootstrap sampling distribution.

The next result provides conditions under which the bootstrap percentile method is asymptotically valid when $\hat{P}_n$ is the distribution corresponding to the kernel density estimator $\eta_n$.

\begin{restatable}{prop}{propautobiassmoothbootstrapaverage}\label{prop: auto bias correction smooth}
    Suppose that $\eta_0$ is uniformly bounded and $m$-times continuously differentiable, and for all $|\alpha| = m$, $D^{\alpha} \eta_0$ is uniformly bounded and $\int[D^{\alpha} \eta_0 (x)]^2 \sd x < \infty$. If both $\eta_n$ and $\eta_n^*$ are kernel density estimators with common uniformly bounded symmetric $m$th order kernel function $K$ and common bandwidth $h$, $\hat{P}_n$ is the distribution corresponding to $\eta_n$, and $n^{-1/(2d)} \prec h \prec n^{-1/(4m)}$, 
    then bootstrap percentile intervals based on $\psi_{n,1}^*$, $\psi_{n,2}^*$, and $\psi_{n,3}^*$ are asymptotically valid. 
\end{restatable}

To the best of our knowledge, \Cref{prop: auto bias correction smooth} is the first result establishing automatic bias correction of bootstrap confidence intervals using the smooth bootstrap. The bandwidth condition   $n^{-1/(2d)} \prec h \prec n^{-1/(4m)}$ 
in \Cref{prop: auto bias correction smooth} can be satisfied if $m > d /2$. However, the condition is different from the bandwidth condition of \Cref{prop: auto bias correction} for the empirical bootstrap in several interesting ways. For the one-step estimator, the conditions of \Cref{prop: auto bias correction smooth} are the same as those used for (conditional) asymptotic linearity of the one-step estimator. Hence, unlike the empirical bootstrap, it does not appear that the smooth bootstrap produces valid confidence intervals based on the one-step estimator under weaker smoothness or dimension requirements than non-bootstrap Wald intervals. However, for the plug-in and empirical mean plug-in estimators, the requirements of \Cref{prop: auto bias correction smooth} are weaker than those required for (conditional) asymptotic linearity of the estimators because they require that 
$h \prec n^{-1/(4m)}$ rather than $h \prec n^{-1/(2m)}$. 
Hence, both empirical and smooth bootstrap percentile confidence intervals based on the plug-in estimators can be valid if $m > d/2$. However, the conditions for the plug-in estimators based on the smooth bootstrap are satisfied if the optimal bandwidth is used, while the conditions for the empirical bootstrap require under-smoothing. Thus, in this case, the empirical bootstrap is preferable for the one-step estimator, while the smooth bootstrap is preferable for the plug-in and empirical mean plug-in estimators. We emphasize that it is not presently clear whether these conclusions would remain true with other nuisance estimators or smooth bootstrap sampling distributions, or for other parameter mappings.

\subsection{G-computed conditional mean}

The second parameter we will use to illustrate the use of our general results is the G-computed conditional mean. Suppose that $\s{X} = \d{R}\times\{0, 1\} \times \d{R}^d$ and $X = (Y, A, Z)$, where $Y \in \d{R}$ is an outcome of interest, $A \in \{0, 1\}$ is a binary treatment or exposure, and $Z \in \d{R}^d$ is a vector of adjustment covariates. We then define the G-computed conditional mean as $\psi(P) = \E_P[\mu_{P}(Z) \mid A=1]$, where $\mu_{P}(z) := \E_P (Y \mid A=0, Z=z)$. Under the no unobserved confounding causal model, $\psi_0$ corresponds to the mean outcome among treated units (i.e., those with $A = 1$) had they been assigned to receive control $A = 0$ \citep{robins1986new, gill2001causal}. We use this parameter as an example rather than the simpler G-computed mean, $E_P[\mu_P(Z)]$ because the one-step and estimating equations-based estimators are different for the conditional mean, which gives us the chance to illustrate the use of our results for estimating equations-based estimators.

The efficient influence function of $\psi$ at $P$ relative to a nonparametric model is given by 
\[
    \phi_P(y, a, z) = \frac{I(a=0) g_P(z)}{\pi_P \left[1-g_P(z)\right]} \left[y - \mu_{P}(z)\right] + \frac{I(a=1)}{\pi_P} \left[\mu_{P}(z) - \psi_P\right],
\]
where $g_P(z) := P(A=1 \mid Z=z)$ is the propensity score function and $\pi_P := P(A=1)$.  In this example, the nuisance parameter is $\eta_P = (\mu_P, g_P, Q_P)$, where $Q_P$ is the marginal distribution of $Z$ under $P$, and $\psi_P$ and $\pi_P$ are defined through $\eta_P$ as $\psi_P = \int \mu_{P}(z) g_P(z) \pi_P^{-1} \sd Q_P(z)$ and $\pi_P = \int g_P \sd Q_P$.

We consider two approaches to constructing an asymptotically linear estimator of $\psi_0$. We let $\eta_n = (\mu_n, g_n, Q_n)$ be an estimator of the nuisance $\eta_0$, where $Q_n$ is the marginal empirical distribution of $Z$. First, we consider the one-step estimator discussed in Section~\ref{sec: one-step}, which is given by 
\begin{align*}
    \psi_{n, 1} & = \frac{1}{n}\sum_{i=1}^n \left\{ \frac{I(A_i = 0) g_n(Z_i)}{ \pi_n [1-g_n(Z_i)]}[Y_i - \mu_{n}(Z_i)] + \left[ 2 - \frac{\bar\pi_n}{\pi_n}\right] \frac{I(A_i = 1)}{\pi_n} \mu_{n}(Z_i)\right\},
\end{align*}
where $\pi_n :=\int g_n \sd Q_n$ and $\bar\pi_n := \tfrac{1}{n}\sum_{i=1}^n A_i$. Second, we consider the estimating equations-based estimator discussed in Section~\ref{sec: estimating-equations}. We define the estimating function $G_{P,\eta}(\psi) := P \phi_{\psi, \eta}$, and we note that with $\eta = (\mu, g, Q)$ and $\pi := \int g \sd Q$, 
\[ G_{0,\eta}(\psi) = \pi^{-1} P_0 \left[ \frac{(g- g_0)(\mu - \mu_0)}{1 - g} \right] + \frac{\pi_0}{\pi} (\psi_0 - \psi).\]
In particular, $G_0(\psi) = \psi_0 - \psi$, so $\psi_0$ is the unique solution to the population estimating equation, and $G_{0,\eta}'(\psi) = -\pi_0 / \pi$, which approaches $-1$ as $\eta \to \eta_0$. We then define the estimating equations-based estimator $\psi_{n, 2} = T_2(\eta_n, \d{P}_n)$ as the solution to the sample estimating function $G_{n,\eta_n}(\psi) := \d{P}_n \phi_{\psi, \eta_n}$, which it is easy to see equals
\begin{align*}
    \psi_{n, 2} = \frac{1}{n}\sum_{i=1}^n \left\{ \frac{I(A_i = 0) g_n(Z_i)}{\bar\pi_n [1-g_n(Z_i)]} [Y_i - \mu_{n}(Z_i)] + \frac{I(A_i = 1)}{\bar\pi_n} \mu_{n}(Z_i) \right\}.
\end{align*}
We note that if $\pi_n = \bar\pi_n$, then $\psi_{n, 1}=\psi_{n, 2}$. The next result provides conditions under which these two estimators are asymptotically linear using Theorem~\ref{thm: classical}. 
\begin{restatable}{prop}{propclassicalcounterfactual}\label{prop: classical counterfactual}
    If $\mu_n$, $g_n$, and $(y,a,z) \mapsto (1-a)y g_n(z) / [1-g_n(z)]$ fall in $P_0$-Donsker classes with probability tending to 1, $\E_0 (Y^2) < \infty$, there exists constants $0 < a < b < 1$ such that $P_0(g_0(Z) \in (a,b)) = 1$, $P_0(g_n(Z) \in (a,b)) = 1$, and $P_0( |\mu_n(Z)| \leq b) = 1$, $\|g_n - g_0\|_{L_2(P_0)} = o_{\prob_0^*}(1)$, $\|\mu_{n} - \mu_{0}\|_{L_2(P_0)} = o_{\prob_0^*}(1)$, and $P_0 \left\{ (g_n - g_0)(\mu_{n}- \mu_{0}) / (1-g_n)\right\} = o_{\prob_0^*}(n^{-1/2})$, then $\psi_{n, 2}$ is asymptotically linear with influence function $\phi_0$. If in addition  $(\pi_n-\pi_0)(\psi_n - \psi_0) = o_{\prob_0^*}(n^{-1/2})$, then $\psi_{n,1}$ is asymptotically linear with influence function $\phi_0$.
\end{restatable}

We now turn to methods of bootstrapping $\psi_{n, 1}$ and $\psi_{n, 2}$. We define $\eta_n^* = (\mu_n^*, g_n^*, Q_n^*)$ as a bootstrap estimator of $\eta_0$ based on $n$ bootstrap observations, where $Q_n^*$ is the empirical distribution of the bootstrap covariates. We then consider $\psi_{n, 1}^* = T_1(\eta_n^*, \d{P}_n^*)$ and $\psi_{n, 2}^* = T_2(\eta_n^*, \d{P}_n^*)$. The next result provides conditions under which these estimators are conditionally asymptotically linear for the empirical bootstrap. 

\begin{restatable}{prop}{propempbootstrapcounterfactual}\label{prop: empirical bootstrap counterfactual}
    Suppose $\hat{P}_n = \d{P}_n$ is the empirical bootstrap and the conditions of \Cref{prop: classical counterfactual} hold. If $\mu_n^*$, $g_n^*$, and $(y,a,z) \mapsto (1-a)y g_n^*(z) / [1-g_n^*(z)]$ fall in $P_0$-Donsker classes with conditional probability tending to 1, there exist constants $0 < a < b < 1$ such that  $P_W^*(g_n^*(Z) \in (a,b)) = 1$, and $P_W^*( |\mu_n^*(Z)| \leq b) = 1$, $\|g_n^* - g_0\|_{L_2(P_0)} = o_{\prob_W^*}(1)$, $\|\mu_{n}^* - \mu_{0}\|_{L_2(P_0)} = o_{\prob_W^*}(1)$, and $P_0 \left\{ (g_n^* - g_0)(\mu_{n}^*- \mu_{0}) / (1-g_n^*)\right\} = o_{\prob_W^*}(n^{-1/2})$, then $\psi_{n, 2}^*$ is conditionally asymptotically linear with influence function $\phi_0$. If in addition  $(\pi_n^*-\pi_0)(\psi_n^* - \psi_0) = o_{\prob_W^*}(n^{-1/2})$, then $\psi_{n,1}^*$ is conditionally asymptotically linear with influence function $\phi_0$.
\end{restatable}

\Cref{prop: empirical bootstrap counterfactual} requires that $\mu_n^*$ and $g_n^*$ converge fast enough to $\mu_0$ and $g_0$, respectively, conditional on the data. 
We also note that, as above, percentile confidence intervals based on the empirical bootstrap might be consistent under weaker conditions than \Cref{prop: empirical bootstrap counterfactual}.

The next result addresses the case where $\hat{P}_n$ is a non-empirical bootstrap sampling distribution. We let $\hat{Q}_n$ be the marginal distribution of $Z$ under $\hat{P}_n$, $\hat{g}_n(z) := \hat{P}_n(A=1 | Z=z)$, $\hat{\mu}_{n}(z) := \E_{\hat{P}_n}(Y | A=0, Z=z)$, and $\hat\sigma_n^2(z) := \mathrm{Var}_{\hat{P}_n}(Y \mid A = 0, Z = z)$, all of which we assume are well-defined. We note that $\hat{Q}_n$ need not be a smooth distribution. We also define $\sigma_0^2(z) :=  \mathrm{Var}_{0}(Y \mid A = 0, Z = z)$. We then have the following result regarding conditional asymptotic linearity of the bootstrap one-step and estimating equations-based estimators when sampling from $\hat{P}_n$.

\begin{restatable}{prop}{propsmobootstrapcounterfactual}\label{prop: smooth bootstrap counterfactual}
    Suppose $P_0^*(\hat{P}_n \in \s{P}) \to 1$, where $\s{P}$ is such that $\lim_{M \to \infty} \sup_{P \in \s{P}} E_P [ Y^2 I( Y^2 > M)] = 0$, $\hat\mu_n \in \s{F}_\mu$ and $\hat{g}_n \in \s{F}_g$ with probability tending to one, where $\s{F}_\mu$ is uniformly bounded, $\s{F}_g$ is uniformly bounded away from zero, and $\s{F}_\mu$ and $\s{F}_g$ possess finite uniform entropy integrals, $\| \hat{g}_n - g_0 \|_{L_2(P_0)}$,  $\| \hat{\mu}_n - \mu_0 \|_{L_2(P_0)}$, $\| \hat\sigma_n^2 - \sigma_0^2 \|_{L_2(P_0)}$ are each  $o_{\prob_0^*}(1)$, and each of the following is $o_{\prob_0^*}(1)$:
    \begin{align}\label{eq:hatQcond}
    \begin{split}
        \sup_{g, \bar{g}} \left| (\hat{Q}_n - Q_0) \left[ \frac{g\bar{g}(1-\hat{g}_n)}{(1-g) (1-\bar{g})}(\hat\sigma_n^2 + \hat\mu_n^2) \right]  \right|, \, \sup_{g, \bar{g}, \mu, \bar\mu}\left|(\hat{Q}_n - Q_0) \left[ \frac{g\bar{g}(1-\hat{g}_n)}{(1-g) (1-\bar{g})} \mu \bar\mu \right] \right|,\\
        \sup_{\mu} \left| (\hat{Q}_n - Q_0) \left[  \mu  \hat{g}_n  \right] \right|,  \, \sup_{\mu ,\bar\mu}\left| (\hat{Q}_n - Q_0) \left[ \mu  \bar\mu \hat{g}_n  \right] \right|,  \,\sup_{\mu, g}\abs{ (\hat{Q}_n - Q_0) \left[\frac{g(1-\hat{g}_n)}{1-g}\mu \right]}, \, (\hat{Q}_n - Q_0) \hat{g}_n. 
    \end{split}
    \end{align}
    where the suprema over $\mu$ and $\bar\mu$ are taken over $\s{F}_{\mu}$ and the suprema over $g$ and $\bar{g}$ are taken over $\s{F}_g$, and $\hat{Q}_n \left\{ (\hat{g}_n - g_n)(\hat\mu_{n} - \mu_n) / (1-g_n)\right\} = o_{\prob_0^*}(n^{-1/2})$.  Suppose also that $\mu_n^* \in \s{F}_\mu$ and $g_n^* \in \s{F}_g$ with conditional probability tending to one,  $\left\| g_n^* - g_0 \right\|_{L_2(P_0)} = o_{\prob_W^*}(1)$, $ \left\| \mu_n^* - \mu_0 \right\|_{L_2(P_0)} = o_{\prob_W^*}(1)$, and $\hat{Q}_n \left\{ (g_n^* - \hat{g}_n)(\mu_{n}^*- \hat\mu_{n}) / (1-g_n^*)\right\} = o_{\prob_W^*}(n^{-1/2})$. If $(\pi_n^* - \hat\pi_n) ( \psi_n^* - \hat\psi_n) =o_{\prob_W^*}(n^{-1/2})$ and $(\pi_n-\hat\pi_n)(\psi_n - \hat\psi_n) = o_{\prob_0^*}(n^{-1/2})$, then $\psi_{n,1}^*$ is conditionally asymptotically linear with influence function $\phi_0$. If the conditions of \Cref{prop: classical counterfactual} hold and $\pi_n^* - \hat\pi_n = o_{\prob_W^*}(1)$, then $\psi_{n,2}^*$ is conditionmally asymptotically linear with influence function $\phi_0$.
\end{restatable}

\Cref{prop: smooth bootstrap counterfactual} illustrates that the bootstrap sampling distribution $\hat{P}_n$ can produce valid bootstrap confidence intervals even if it is not globally consistent. In this case, it is sufficient that the propensity score, conditional mean, and conditional variance functions induced by $\hat{P}_n$ be consistent, and that certain means of the marginal distribution of the covariates $\hat{Q}_n$ be consistent. If $\hat{Q}_n = Q_n$ is the empirical distribution of the observed covariates, then the conditions in~\eqref{eq:hatQcond} hold by the assumption that $\s{F}_\mu$ and $\s{F}_g$ possess finite uniform entropy integrals. However, it may be of interest to use something other than the empirical distribution for $\hat{Q}_n$ in order to, for instance, produce unique bootstrap covariate values. We also note that it is possible that stronger notions of consistency of $\hat{P}_n$ have implications for higher-order properties of bootstrap confidence intervals. Finally, we note that some of the conditions of \Cref{prop: smooth bootstrap counterfactual} hold automatically if $\hat\mu_n = \mu_n$ or $\hat{g}_n = g_n$.

\section{Numerical study}\label{sec: simulation}

We conducted a simulation study to assess the finite-sample performance of the methods of inference for the average density value parameter studied in  Section~\ref{sec: av dens}. We set $P_0$ as the standard normal distribution. 
For each sample size $n \in \{50, 100, 200, 300, 400, 500, 1000, 2000, 3000, 4000, 5000\}$, we simulated 1000 datasets of $n$ independent and identically distributed observations from $P_0$. 
For each dataset, we considered the three estimator constructions defined in Section~\ref{sec: applications}: the one-step estimator $\psi_{n,1} = T_1(\eta_n, \d{P}_n)$, the plug-in estimator $\psi_{n,2} = T_2(\eta_n, \d{P}_n)$, and the empirical mean plug-in estimator $\psi_{n,3} = T_3(\eta_n, \d{P}_n)$. For each estimator, we used three different nuisance estimators $\eta_n$: (1) a KDE with Gaussian kernel and bandwidth $h$ selected at the optimal rate $h \propto n^{-1/5}$ for twice-differentiable densities using the method of \cite{sheather1991bandwidth}; (2) a KDE with under-smoothed bandwidth $h / n^{1/10}$; and (3) TMLE using (1) as the initial estimator. Hence, we constructed a total of nine distinct estimators for each dataset.

We considered four bootstrap sampling distributions: the empirical distribution, and three smooth distributions corresponding to the three density estimators defined above. For each dataset, we generated $B = 1000$ bootstrap datasets using these four bootstrap distributions. For each bootstrap dataset, we then considered the same three estimator constructions using the bootstrap data: $\psi_{n,1}^* = T_1(\eta_n^*, \d{P}_n^*)$, $\psi_{n,2}^* = T_2(\eta_n^*, \d{P}_n^*)$, and $\psi_{n,3}^* = T_3(\eta_n^*, \d{P}_n^*)$. We considered two bootstrap nuisance estimators $\eta_n^*$: the same nuisance estimation procedure used for the original data applied to the bootstrap sample, with bandwidth fixed at the value selected using the original data, and using the fixed nuisance estimator obtained from the original data, i.e.\ $\eta_n^* = \eta_n$. Finally, we used all four methods of constructing bootstrap confidence intervals defined in Section~\ref{sec:conf_int} to construct two-sided, equi-tailed 95\% confidence intervals for $\psi_0$ based on each bootstrap sample. For the percentile $t$-method, we used the influence function-based variance estimator. For comparison, we also constructed ordinary Wald-style confidence intervals using the influence function-based variance estimator. We evaluated the performance of these confidence intervals by computing their empirical coverage and average width over the 1000 simulations.

\begin{figure}[!hbtp]
    \centering
    \includegraphics[height=8in, width=\textwidth]{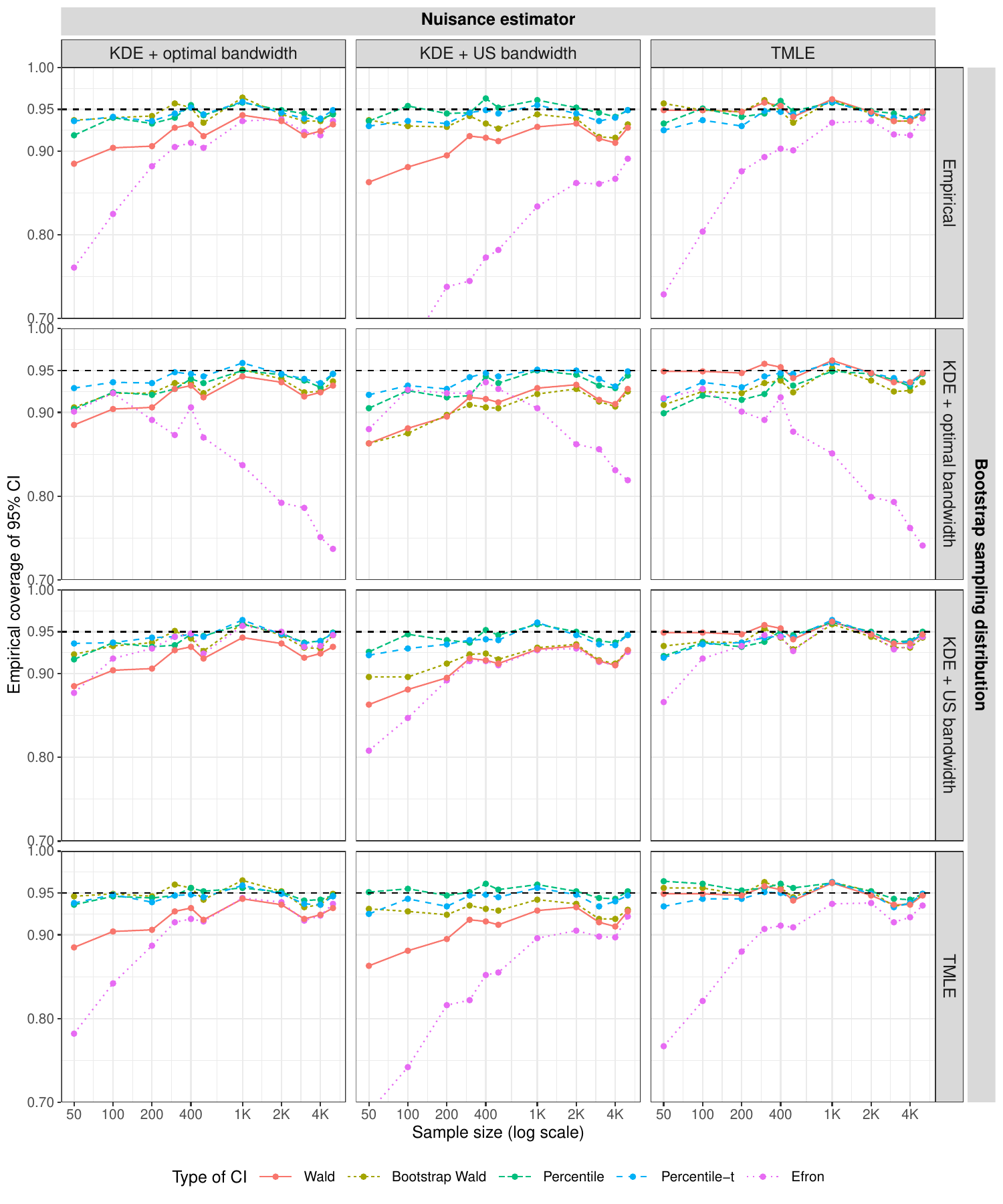}
    \caption{Empirical coverage of 95\% confidence intervals based on the bootstrap one-step estimator when re-estimating the nuisance using the bootstrap sample. ``KDE" stands for kernel density estimator; ``US" stands for under-smoothed; TMLE stands for targeted maximum likelihood estimator.}
    \label{fig: onestep}
\end{figure}
\begin{figure}[!hbtp]
    \centering
    \includegraphics[height=8in, width=\linewidth]{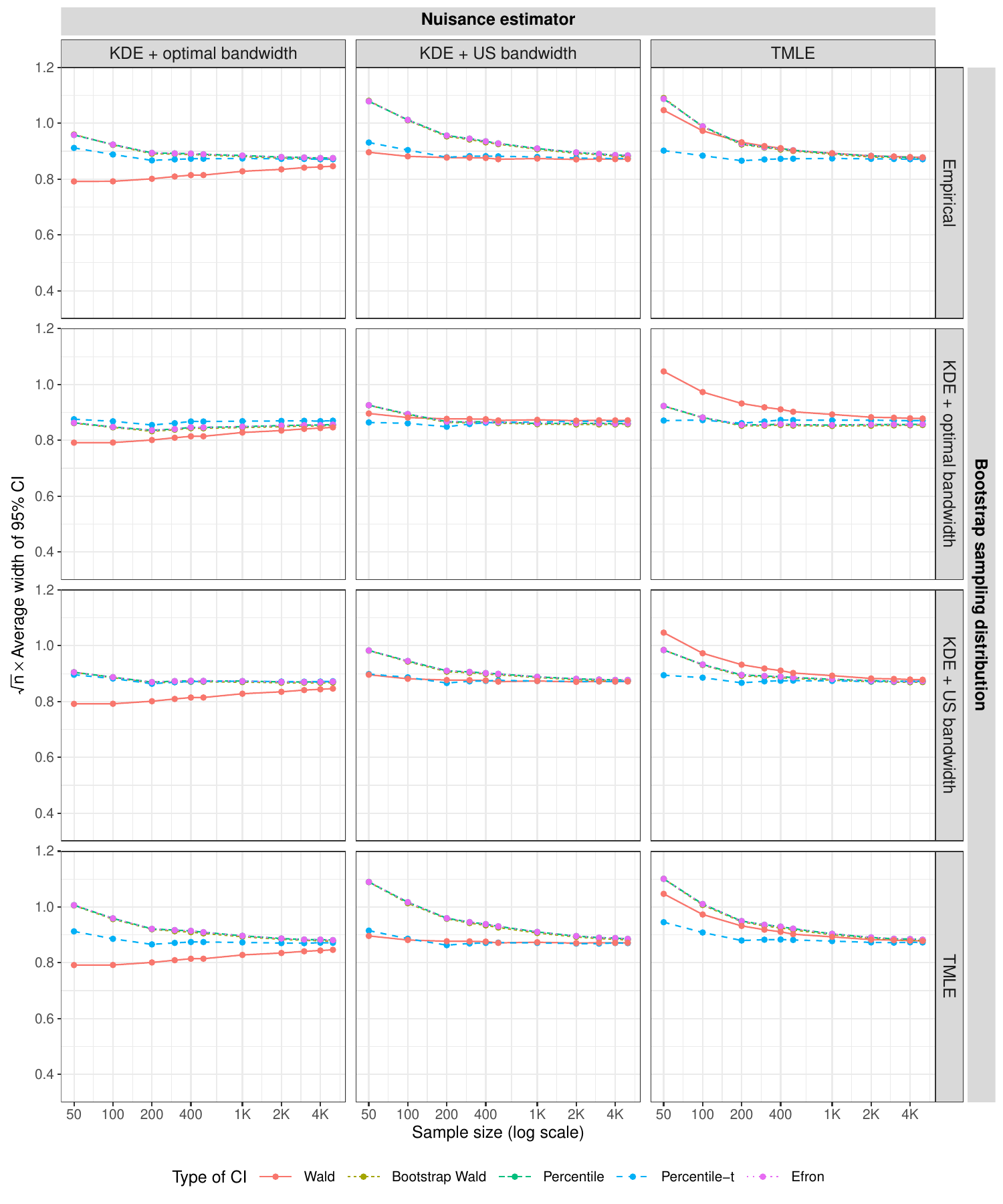}
    \caption{Scaled average width of 95\% confidence intervals based on the bootstrap one-step estimator when re-estimating the nuisance using the bootstrap sample. Abbreviations as in Figure~\ref{fig: onestep}.}
    \label{fig: onestep width}
\end{figure}

We now turn to the results of the simulation study. Figures~\ref{fig: onestep},~\ref{fig: plugin}, and~\ref{fig: mean_plugin} display empirical coverage and Figures~\ref{fig: onestep width},~\ref{fig: plugin width} and~\ref{fig: mean_plugin width} display the average width of 95\% confidence intervals based on the one-step, plug-in, and empirical mean plug-in estimators, respectively, when the bootstrap nuisance was re-estimated using the bootstrap sample. Figures~\ref{fig: fixed onestep} and~\ref{fig: fixed meanplugin} display empirical coverage and Figures~\ref{fig: fixed onestep width} and~\ref{fig: fixed mean_plugin width} display average width based on the one-step and empirical mean plug-in constructions, respectively, when the bootstrap nuisance was fixed. The results for the plug-in construction with the nuisance fixed are not displayed because the coverage in this case was always zero, as discussed in Section~\ref{sec: plug-in}. In each figure, the rows represent the bootstrap sampling distribution and the columns represent the method of construction of the nuisance estimator. For example, the top left panel of Figure~\ref{fig: onestep} shows the empirical coverage rate of confidence intervals based on the one-step estimator where the bootstrap sampling distribution was the empirical distribution $\d{P}_n$ and the nuisance estimator was the KDE with optimal bandwidth. 

We first discuss the results displayed in Figure~\ref{fig: onestep} for the one-step estimator with re-estimated nuisance. Efron's percentile method did not yield valid confidence intervals at large sample sizes when the bootstrap sampling distribution was based on a KDE with optimal bandwidth (second row from the top). This was expected based on the results of Section~\ref{sec:conf_int}. The bias  $T_1(\eta_n, \hat{P}_n) - T_1(\eta_n, \d{P}_n) = (\d{P}_n -\hat{P}_n) \phi_{\eta_n}$ in this case is not $o_{\prob_0^*}(n^{-1/2})$ because $\hat{P}_n$ was not under-smoothed. The coverage of all other confidence intervals for the bootstrap one-step estimator approached 95\% as the sample size increased, which is in line with Propositions~\ref{prop: empirical bootstrap average} and~\ref{prop: smooth bootstrap average}. Efron's percentile confidence intervals had poor coverage for small and moderate sample sizes in some cases, which we hypothesize is due to excess bias in this method, as discussed in Section~\ref{sec:conf_int}.  The coverage of (non-bootstrap) Wald-type confidence intervals approached 95\% in all cases, as expected, though its performance for smaller sample sizes was not always good. An exception was when the nuisance estimator was TMLE (third column from the left), in which case the coverage of the Wald-type estimator was excellent at all sample sizes considered. Otherwise, there was no clear and consistent best nuisance estimator or bootstrap sampling distribution. The average widths scaled by $n^{1/2}$ displayed in Figure~\ref{fig: onestep width} all converge to the same value.

\begin{figure}[!hbtp]
    \centering
    \includegraphics[height=8in, width=\linewidth]{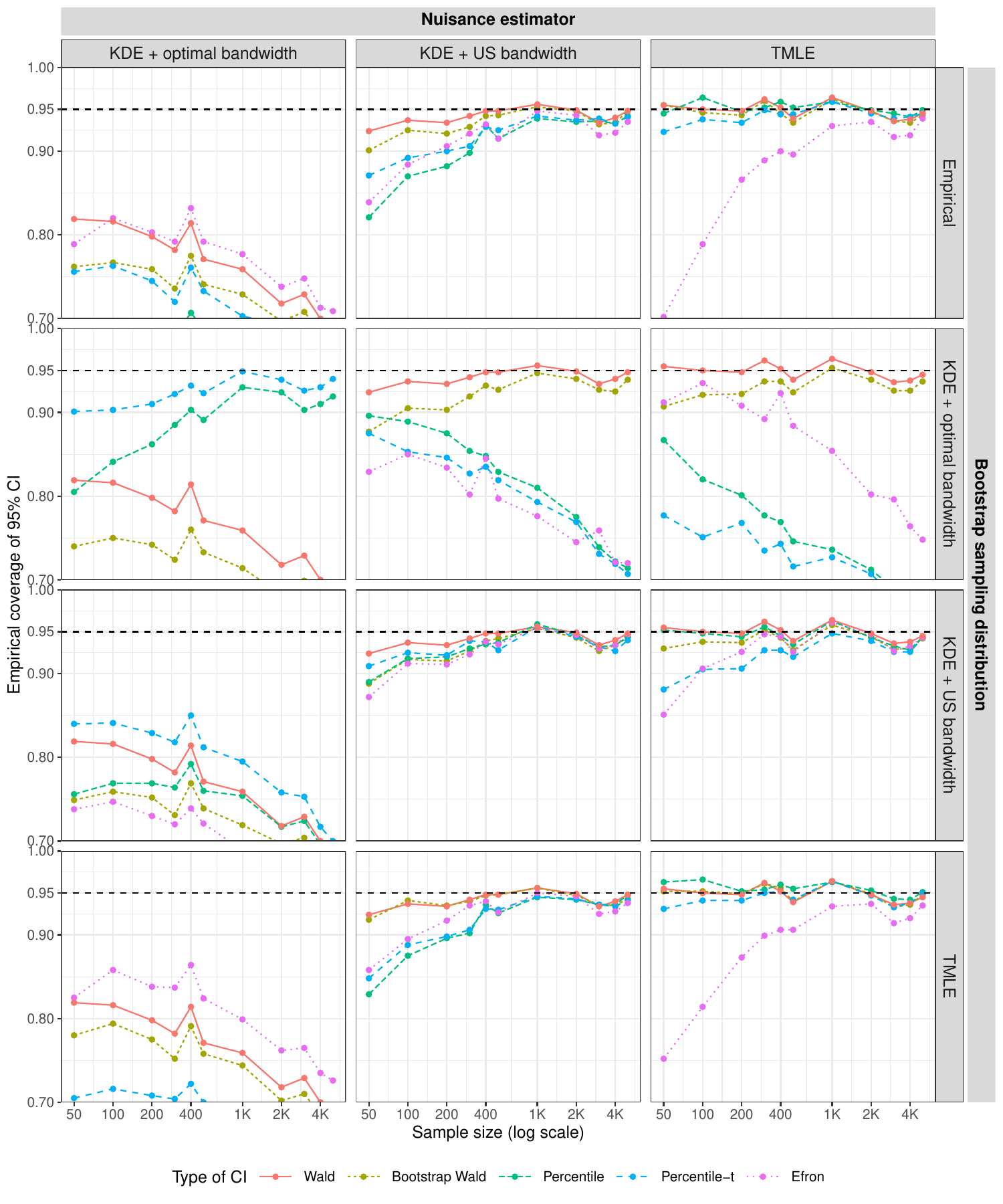}
    \caption{Empirical coverage of 95\% confidence intervals based on the bootstrap plug-in estimator when re-estimating the nuisance using the bootstrap sample. Abbreviations as in Figure~\ref{fig: onestep}.}
    \label{fig: plugin}
\end{figure}

\begin{figure}[!hbtp]
    \centering
    \includegraphics[height=8in, width=\linewidth]{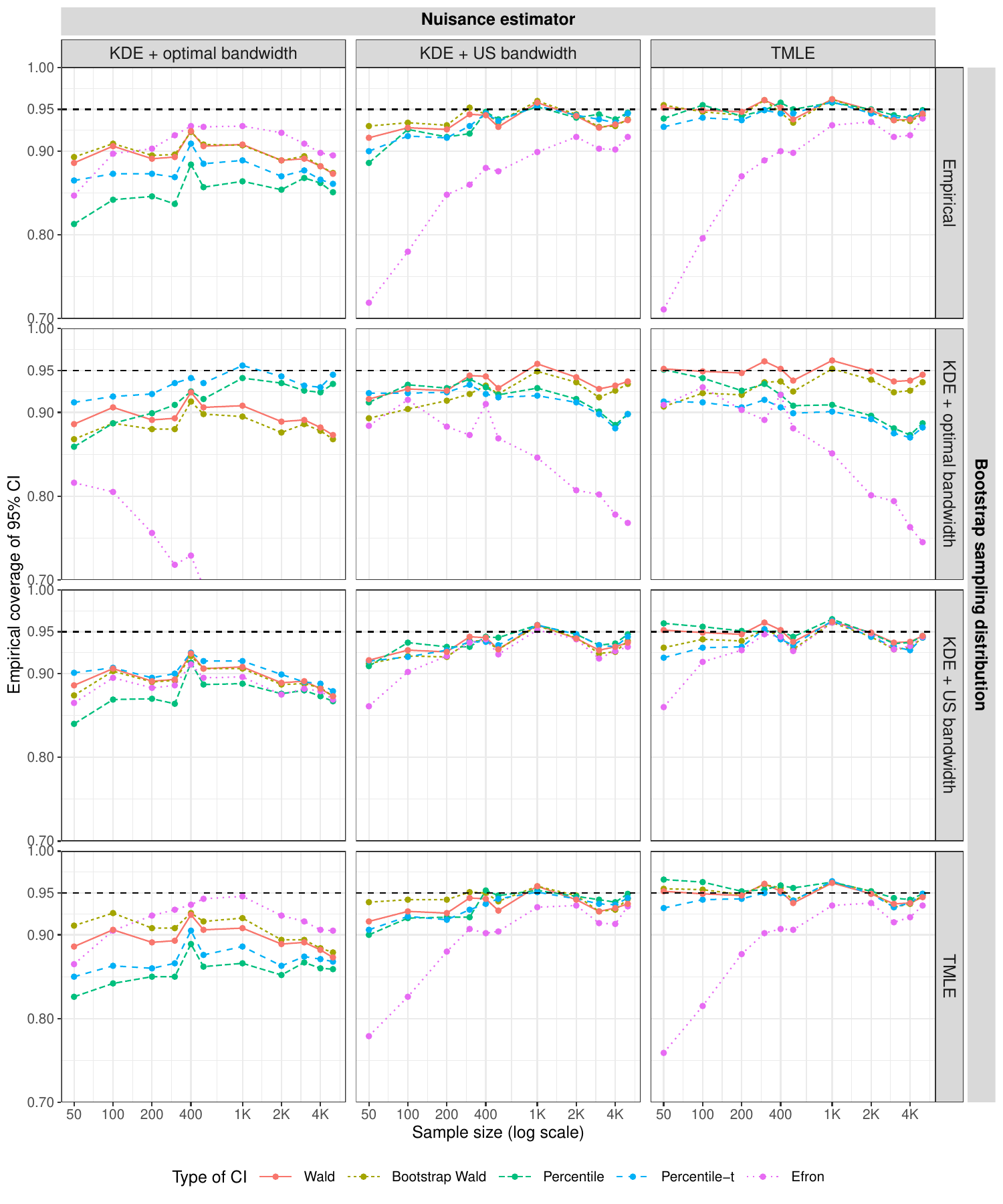}
    \caption{Empirical coverage of 95\% confidence intervals based on the bootstrap empirical mean plug-in estimator when re-estimating the nuisance using the bootstrap sample. Abbreviations as in Figure~\ref{fig: onestep}.}
    \label{fig: mean_plugin}
\end{figure}

\begin{figure}[!hbtp]
    \centering
    \includegraphics[height=8in, width=\linewidth]{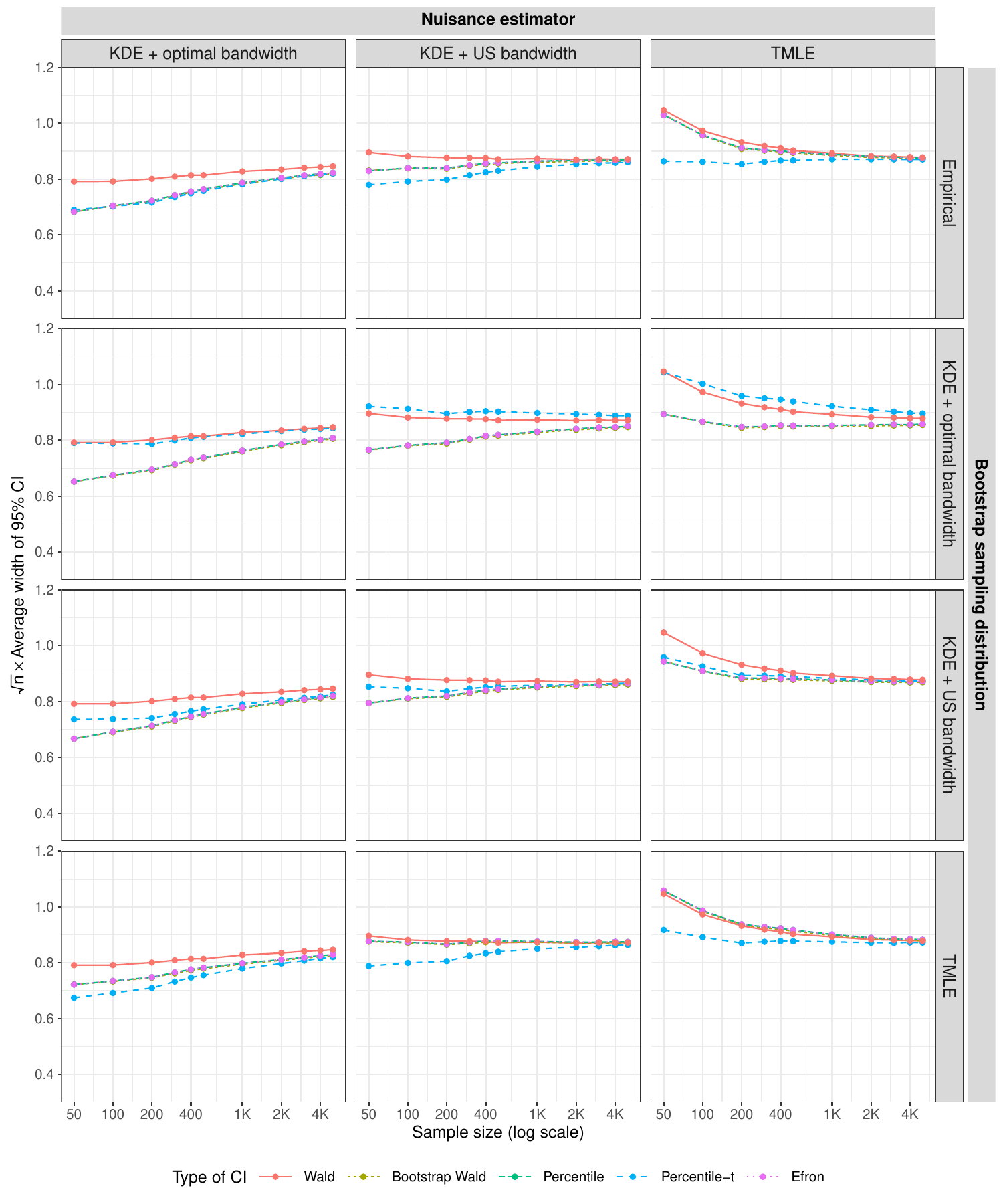}
    \caption{Scaled average width of 95\% confidence intervals based on the bootstrap plug-in estimator when re-estimating the nuisance using the bootstrap sample. Abbreviations as in Figure~\ref{fig: onestep}.}
    \label{fig: plugin width}
\end{figure}
\begin{figure}[!hbtp]
    \centering
    \includegraphics[height=8in, width=\linewidth]{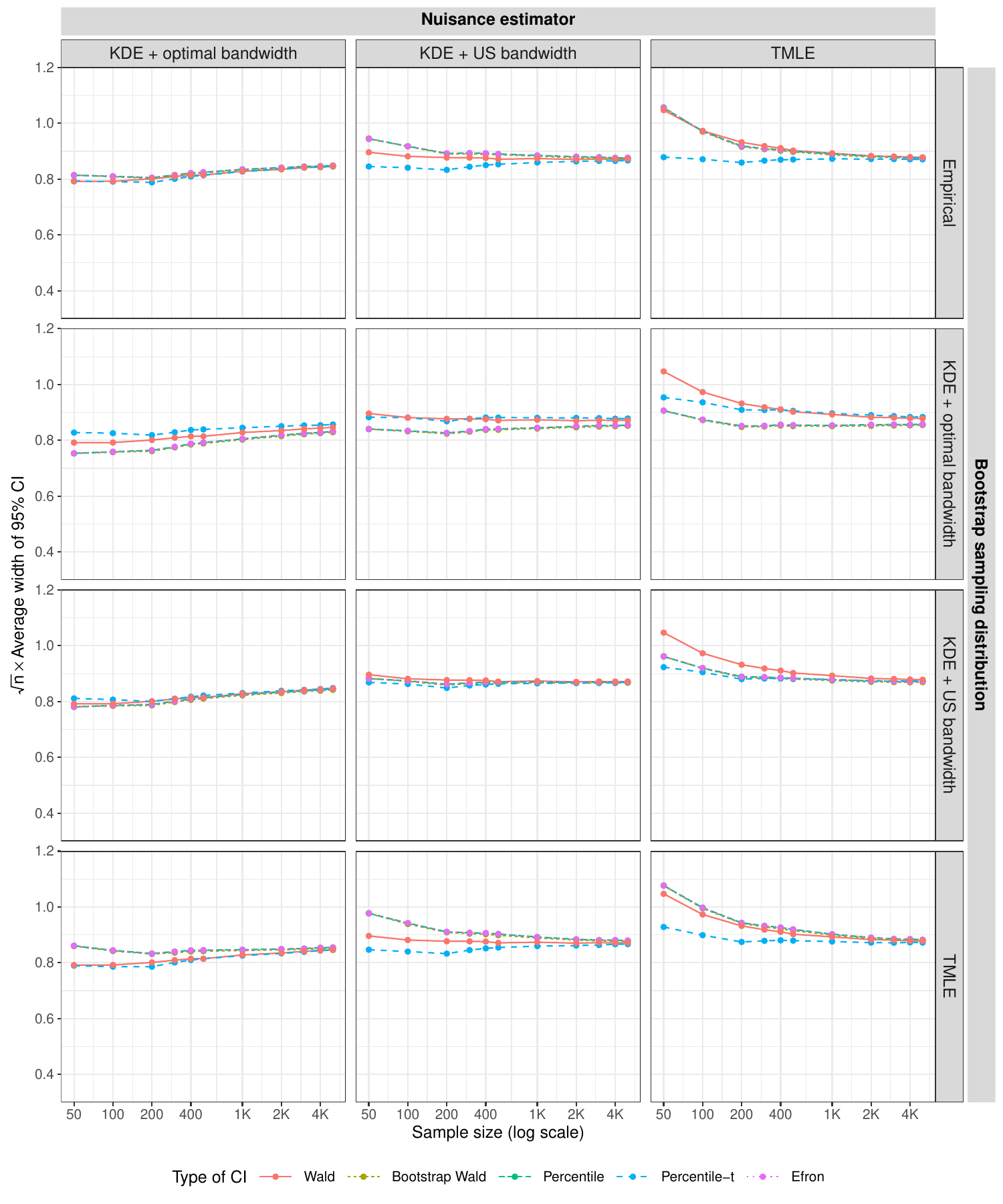}
    \caption{Scaled average width of 95\% confidence intervals based on the bootstrap empirical mean plug-in estimator when re-estimating the nuisance using the bootstrap sample. Abbreviations as in Figure~\ref{fig: onestep}.}
    \label{fig: mean_plugin width}
\end{figure}

We next discuss the results displayed in Figures~\ref{fig: plugin} and~\ref{fig: mean_plugin} for the plug-in and empirical mean plug-in estimators with re-estimated nuisance. Using the KDE with optimal bandwidth as nuisance estimator did not generally yield valid Wald or bootstrap confidence intervals 
because neither estimator is asymptotically linear in this case (first column from the left). However, the percentile and percentile-$t$ confidence intervals had close to nominal coverage in large samples for both estimators when both the nuisance estimator and bootstrap sampling distribution were the KDE with optimal bandwidth (second row from the top and first column from the left). This was expected based on \Cref{prop: auto bias correction smooth}. Besides the bootstrap Wald method, using the KDE with optimal bandwidth as the bootstrap sampling distribution did not yield valid bootstrap confidence intervals in large samples for other nuisance estimators (second row from the top and second and third columns from the left). All other confidence intervals based on the plug-in and empirical mean plug-in estimators with re-estimated nuisance had close to 95\% coverage for large sample sizes, which aligns with Propositions~\ref{prop: classical average},~\ref{prop: empirical bootstrap average}, and~\ref{prop: smooth bootstrap average}. Among the methods with good large-sample coverage, using TMLE as the nuisance estimator and bootstrap sampling distribution along with the Wald, bootstrap Wald, percentile, or percentile $t$-confidence interval yielded the best coverage for small and moderate sample sizes. In many cases, Efron's method  again had poor coverage in small and moderate samples. The average widths scaled by $n^{1/2}$ displayed in Figures~\ref{fig: plugin width} and~\ref{fig: mean_plugin width} appear to again converge to roughly the same value.

\begin{figure}[!hbtp]
    \centering
    \includegraphics[height=8in, width=\linewidth]{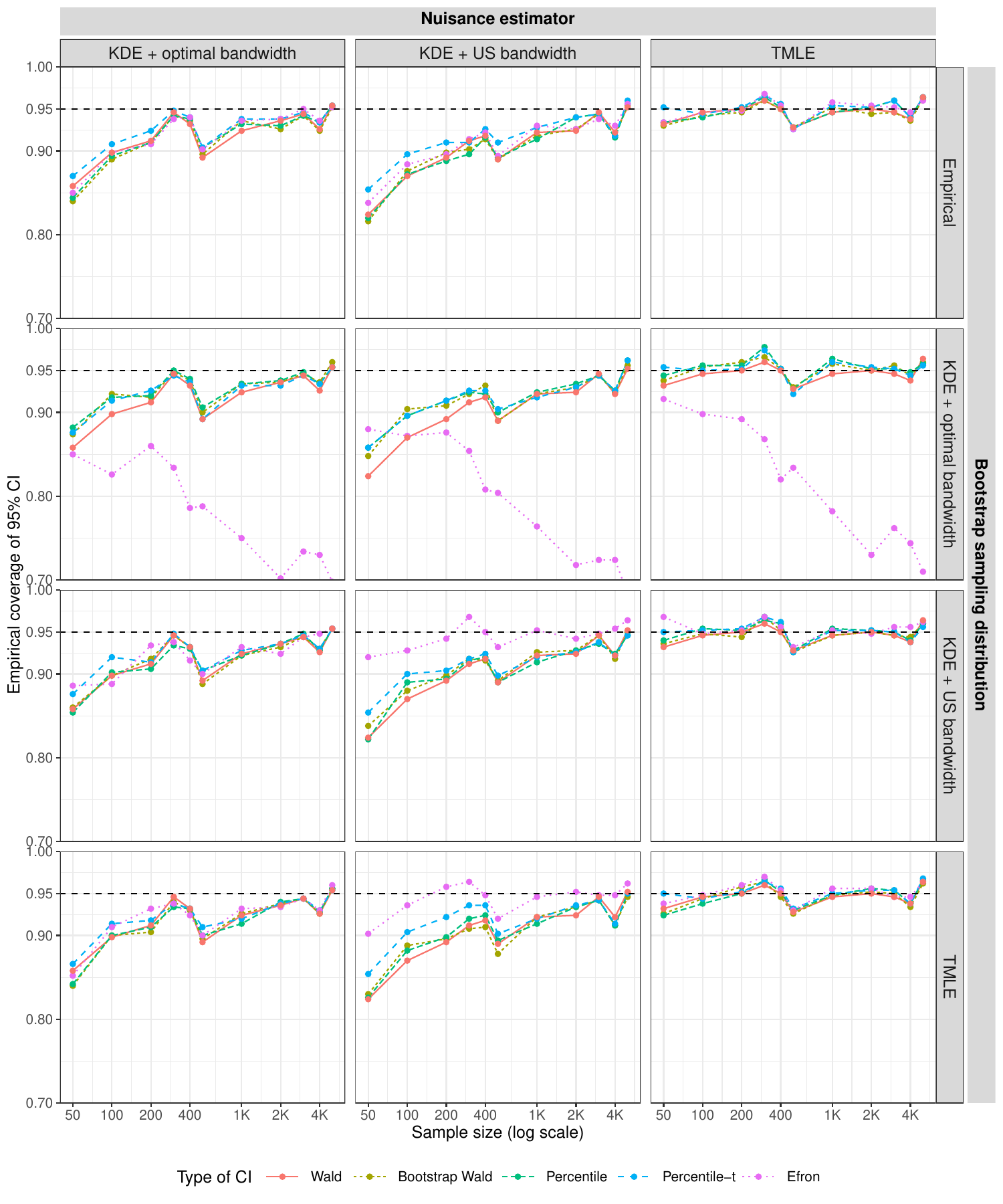}
    \caption{Empirical coverage of 95\% confidence intervals based on the bootstrap one-step estimator when the nuisance was not re-estimated using the bootstrap sample.  Abbreviations  as in Figure~\ref{fig: onestep}.}
    \label{fig: fixed onestep}
\end{figure}

\begin{figure}[!hbtp]
    \centering
    \includegraphics[height=8in, width=\linewidth]{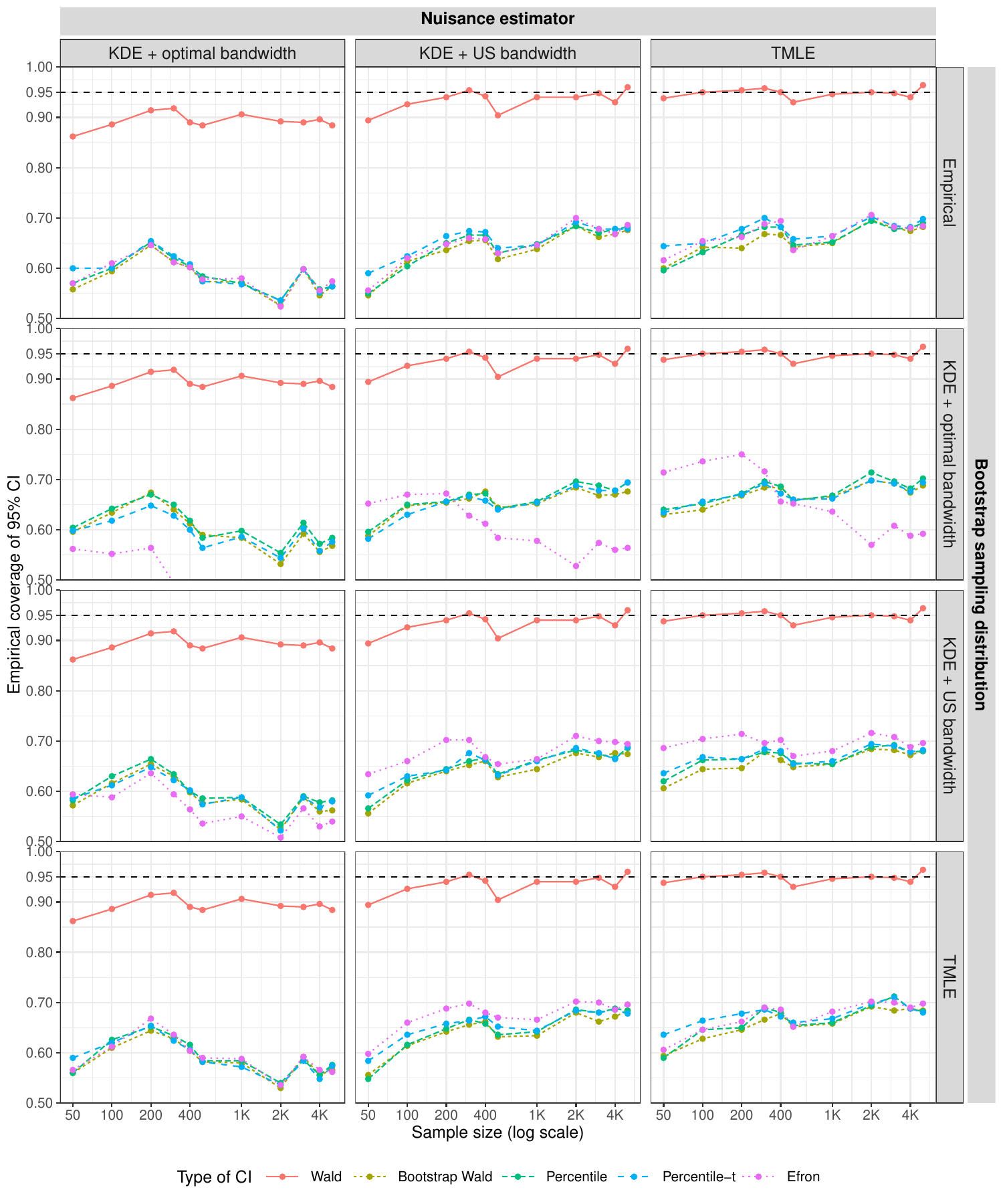}
    \caption{Empirical coverage of 95\% confidence intervals based on the bootstrap empirical mean plug-in estimator when the nuisance was not re-estimated using the bootstrap sample.  Abbreviations  as in Figure~\ref{fig: onestep}.}
    \label{fig: fixed meanplugin}
\end{figure}

\begin{figure}[!hbtp]
    \centering
    \includegraphics[height=8in, width=\linewidth]{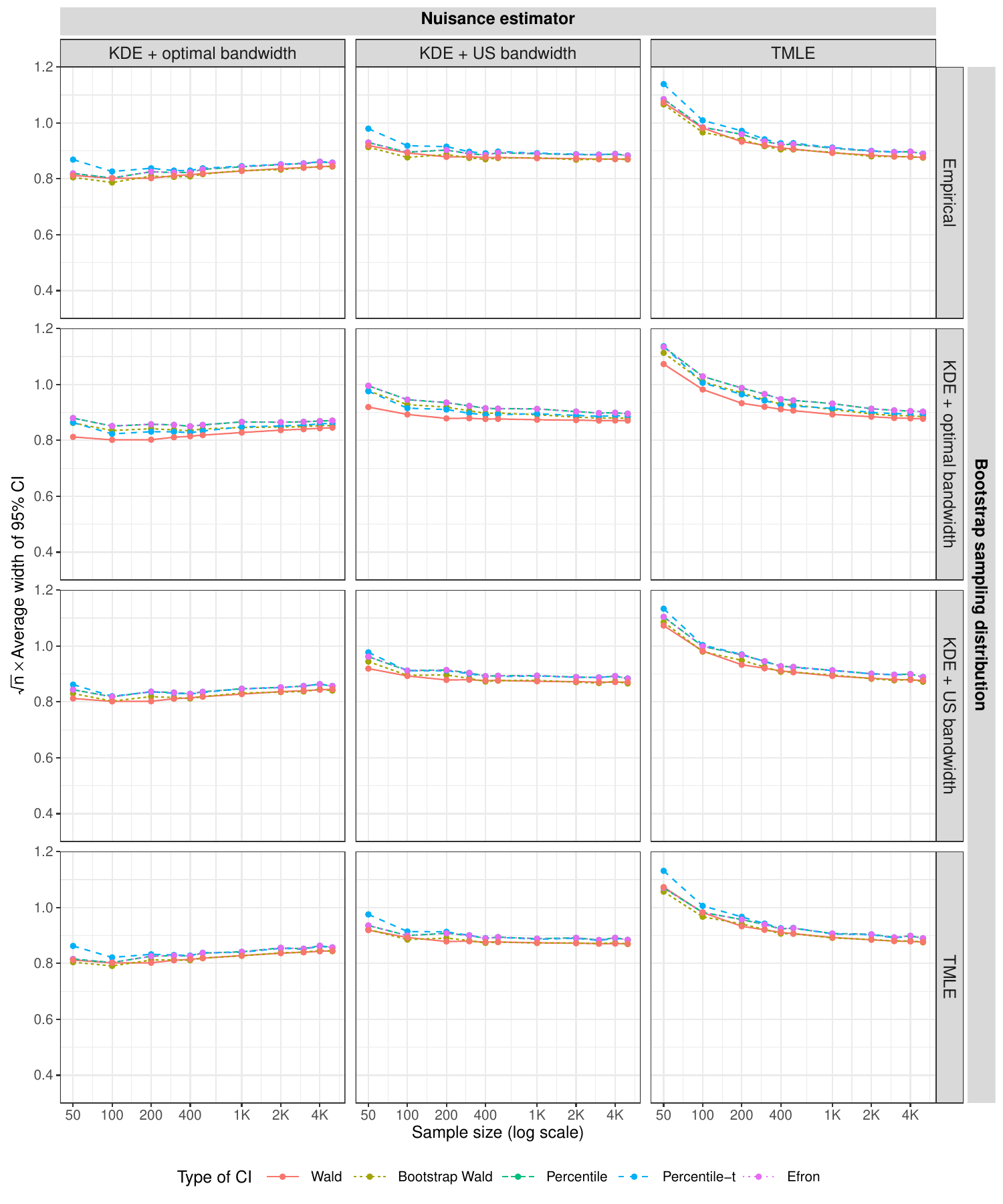}
    \caption{Scaled average width of 95\% confidence intervals based on the bootstrap one-step when the nuisance was not re-estimated using the bootstrap sample.  Abbreviations  as in Figure~\ref{fig: onestep}.}
    \label{fig: fixed onestep width}
\end{figure}

\begin{figure}[!hbtp]
    \centering
    \includegraphics[height=8in, width=\linewidth]{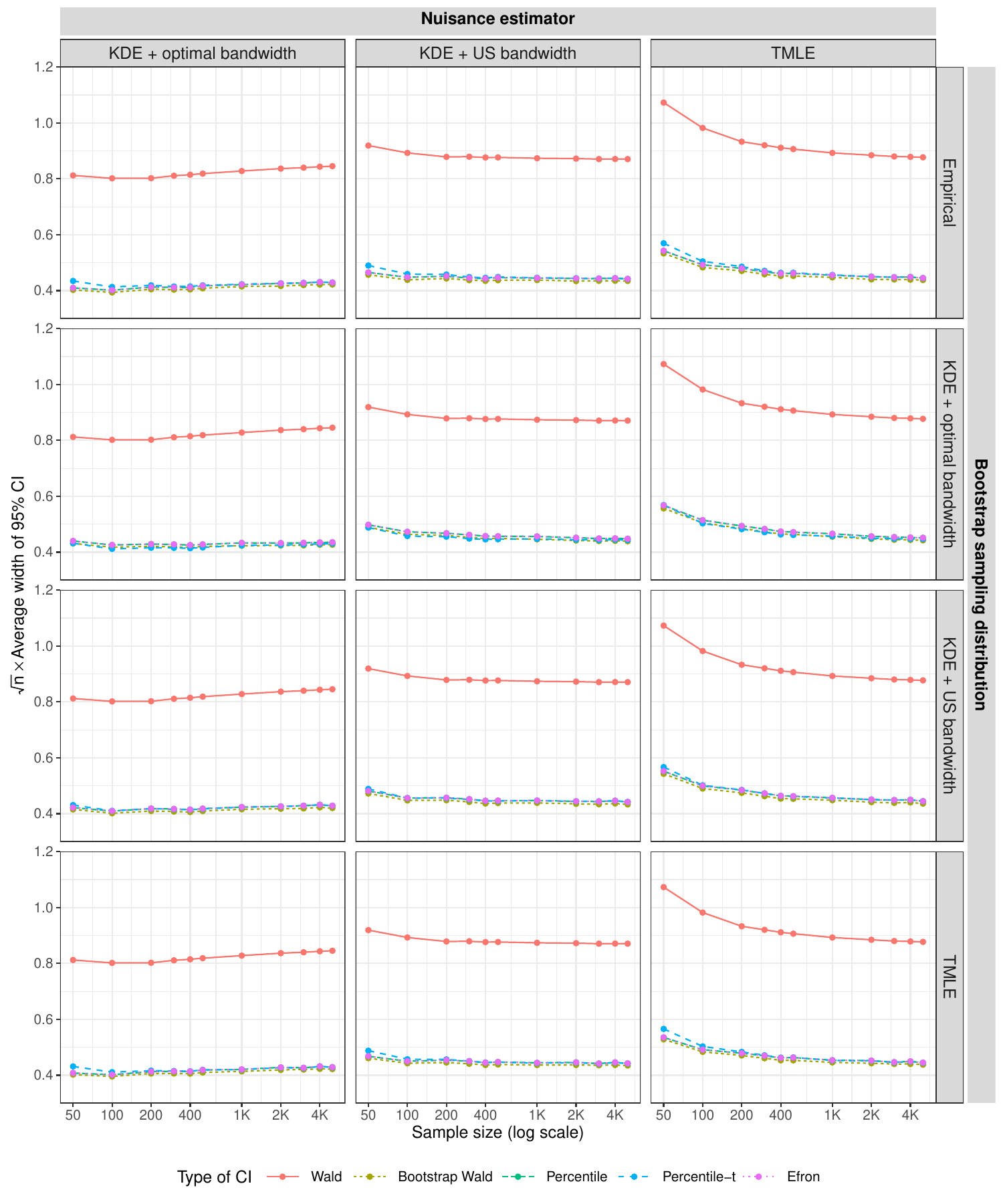}
    \caption{Scaled average width of 95\% confidence intervals based on the bootstrap empirical mean plug-in estimator when the nuisance was not re-estimated using the bootstrap sample.  Notes as in Figure~\ref{fig: onestep}.}
    \label{fig: fixed mean_plugin width}
\end{figure}

We now turn to the results displayed in Figure~\ref{fig: fixed onestep} for the one-step estimator with fixed bootstrap nuisance. Efron's percentile method again did not yield valid coverage in large samples when the bootstrap sampling distribution was based on a KDE with optimal bandwidth (second row from the top). As discussed following Figure~\ref{fig: onestep}, this was expected based on the results of Section~\ref{sec:conf_int}. All other confidence intervals for the bootstrap one-step estimator with fixed bootstrap nuisance estimator had good coverage in large samples, as expected based on \Cref{cor: bootstrap one-step fixed}. Compared to the results in Figure~\ref{fig: onestep} for the bootstrap one-step estimator with re-estimated nuisance, the coverage of bootstrap Wald, percentile, and percentile $t$-confidence intervals were mostly worse when fixing the bootstrap nuisance than when re-estimating it. However, interestingly, the coverage of Efron's percentile intervals was often better. For the empirical bootstrap, percentile $t$-confidence intervals had slightly better coverage in small samples than other bootstrap intervals, but in many other cases the various bootstrap methods had very similar coverage when the nuisance was fixed. For many of the cases considered, bootstrap intervals had better coverage in small and moderate samples than Wald-type intervals, indicating there may still be a benefit of the bootstrap even when fixing the bootstrap nuisance. There was not a substantial difference between the bootstrap sampling distributions, but the KDE with optimal bandwidth had the best overall performance in small and moderate samples. As with the bootstrap one-step estimator with the nuisance re-estimated, TMLE had the best performance among the nuisance estimators considered, with excellent coverage even for small sample sizes. The average widths scaled by $n^{1/2}$ displayed in Figures~\ref{fig: fixed onestep width} appear to again converge to roughly the same value.

Finally, from Figures~\ref{fig: fixed meanplugin} and~\ref{fig: fixed mean_plugin width}, all of the bootstrap confidence intervals for the bootstrap empirical mean plug-in estimator with fixed bootstrap nuisance had poor coverage at all sample sizes because they were on average too narrow. This is in line with the results and discussion of Section~\ref{sec: mean plug-in}.

\section{Conclusion}\label{sec: conclusion}

In this paper, we studied the problem of bootstrapping  asymptotically estimators that rely on a data-adaptive nuisance estimator. We proposed a framework that encompasses many approaches to constructing asymptotically linear estimators and bootstrapping them, and we provided high-level conditions for consistency of the bootstrap in this framework. We also provided more detailed conditions for several bootstrap distributions and estimator constructions. We used our general results to  demonstrate that a wide variety of bootstrap confidence intervals are asymptotically valid in this setting, and our simulation study confirmed this. 
It is our hope that the generality of our framework and theory ensures that there are many potential future applications of our results. 

An important area of future research is establishing rates of convergence for the bootstrap methods considered here. We focused on consistency of the bootstrap because it is an important first step, and because we expect that rates of convergence will require stronger assumptions than we used here. However, understanding how the different components of the original estimator and the bootstrap sampling scheme contribute to the accuracy of bootstrap confidence intervals is crucial for deciding which method to use in practice. For instance, while we showed that the precise behavior of the bootstrap nuisance estimator does not play a role in the first-order asymptotic behavior of the bootstrap estimator as long as our high-level conditions hold, we expect that it plays an important role in the finite-sample accuracy of the bootstrap. This was sometimes, but not always, the case in our numerical studies. Similarly, while we showed that both the empirical and smooth bootstraps can yield asymptotically valid bootstrap confidence intervals, our results did not reveal which approach will yield better finite-sample coverage.
 

\section*{Acknowledgements}

The authors gratefully acknowledge support from the University of Massachusetts Amherst Department of Mathematics and Statistics startup fund and NSF Award 2113171.

\clearpage

\bibliographystyle{apalike2}
\bibliography{reference.bib}

\begin{appendices}

\clearpage

\begin{center}
    \bfseries{\huge Supplementary Material}
\end{center}

\section{Proof of results in Section 2}

\thmclassical*
\begin{proof}[\bfseries{Proof of Theorem~\ref{thm: classical}}]
    Consider the decomposition
    \begin{align*}
        T(\eta_n, \d{P}_n) - T(\eta_0, P_0) = \d{P}_n \phi_0 &+ (\d{P}_n - P_0) (\phi_n - \phi_0) + [T(\eta_n, \d{P}_n) - T(\eta_0, P_0) - (\d{P}_n - P_0) \phi_n].
    \end{align*}
    By conditions~\ref{cond: limited complexity}--\ref{cond: weak consistency} and Lemma~19.24 of \cite{van2000asymptotic}, $(\d{P}_n - P_0) (\phi_n - \phi_0) = o_{\prob_0^*}(n^{-1/2})$. By condition~\ref{cond: second order}, the remainder term is $o_{\prob_0^*}(n^{-1/2})$. The second statement follows by the central limit theorem since $P_0 \phi_0 = 0$ and $P_0 \phi_0^2 < \infty$ by assumption.
\end{proof}

\propvar*
\begin{proof}[\bfseries{Proof of Proposition~\ref{prop: variance}}]
    By adding and subtracting terms, we have
    \[
       \d{P}_n \phi_n^2 - \sigma_0^2 = (\d{P}_n - P_0) \phi_n^2  + P_0 (\phi_n^2 - \phi_0^2).
    \]
    The second term is $o_{\prob_0^*}(1)$ by assumption. For the first term, for any $\varepsilon >0$, we have 
    \begin{align*}
        \prob_0^*\left( \abs{(\d{P}_n - P_0) \phi_n^2} \geq \varepsilon \right) &= \prob_0^*\left( \abs{(\d{P}_n - P_0) \phi_n^2} \geq \epsilon,\, \phi_n^2 \in \s{G} \right) + \prob_0^*\left( \abs{(\d{P}_n - P_0) \phi_n^2} \geq \varepsilon,\, \phi_n^2 \notin \s{G} \right)\\
        & \leq \prob_0^*\left( \sup_{g \in \s{G}} \abs{(\d{P}_n - P_0) g} \geq \varepsilon \right) + \prob_0^*\left( \phi_n^2 \notin \s{G} \right).
    \end{align*}
    Since $\s{G}$ is a $P_0$-Glivenko Cantelli class, the first term is $o(1)$, and since $\phi_n^2 \in \s{G}$ with probability tending to one the second term is $o(1)$. Hence, $(\d{P}_n - P_0) \phi_n^2 = o_{\prob_0^*}(1)$, which completes the proof of the first statement.
    
    By Theorem~2.10.14 of \cite{van1996weak}, since $\sup_{f \in \s{F}}|P_0 f| < \infty$, condition~\ref{cond: limited complexity} implies that $\s{F}^2 = \{f^2: f\in \s{F}\}$ is $P_0$-Glivenko-Cantelli in probability, and since $\prob_0^*(\phi_n^2 \in \s{F}^2) \to 1$, this implies (i). By Cauchy-Schwarz and Minkowski's inequalities, we also have
    \begin{align*}
        \left|P_0 (\phi_n^2 - \phi_0^2) \right| &\leq P_0 | \phi_n^2 - \phi_0^2| \\
        &= P_0 \left| (\phi_n - \phi_0)(\phi_n + \phi_0)\right|\\
        & \leq \norm{\phi_n-\phi_0}_{L_2(P_0)} \norm{\phi_n+\phi_0}_{L_2(P_0)} \\
        & \leq \norm{\phi_n-\phi_0}_{L_2(P_0)} [\norm{\phi_n-\phi_0}_{L_2(P_0)} + 2\norm{\phi_0}_{L_2(P_0)}],
    \end{align*}
    which is $o_{\prob_0^*}(1)$ by~\ref{cond: weak consistency}.
\end{proof}

\section{Proof of results in Section~\ref{sec:theory}}
    

\lemmabootran*
\begin{proof}[\bfseries{Proof of Lemma~\ref{lemma: boot_ran}}]
     A standard way to prove results of this type outside the setting of the bootstrap is the continuous mapping theorem (see, e.g., \citealp[Theorem~19.24]{van2000asymptotic}). However, the continuous mapping theorem might not be applicable in the bootstrap setting as the map $(w_1, \dotsc, w_n) \mapsto h(\d{G}_n^*(x_1, \dotsc, x_n, w_1, \dotsc, w_n))$ might not be measurable given almost $x_1, \dotsc, x_n$ for all $h \in C_b(\ell^\infty(\s{F}))$.

    For any $\tau > 0$, we define $\s{F}_\tau = \{f-f_\infty: f \in \s{F}, \rho(f, f_\infty) < \tau\}$. We then have
    \begin{align*}
        \left\{ | \d{G}_n^* (f_n^* - f_\infty )| \geq \varepsilon \right\}
        &\subseteq \left\{ \| \d{G}_n^* (f_n^* - f_\infty )| \geq \varepsilon,\ \rho(f_n^*, f_\infty) < \tau, \ f_n^* \in \s{F} \right\}\\
        &\qquad\qquad \cup \left\{ \rho(f_n^*, f_\infty) \geq \tau \ \text{or}\ f_n^* \notin \s{F} \right\} \\
        &\subseteq 
        \left\{\| \d{G}_n^* \|_{\s{F}_\tau} \geq \varepsilon \right\} \cup 
        \left\{ \rho(f_n^*, f_\infty)\geq \tau \right\} \cup 
        \left\{ f_n^* \notin \s{F} \right\}
    \end{align*}
    for any $\varepsilon, \tau > 0$. Thus,
    \begin{align*}
        \prob_W^* \left( | \d{G}_n^* (f_n^* - f_\infty)| \geq \varepsilon \right) &\leq \prob^*_W \left( \| \d{G}_n^* \|_{\s{F}_\tau} \geq \varepsilon \right) + \prob^*_W \left( \rho(f_n^*, f_\infty) \geq \tau  \right) + \prob^*_W \left( f_n^*  \notin \s{F} \right).
    \end{align*}
    The second and third terms on the right-hand side conditionally converge to zero in outer probability for any $\tau > 0$ by assumption.
    Hence, the proof is complete if we can show that for all $\varepsilon, \delta,\gamma > 0$ there exists $\tau > 0$ such that $\prob_0^*\left( \prob^*_W \left( \| \d{G}_n^* \|_{\s{F}_\tau} \geq \varepsilon \right) >\delta \right) < \gamma$ for all $n$ large enough.
    
    For each $\varepsilon, \tau > 0$ there exists a sequence of functions $h_m: \ell^\infty(\s{F}) \to \d{R}$ such that $h_m$ is $m$-Lipschitz for all $m$, $1\{ \| z\|_{\s{F}_\tau} \geq \varepsilon\}\leq h_m(z) \leq 1$ for all $z \in \ell^\infty(\s{F})$ and $m$, and $h_m(z)$ monotonically decreases to $1\{ \| z\|_{\s{F}_\tau} \geq \varepsilon\}$ as $m \to \infty$ for each $z \in \ell^\infty(\s{F})$. For instance, $h_m : z \mapsto \min\{ \max(m[\| z\|_{\s{F}_\tau} - \varepsilon] + 1, 0), 1\}$ satisfies these criteria. For any $\varepsilon, \tau > 0$ and $m \in \{1, 2, \dots\}$, we can now write
    \begin{align*}
        \prob^*_W \left( \| \d{G}_n^* \|_{\s{F}_\tau} \geq \varepsilon \right) &= \left[ \E^*_W 1\left\{ \| \d{G}_n^* \|_{\s{F}_\tau} \geq \varepsilon \right\} - \E^*_W h_m( \d{G}_n^* ) \right] + \left[\E^*_W h_m( \d{G}_n^* ) - \E_0 h_m(\d{G}_0) \right] \\
        &\qquad + \left[ \E_0 h_m(\d{G}_0)  -  \E_0 1\left\{ \| \d{G}_0 \|_{\s{F}_\tau} \geq \varepsilon \right\}  \right] + \prob_0 \left( \norm{ \d{G}_0 }_{\s{F}_\tau} \geq \varepsilon \right) \\
        &\leq \left[\E^*_W h_m( \d{G}_n^*) - \E_0 h_m(\d{G}_0) \right]  + \left[ \E_0 h_m(\d{G}_0)  -  \E_0 1\left\{ \| \d{G}_0 \|_{\s{F}_\tau} \geq \varepsilon \right\}  \right] \\
        & \qquad + \prob_0 \left( \norm{ \d{G}_0 }_{\s{F}_\tau} \geq \varepsilon \right),
    \end{align*}
    where the second inequality follows because $\E^*_W 1\left\{ \| \d{G}_n^* \|_{\s{F}_\tau} \geq \varepsilon \right\} \leq \E^*_W h_m\left( \d{G}_n^* \right)$ for all $\varepsilon, \tau >0$ and $m \in \{1,2,\dots\}$ by assumption. For the final term on the right-hand side, since almost all sample paths of $\d{G}_0$ are uniformly continuous in $\ell^\infty(\s{F})$ with respect to $\rho$, 
    for any $\varepsilon, \delta > 0$, we can choose $\tau > 0$ such that $\prob_0 \left( \norm{ \d{G}_0 }_{\s{F}_\tau} \geq \varepsilon \right) <\delta$. For the second term on the right-hand side, by the monotone convergence theorem, $\lim_{m \to \infty} \E_0 h_m(\d{G}_0)  = \E_0 1\{ \| \d{G}_0 \|_{\s{F_\tau}} \geq \varepsilon\}$. Hence, for any $\varepsilon, \delta, \tau >0$, we can choose $m$ such that $\E_0 h_m(\d{G}_0)  -  \E_0 1\left\{ \| \d{G}_0 \|_{\s{F}_\tau} \geq \varepsilon \right\} < \delta$. Finally, for the first  term on the right-hand side, since $h_m$ is bounded and $m$-Lipschitz, $h_m / m \in \BL_1(\ell^\infty(\s{F}))$ for each $m$. Therefore, by the assumed conditional weak convergence of $\d{G}_n^*$ to $\d{G}_0$ in $\ell^\infty(\s{F})$, $\E^*_W \frac{1}{m}h_m\left( \d{G}_n^* \right) \inoutprob \E_0 \frac{1}{m}h_m(\d{G}_0)$ for each $m$, so that $\E^*_W h_m\left( \d{G}_n^* \right) \inoutprob \E_0 h_m(\d{G}_0)$ as well. Hence, for any $\varepsilon, \delta, \gamma, \tau > 0$ and $m \in \{1, 2, \dots\}$, $\prob_0^*\left(\E^*_W h_m\left( \d{G}_n^* \right) - \E_0 h_m(\d{G}_0) > \delta\right) < \gamma$ for all $n$ large enough. This completes the proof. 
    
\end{proof}

\thmbootstrap*
\begin{proof}[\bfseries{Proof of Theorem~\ref{thm: bootstrap}}]
    We let $h$ be an arbitrary element of $\BL_1(\d{R})$. By adding and subtracting terms and the triangle inequality, we have
    \begin{align}\label{eq:thm2proof_bound}
        \begin{split}
            & \abs{ \E_W^* h\left( n^{1/2}\left[ T(\eta_n^*, \d{P}_n^*) - T(\eta_n, \hat{P}_n) \right] \right)  - \E_0 h\left(\d{G}_0 \phi_0\right)} \\
            & \qquad \leq  \E_W^*\abs{ h\left( n^{1/2}\left[T(\eta_n^*, \d{P}_n^*) - T(\eta_n, \hat{P}_n)\right] \right) -  h\left(\d{G}_n^*\phi_0\right)}  + \abs{ \E_W^* h\left(\d{G}_n^* \phi_0\right)  - \E_0h\left(\d{G}_0 \phi_0\right) }.
        \end{split}
    \end{align}

    For the first term on the right-hand side of \eqref{eq:thm2proof_bound}, we note that for any $x_1, x_2 \in \d{R}$, $\abs{h(x_1) - h(x_2)} \leq 2 \wedge |x_1-x_2|$, and for any $\varepsilon > 0$,  $2 \wedge |x_1-x_2| \leq \varepsilon + 2 I(|x_1-x_2| > \varepsilon)$, which implies that for any $\varepsilon > 0$,
    \begin{align}\label{eq:thm2weak}
        \begin{split}
            &\E_W^*\abs{ h\left( n^{1/2}\left[T(\eta_n^*, \d{P}_n^*) - T(\eta_n, \hat{P}_n)\right] \right) -  h\left(\d{G}_n^*\phi_0\right)} \\
            &\qquad \leq \varepsilon + 2 \prob^*_W \left( \left|n^{1/2}\left[T(\eta_n^*, \d{P}_n^*) - T(\eta_n, \hat{P}_n)\right] - \d{G}_n^*\phi_0\right| \geq \varepsilon \right). 
        \end{split}
     \end{align}
    We now write
    \begin{align*}
        n^{1/2}\left[T(\eta_n^*, \d{P}_n^*) - T(\eta_n, \hat{P}_n)\right] - \d{G}_n^*\phi_0 &=  \d{G}_n^* (\phi_n^* - \phi_0)  + n^{1/2}\left[T(\eta_n^*, \d{P}_n^*) - T(\eta_n, \hat{P}_n) - (\d{P}_n^* - \hat{P}_n) \phi_n^* \right]\\
        & = \d{G}_n^* (\phi_n^* - \phi_0) + n^{1/2} R_n^*.
    \end{align*}
    Note that $\d{G}_n^* (\phi_n^* - \phi_0) = o_{\prob_W^*}(1)$ by conditions~\ref{cond: bootstrap limited complexity}--\ref{cond: bootstrap weak consistency} and \Cref{lemma: boot_ran}, and $n^{1/2}R_n^* = o_{\prob_W^*}(1)$ by condition~\ref{cond: bootstrap second order}. This implies the first term on the right-hand side of~\eqref{eq:thm2proof_bound} is $o_{\prob_0^*}(1)$. 
    For the second term on the right hand side of~\eqref{eq:thm2proof_bound}, we define the function $g : \ell^\infty(\s{F}) \to \d{R}$ as $g(z) := h(z(\phi_0))$. For any $z_1, z_2 \in \ell^\infty(\s{F})$ with 
    $\| z_1 - z_2\|_{\s{F}} > 0$, we have
    \[ \frac{|g(z_1) - g(z_2)|}{\| z_1 - z_2\|_{\s{F}}} = \frac{|h(z_1(\phi_0)) - h(z_2(\phi_0))|}{\| z_1 - z_2\|_{\s{F}}} \leq\frac{|z_1(\phi_0) - z_2(\phi_0)|}{\| z_1 - z_2\|_{\s{F}}} \leq 1 \]
     because $h \in \BL_1(\d{R})$ and $\phi_0 \in \s{F}$. Hence, $g \in \BL_1(\ell^\infty(\s{F}))$. Therefore,
     \[\abs{ \E_W^* h\left(\d{G}_n^* \phi_0\right)  - \E_0h\left(\d{G}_0 \phi_0\right) } \leq \sup_{g \in \BL_1(\ell^\infty(\s{F}))}\abs{ \E_W^* g\left(\d{G}_n^* \right)  - \E_0g\left(\d{G}_0\right) } \inoutprob 0\] 
     by condition~\ref{cond: bootstrap limited complexity}. This implies conditional asymptotic linearity of $T(\eta_n^*, \d{P}_n^*)$. Conditional weak convergence of $n^{1/2}\left[T(\eta_n^*, \d{P}_n^*) - T(\eta_n, \hat{P}_n)\right]$ is then implied by \eqref{eq:thm2weak}.
\end{proof}

\lemmasmooth*
\begin{proof}[\bfseries{Proof of \Cref{prop: uniform donsker}}]
    We argue along subsequences using a standard argument structure. By Lemma~1.9.2 of \cite{van1996weak}, every subsequence $\{ n_k : k =1, 2, \dots\}$ has a further subsequence $\{n_{k_l} : l = 1, 2, \dots\}$ such
    \[ 
        \sup_{f, g \in \s{F}} \left| \| f - g\|_{L_2(\hat{P}_{n_{k_l}})} - \| f - g\|_{L_2(P_0)} \right| \stackrel{\text{a.s.*}}{\longrightarrow} 0.
    \]
    By Theorem~4.5 of \cite{anne1992uniform}, we then have $n_{k_l}^{1/2}(\d{P}_{n_{k_l}}^* - \hat{P}_{n_{k_l}}) \underset{W}{\stackrel{\text{a.s.*}}{\leadsto}} \d{G}_0$. Hence, every $\{ n_k \}$ has a further subsequence $\{n_{k_l} \}$ such that 
    \[
        \sup_{h \in \BL_1(\ell^\infty(\s{F}))} \abs{\E_W^* h(\d{G}_{n_{k_l}}^*) - \E_0 h(\d{G}_0)} \stackrel{\text{a.s.*}}{\longrightarrow} 0, 
    \]
    which by Lemma~1.9.2 of~\cite{van1996weak} again implies that 
    \[
        \sup_{h \in \BL_1(\ell^\infty(\s{F}))} \abs{\E_W^* h(\d{G}_{n}^*) - \E_0 h(\d{G}_0)} \inoutprob 0, 
    \]
    i.e., $\d{G}_{n}^* \condinoutdist \d{G}_0$.

\end{proof}

\propconvolution*
\begin{proof}[\bfseries{Proof of Proposition~\ref{prop: weak convergence of convolution smooth bootstrap}}]
    Since $\s{F}$ has uniformly bounded envelope function and finite uniform entropy integral, by Theorem~2.8.3 of \cite{van1996weak}, conditions~(i)--(ii) of Lemma~\ref{prop: uniform donsker} are satisfied. We next show the uniform convergence of semi-metric. For simplicity, we denote $\s{F}_\infty:=\{ f-g : f, g \in \s{F}\}$ and $\s{F}_\infty^2 :=\{ (f-g)^2  : f, g \in \s{F}\}$. Since $|\sqrt{x} - \sqrt{y}| \leq \sqrt{|x-y|}$ for all $x, y \in \d{R}$, we have 
    \begin{align*}
        \sup_{f,g \in \s{F}} \abs{ \norm{f-g}_{L_2(\hat{P}_n)} - \norm{f-g}_{L_2(P_0)}} \leq \sup_{f,g \in \s{F}} \abs{\int (f-g)^2 \sd (\hat{P}_n - P_0)}^{1/2}
    \end{align*}
    Hence, condition~(iii) of Lemma~\ref{prop: uniform donsker} is satisfied if $\|\hat{P}_n - P_0\|_{\s{F}_\infty^2}=o_{\prob_0^*}(1)$.

    We show that $\prob_0^*( \|\hat{P}_n - P_0\|_{\s{F}_\infty^2} \geq \delta ) \to 0$ for all $\delta >0$ using Theorem~2.2 of \cite{gaenssler2000uniform}. To do so, we verify their conditions~(2.2)--(2.5) with $\mu_n = L_n$ and $\nu_n = \d{P}_n$. Since $\s{F}_\infty^2$ is uniformly bounded by the uniform boundedness of $\s{F}$, conditions (2.2) and (2.3) are satisfied, as discussed following Theorem 2.2 of \cite{gaenssler2000uniform}. 
   To show condition (2.4), we firstly note that following inequality holds for all $\varepsilon >0$ and measures $Q$  on $(\s{X}, \s{B})$,
    \[
        N(4\varepsilon, \s{F}_\infty^2, L_1(Q)) \leq N(4\varepsilon, \s{F}_\infty^2, L_2(Q))  \leq N(2\varepsilon, \s{F}_\infty, L_2(Q)) \leq N(\varepsilon, \s{F}, L_2(Q))^2.
    \]
    The first inequality is because $\norm{\,\cdot\,}_{L_1(Q)} \leq \norm{\,\cdot\,}_{L_2(Q)}$ by H{\"{o}}lder's inequality. The second inequality is because for any $f_\infty, g_\infty \in \s{F}_\infty$, we have $Q(f_\infty^2 - g_\infty^2)^2 = Q(f_\infty - g_\infty)^2(f_\infty + g_\infty)^2 \leq 4Q(f_\infty - g_\infty)^2$ by uniform boundedness of $\s{F}$. The third inequality is because for any $f_1, g_1, f_2, g_2 \in \s{F}$, we have $Q((f_1 - g_1) - (f_2 - g_2))^2 \leq Q(f_1 - f_2)^2 + Q(g_1 - g_2)^2 + 2\|f_1-f_2\|_{L_2(Q)}\|g_1-g_2\|_{L_2(Q)}$. Since $\s{F}$ has uniformly integrable $L_2$ entropy, we have 
    $
        \int_0^\infty \sup_{Q} \left\{ \log N(\varepsilon, \s{F}, L_2(Q)) \right\}^{1/2} \sd \varepsilon < \infty,
    $
    where the supremum is taken over all finitely discrete probability measures $Q$ on $(\s{X}, \s{B})$. By exercise 1 in Chapter~2.5 of \cite{van1996weak}, the supremum over finitely discrete $Q$ can be replaced by the supremum over all probability measures $Q$ such that $0< QF^2 < \infty$ without changing the assumption of uniformly integrable $L_2$ entropy of $\s{F}$. So, we have 
    \begin{equation}
        \int_0^\infty \sup_{Q} \left\{ \log N(\varepsilon, \s{F}_\infty^2, L_1(Q)) \right\}^{1/2} \sd \varepsilon < \infty,
        \label{eq: uniform entropy 2}
    \end{equation}
    where the supremum is taken over all probability measures $Q$ on $(\s{X}, \s{B})$. Since $N(\varepsilon, \s{F}_\infty^2, L_1(\hat{P}_n)) \leq \sup_Q N(\varepsilon, \s{F}_\infty^2, L_1(Q))$ almost surely, and the latter is finite for $\varepsilon > 0$ by~\eqref{eq: uniform entropy 2}, $N(\varepsilon, \s{F}_\infty^2, L_1(\hat{P}_n))$ is stochastically bounded for all $\varepsilon > 0$. Furthermore, since $\| \cdot \|_{d_{\tilde\nu_n}^{(1)}} \leq \|\cdot \|_{L_1(\hat{P}_n)}$ for $d_{\tilde\nu_n}^{(1)}$ defined in \cite{gaenssler2000uniform}, this further implies that condition (2.4) of Theorem~2.2 of \cite{gaenssler2000uniform} holds.
    
    Lastly, to show condition (2.5) in Theorem 2.2 of \cite{gaenssler2000uniform} holds, we show that
    \begin{align*}
        & \lim_{n\to \infty}\sup_{f_\infty \in \s{F}_\infty} \abs{P_0*L_n (f_\infty^2) - P_0(f_\infty^2)} \\
        & \qquad = \lim_{n\to \infty}\sup_{f_\infty \in \s{F}_\infty} \abs{\iint f_\infty(x+y)^2 \sd P_0(x) \sd  L_n(y) - \iint  f_\infty(x+y)^2 \sd P_0(x) \sd  \delta_0(y)} = 0.
    \end{align*}
   The function $y \mapsto \int f_\infty^2(x+y) \sd P_0(x)$ is uniformly bounded by 1 since  $\s{F}_\infty^2$ is uniformly bounded by 1. Furthermore, for any $y_1, y_2 \in \s{X}$ and $f_\infty \in \s{F}_\infty$, using the assumption that $p_0$ is Lipschitz, we have 
    \begin{align*}
        &\abs{ \int f_\infty(x+y_1)^2 \sd P_0(x) - \int f_\infty(x+y_2)^2 \sd P_0(x) } \\
        & \qquad = \abs{ \int f_\infty(x+y_1)^2 p_0(x)\sd \lambda(x) - \int f_\infty(x+y_2)^2 p_0(x)\sd \lambda(x) }\\
        & \qquad = \abs{ \int f_\infty(z)^2 [p_0(z-y_1) - p_0(z-y_2)] \sd \lambda(z) }\\
        & \qquad \leq  \gamma\abs{y_1 - y_2} \int f_\infty^2 \sd \lambda\\
        & \qquad \leq  4\gamma \abs{y_1 - y_2} \int F^2 \sd \lambda.
    \end{align*}
    for some $\gamma >0$. Since $ \int F^2 \sd \lambda <\infty$ by assumption, this implies $y \mapsto \int f_\infty^2(x+y) \sd P_0(x)$ is uniformly bounded by 1 and $\gamma'$-Lipschitz for $\gamma' := 4\gamma\int F^2 \sd \lambda$ for all $f_\infty \in \s{F}_\infty$. Thus, the function $\min(\gamma'^{-1}, 1) \times \int f_\infty^2(x+\cdot) \sd P_0(x) \in \BL_1(\d{R})$, which implies that 
    \begin{align*}
        &\sup_{f_\infty \in \s{F}_\infty} \abs{\iint f_\infty(x+y)^2 \sd P_0(x) \sd  L_n(y) - \iint f_\infty(x+y)^2 \sd P_0(x) \sd  \delta_0(y)} \\
        & \qquad \leq \frac{1}{\min(\gamma'^{-1}, 1)} \sup_{h \in \BL_1(\s{X})} \abs{\int h(y) \sd  L_n(y) - \int h(y) \sd  \delta_0(y)},
    \end{align*}
    which goes to $0$ as $n \to \infty$ because $L_n$ converges weakly to $\delta_0$ by assumption. We have now checked all the conditions of Theorem~2.2 of \citealp{gaenssler2000uniform}, so we conclude that $E_0^*\| \hat{P}_n - P_0 \|_{\s{F}_\infty^2}^p \to 0$ for all $p \geq 1$. This demonstrates that condition (iii) of Lemma~\ref{prop: uniform donsker} holds, so the result follows.

\end{proof}

\lemmabootstrapvar*
\begin{proof}[\bfseries{Proof of Lemma~\ref{lemma: bootstrap variance}}]
    By adding and subtracting terms, 
    \begin{equation}\label{eq: bootstrap variance decomposition}
        \d{P}_n^* \phi_n^{*2} - \sigma_0^2 = (\d{P}_n^* - \hat{P}_n) \phi_n^{*2}  + (\hat{P}_n -P_0)\phi_n^{*2} + P_0(\phi_n^{*2}-\phi_0^2).
    \end{equation}
    For the first term on the right-hand side of (\ref{eq: bootstrap variance decomposition}), for any $\varepsilon >0$, we have 
    \begin{align*}
        \prob_W^* \left( \abs{(\d{P}_n^* - \hat{P}_n) \phi_n^{*2}} \geq \varepsilon \right) &= \prob_W^* \left( \abs{(\d{P}_n^* - \hat{P}_n) \phi_n^{*2}} \geq \varepsilon,\, \phi_n^{*2} \in \s{G} \right) \\
        & \qquad + \prob_W^* \left( \abs{(\d{P}_n^* - \hat{P}_n) \phi_n^{*2}} \geq \varepsilon,\, \phi_n^{*2} \notin \s{G} \right)\\
        & \leq \prob_W^* \left( \sup_{g \in \s{G}} \abs{(\d{P}_n^* - \hat{P}_n) g} \geq \varepsilon \right) + \prob_W^* \left( \phi_n^{*2} \notin \s{G} \right)\\
        & = o_{\prob_0^*}(1).
    \end{align*}
    For the second term on the right-hand side of (\ref{eq: bootstrap variance decomposition}), for any $\varepsilon > 0$, we have
    \begin{align*}
        \prob_W^*\left( \abs{(\hat{P}_n - P_0) \phi_n^{*2}} \geq \varepsilon \right) &= \prob_W^* \left( \abs{(\hat{P}_n - P_0) \phi_n^{*2}} \geq \varepsilon,\, \phi_n^{*2} \in \s{G} \right) \\
        & \qquad + \prob_W^* \left( \abs{(\hat{P}_n - P_0) \phi_n^{*2}} \geq \varepsilon,\, \phi_n^{*2} \notin \s{G} \right)\\
        & \leq \prob_W^* \left( \sup_{g \in \s{G}} \abs{(\hat{P}_n - P_0) g} \geq \varepsilon \right) + \prob_W^* \left( \phi_n^{*2} \notin \s{G} \right)\\
        & = o_{\prob_0^*}(1).
    \end{align*}
    Therefore, $\d{P}_n^* \phi_n^* - \sigma_0^2 = o_{\prob_W^*}(1)$ by assumption. Furthermore, $P_0 (\phi_n^{*2} - \phi_0^2) = o_{\prob_W^*}(1)$ is implied by condition~\ref{cond: bootstrap weak consistency}. This is because by Cauchy-Schwarz and Minkowski's inequalities,
    \begin{align*}
        \abs{P_0 (\phi_n^{*2} - \phi_0^2)} &\leq P_0 | \phi_n^{*2} - \phi_0^2| \\
        & = P_0 \abs{(\phi_n^* - \phi_0)(\phi_n^* + \phi_0)} \\
        & \leq \norm{\phi_n^* - \phi_0}_{L_2(P_0)} \norm{\phi_n^* + \phi_0}_{L_2(P_0)} \\
        & \leq \norm{\phi_n^* - \phi_0}_{L_2(P_0)} [\norm{\phi_n^* - \phi_0}_{L_2(P_0)} + 2\norm{\phi_0}_{L_2(P_0)}]\\
        & = o_{\prob_W^*}(1).
    \end{align*}
    Specifically, for the empirical bootstrap $\hat{P}_n = \d{P}_n$, condition~\ref{cond: bootstrap limited complexity}(b) implies $\s{F}$ is $P_0$-Donsker by Theorem~3.6.1 of \cite{van1996weak}. Hence, $\s{G} = \{f^2: f \in \s{F}\}$ is $P_0$-Glivenko-Cantelli in probability by  Lemma~2.10.14 of \cite{van1991differentiable}, which implies (ii) and (iii) by Theorem~2.6 of \cite{gine1990bootstrapping}. Meanwhile, (i) is implied by condition~\ref{cond: bootstrap limited complexity}(a), where $\s{G} = \{f^2: f \in \s{F}\}$.
\end{proof}

\begin{lemma}[Bootstrap Slutsky's Theorem]\label{lemma: slutsky}
    Let $X_n^*$ and $Y_n^*$ be two sequences of real-valued random variables defined on the product probability space $(\s{X}, \s{B}, P_0)^n \times (\s{W}_n, \s{C}_n, Q_n)^n$. Assume that $X_n^* \condinoutdist X$ and $Y_n^* \condinoutprob c$ in $\d{R}$, where $X$ is a tight, Borel measurable limit and $c$ is a constant. We then have (i) $Y_n^* X_n^* \condinoutdist cX$ in $\d{R}$, (ii) $(Y_n^*)^{-1} X_n^* \condinoutdist c^{-1}X$ in $\d{R}$ provided $c \neq 0$, and (iii)  $X_n^* + Y_n^* \condinoutdist X + Y$ in $\d{R}$.
\end{lemma}
\begin{proof}[\bfseries{Proof of Lemma~\ref{lemma: slutsky}}]
    We only show (i); a similar argument can be applied to yield (ii) and (iii). As in the proof of Lemma \ref{lemma: boot_ran}, a standard way to prove results of this type outside the setting of the
    bootstrap is the continuous mapping theorem (e.g., Lemma~2.8 of \citealp{van2000asymptotic}). However, the continuous mapping theorem might not be applicable in the bootstrap setting as the map $(w_1, \dotsc, w_n) \mapsto h(\d{G}_n^*(x_1, \dotsc, x_n, w_1, \dotsc, w_n))$ might not be measurable given almost every $x_1, \dotsc, x_n$ for all $h \in C_b(\ell^\infty(\s{F}))$.
    
    We let $h$ be an arbitrary element of $\BL_1(\d{R})$. By adding and subtracting terms and the triangle inequality, for any $\varepsilon > 0$,
    \begin{align}
        | \E_W^*h(X_n^* Y_n^*) - \E_0 h(X c) | &\leq | \E_W^*h(X_n^* Y_n^*) - \E_W^* h(X_n^* c) | + | \E_W^*h(X_n^* c) - \E_0 h(X c) | \notag\\
        &\leq \varepsilon + 2 \prob_W^* (|X_n^* Y_n^* - X_n^* c| \geq \varepsilon) + | \E_W^*h(X_n^* c) - \E_0 h(X c) |\label{eq: multiple slutsky},
    \end{align}
    where the second inequality is because $\abs{h(x_1) - h(x_2)} \leq 2 \wedge |x_1-x_2| \leq \varepsilon + 2 I(|x_1-x_2| > \varepsilon)$  for any $x_1, x_2 \in \d{R}$. By defining $g: \d{R} \mapsto \d{R}$ as $g(x)=cx$, we note that $\min(|c|{-1}, 1)\times (h \circ g) \in \BL_1(\d{R})$ because for any $x_1, x_2 \in \d{R}$ and $x_1 \neq x_2$, 
    \begin{align*}
        & \frac{|\min(|c|^{-1}, 1) \times (h \circ g)(x_1) - \min(|c|^{-1}, 1) \times (h \circ g)(x_2)|}{|x_1 - x_2|} \\
        &\qquad = \min(|c|^{-1}, 1) \frac{|(h \circ g)(x_1) - (h \circ g)(x_2)|}{|g(x_1) - g(x_2)|}\frac{|g(x_1) - g(x_2)|}{|x_1 - x_2|}\\
        & \qquad \leq \min(|c| , 1) \frac{|(h \circ g)(x_1) - (h \circ g)(x_2)|}{|g(x_1) - g(x_2)|} \\
        &\qquad \leq 1
    \end{align*}
    because $h \in \BL_1(\d{R})$. Since in addition $|\min(|c|^{-1}, 1)\times (h \circ g)| \leq 1$,  $\min(|c|^{-1}, 1)\times (h \circ g) \in \BL_1(\d{R})$. Hence, for the third term on the right-hand side of \eqref{eq: multiple slutsky}, we have 
    \begin{align*}
        | \E_W^*h(X_n^* c) - \E_0 h(X c) | & = \frac{1}{\min(|c|^{-1}, 1)} \abs{\E_W^*  (\min(|c|^{-1}, 1) \times (h \circ g)) (X_n^*) - \E_0 (\min(|c|^{-1}, 1) \times (h\circ g))(X)} \\
        & \leq \frac{1}{\min(|c|^{-1}, 1)} \sup_{h \in \BL_1(\d{R})} | \E_W^*h(X_n^*) - \E_0 h(X) |,
    \end{align*}
    which converges to 0 in outer probability since $X_n^* \condinoutdist X$ by assumption. For the second term on the right-hand side of \eqref{eq: multiple slutsky}, for any $M>0$, 
    \begin{align*}
        \prob_W^* (|X_n^* Y_n^* - X_n^* c| \geq \varepsilon) & = \prob_W^* (|X_n^* Y_n^* - X_n^* c| \geq \varepsilon, |X_n^*| < M) \\
        & \qquad + \prob_W^* (|X_n^* Y_n^* - X_n^* c| \geq \varepsilon, |X_n^*| \geq M)\\
        & \leq \prob_W^* (|Y_n^* - c| \geq \varepsilon/M) + \prob_W^* (|X_n^*| \geq M).
    \end{align*}
    By assumption, $Y_n^* \condinoutprob c$ so that $\prob_W^* (|Y_n^* - c| \geq \varepsilon/M) = o_{\prob_0^*}(1)$. Hence, the proof is complete if we can show that given any $\varepsilon, \gamma > 0$, there exists $M>0$ such that $\prob_0^*(\prob_W^* (|X_n^*| \geq M) \geq \varepsilon) < \gamma$ for all $n$ large enough. 
    
    For each $\varepsilon, \gamma >0$ there exists a sequence of functions $h_m: \d{R} \to \d{R}$ such that $h_m$ is $m$-Lipschitz for all $m$, $1\{ |x| \geq M\}\leq h_m(z) \leq 1$ for all $x \in \d{R}$ and $m$, and $h_m(z)$ monotonically decreases to $1\{ |x| \geq M\}$ as $m \to \infty$ for each $x \in \d{R}$. For instance, $h_m : x \mapsto \min\{ \max(m[|x| - M] + 1, 0), 1\}$ satisfies these criteria. For any $\varepsilon > 0$ and $m \in \{1, 2, \dots\}$, we can now write
    \begin{align*}
        \prob^*_W \left( |X_n^*| \geq M \right) &= \left[ \E^*_W 1\left\{ |X_n^*| \geq M \right\} - \E^*_W h_m\left( X_n^* \right) \right] + \left[\E^*_W h_m\left( X_n^* \right) - \E_0 h_m(X) \right] \\
        &\qquad + \left[ \E_0 h_m(X)  -  \E_0 1\left\{ |X| \geq M \right\}  \right] + \prob_0 \left( |X| \geq M \right) \\
        &\leq \left[\E^*_W h_m\left( X_n^* \right) - \E_0 h_m(X) \right]  + \left[ \E_0 h_m(X)  -  \E_0 1\left\{ |X| \geq M \right\}  \right] \\
        & \qquad + \prob_0 \left( |X| \geq M \right),
    \end{align*}
   where the second inequality follows because $\E^*_W 1\left\{ |X_n^*| \geq M \right\} \leq \E^*_W h_m\left( X_n^* \right)$ for all $\varepsilon >0$ and $m \in \{1,2,\dots\}$ by assumption. We therefore have 
   \begin{align}\label{eq: slutsky decomposition}
       &\prob_0^*(\prob_W^* (|X_n^*| \geq M) \geq \varepsilon) \notag\\
       & \qquad \leq \prob_0^*\left([\E^*_W h_m\left( X_n^* \right) - \E_0 h_m(X)]  \geq \varepsilon/3 \right) + \prob_0^*\left( [\E_0 h_m(X)  -  \E_0 1\left\{ |X| \geq M \right\}] \geq \varepsilon/3\right)  \notag\\
       & \qquad \qquad + \prob_0^*(  \prob_0 \left( |X| \geq M \right) \geq \varepsilon/3) \notag\\
       & \qquad = \prob_0^*\left([\E^*_W h_m\left( X_n^* \right) - \E_0 h_m(X)]  \geq \varepsilon/3 \right) + 1\left\{ [\E_0 h_m(X)  -  \E_0 1\left\{ |X| \geq M \right\}] \geq \varepsilon/3\right\}  \notag\\
       & \qquad \qquad + 1\left\{  \prob_0 \left( |X| \geq M \right) \geq \varepsilon/3\right\}. 
   \end{align}
   First, since $X$ is tight by assumption, we can choose $M$ suh that the third term on the right-hand side of \eqref{eq: slutsky decomposition} is 0. Next, for this choice of $M$, by the monotone convergence theorem we can choose $m$ large enough so that the second term on the right-hand side of \eqref{eq: slutsky decomposition} is 0. Finally, for these choices of $m$ and $M$, for the first  term on the right-hand side of \eqref{eq: slutsky decomposition}, we have $\min(m^{-1}, 1) \times h_m \in \BL_1(\d{R})$ because $h_m$ is bounded and $m$-Lipschitz. Therefore, by the assumed conditional weak convergence of $X_n^*$ to $X$ in $\d{R}$, $\E^*_W [\min(m^{-1}, 1) \times h_m\left( X_n^* \right)] \inoutprob \E_0 [\min(m^{-1}, 1) \times h_m(X)]$ for each $m$, so that $\E^*_W h_m\left( X_n^* \right) \inoutprob \E_0 h_m(X)$ as well. Hence, the first  term on the right-hand side of \eqref{eq: slutsky decomposition} converges to 0 for fixed $m$ and $\varepsilon$, so we can make it as small as we like for $n$ large enough.  Hence, for any $\varepsilon$, there exist $M>0$ and $m \in \{1, 2, \dots\}$ such that $ \prob_W^* (|X_n^*| \geq M) = o_{\prob_0^*}(1)$ for $n$ large enough.
\end{proof}

We next have a result that generalizes Theorem~23.4 of \cite{van2000asymptotic} to include smooth bootstraps. 
\begin{lemma}[Conditional CLT]\label{thm: clt}
    Let $\phi_0: \d{R}^d \mapsto \d{R}$ satisfy $P_0 \phi_0 = 0$ and $P_0 \phi_0^2 < \infty$. Let $X_1, X_2, \dots$ be IID random vectors with mean $\mu$ and covariance matrix $\Sigma$. If $\hat{P}_n \phi_0^2 \inoutprob P_0 \phi_0^2  $ and $(\hat{P}_n - P_0) [\phi_0^2 1\{|\phi_0| > M\}] \inoutprob 0$ for any $M>0$, then conditionally on every sequence $X_1, X_2, \dots$, in outer probabiliy,
    \[
        n^{1/2} \left( \d{P}_n^* - \hat{P}_n \right)\phi_0  \condinoutdist N(0, P_0 \phi_0^2).
    \]
\end{lemma}
\begin{proof}[\bfseries{Proof of \Cref{thm: clt}}]
    We use the Lindeberg-Feller CLT for triangular arrays (e.g., Proposition 2.27 of \citealp{van2000asymptotic}). Since the bootstrap sample $X_1^*, \dotsc, X_n^*$ is drawn IID from $\hat{P}_n$ given $X_1, \dotsc, X_n$, the conditional mean and variance of $\phi_0(X_i^*)$ are given by
    \begin{align*}
        \E_W \phi_0(X_i^*) &= \int \phi_0(x) \sd \hat{P}_n(x) = \hat{P}_n \phi_0, \text{ and}\\
        \E_W [\phi_0(X_i^*) - \E_W \phi_0(X_i^*)]^2 &= \E_W \phi_0^2(X_i^*) - [\E_W \phi_0(X_i^*)]^2 = \hat{P}_n \phi_0^2 - (\hat{P}_n \phi_0)^2 \inoutprob P_0 \phi_0^2.
    \end{align*}
    We next verify the Lindeberg condition. For any $\varepsilon, \gamma > 0$, there exists $M > 0$ such that $\int \phi_0(x)^2 1\{ |\phi_0(x)| > M\} \sd P_0(x) < \varepsilon/2$ and for all $n$ large enough, $M \leq \gamma \sqrt{n}$. Therefore, 
    \begin{align*}
        & \prob_0\left( \E_W [\phi_0(X_i^*)^2 1\{ |\phi_0(X_i^*) |> \gamma \sqrt{n}\}] \geq \varepsilon \right)\\
        & \qquad = \prob_0\left( \hat{P}_n [\phi_0^2 1\{ |\phi_0| > \gamma \sqrt{n}\}] \geq \varepsilon \right)\\
        & \qquad \leq \prob_0\left( \hat{P}_n [\phi_0^2 1\{ |\phi_0| > M\}] \geq \varepsilon \right)\\
        & \qquad \leq \prob_0\left( |(\hat{P}_n - P_0) [\phi_0^2 1\{ |\phi_0| > M\}]| \geq \varepsilon/2 \right) + \prob_0\left(  P_0 [\phi_0^2 1\{ |\phi_0| > M\}] \geq \varepsilon/2 \right)\\
        & \qquad = \prob_0\left( |(\hat{P}_n - P_0) [\phi_0^2 1\{ |\phi_0| > M\}]| \geq \varepsilon/2 \right) + 1\left\{  P_0 [\phi_0^2 1\{ |\phi_0| > M\}] \geq \varepsilon/2 \right\}\\
        & \qquad = \prob_0\left( |(\hat{P}_n - P_0) [\phi_0^2 1\{ |\phi_0| > M\}]| \geq \varepsilon/2 \right)\\
        & \qquad \longrightarrow 0
    \end{align*}
    
     
     Thus, $\E_W [\phi_0(X_i^*)^2 1\{ |\phi_0(X_i^*)| > \varepsilon \sqrt{n}\}] \inoutprob  0$. Therefore, the conditions of the Lindeberg-Feller CLT holds conditionally on every sequence $X_1, X_2, \dots$, in outer probability, which yields the result.
\end{proof}

\thmprecinterval* 
\begin{proof}[\bfseries{Proof of \Cref{thm: perc bootstrap CI}}]
    We prove the result for the percentile $t$-method; the result for the percentile method follows the same argument setting $\sigma_n = \sigma_n^* = \sigma_0 = 1$ (so that the additional conditions for the percentile $t$-method hold automatically).

    We first show that $\sigma_n^{*-1}(S_n^* + R_n^*) - \sigma_n^{-1}(S_n + R_n) = o_{\prob_W^*}(n^{-1/2})$. By adding and subtracting terms, we note that 
    \begin{align}\label{eq: bias ci decom}
        \frac{S_n^* + R_n^*}{\sigma_n^*} - \frac{S_n + R_n}{\sigma_n} &= \frac{\sigma_n(S_n^* + R_n^*) - \sigma_n(S_n + R_n) + \sigma_n(S_n + R_n) - \sigma_n^*(S_n + R_n)}{\sigma_n\sigma_n^*} \notag \\
        & = \frac{(S_n^* - S_n) + (R_n^* - R_n)}{\sigma_n^*} + \frac{(\sigma_n - \sigma_n^*)(S_n + R_n) }{\sigma_n\sigma_n^*}.
    \end{align}
    Since $\sigma_n^{*} \condinoutprob \sigma_0 > 0$, we can show that $\sigma_n^{*-1} = O_{\prob_W^*}(1)$ using the same logic as the unconditional result. Hence, the above equals $o_{\prob_W^*}(n^{-1/2})O_{\prob_W^*}(1) + o_{\prob_W^*}(n^{-1/2})O_{\prob_W^*}(1)$ by assumption.
    
    For any $t \in \d{R}$, by definition of $S_n$ and $R_n$, we have
    \begin{align*}
        \prob_0^* \left( \frac{T(\eta_n, \d{P}_n) - T(\eta_0, P_0)}{\sigma_n} \leq t \right) &= \prob_0^* \left(\sigma_n^{-1}\d{G}_n \phi_0 + n^{1/2}\sigma_n^{-1} \left[S_n + R_n\right]   \leq n^{1/2} t \right) \\
        &= \prob_0^* \left(\sigma_n^{-1}\d{G}_n \phi_0   \leq T_n\right),
    \end{align*}
    where $T_n := n^{1/2} t -  n^{1/2}\sigma_n^{-1} \left[S_n + R_n\right]$. Thus, by definition of $S_n^*$ and $R_n^*$, 
    \begin{align*}
        \prob_W^* \left( \frac{T(\eta_n^*, \d{P}_n^*) - T(\eta_n, \hat{P}_n)}{\sigma_n^*} \leq t \right) &= \prob_W^* \left(\sigma_n^{*-1}\d{G}_n^* \phi_0 + n^{1/2}\sigma_n^{*-1} \left[S_n^* + R_n^*\right]   \leq n^{1/2} t \right)\\
        &= \prob_W^* \left(\sigma_n^{*-1}\d{G}_n^* \phi_0 + n^{1/2}\sigma_n^{*-1} \left[S_n^* + R_n^*\right]   \leq T_n + n^{1/2}\sigma_n^{-1} \left[S_n + R_n\right] \right) \\
        &= \prob_W^* \left(\sigma_n^{*-1}\d{G}_n^* \phi_0 + A_n^* \leq T_n  \right),
    \end{align*}
    where $A_n^* :=n^{1/2} \left\{\sigma_n^{*-1} \left[S_n^* + R_n^*\right] - \sigma_n^{-1} \left[ S_n + R_n \right]\right\}$. Thus,
    \begin{align*}
    \sup_{t \in \d{R}} &\left|\prob_W^* \left( \frac{T(\eta_n^*, \d{P}_n^*) - T(\eta_n, \hat{P}_n)}{\sigma_n^*} \leq t \right) -  \prob_0^* \left( \frac{T(\eta_n, \d{P}_n) - T(\eta_0, P_0)}{\sigma_n} \leq t \right) \right| \\
    &= \sup_{t \in \d{R}} \left| \prob_W^* \left(\sigma_n^{*-1}\d{G}_n^* \phi_0 + A_n^* \leq t  \right) - \prob_0^* \left(\sigma_n^{-1}\d{G}_n \phi_0   \leq t \right) \right| \\
    &\leq \sup_{t \in \d{R}} \left| \prob_W^* \left(\sigma_n^{*-1}\d{G}_n^* \phi_0 + A_n^* \leq t  \right)  - \prob_0^* \left( \sigma_0^{-1}\d{G}_0 \phi_0 \leq t \right) \right| +  \sup_{\xi \in \d{R}} \left| \prob_0^* \left(\sigma_n^{-1}\d{G}_n \phi_0   \leq t \right)  - \prob_0^* \left( \sigma_0^{-1} \d{G}_0 \phi_0 \leq t \right)  \right|.
    \end{align*}
     By \Cref{thm: clt}, we have $\d{G}_n^*\phi_0 \condinoutdist \d{G}_0 \phi_0$. Since $\sigma_n^{*2} \condinoutprob \sigma_0^2$, $\sigma_n^{*-1} \d{G}_n^*\phi_0 \condinoutdist \sigma_0^{-1}\d{G}_0 \phi_0$ by \Cref{lemma: slutsky}. Since $A_n^* = o_{\prob_W^*}(1)$, it follows that $\sigma_n^{*-1} \d{G}_n^*\phi_0  + A_n^* \condinoutdist \sigma_0^{-1}\d{G}_0 \phi_0$. Hence, by the Portmanteau theorem (e.g., Lemma~1.3.4 of \citealp{van1996weak}), 
     \[ 
         \sup_{t \in \d{R}} \left| \prob_W^* \left(\sigma_n^{*-1}\d{G}_n^* \phi_0 + A_n^* \leq t  \right)  - \prob_0^* \left( \sigma_0^{-1}\d{G}_0 \phi_0 \leq t \right) \right| \inoutprob 0.
     \]
     Similarly, $\sigma_n^{-1} \d{G}_n\phi_0 \leadsto \sigma_0^{-1}\d{G}_0 \phi_0$, so 
     \[ \sup_{t \in \d{R}} \left| \prob_0^* \left(\sigma_n^{-1}\d{G}_n \phi_0   \leq \xi \right)  - \prob_0^* \left( \sigma_0^{-1} \d{G}_0 \phi_0 \leq t \right)  \right| \to 0.\]
     The first statement follows. 

    By definition of $\xi_{n, p}^*,$ $S_n^*$, and $R_n^*$, we have
    \begin{align*}
        \xi_{n, p}^* :=& \inf\left\{ \xi \in \d{R} : \prob_W^* \left( \frac{T(\eta_n^*, \d{P}_n^*) - T(\eta_n, \hat{P}_n)}{\sigma_n^*} \leq \xi \right) \geq p \right\}\\
        =& \inf\left\{ \xi \in \d{R} : \prob_W^* \left(  n^{-1/2}\sigma_n^{*-1}\d{G}_n^* \phi_0 + \sigma_n^{*-1} \left[S_n^* + R_n^*\right] \leq \xi \right) \geq p \right\}.
    \end{align*}
    Since $\sigma_n^{-1} \left[ S_n + R_n \right]$ is a function of the original data, we have
    \begin{align*}
        &n^{1/2} \left( \xi_{n, p}^* - \sigma_n^{-1} \left[ S_n + R_n \right] \right) \\
        &\qquad= \inf\left\{ \xi \in \d{R} : \prob_W^* \left(  n^{-1/2}\sigma_n^{*-1}\d{G}_n^* \phi_0 + \sigma_n^{*-1} \left[S_n^* + R_n^*\right] \leq n^{-1/2}\xi + \sigma_n^{-1} \left[ S_n + R_n \right] \right) \geq p \right\} \\
        &\qquad = \inf\left\{ \xi \in \d{R} : \prob_W^* \left( \sigma_n^{*-1}\d{G}_n^* \phi_0  + A_n^*  \leq \xi\right) \geq p \right\}.
    \end{align*}
    By the above derivations and \citealp[Lemma~21.2]{van2000asymptotic}, for any $p \in (0,1)$, we then have 
    \[ n^{1/2} \left( \xi_{n, p}^* - \sigma_n^{-1} \left[ S_n + R_n \right] \right) \inoutprob \Phi^{-1}(p).\]
    Therefore,
    \begin{align*}
        &\prob_0^* \left( T(\eta_n, \d{P}_n) - \xi_{n, 1-\alpha}^* \sigma_n \leq T(\eta_0, P_0) \leq T(\eta_n, \d{P}_n) - \xi_{n, \beta}^* \sigma_n \right)\\
        & \qquad = \prob_0^* \left( n^{1/2}\xi_{n, \beta}^* \leq n^{1/2}[T(\eta_n, \d{P}_n) - T(\eta_0, P_0)]/\sigma_n \leq n^{1/2}\xi_{n, 1-\alpha}^* \right) \\
        & \qquad = \prob_0^* \left( n^{1/2}\left[\xi_{n, \beta}^* - \sigma_n^{-1}(S_n + R_n)\right] \leq \sigma_n^{-1} \d{G}_n \phi_0 \leq n^{1/2}\left[\xi_{n, 1-\alpha}^* - \sigma_n^{-1}(S_n + R_n)\right]\right)\\
        & \qquad \longrightarrow \prob_0^* \left( \Phi^{-1}(\beta) \leq \sigma_0^{-1}\d{G}_0 \phi_0 \leq \Phi^{-1}(1-\alpha) \right) \\
        &\qquad= 1-\alpha - \beta.
    \end{align*}

\end{proof}

\corprecinterval*
\begin{proof}[\bfseries{Proof of \Cref{cor: perc bootstrap CI}}]
    We note that in the proof of~\Cref{thm: perc bootstrap CI}, the conditions $\hat{P}_n \phi_0^2 \inoutprob P_0 \phi_0^2$ and $(\hat{P}_n - P_0) [\phi_0^2 1\{|\phi_0| > M\}] \inoutprob 0$ were only used to establish that $\d{G}_n^*\phi_0 \condinoutdist \d{G}_0 \phi_0$ using Lemma~\ref{thm: clt}. However, condition~\ref{cond: bootstrap limited complexity} directly implies $\d{G}_n^*\phi_0 \condinoutdist \d{G}_0 \phi_0$, so it suffices to check the remainder of the conditions of  \Cref{thm: perc bootstrap CI}. By conditions~\ref{cond: limited complexity}--\ref{cond: second order}, we have $S_n = o_{\prob_0^*}(n^{-1/2})$ and $R_n = o_{\prob_0^*}(n^{-1/2})$. By conditions~\ref{cond: bootstrap limited complexity}--\ref{cond: bootstrap second order}, we have $S_n^* = o_{\prob_W^*}(n^{-1/2})$ and $R_n^* = o_{\prob_W^*}(n^{-1/2})$. Hence, $S_n^* - S_n = o_{\prob_W^*}(n^{-1/2})$ and $R_n^* - R_n = o_{\prob_W^*}(n^{-1/2})$, which implies the consistency of the bootstrap percentile confidence interval by \Cref{thm: perc bootstrap CI}. For the percentile $t$ interval, $\sigma_n^* \condinoutprob \sigma_0$ and $\sigma_n \inoutprob \sigma_0$ by assumption, which also implies $(S_n + R_n)(\sigma_n^* - \sigma_n) = o_{\prob_W^*}(n^{-1/2})o_{\prob_W^*}(1) = o_{\prob_W^*}(n^{-1/2})$.
\end{proof}

\thmefroninterval* 
\begin{proof}[\bfseries{Proof of \Cref{thm: efron bootstrap CI}}]
    For any $t \in \d{R}$, by definition of $S_n$ and $R_n$, we have
    \begin{align*}
        \prob_0^* \left( -\left[T(\eta_n, \d{P}_n) - T(\eta_0, P_0)\right] \leq t \right) &= \prob_0^* \left(-\d{G}_n \phi_0 -n^{1/2} \left[S_n + R_n\right]   \leq n^{1/2} t \right) = \prob_0^* \left(-\d{G}_n \phi_0   \leq T_n\right),
    \end{align*}
    where $T_n := n^{1/2} t +  n^{1/2}\left[S_n + R_n\right]$. Thus, by definition of $S_n^*$ and $R_n^*$, 
    \begin{align*}
        \prob_W^* \left( T(\eta_n^*, \d{P}_n^*) - T(\eta_n, \d{P}_n)\leq t \right) &= \prob_W^* \left(\d{G}_n^* \phi_0 + n^{1/2}\left[S_n^* + R_n^* + T(\eta_n, \hat{P}_n)  -T(\eta_n, \d{P}_n) \right]   \leq n^{1/2} t \right)\\
        &= \prob_W^* \left(\d{G}_n^* \phi_0 + B_n^* \leq T_n  \right),
    \end{align*}
    where $B_n^* :=n^{1/2} \left[S_n^* + S_n +  R_n^* + R_n  + T(\eta_n, \hat{P}_n)  -T(\eta_n, \d{P}_n)\right]$. Thus,
    \begin{align*}
    \sup_{t \in \d{R}} &\left|\prob_W^* \left( T(\eta_n^*, \d{P}_n^*) - T(\eta_n, \d{P}_n) \leq t \right) -  \prob_0^* \left( -\left[T(\eta_n, \d{P}_n) - T(\eta_0, P_0)\right] \leq t \right) \right| \\
    &= \sup_{t \in \d{R}} \left| \prob_W^* \left(\d{G}_n^* \phi_0 + B_n^* \leq t  \right) - \prob_0^* \left(-\d{G}_n \phi_0   \leq t \right) \right| \\
    &\leq \sup_{t \in \d{R}} \left| \prob_W^* \left(\d{G}_n^* \phi_0 + B_n^* \leq t  \right)  - \prob_0^* \left( \d{G}_0 \phi_0 \leq t \right) \right| +  \sup_{t \in \d{R}} \left| \prob_0^* \left(-\d{G}_n \phi_0   \leq t \right)  - \prob_0^* \left(  \d{G}_0 \phi_0 \leq t \right)  \right|.
    \end{align*}
     By \Cref{thm: clt} and since $B_n^* = o_{\prob_W^*}(1)$ by assumption, we have $\d{G}_n^*\phi_0  + B_n^* \condinoutdist \d{G}_0 \phi_0$. Hence, by the Portmanteau theorem (e.g., Lemma~1.3.4 of \citealp{van1996weak}), 
     \[ 
         \sup_{t \in \d{R}} \left| \prob_W^* \left(\d{G}_n^* \phi_0 + B_n^* \leq t  \right)  - \prob_0^* \left( \d{G}_0 \phi_0 \leq t \right) \right| \inoutprob 0.
     \]
     Similarly, $\d{G}_n\phi_0 \leadsto \d{G}_0 \phi_0$, and since $\d{G}_0\phi_0$ follows a mean-zero normal distribution, $\d{G}_n\phi_0 \leadsto -\d{G}_0 \phi_0$ as well. Therefore,
     \[ \sup_{t \in \d{R}} \left| \prob_0^* \left(\d{G}_n \phi_0   \leq t \right)  - \prob_0^* \left( -\d{G}_0 \phi_0 \leq t \right)  \right| \to 0.\]
     The first statement follows.
    
    By definition of $\zeta_{n, p}^*, S_n^*$ and $R_n^*$, we have 
    \begin{align*}
        \zeta_{n, p}^* :=& \inf\left\{ \zeta \in \d{R}: \prob_W^*\left( T(\eta_n^*, \d{P}_n^*) \leq \zeta \right) \geq p\right\}\\
        = & \inf\left\{ \zeta \in \d{R}: \prob_W^*\left( n^{-1/2} \d{G}_n^* \phi_0 + (S_n^* + R_n^*) + T(\eta_n, \hat{P}_n) \leq \zeta \right) \geq p\right\}.
    \end{align*}
    Since $[S_n + R_n - T(\eta_n, \d{P}_n)]$ is a function of the original data, we have
    \begin{align*}
        &n^{1/2} \left\{ \zeta_{n, p}^* + [S_n + R_n - T(\eta_n, \d{P}_n)] \right\} \\
        &\qquad= \inf\left\{ \zeta \in \d{R} : \prob_W^* \left(  n^{-1/2}\d{G}_n^* \phi_0 + (S_n^* + R_n^*) + T(\eta_n, \hat{P}_n)\leq n^{-1/2}\zeta -  [S_n + R_n - T(\eta_n, \d{P}_n)] \right) \geq p \right\} \\
        &\qquad = \inf\left\{ \zeta \in \d{R} : \prob_W^* \left( \d{G}_n^* \phi_0  + B_n^*  \leq \zeta\right) \geq p \right\},
    \end{align*}
    As argued above, $\d{G}_n^*\phi_0  + B_n^* \condinoutdist \d{G}_0 \phi_0$, so by Lemma~21.2 of \cite{van2000asymptotic}, for any $p \in (0,1)$, we then have 
    \[ 
        n^{1/2} \left\{ \zeta_{n, p}^* + \left[ S_n + R_n - T(\eta_n, \d{P}_n)\right] \right\} \inoutprob \sigma_0\Phi^{-1}(p).
    \]
    Therefore,
    \begin{align*}
        &\prob_0^* \left( \zeta_{n, 1-\alpha}^* \leq T(\eta_0, P_0) \leq \zeta_{n, \beta}^* \right)\\
        & \qquad = \prob_0^* \left( -n^{1/2}[\zeta_{n, \beta}^* - T(\eta_n, \d{P}_n)] \leq n^{1/2}[T(\eta_n, \d{P}_n) - T(\eta_0, P_0)] \leq -n^{1/2}[\zeta_{n, 1-\alpha}^* - T(\eta_n, \d{P}_n)] \right) \\
        & \qquad = \prob_0^* \left( -n^{1/2}[\zeta_{n, \beta}^* - T(\eta_n, \d{P}_n)] \leq \d{G}_n \phi_0 + n^{1/2} (S_n + R_n)\leq -n^{1/2}[\zeta_{n, 1-\alpha}^* - T(\eta_n, \d{P}_n)] \right) \\
        & \qquad = \prob_0^* \left( -n^{1/2}[\zeta_{n, \beta}^*+S_n + R_n - T(\eta_n, \d{P}_n)] \leq \d{G}_n \phi_0 \leq -n^{1/2}[\zeta_{n, 1-\alpha}^* +S_n + R_n- T(\eta_n, \d{P}_n)] \right) \\
        & \qquad \longrightarrow \prob_0^* \left( -\Phi^{-1}(\beta) \leq \sigma_0^{-1}\d{G}_0 \phi_0 \leq -\Phi^{-1}(1-\alpha) \right) \\
        &\qquad= 1-\alpha - \beta,
    \end{align*}
    where the last equality is due to the symmetry of the standard normal distribution. This shows the consistency of Efron’s percentile method.

\end{proof}

\corefroninterval*
\begin{proof}[\bfseries{Proof of \Cref{cor: efron bootstrap CI}}]
    As in the proof of Corollary~\ref{cor: perc bootstrap CI}, we note that in the proof of~\Cref{thm: efron bootstrap CI}, the conditions $\hat{P}_n \phi_0^2 \inoutprob P_0 \phi_0^2$ and $(\hat{P}_n - P_0) [\phi_0^2 1\{|\phi_0| > M\}] \inoutprob 0$ were only used to establish that $\d{G}_n^*\phi_0 \condinoutdist \d{G}_0 \phi_0$ using Lemma~\ref{thm: clt}. However, condition~\ref{cond: bootstrap limited complexity} directly implies $\d{G}_n^*\phi_0 \condinoutdist \d{G}_0 \phi_0$, so it suffices to check the remainder of the conditions of  \Cref{thm: efron bootstrap CI}.

    By conditions~\ref{cond: limited complexity}--\ref{cond: second order}, we have $S_n = o_{\prob_0^*}(n^{-1/2})$ and $R_n = o_{\prob_0^*}(n^{-1/2})$. By conditions~\ref{cond: bootstrap limited complexity}--\ref{cond: bootstrap second order}, we have $S_n^* = o_{\prob_W^*}(n^{-1/2})$ and $R_n^* = o_{\prob_W^*}(n^{-1/2})$. Since $T(\eta_n, \hat{P}_n) - T(\eta_n, \d{P}_n) = o_{\prob_0^*}(n^{-1/2})$ by assumption, we then have $S_n^* + S_n + R_n^* + R_n +T(\eta_n, \hat{P}_n) - T(\eta_n, \d{P}_n) = o_{\prob_W^*}(n^{-1/2})$. Hence, the result follows by~\Cref{thm: efron bootstrap CI}.

\end{proof}

\thmwaldbootinterval*
\begin{proof}[\bfseries{Proof of \Cref{thm: wald bootstrap CI}}]
    Let $h(x) := x^2$. There exists a sequence of functions $h_m: \d{R} \mapsto \d{R}$, $m = 1, 2, \dots$ such that $h_m \in \BL_1(\d{R})$ for all $m$ and $mh_m$ monotonically increases to $h$ as $m\to \infty$. For instance, $h_m(x) := \min\{x^2 / (m \vee 4), 1\}$ satisfies these criteria. We then write
    \begin{align*}
        \E_W^*h\left(T_n^*\right) - \E_0 h(\d{G}_0 \phi_0) & = [\E_W^* h\left(T_n^*\right) - m\E_W^* h_m\left(T_n^*\right)] + [m\E_W^* h_m\left(T_n^*\right) - m\E_0 h_m\left(\d{G}_0 \phi_0\right)] \\
        & \qquad \qquad + [m\E_0h_m\left(\d{G}_0 \phi_0\right)  - \E_0h(\d{G}_0 \phi_0)].
    \end{align*}
    Hence, for any $\varepsilon > 0$, we have 
    \begin{align*}
        &  \prob_0^* \left( \abs{\E_W^*h\left(T_n^*\right) - \E_0 h(\d{G}_0 \phi_0)} \geq \varepsilon\right) \notag \\
        & \qquad \leq  \prob_0^* \left( \abs{\E_W^* h\left(T_n^*\right) - m\E_W^* h_m\left(T_n^*\right)} \geq \varepsilon/3\right) + \prob_0^* \left( \abs{\E_W^* h_m\left(T_n^*\right) - \E_0 h_m\left(\d{G}_0 \phi_0\right)} \geq \varepsilon/(3m) \right) \notag\\
        & \qquad \qquad + \prob_0^* \left( \abs{m\E_0h_m\left(\d{G}_0 \phi_0\right)  - \E_0h(\d{G}_0 \phi_0)} \geq \varepsilon/3 \right) \notag\\
        & \qquad \leq \prob_0^* \left( \abs{\E_W^* h\left(T_n^*\right) - m\E_W^* h_m\left(T_n^*\right)} \geq \varepsilon/3\right) + \prob_0^* \left( \sup_{g \in \BL_1(\d{R})}\abs{\E_W^* g\left(T_n^*\right) - \E_0 g\left(\d{G}_0 \phi_0\right)} \geq \varepsilon/(3m) \right) \notag\\
        & \qquad \qquad + 1 \left\{ \E_0\abs{mh_m\left(\d{G}_0 \phi_0\right)  - h(\d{G}_0 \phi_0)} \geq \varepsilon/3 \right\}.
    \end{align*}
    By conditions~\ref{cond: bootstrap limited complexity}-\ref{cond: bootstrap second order}, we have $ \E_W^* T_n^* \condinoutdist \d{G}_0 \phi_0$, i.e., $\sup_{g \in \BL_1(\d{R})} \abs{ \E_W^* g\left( T_n^* \right) - \E_0 g(\d{G}_0 \phi_0)} \inoutprob 0$. Therefore, for each fixed $m$, the second term on the right-hand side of previous display converges to 0. By the monotone convergence theorem, the third term on the right-hand side of previous display converges to 0 as $m \longrightarrow \infty$. For the first term on the right hand side of previous display, we note that 
    \begin{align*}
        \prob_0^* \left( \abs{\E_W^* h\left(T_n^*\right) - m\E_W^* h_m\left(T_n^*\right)} \geq \varepsilon/3\right) &  = \prob_0^* \left( \abs{\E_W^* T_n^{*2} - \E_W^* \min\{T_n^{*2}, m\}} \geq \varepsilon/3\right)\\
        & = \prob_0^* \left( \abs{\E_W^* (T_n^{*2} - m)1\{T_n^{*2} \geq m \} } \geq \varepsilon/3\right)\\
        & \leq \prob_0^* \left( \abs{\E_W^* T_n^{*2}1\{T_n^{*2} \geq m \} } \geq \varepsilon/3\right)\\
        & \leq \E_0^* \E_W^* T_n^{*2}1\{T_n^{*2} \geq m \} /(\varepsilon/3),
    \end{align*}
    and $\lim_{m\to \infty} \limsup_{n\to\infty}$ of this latter expression is 0 by assumption. Therefore, we have 
    \[
        \limsup_{n\to\infty} \prob_0^* \left( \abs{\E_W^*h\left(T_n^*\right) - \E_0 h(\d{G}_0 \phi_0)} \geq \varepsilon\right) = \lim_{m\to \infty} \limsup_{n\to\infty} \prob_0^* \left( \abs{\E_W^*h\left(T_n^*\right) - \E_0 h(\d{G}_0 \phi_0)} \geq \varepsilon\right) = 0
    \]
    for every $\varepsilon > 0$.  Hence, $\bar{\sigma}_n^2 = \E_W^* T_n^{*2} \inoutprob \E_0 (\d{G}_0 \phi_0)^2 = \sigma_0^2$. By conditions~\ref{cond: limited complexity}--\ref{cond: second order}, we have $n^{1/2}[T(\eta_n, \d{P}_n) - T(\eta_0, P_0)] \leadsto \d{G}_0 \phi_0$. By Slutsky's theorem, we have $n^{1/2} \bar{\sigma}_n^{-1}[T(\eta_n, \d{P}_n) - T(\eta_0, P_0)] \leadsto \sigma_0^{-1} \d{G}_0\phi_0$. We then have 
    \begin{align*}
        & \prob_0 \left( T(\eta_n, \d{P}_n) + z_{\beta} \bar{\sigma}_n n^{-1/2} \leq T(\eta_0, P_0) \leq T(\eta_n, \d{P}_n) + z_{1-\alpha} \bar{\sigma}_n n^{-1/2} \right) \\
        &\qquad = \prob_0 \left( -z_{1-\alpha} \leq n^{1/2}[T(\eta_n, \d{P}_n) - T(\eta_0, P_0)]/\bar{\sigma}_n \leq -z_{\beta} \right)\\
        & \qquad \to \prob_0 \left( -z_{1-\alpha} \leq \d{G}_0 \phi_0/\sigma_0 \leq -z_{\beta} \right)\\
        & \qquad = 1-\alpha - \beta,
    \end{align*}
    and the result follows.

\end{proof}

\section{Proof of results in Section~\ref{sec: T examples}}

We first introduce a Lemma providing conditions under which the (bootstrap) estimating equations-based estimator is (conditionally) consistent, which may be useful in its own right.

\begin{lemma} \label{lemma: consistency of ee}
    If $\psi_0$ is a well-separated solution of $\psi \mapsto G_{0,\eta_0}(\psi)$, $\d{P}_n \phi_{\psi_n, \eta_n} = o_{\prob_0^*}(1)$, $\phi_{\psi_n, \eta_n}$ is contained in a $P_0$-Glivenko Cantelli class with probability tending to one, $\sup_{|\psi| \leq M} \left| P_0( \phi_{\psi, \eta_n} - \phi_{\psi, \eta_0})\right| = o_{\prob_0^*}(1)$ for all $M > 0$, and $\psi_n = O_{\prob_0^*}(1)$, then $\psi_n - \psi_0 = o_{\prob_0^*}(1)$. If $\psi_0$ is a well-separated solution of $\psi \mapsto G_{0,\eta_0}(\psi)$, $\d{P}_n^* \phi_{\psi_n^*, \eta_n^*} = o_{\prob_W^*}(1)$,  $\phi_{\psi_n^*, \eta_n^*}$ is contained in a $P_0$-Glivenko Cantelli class $\s{F}$ with conditional probability tending to one, $\sup_{|\psi| \leq M} \left| P_0( \phi_{\psi, \eta_n^*} - \phi_{\psi, \eta_0})\right| = o_{\prob_W^*}(1)$ for all $M > 0$, $\psi_n^* = O_{\prob_W^*}(1)$, and $\|\hat{P}_n - P_0\|_{\s{F}} = o_{\prob_0^*}(1)$, then $\psi_n^* - \psi_0 = o_{\prob_W^*}(1)$.
\end{lemma}
\begin{proof}[\bfseries{Proof of Lemma~\ref{lemma: consistency of ee}}]
By the well-separated assumption, for every $\varepsilon > 0$ there exists $\delta > 0$ such that $|P_0 \phi_{\psi, \eta_0}| \geq \delta$ for every $\psi \in \d{R}$ such that $|\psi - \psi_0| \geq \varepsilon$. Therefore, the event $\{|\psi_n - \psi_0| \geq \varepsilon\}$ is contained in the event $\{ |P_0 \phi_{\psi_n, \eta_0}| \geq \delta \}$. We then write
        \[ 
            P_0 \phi_{\psi_n, \eta_0} = -(\d{P}_n - P_0) \phi_{\psi_n, \eta_n} + \d{P}_n \phi_{\psi_n, \eta_n} - P_0 \left(\phi_{\psi_n, \eta_n} -\phi_{\psi_n, \eta_0}\right).
        \]
        By the assumption that $\phi_{\psi_n, \eta_n}$ is contained in a $P_0$-Glivenko Cantelli class with probability tending to one, $(\d{P}_n - P_0) \phi_{\psi_n, \eta_n} = o_{\prob_0^*}(1)$, and $\d{P}_n \phi_{\psi_n, \eta_n} = o_{\prob_0^*}(1)$ by assumption.  Since $\psi_n = O_{\prob_0^*}(1)$, for all $\gamma >0$ there exists $M$ such that $P_0^*( |\psi_n| > M) < \gamma$ for all $n$ large enough. Thus, 
        \[ 
            P_0\left( \left|  P_0 \left(\phi_{\psi_n, \eta_n} -\phi_{\psi_n, \eta_0} \right)\right| \geq \delta \right) \leq P_0\left(\sup_{|\psi| \leq M} \left|  P_0 \left(\phi_{\psi, \eta_n} -\phi_{\psi, \eta_0} \right)\right| \geq \delta \right) + P_0(|\psi_n| > M),
        \]
        which is less than $2\gamma$ for all $n$ large enough since $\sup_{|\psi| \leq M} \left| P_0( \phi_{\psi, \eta_n} - \phi_{\psi, \eta_0})\right| = o_{\prob_0^*}(1)$ by assumption. Therefore, for any $\varepsilon > 0$,
        \begin{align*}
            &\prob_0^* \left( \left|\psi_n - \psi_0\right| \geq \varepsilon \right) \leq \prob_0^* \left( \left|P_0 \phi_{\psi_n, \eta_0} \right| \geq \delta \right) \longrightarrow 0.
        \end{align*} 
        
        We show that $\psi_n^* \condinoutprob \psi_0$ using a similar method. By the well-separated condition,  the event $\{|\psi_n^* - \psi_0| \geq \varepsilon\}$ is contained in the event $\{ |P_0 \phi_{\psi_n^*, \eta_0}| \geq \delta \}$. We then write
        \[ 
            P_0 \phi_{\psi_n^*, \eta_0} = -(\d{P}_n^* - \hat{P}_n) \phi_{\psi_n^*, \eta_n^*} -(\hat{P}_n - P_0) \phi_{\psi_n^*, \eta_n^*} + \d{P}_n^* \phi_{\psi_n^*, \eta_n^*} - P_0 \left(\phi_{\psi_n^*, \eta_n^*} -\phi_{\psi_n^*, \eta_0}\right).
        \]
        By the assumption that $\phi_{\psi_n^*, \eta_n^*}$ is contained in a $P_0$-Glivenko Cantelli class $\s{F}$ with conditional probability tending to one, $(\d{P}_n^* - \hat{P}_n) \phi_{\psi_n^*, \eta_n^*} = o_{\prob_W^*}(1)$ and $\left|(\hat{P}_n - P_0) \phi_{\psi_n^*, \eta_n^*}\right| \leq \|\hat{P}_n - P_0\|_{\s{F}} = o_{\prob_W^*}(1)$, and $\d{P}_n^* \phi_{\psi_n^*, \eta_n^*} = o_{\prob_W^*}(1)$ by assumption. Since $\psi_n^* = O_{\prob_W^*}(1)$, for all $\gamma >0$ there exists $M$ such that $P_0^*(P_W^*( |\psi_n^*| > M) \geq \gamma/2) = o(1)$. Thus, 
        \begin{align*}
            P_0^* \left( P_W^*\left( \left|  P_0 \left(\phi_{\psi_n^*, \eta_n^*} -\phi_{\psi_n^*, \eta_0} \right)\right| \geq \delta \right) \geq \gamma \right) &\leq P_0^*\left(P_W^*\left(\sup_{|\psi| \leq M} \left|  P_0 \left(\phi_{\psi, \eta_n^*} -\phi_{\psi, \eta_0} \right)\right| \geq \delta \right) > \gamma / 2 \right) \\
            &\qquad+ P_0^*\left(P_W^*(|\psi_n^*| > M) > \gamma/2 \right),
        \end{align*}
        and both terms are $o(1)$. The result follows.
\end{proof}

\lemmaclassicalestimating*
\begin{proof}[\bfseries{Proof of Lemma~\ref{lemma: classical estimating-equations}}]
    We first show that the assumptions imply that $\psi_n \inoutprob \psi_0$. By assumption, $\psi_0$ is a well-separated solution of $\psi \mapsto G_{0, \eta_0}(\psi)$ and $\d{P}_n \phi_{\psi_n, \eta_n} =  o_{\prob_0^*}(1)$. By condition~\ref{cond: limited complexity}, $\phi_n:= \phi_{\psi_n, \eta_n}$ is contained in a $P_0$-Donsker class with probability tending to one, which implies that it is contained in a $P_0$-Glivenko Cantelli class with probability tending to one. By adding and subtracting terms, we can write
    \begin{align*}
        P_0 \left( \phi_{\psi, \eta_n} - \phi_{\psi, \eta_0}\right) &= \Gamma_{0, \eta_n} (\psi) - \Gamma_{0, \eta_0}(\psi)  + P_0 \phi_{\psi_0, \eta_n} + \left[ G_{0, \eta_n}'(\psi_0) + 1 \right] \left(\psi - \psi_0\right) 
    \end{align*} 
    By assumption, $\sup_{|\psi| \leq M} \left| \Gamma_{0, \eta_n} (\psi) - \Gamma_{0, \eta_0}(\psi)\right|=  o_{\prob_0^*}(1)$, $P_0 \phi_{\psi_0, \eta_n} = o_{\prob_0^*}(1)$, and 
    \[ \sup_{| \psi| \leq M} \left| \left[ G_{0, \eta_n}'(\psi_0) + 1 \right] \left(\psi - \psi_0\right) \right| = o_{\prob_0^*}(1)\]
    for every $M > 0$ since $\eta_n \inoutprob \eta_0$. Finally, $\psi_n =  O_{\prob_0^*}(1)$ by assumption. Thus, the conditions of  \Cref{lemma: consistency of ee} are satisfied, so $\psi_n \inoutprob \psi_0$.
    
    We next show that $\psi_n - \psi_0 = O_{\prob_0^*}(n^{-1/2})$ under the assumptions of the lemma. We have
    \[  
        n^{1/2} (\psi_n  - \psi_0) =  \d{G}_n \phi_{\psi_n, \eta_n} - n^{1/2}\d{P}_n\phi_{\psi_n, \eta_n} + n^{1/2} P_0 \phi_{\psi_0, \eta_n}+ n^{1/2}\left[P_0(\phi_{\psi_n, \eta_n} - \phi_{\psi_0, \eta_n}) + (\psi_n - \psi_0) \right] .
    \]
    Now, $\d{G}_n \phi_{\psi_n, \eta_n} = O_{\prob_0^*}(1)$ by~\ref{cond: limited complexity}, and $n^{1/2}\d{P}_n\phi_{\psi_n, \eta_n}$ and $n^{1/2} P_0 \phi_{\psi_0, \eta_n}$ are both $o_{\prob_0^*}(1)$ by assumption. Since $\norm{\eta_n - \eta_0}_{\s{H}} = o_{\prob_0^*}(1)$ by assumption, with probability tending to one it holds that
    \begin{align*}
        \left| P_0(\phi_{\psi_n, \eta_n} - \phi_{\psi_0, \eta_n}) + (\psi_n - \psi_0) \right| &= \left| G_{0,\eta_n}(\psi_n) - G_{0,\eta_n}(\psi_0) + (\psi_n - \psi_0)\right| \\
        &\leq  \left|  \Gamma_{0,\eta_n}(\psi_n)\right| + \left|G_{0,\eta_n}'(\psi_0) + 1 \right| |\psi_n - \psi_0| \\ 
        &\leq \sup_{\eta : \|\eta - \eta_0\|_{\s{H}} < \delta}\left| \Gamma_{0,\eta}(\psi_n)\right|  +  \left|G_{0,\eta_n}'(\psi_0) + 1 \right| \left| \psi_n- \psi_0\right|.
    \end{align*}
    Now since $\|\eta_n - \eta_0\|_{\s{H}} = o_{\prob_0^*}(1)$ and $G_{0,\eta}(\psi_0)$ is continuous in $\eta$ at $\eta_0$ with $G_{0,\eta}(\psi_0) = -1$, $\left|G_{0,\eta_n}'(\psi_0) + 1 \right| = o_{\prob_0^*}(1)$. Since $\psi_n \inoutprob \psi_0$ as established above, we then have $\sup_{\eta : \|\eta - \eta_0\|_{\s{H}} < \delta}\left| \Gamma_{0,\eta}(\psi_n)\right|= o_{\prob_0^*}(|\psi_n - \psi_0|)$ by the assumed differentiability. Hence, we have 
    \begin{align*}
        n^{1/2} |\psi_n  - \psi_0| = O_{\prob_0^*}(1) + n^{1/2}o_{\prob_0^*}(|\psi_n - \psi_0|),
    \end{align*}
    which implies that $\psi_n - \psi_0 = O_{\prob_0^*}(n^{-1/2})$. 

    Now since $G_{0,\eta_n}'(\psi_0) + 1 = o_{\prob_0^*}(1)$, $P_0\phi_{\psi_0, \eta_n} = o_{\prob_0^*}(n^{-1/2})$, and $\psi_n - \psi_0 = O_{\prob_0^*}(n^{-1/2})$, we have
    \begin{align*} 
        \psi_n  - \psi_0 + P_0 \phi_{\psi_n, \eta_n} &= \left[P_0(\phi_{\psi_n, \eta_n} - \phi_{\psi_0, \eta_n})- G_{0,\eta_n}'(\psi_0)(\psi_n - \psi_0)\right] \\
        &\qquad + \left[ G_{0,\eta_n}'(\psi_0) + 1\right] \left(\psi_n - \psi_0\right) + P_0\phi_{\psi_0, \eta_n}  \\
        &= \Gamma_{0, \eta_n}(\psi_n)  + o_{\prob_0^*}(n^{-1/2}).
    \end{align*}
    The first term is $o_{\prob_0^*}(n^{-1/2})$ by the differentiability assumption since $\|\eta_n - \eta_0\|_{\s{H}} = o_{\prob_0^*}(1)$ and $\psi_n - \psi_0 = O_{\prob_0^*}(n^{-1/2})$ as discussed above. The result follows. 
\end{proof}

\lemmaempbootstrapestimating*
\begin{proof}[\bfseries{Proof of Lemma~\ref{lemma: empirical bootstrap estimating-equations}}]
    We first show that the assumptions imply that $\psi_n^* \condinoutprob \psi_0$. By assumption, $\psi_0$ is a well-separated solution of $\psi \mapsto G_{0, \eta_0}(\psi)$ and $\d{P}_n^* \phi_{\psi_n^*, \eta_n^*} =  o_{\prob_0^*}(1)$. By condition~\ref{cond: limited complexity}, $\phi_n^*:= \phi_{\psi_n^*, \eta_n^*}$ is contained in a $P_0$-Donsker class with probability tending to one, which implies that it is contained in a $P_0$-Glivenko Cantelli class with probability tending to one. Adding and subtracting terms, we can write
    \begin{align*}
        P_0 \left( \phi_{\psi, \eta_n^*} - \phi_{\psi, \eta_0}\right) &= \Gamma_{0, \eta_n^*} (\psi) - \Gamma_{0, \eta_0}(\psi)  + P_0 \phi_{\psi_0, \eta_n^*} + \left[ G_{0, \eta_n^*}'(\psi_0) + 1 \right] \left(\psi - \psi_0\right) 
    \end{align*} 
    By assumption, $\sup_{|\psi| \leq M} \left| \Gamma_{0, \eta_n^*} (\psi) - \Gamma_{0, \eta_0}(\psi_0)\right|=  o_{\prob_W^*}(1)$, $P_0 \phi_{\psi_0, \eta_n^*} = o_{\prob_0^*}(1)$, and 
    \[ \sup_{| \psi| \leq M} \left| \left[ G_{0, \eta_n^*}'(\psi_0) + 1 \right] \left(\psi - \psi_0\right) \right| = o_{\prob_0^*}(1)\]
    for every $M > 0$ since $\eta_n^* \condinoutprob \eta_0$. Finally, $\psi_n^* =  O_{\prob_W^*}(1)$ and $\|\d{P}_n - P_0 \|_{\s{F}} = o_{\prob_W^*}(1)$ by condition~\ref{cond: limited complexity}. Thus, the conditions of  \Cref{lemma: consistency of ee} are satisfied, so $\psi_n^* \condinoutprob \psi_0$.
    
    We next show that $|P_0(\phi_{\psi_n^*, \eta_n^*} - \phi_{\psi_0, \eta_n^*}) + (\psi_n^* - \psi_0)| = o_{\prob_W^*}(|\psi_n^* - \psi_0|)$ under the assumptions of the lemma. Since $\|\eta_n^* - \eta_0\|_{\s{H}} = o_{\prob_W^*}(1)$ by assumption, with conditional probability tending to one it holds that
    \begin{align*}
        \left| P_0(\phi_{\psi_n^*, \eta_n^*} - \phi_{\psi_0, \eta_n^*}) + (\psi_n^* - \psi_0) \right| &= \left| G_{0,\eta_n^*}(\psi_n^*) - G_{0,\eta_n^*}(\psi_0) + (\psi_n^* - \psi_0)\right| \\
        &\leq  \left| \Gamma_{0,\eta_n^*}(\psi_n^*)\right| + \left|G_{0,\eta_n^*}'(\psi_0) + 1 \right| |\psi_n^* - \psi_0| \\ 
        &\leq \sup_{\eta : \|\eta - \eta_0\|_{\s{H}} < \delta}\left| \Gamma_{0,\eta}(\psi_n^*)\right|  +  \left|G_{0,\eta_n^*}'(\psi_0) + 1 \right| \left| \psi_n^* - \psi_0\right|.
    \end{align*}
    Now since $\|\eta_n^* - \eta_0\|_{\s{H}} = o_{\prob_W^*}(1)$ and $G_{0,\eta}(\psi_0)$ is continuous in $\eta$ at $\eta_0$ with $G_{0,\eta_0}(\psi_0) = -1$, $\left|G_{0,\eta_n^*}'(\psi_0) + 1 \right| = o_{\prob_W^*}(1)$. Since $\psi_n^* \condinoutprob \psi_0$ as shown above, we also have $\sup_{\eta : \|\eta - \eta_0\|_{\s{H}} < \delta}\left| \Gamma_{0,\eta}(\psi_n^*)\right| = o_{\prob_W^*}(|\psi_n^* - \psi_0|)$. The result follows.

    We now show that $\psi_n^* - \psi_0 = O_{\prob_W^*}(n^{-1/2})$ under the assumptions of the lemma. We have
    \begin{align*}
        n^{1/2} (\psi_n^*  - \psi_0) &=  \d{G}_n^* \phi_{\psi_n^*, \eta_n^*} +  \d{G}_n \phi_{\psi_n^*, \eta_n^*} - n^{1/2}\d{P}_n^*\phi_{\psi_n^*, \eta_n^*} + n^{1/2} P_0 \phi_{\psi_0, \eta_n^*}\\
        & \qquad + n^{1/2}\left[P_0(\phi_{\psi_n^*, \eta_n^*} - \phi_{\psi_0, \eta_n^*}) + (\psi_n^* - \psi_0) \right].
    \end{align*}
    Now, $\d{G}_n^* \phi_{\psi_n^*, \eta_n^*} = O_{\prob_W^*}(1)$ by~\ref{cond: bootstrap limited complexity}, $\d{G}_n \phi_{\psi_n^*, \eta_n^*} = O_{\prob_W^*}(1)$ by~\ref{cond: limited complexity}, and $n^{1/2}\d{P}_n^*\phi_{\psi_n^*, \eta_n^*}$ and $n^{1/2} P_0 \phi_{\psi_0, \eta_n^*}$ are both $o_{\prob_W^*}(1)$ by assumption. We therefore have 
    \begin{align*}
        n^{1/2} |\psi_n^*  - \psi_0| &= O_{\prob_W^*}(1) + n^{1/2}\left|P_0(\phi_{\psi_n^*, \eta_n^*} - \phi_{\psi_0, \eta_n^*}) + (\psi_n^* - \psi_0) \right| \\
        &= O_{\prob_W^*}(1) + n^{1/2}o_{\prob_W^*}(|\psi_n^* - \psi_0|),
    \end{align*}
    which implies that $\psi_n^* - \psi_0 = O_{\prob_W^*}(n^{-1/2})$.
    
    We now show the first statement of \Cref{lemma: empirical bootstrap estimating-equations}. We write
    \begin{align}
        \begin{split}\label{eq: empirical bootstrap ee}
            \psi_n^* - \psi_n - (\d{P}_n^* - \d{P}_n) \phi_{\psi_n^*, \eta_n^*}  & = [\psi_n^* - \psi_0 + P_0 \phi_{\psi_n^*, \eta_n^*}] - [\psi_n - \psi_0 + P_0 \phi_{\psi_n, \eta_n}] \\
            & \qquad  + (\d{P}_n - P_0)(\phi_{\psi_n^*, \eta_n^*} - \phi_{\psi_n, \eta_n} ) + \d{P}_n \phi_{\psi_n, \eta_n} - \d{P}_n^* \phi_{\psi_n^*, \eta_n^*}.
        \end{split}
    \end{align}
   
    By \Cref{lemma: classical estimating-equations}, we have $\psi_n - \psi_0 + P_0 \phi_{\psi_n, \eta_n} = o_{\prob_0^*}(n^{-1/2})$. By the definitions of $\psi_n$ and $\psi_n^*$, we have $\d{P}_n \phi_{\psi_n, \eta_n} = o_{\prob_W^*}(n^{-1/2})$ and $\d{P}_n^* \phi_{\psi_n^*, \eta_n^*} = o_{\prob_W^*}(n^{-1/2})$. By conditions~\ref{cond: limited complexity} and~\ref{cond: weak consistency} and Lemma~19.24 of \cite{van2000asymptotic}, we have  $(\d{P}_n - P_0)(\phi_{\psi_n, \eta_n} - \phi_{\psi_0, \eta_0}) = o_{\prob_0^*}(n^{-1/2})$. Similarly, by \ref{cond: limited complexity}(b), \ref{cond: bootstrap limited complexity}(a), \ref{cond: bootstrap weak consistency}  and a minor modification of Lemma~19.24 \cite{van2000asymptotic}, we have $(\d{P}_n - P_0)(\phi_{\psi_n^*, \eta_n^*} - \phi_{\psi_0, \eta_0}) = o_{\prob_W^*}(n^{-1/2})$. Hence, $(\d{P}_n - P_0)(\phi_{\psi_n^*, \eta_n^*} - \phi_{\psi_n, \eta_n} ) = o_{\prob_W^*}(n^{-1/2})$.  Finally, since $G_{0,\eta_n^*}'(\psi_0) + 1 = o_{\prob_W^*}(1)$, $P_0\phi_{\psi_0, \eta_n^*} = o_{\prob_W^*}(n^{-1/2})$, and $\psi_n^* - \psi_0 = O_{\prob_W^*}(n^{-1/2})$, we have
    \begin{align*} 
        |\psi_n^* - \psi_0 + P_0 \phi_{\psi_n^*, \eta_n^*}| &\leq \left| P_0 (\phi_{\psi_n^*, \eta_n^*} - \phi_{\psi_0, \eta_n^*}) + (\psi_n^* - \psi_0) \right|  + \abs{P_0 \phi_{\psi_0, \eta_n^*}}\\
        & = o_{\prob_W^*}(|\psi_n^* - \psi_0|) + o_{\prob_W^*}(n^{-1/2})\\
        & = o_{\prob_W^*}(n^{-1/2}).
    \end{align*}
    The result follows. 

    We next show the second statement in \Cref{lemma: empirical bootstrap estimating-equations}. Since $R_n = \psi_n - \psi_0 - (\d{P}_n - P_0) \phi_{\psi_n, \eta_n}$ and $R_n^* = \psi_n^* - \psi_n - (\d{P}_n^* - \d{P}_n) \phi_{\psi_n^*, \eta_n^*}$, by adding and subtracting terms, we have 
    \begin{align*}
        R_n^* - R_n & = \left(\psi_n^* - \psi_0 + P_0 \phi_{\psi_n^*, \eta_n^*} - P_0 \phi_{\psi_0, \eta_n^*}\right) -2\left(\psi_n - \psi_0 + P_0 \phi_{\psi_n, \eta_n} - P_0 \phi_{\psi_0, \eta_n} \right) \\
        &  \qquad + (P_0 \phi_{\psi_0, \eta_n^*} - 2 P_0 \phi_{\psi_0, \eta_n}) +  \left(\d{P}_n - P_0\right)\left(\phi_{\psi_n^*, \eta_n^*} - \phi_{\psi_n, \eta_n} \right) + \left(2\d{P}_n \phi_{\psi_n, \eta_n} - \d{P}_n^* \phi_{\psi_n^*, \eta_n^*}\right).
    \end{align*}
    As we showed above, $\psi_n^* - \psi_0 + P_0 \phi_{\psi_n^*, \eta_n^*} - P_0 \phi_{\psi_0, \eta_n^*} = o_{\prob_W^*}(n^{-1/2})$ and $(\d{P}_n - P_0)(\phi_{\psi_n^*, \eta_n^*} - \phi_{\psi_n, \eta_n} ) = o_{\prob_W^*}(n^{-1/2})$. By assumption, $P_0 \phi_{\psi_0, \eta_n^*} - 2 P_0 \phi_{\psi_0, \eta_n} = o_{\prob_W^*}(n^{-1/2})$. By the definitions of $\psi_n^*$ and $\psi_n$, $2\d{P}_n \phi_{\psi_n, \eta_n}$ and  $\d{P}_n^* \phi_{\psi_n^*, \eta_n^*}$ are both $o_{\prob_W^*}(n^{-1/2})$. Therefore,
    \begin{align*}
        R_n^* - R_n & = -2\left(\psi_n - \psi_0 + P_0 \phi_{\psi_n, \eta_n} - P_0 \phi_{\psi_0, \eta_n} \right) + o_{\prob_W^*}(n^{-1/2}) \\
        & = - 2\Gamma_{0,\eta_n}(\psi_n)  - 2\left[ G_{0, \eta_n}'(\psi_0) + 1\right] (\psi_n - \psi_0)  + o_{\prob_W^*}(n^{-1/2}).
    \end{align*}
    The first two terms are $o_{\prob_0^*}(n^{-1/2})$ as discussed above. The result follows. 
    
\end{proof}

\lemmasmobootstrapestimating*
\begin{proof}[\bfseries{Proof of Lemma~\ref{lemma: smooth bootstrap estimating-equations}}]
    By the same argument as used in \Cref{lemma: empirical bootstrap estimating-equations}, and since  $\| \hat{P}_n - P_0\|_{\s{F}} =o_{\prob_W^*}(1)$ by assumption, we can show that the conditions of \Cref{lemma: consistency of ee} are satisfied, so that $\psi_n^* \condinoutprob \psi_0$.
    
    We recall that $\hat\psi_n := T(\hat\eta_n, \hat{P}_n)$. By the same argument as used in \Cref{lemma: empirical bootstrap estimating-equations},  $P_0 \phi_{\hat\psi_n, \eta_0} = o_{\prob_0^*}(1)$ implies that $\hat\psi_n - \psi_0 = o_{\prob_0^*}(1)$. We write
    \[ 
        P_0 \phi_{\hat\psi_n, \eta_0} = \hat{P}_n \phi_{\hat\psi_n, \hat\eta_n} - \left(\hat{P}_n - P_0\right) \left(\phi_{\hat\psi_n, \hat\eta_n}\right) - P_0 \left(\phi_{\hat\psi_n, \hat\eta_n} -\phi_{\hat\psi_n, \eta_0}\right).
    \]
    By the definition of $\hat\psi_n$, we have $\hat{P}_n \phi_{\hat\psi_n, \hat\eta_n} = 0$. Under the assumptions $\| \hat{P}_n - P_0\|_{\s{F}} = o_{\prob_0^*}(1)$ and $\phi_{\hat\psi_n, \hat\eta_n}$ is contained in $\s{F}$ with probability tending to one, we have  $(\hat{P}_n - P_0) \phi_{\hat\psi_n, \hat\eta_n} = o_{\prob_0^*}(1)$. Finally, since $\hat\psi_n = O_{\prob_0^*}(1)$ and $\sup_{|\psi| \leq M} \left| \Gamma_{0, \hat\eta_n}(\psi) - \Gamma_{0, \eta_0}(\psi) \right| = o_{\prob_0^*}(1)$ for every $M > 0$, $P_0 \phi_{\psi_0, \hat\eta_n} = o_{\prob_0^*}(1)$, and $\| \hat\eta_n - \eta_0\|_{\s{H}} = o_{\prob_0^*}(1)$, by the same argument as in \Cref{lemma: classical estimating-equations}, $P_0 \left(\phi_{\hat\psi_n, \hat\eta_n} -\phi_{\hat\psi_n, \eta_0}\right) = o_{\prob_0^*}(1)$. Thus, $P_0 \phi_{\hat\psi_n, \eta_0} = o_{\prob_0^*}(1)$, so that $\hat\psi_n \inoutprob \psi_0$. Hence, we also have $\psi_n^* - \hat\psi_n = (\psi_n^* - \psi_0) - (\hat\psi_n - \psi_0) \condinoutprob 0$.

    We now show that $\hat{P}_n \left(\phi_{\psi_n^*, \eta_n^*} - \phi_{\hat\psi_n, \eta_n^*} \right) + (\psi_n^* - \hat\psi_n) = o_{\prob_W^*}(\psi_n^* - \hat\psi_n)$ under the assumptions of the lemma. With conditional probability tending to one, it holds that
    \begin{align*}
        &\abs{\hat{P}_n \big(\phi_{\psi_n^*, \eta_n^*} - \phi_{\hat\psi_n, \eta_n^*} \big) + (\psi_n^* - \hat\psi_n)} \\
        & \qquad =  \abs{\hat{G}_{n, \eta_n^*}(\psi_n^*) - \hat{G}_{n, \eta_n^*}(\hat\psi_n) + (\psi_n^* - \hat\psi_n)}\\
        & \qquad \leq \abs{\hat{G}_{n, \eta_n^*}(\psi_n^*) - \hat{G}_{n, \eta_n^*}(\hat\psi_n) - \hat{G}_{n, \eta_n^*}'(\hat\psi_n) (\psi_n^* -  \hat\psi_n)} + \abs{\hat{G}_{n, \eta_n^*}'(\hat\psi_n) + 1} \abs{\psi_n^* -  \hat\psi_n}\\
        & \qquad \leq \sup_{\eta: \norm{\eta - \hat\eta_n}_{\s{H}} < \delta}\abs{\hat\Gamma_{n, \eta}(\psi_n^*)} + \abs{\hat{G}_{n, \eta_n^*}'(\hat\psi_n) + 1} \abs{\psi_n^* -  \hat\psi_n}
    \end{align*}
    Since by assumption $\hat{G}_{n,\eta_n^*}'(\psi_n) + 1 = o_{\prob_W^*}(1)$, the second term on the right hand side of previous display is $o_{\prob_W^*}(\psi_n^* - \hat\psi_n)$. The result follows. 
  
    We next show that $\psi_n^* - \hat\psi_n = O_{\prob_W^*}(n^{-1/2})$  under the assumptions of the lemma.  We have 
    \begin{align*}
        n^{1/2}(\psi_n^* - \hat\psi_n) &=  \d{G}_n^*\phi_{\psi_n^*, \eta_n^*} - n^{1/2}\d{P}_n^* \phi_{\psi_n^*, \eta_n^*} + n^{1/2}\hat{P}_n \phi_{\hat\psi_n, \eta_n^*} + n^{1/2}\left[\hat{P}_n \left(\phi_{\psi_n^*, \eta_n^*} - \phi_{\hat\psi_n, \eta_n^*} \right) + \psi_n^* - \hat\psi_n\right].
    \end{align*}
    Now, $\d{G}_n^* \phi_{\psi_n^*, \eta_n^*} = O_{\prob_W^*}(1)$ because \ref{cond: bootstrap limited complexity} holds for $\phi_n^* = \phi_{\psi_n^*, \eta_n^*}$ by assumption, and $n^{1/2}\d{P}_n^*\phi_{\psi_n^*, \eta_n^*}$ and $n^{1/2}\hat{P}_n \phi_{\hat\psi_n, \eta_n^*}$ are both $o_{\prob_W^*}(1)$ by assumption. As discussed above, $\hat{P}_n \left(\phi_{\psi_n^*, \eta_n^*} - \phi_{\hat\psi_n, \eta_n^*} \right) + \psi_n^* - \hat\psi_n = o_{\prob_W^*}(\psi_n^* - \hat\psi_n)$. Hence, we have 
    \begin{align*}
        n^{1/2}|\psi_n^*  - \hat\psi_n| = O_{\prob_W^*}(1) + n^{1/2}o_{\prob_W^*}(|\psi_n^* - \hat\psi_n|),
    \end{align*}
    which implies that $\psi_n^* - \hat\psi_n = O_{\prob_W^*}(n^{-1/2})$. 
    
    We can now show the first statement of \Cref{lemma: smooth bootstrap estimating-equations}. We recall that $\psi_n^\circ:= T(\eta_n, \hat{P}_n)$. We  can write
    \begin{align*}
        R_n^* &=\psi_n^* - \psi_n^\circ - (\d{P}_n^* - \hat{P}_n) \phi_{\psi_n^*, \eta_n^*} \\
        &= \left(\psi_n^* - \hat\psi_n + \hat{P}_n \phi_{\psi_n^*, \eta_n^*} - \hat{P}_n \phi_{\hat\psi_n, \eta_n^*}\right)- \left(\psi_n^\circ - \hat\psi_n\right) + \hat{P}_n \phi_{\hat\psi_n, \eta_n^*}  - \d{P}_n^* \phi_{\psi_n^*, \eta_n^*};
    \end{align*}
    We have $\psi_n^\circ - \hat\psi_n = o_{\prob_0^*}(n^{-1/2})$ and $\hat{P}_n \phi_{\hat\psi_n, \eta_n^*} = o_{\prob_W^*}(n^{-1/2})$ by assumption, $\d{P}_n^* \phi_{\psi_n^*, \eta_n^*} = o_{\prob_W^*}(n^{-1/2})$ by definition of $\psi_n^*$, and 
    \[
        \psi_n^* - \hat\psi_n + \hat{P}_n \phi_{\psi_n^*, \eta_n^*} - \hat{P}_n \phi_{\hat\psi_n, \eta_n^*} = o_{\prob_W^*}\left(\psi_n^* - \hat\psi_n\right) =  o_{\prob_W^*}(n^{-1/2})
    \]
    by the derivations above. Hence, $R_n^* = o_{\prob_W^*}(n^{-1/2})$.

    Lastly, we show the second statement in \Cref{lemma: smooth bootstrap estimating-equations}. Recall that $R_n = \psi_n - \psi_0 - (\d{P}_n - P_0) \phi_{\psi_n, \eta_n}$ and $R_n^* = \psi_n^* - \psi_n^\circ - (\d{P}_n^* - \hat{P}_n) \phi_{\psi_n^*, \eta_n^*}$. By adding and subtracting terms, we can write
    \begin{align*}
        R_n^* - R_n & = \left(\psi_n^* - \hat\psi_n + \hat{P}_n \phi_{\psi_n^*, \eta_n^*} - \hat{P}_n \phi_{\hat\psi_n, \eta_n^*} \right) - \left(\psi_n - \psi_0 + P_0 \phi_{\psi_n, \eta_n} - P_0 \phi_{\psi_0, \eta_n} \right) \\
        & \qquad + \left(\hat{P}_n \phi_{\hat\psi_n, \eta_n^*} - P_0 \phi_{\psi_0, \eta_n} \right) + \left(\d{P}_n \phi_{\psi_n, \eta_n} - \d{P}_n^* \phi_{\psi_n^*, \eta_n^*} \right) + \left(\hat\psi_n - \psi_n^\circ\right).
    \end{align*} 
    We have $\psi_n^* - \hat\psi_n + \hat{P}_n \phi_{\psi_n^*, \eta_n^*} - \hat{P}_n \phi_{\hat\psi_n, \eta_n^*} = o_{\prob_W^*}(n^{-1/2})$ by the derivations above, $\psi_n - \psi_0 + P_0 \phi_{\psi_n, \eta_n} - P_0 \phi_{\psi_0, \eta_n} =o_{\prob_0^*}(n^{-1/2})$ by \Cref{lemma: classical estimating-equations}, $\hat\psi_n - \psi_n^\circ = o_{\prob_W^*}(n^{-1/2})$ and $\hat{P}_n \phi_{\hat\psi_n, \eta_n^*}  - P_0\phi_{\psi_0, \eta_n}= o_{\prob_W^*}(n^{-1/2})$ by assumption, and $\d{P}_n \phi_{\psi_n, \eta_n} - \d{P}_n^* \phi_{\psi_n^*, \eta_n^*} = o_{\prob_W^*}(n^{-1/2})$ by the definitions of $\psi_n$ and $\psi_n^*$. 
\end{proof}

\section{Proof of results in Section~\ref{sec: applications}}

\propclassicalaverage*
\begin{proof}[\bfseries{Proof of \Cref{prop: classical average}}]
    We use \Cref{thm: classical}. Condition~\ref{cond: limited complexity} holds because $\eta_n$ falls in a $P_0$-Donsker class by assumption. For any $\varepsilon >0$, by the triangle and Cauchy–Schwarz inequalities, we have 
    \begin{align*}
        &\prob_0^* \left(\norm{\phi_n - \phi_0}_{L_2(P_0)} \geq \varepsilon \right)  \\
        & \qquad \leq  \prob_0^* \left(2\norm{\eta_n - \eta_0}_{L_2(P_0)} \geq \varepsilon/2 \right) + \prob_0^* \left(2\abs{\int \eta_n^{2} - \int \eta_0^2} \geq \varepsilon/2 \right) \\
        & \qquad \leq \prob_0^* \left(2M\norm{\eta_n - \eta_0}_{L_2(\lambda)} \geq \varepsilon/2 \right) + \prob_0^* \left(2\norm{\eta_n - \eta_0}_{L_2(\lambda)}\norm{\eta_n + \eta_0}_{L_2(\lambda)} \geq \varepsilon/2 \right) \\
        & \qquad \leq \prob_0^* \left(2M\norm{\eta_n - \eta_0}_{L_2(\lambda)} \geq \varepsilon/2 \right) + \prob_0^* \left(4M\norm{\eta_n - \eta_0}_{L_2(\lambda)} \geq \varepsilon/2 \right) + \prob_0^* \left(\norm{\eta_n}_{L_2(\lambda)} \geq M \right).
    \end{align*}
    Each term on the right-hand side converges to 0, which implies condition~\ref{cond: weak consistency}.    
    
    For condition~\ref{cond: second order} for $\psi_{n,1}$, we have 
    \begin{align*}
        T_1(\eta_n, \d{P}_n) - T_1(\eta_0, P_0) - (\d{P}_n - P_0)\phi_n &= \left(2\d{P}_n \eta_n - \int \eta_n^2 \right) - \int \eta_0^2 - 2(\d{P}_n - P_0)\eta_n \\
        &= - \int (\eta_n - \eta_0)^2 \\
        &=- \norm{ \eta_n - \eta_0}_{L_2(\lambda)}^2
    \end{align*}
    which is $o_{\prob_0^*}(n^{1/2})$ by assumption. For $\psi_{n,2}$ and $\psi_{n,3}$, we have
    \begin{align*}
        T_2(\eta_n, \d{P}_n) - T_2(\eta_0, P_0) - (\d{P}_n - P_0)\phi_n &= \int \eta_n^2 - \int \eta_0^2 - 2(\d{P}_n - P_0)\eta_n \\
        & = - \int (\eta_n - \eta_0)^2 + 2 \left[\int \eta_n^2 - \d{P}_n \eta_n\right] \\
        T_3(\eta_n, \d{P}_n) - T_3(\eta_0, P_0) - (\d{P}_n - P_0)\phi_n &= \d{P}_n \eta_n - P_0 \eta_0 - 2(\d{P}_n - P_0)\eta_n \\
        & = - \int (\eta_n - \eta_0)^2 + \left[\int \eta_n^2 - \d{P}_n \eta_n\right], 
    \end{align*}
    which are both $o_{\prob_0^*}(n^{-1/2})$ when $\int \eta_n^2 - \d{P}_n \eta_n = o_{\prob_0^*}(n^{-1/2})$ and  $\norm{ \eta_n - \eta_0}_{L_2(\lambda)}^2 = o_{\prob_0^*}(n^{-1/4})$.
\end{proof}

\propempbootstrapaverage*
\begin{proof}[\bfseries{Proof of \Cref{prop: empirical bootstrap average}}]
     We use \Cref{thm: bootstrap}. Condition~\ref{cond: bootstrap limited complexity} holds by \Cref{lemma: empirical bootstrap} because $\eta_n^*$ is contained in a $P_0$-Donsker class with $\prob_W^*$ probability tending to one. For any $\varepsilon >0$, by the triangle and Cauchy–Schwarz inequalities, we have 
    \begin{align*}
        &\prob_W^* \left(\norm{\phi_n^* - \phi_0}_{L_2(P_0)} \geq \varepsilon \right)  \\
        & \qquad \leq \prob_W^* \left(2\norm{\eta_n^* - \eta_0}_{L_2(P_0)} \geq \varepsilon/2 \right) + \prob_W^* \left(2\abs{\int \eta_n^{*2} - \int \eta_0^2} \geq \varepsilon/2 \right) \\
        & \qquad \leq \prob_W^* \left(2M\norm{\eta_n^* - \eta_0}_{L_2(\lambda)} \geq \varepsilon/2 \right) + \prob_W^* \left(2\norm{\eta_n^* - \eta_0}_{L_2(\lambda)}\norm{\eta_n^* + \eta_0}_{L_2(\lambda)} \geq \varepsilon/2 \right) \\
        & \qquad \leq \prob_W^* \left(2M\norm{\eta_n^* - \eta_0}_{L_2(\lambda)} \geq \varepsilon/2 \right) + \prob_W^* \left(4M\norm{\eta_n^* - \eta_0}_{L_2(\lambda)} \geq \varepsilon/2 \right) + \prob_W^* \left(\norm{\eta_n^*}_{L_2(\lambda)} \geq M \right),
    \end{align*}
   Each term on the right-hand side is $o_{\prob_0^*}(1)$, which implies condition~\ref{cond: bootstrap weak consistency}. 
    
    For condition~\ref{cond: bootstrap second order} for $\psi_{n,1}^*$, we write
    \begin{align*}
        &\abs{T_1(\eta_n^*, \d{P}_n^*) - T_1(\eta_n, \d{P}_n) - (\d{P}_n^* - \d{P}_n)\phi_n^*} \\
        & \qquad = 
        \abs{\left(2\d{P}_n^* \eta_n^* - \int \eta_n^{*2}\right) - \left(2\d{P}_n \eta_n - \int \eta_n^2 \right) - (\d{P}_n^* - \d{P}_n)2\eta_n^*} \\
        & \qquad = \abs{- \int(\eta_n^* - \eta_n)^2 - 2\int \eta_n(\eta_n^* - \eta_n) + 2(\d{P}_n - P_0)(\eta_n^* - \eta_n) + 2P_0(\eta_n^* - \eta_n)}\\
        & \qquad = \abs{- \int(\eta_n^* - \eta_n)^2 - 2\int (\eta_n - \eta_0)(\eta_n^* - \eta_n) + 2(\d{P}_n - P_0)(\eta_n^* - \eta_n)}\\
        & \qquad \leq \norm{\eta_n^* - \eta_n}_{L_2(\lambda)}^2 + 2\norm{\eta_n - \eta_0}_{L_2(\lambda)} \norm{\eta_n^* - \eta_n}_{L_2(\lambda)} + 2n^{-1/2}\abs{\d{G}_n(\eta_n^* - \eta_n)}.
    \end{align*}
    The first two terms are $o_{\prob_W^*}(n^{-1/2})$ by assumption. In addition, $\d{G}_n (\eta_n^* - \eta_0) = o_{\prob_W^*}(1)$ because $\norm{\eta_n^* - \eta_0}_{L_2(P_0)} = o_{\prob_W^*}(1)$ and $\eta_n^*$ is contained in a $P_0$-Donsker class with $\prob_W^*$-probability tending to one. Hence, condition~\ref{cond: bootstrap second order} holds for $\psi_{n,1}^*$, and the conditional asymptotic linearity result for $\psi_{n,1}^*$ follows.
    
    For condition~\ref{cond: bootstrap second order} for $\psi_{n,2}^*$ and $\psi_{n,3}^*$, we have
    \begin{align*}
        T_2(\eta_n^*, \d{P}_n^*) - T_2(\eta_n, \d{P}_n) - (\d{P}_n^* - \d{P}_n)\phi_n^* &= \int \eta_n^{*2} - \int \eta_n^{2} - 2(\d{P}_n^* - \d{P}_n)\eta_n^* \\
        & = - \int (\eta_n^* - \eta_n)^2 +  2 \left(\int \eta_n^{*2} - \d{P}_n^* \eta_n^*\right) - 2\left(\int \eta_n \eta_n^* - \d{P}_n \eta_n^*\right)\\
        T_3(\eta_n^*, \d{P}_n^*) - T_3(\eta_n, \d{P}_n) - (\d{P}_n^* - \d{P}_n)\phi_n^* &= \d{P}_n^* \eta_n^* - \d{P}_n \eta_n - 2(\d{P}_n^* - \d{P}_n)\eta_n^* \\
        & = - \int (\eta_n^* - \eta_n)^2 + \left(\int \eta_n^{*2} - \d{P}_n^* \eta_n^*\right) +\left(\int \eta_n^{2} - \d{P}_n \eta_n\right) \\
        &\qquad - 2\left(\int \eta_n \eta_n^* - \d{P}_n \eta_n^*\right).
    \end{align*}
    Now, $ \int (\eta_n^* - \eta_n)^2$ and $\int \eta_n^{*2} - \d{P}_n^* \eta_n^*$ are both $o_{\prob_W^*}(n^{-1/2})$ by assumption, and $\int \eta_n^2 - \d{P}_n \eta_n = o_{\prob_0^*}(n^{-1/2})$ by assumption. In addition, 
    \begin{align*}
        \abs{\int \eta_n \eta_n^* - \d{P}_n \eta_n^*} =& \abs{\int (\eta_n - \eta_0)(\eta_n^* - \eta_n) - (\d{P}_n - P_0)(\eta_n^* - \eta_n) + \int \eta_n^2 - \d{P}_n \eta_n}\\
        &\leq \norm{\eta_n - \eta_0}_{L_2(\lambda)} \norm{\eta_n^* - \eta_n}_{L_2(\lambda)} + n^{-1/2}\abs{\d{G}_n(\eta_n^* - \eta_n)} + \left|\int \eta_n^2 - \d{P}_n \eta_n\right|,
    \end{align*}
    which is $o_{\prob_W^*}(n^{-1/2})$. Therefore, condition~\ref{cond: bootstrap second order} holds for $\psi_{n,2}^*$ and $\psi_{n,3}^*$.

    We now turn to the setting where $\eta_n$ and $\eta_n^*$ are kernel density estimators using the same bandwidth $h$. We note that for any $x$, since $\hat{P}_n$ is the empirical bootstrap,
    \begin{align*}
        \E_W^* \left[ \eta_n^*(x) \right] &= \E_W^* \left[ \d{P}_n^* h^{-d}K\left(\frac{X_i^* - x}{h}\right) \right] = \d{P}_n h^{-d}K\left(\frac{X_i - x}{h}\right) = \eta_n(x).
    \end{align*} 
    We therefore have 
    \begin{align*}
        \E_W^* \norm{\eta_n^* - \eta_n}_{L_2(\lambda)}^2 & = \int \E_W^* [\eta_n^*(x) - \eta_n(x)]^2 \sd x = \int \mathrm{Var}_W^*(\eta_n^*(x)) \sd x \\
        & = \int \frac{1}{n h^{2d}} \mathrm{Var}_W^* \left[K\left(\frac{X_1^* - x}{h} \right)\right] \sd x \\
        & \leq \frac{1}{n h^{2d}} \int \E_W^* \left[ K\left(\frac{X_1^* - x}{h}\right) \right]^2 \sd x\\
        & = \frac{1}{n h^{2d}} \int \d{P}_n\left[ K\left(\frac{X_i - x}{h}\right) \right]^2 \sd x\\
        & = \frac{1}{n h^d} \int K^2(t) \sd t.
    \end{align*}
    Hence, by Chebyshev's inequality, for any $\varepsilon >0$, we have 
    \begin{align*}
    \prob_W^* \left( n^{1/4}\norm{\eta_n^* - \eta_n}_{L_2(\lambda)} \geq \varepsilon \right) &\leq n^{1/2} \E_W^* \norm{\eta_n^* - \eta_n}_{L_2(\lambda)}^2/\varepsilon^2 \\
    &\leq  \left(n h^{2d}\right)^{-1/2}\varepsilon^{-2}\int K^2(t) \sd t,
    \end{align*}
    which goes to zero since $nh^{2d} \longrightarrow \infty$. Therefore $ \norm{\eta_n^* - \eta_n}_{L_2(\lambda)} = o_{\prob_W^*}(n^{-1/4})$. 
\end{proof}

The proof of \Cref{prop: auto bias correction} relies on technical lemmas that are stated and proved in Appendix~\ref{sec: emp auto bias lemmas}.

\propautobiasempbootstrapaverage*
\begin{proof}[\bfseries{Proof of \Cref{prop: auto bias correction}}]
    Since $\phi_P = 2(\eta_P - \psi_P)$, we then have 
    $S_n^* = (\d{P}_n^* - \d{P}_n)(\phi_n^* - \phi_0) = 2(\d{P}_n^* - \d{P}_n)(\eta_n^* - \eta_0)$ and $S_n = 2(\d{P}_n - P_0)(\eta_n - \eta_0)$. By \Cref{thm: auto empirical process}, we have $ S_n^* - S_n= o_{\prob_W^*}(n^{-1/2})$ under the stated assumptions.
    
    We have $R_n^* = \psi_n^* - \psi_n - (\d{P}_n^* - \d{P}_n)\phi_n^* = \psi_n^* - \psi_n - 2(\d{P}_n^* - \d{P}_n)\eta_n^*$. For the one-step estimator $\psi_{n, 1}^*$, we have $R_n^* = -\int (\eta_n^* - \eta_0)^2 + \int (\eta_n - \eta_0)^2 + (\d{P}_n - P_0)(\phi_n^* - \phi_n)$ and $R_n = -\int(\eta_n - \eta_0)^2 $. Hence
    \[
        R_n^* - R_n = -\int (\eta_n^* - \eta_0)^2 + 2\int (\eta_n - \eta_0)^2 + (\d{P}_n - P_0)(\phi_n^* - \phi_n).
    \]
    If $nh^{4m} \longrightarrow 0$ and $nh^d \longrightarrow \infty$ hold, then $\int (\eta_n^* - \eta_0)^2 - 2\int (\eta_n - \eta_0)^2  = o_{\prob_W^*}(n^{-1/2})$ by \Cref{thm: auto mise}, and  $(\d{P}_n - P_0)(\phi_n^* - \phi_n) = o_{\prob_W^*}(n^{-1/2})$ by \Cref{thm: auto empirical process}. This implies bootstrap percentile confidence intervals based on $\psi_{n, 1}^*$ are asymptotically valid by \Cref{thm: perc bootstrap CI}.

    For the plug-in estimator $\psi_{n, 2}^*$, we have $R_n^* = -\int (\eta_n^* - \eta_0)^2 + \int (\eta_n - \eta_0)^2 - \d{P}_n^*\phi_n^* + \d{P}_n \phi_n + (\d{P}_n - P_0)(\phi_n^* - \phi_n)$ and $R_n = -\int(\eta_n - \eta_0)^2 - \d{P}_n \phi_n$.  Hence, 
    \[
        R_n^* - R_n = -\int (\eta_n^* - \eta_0)^2 + 2\int (\eta_n - \eta_0)^2 - \d{P}_n^*\phi_n^* + 2\d{P}_n \phi_n + (\d{P}_n - P_0)(\phi_n^* - \phi_n).
    \]
    Since $nh^{2m} \longrightarrow 0$ and $nh^d \longrightarrow \infty$ hold, we have $\int (\eta_n^* - \eta_0)^2 - 2\int (\eta_n - \eta_0)^2 = o_{\prob_W^*}(n^{-1/2})$ by \Cref{thm: auto mise}, $\d{P}_n^*\phi_n^* - 2\d{P}_n \phi_n = o_{\prob_W^*}(n^{-1/2})$ by \Cref{thm: auto asym bias} and  $(\d{P}_n - P_0)(\phi_n^* - \phi_n) = o_{\prob_W^*}(n^{-1/2})$ by \Cref{thm: auto empirical process}. This implies bootstrap percentile confidence intervals based on $\psi_{n, 2}^*$ are asymptotically valid by \Cref{thm: perc bootstrap CI}.
    
   For the empirical mean plug-in estimator $\psi_{n, 3}^*$, we have $R_n^* = -\int (\eta_n^* - \eta_0)^2 + \int (\eta_n - \eta_0)^2 - \d{P}_n^*\phi_n^*/2 + \d{P}_n \phi_n/2 + (\d{P}_n - P_0)(\phi_n^* - \phi_n)$ and $R_n = -\int(\eta_n - \eta_0)^2 - \d{P}_n \phi_n/2$.  Hence, 
    \[
        R_n^* - R_n = -\int (\eta_n^* - \eta_0)^2 + 2\int (\eta_n - \eta_0)^2 - \d{P}_n^*\phi_n^* /2+ \d{P}_n \phi_n + (\d{P}_n - P_0)(\phi_n^* - \phi_n).
    \]
    Since $nh^{2m} \longrightarrow 0$ and $nh^d \longrightarrow \infty$ hold, we have $\int (\eta_n^* - \eta_0)^2 = 2\int (\eta_n - \eta_0)^2 + o_{\prob_W^*}(n^{-1/2})$ by \Cref{thm: auto mise}, $\d{P}_n^*\phi_n^* = 2\d{P}_n \phi_n + o_{\prob_W^*}(n^{-1/2})$ by \Cref{thm: auto asym bias} and  $(\d{P}_n - P_0)(\phi_n^* - \phi_n) = o_{\prob_W^*}(n^{-1/2})$ by \Cref{thm: auto empirical process}. This implies bootstrap percentile confidence intervals based on $\psi_{n, 3}^*$ are asymptotically valid by \Cref{thm: perc bootstrap CI}.
\end{proof}

\propsmobootstrapaverage*
\begin{proof}[\bfseries{Proof of Proposition~\ref{prop: smooth bootstrap average}}]
    We use \Cref{thm: bootstrap}. Condition~\ref{cond: bootstrap limited complexity} holds because $\prob_W(\eta_n^* \in \s{F}) \inoutprob 1$ and the conditions of \Cref{prop: uniform donsker} hold by assumption. For any $\varepsilon >0$, by the triangle and Cauchy–Schwarz inequalities, we have
    \begin{align*}
        &\prob_W^* \left(\norm{\phi_n^* - \phi_0}_{L_2(P_0)} \geq \varepsilon \right)  \\
        & \qquad \leq \prob_W^* \left(2\norm{\eta_n^* - \eta_0}_{L_2(P_0)} \geq \varepsilon/2 \right) + \prob_W^* \left(2\abs{\int \eta_n^{*2} - \int \eta_0^2} \geq \varepsilon/2 \right) \\
        & \qquad \leq \prob_W^* \left(2M\norm{\eta_n^* - \eta_0}_{L_2(\lambda)} \geq \varepsilon/2 \right) + \prob_W^* \left(2\norm{\eta_n^* - \eta_0}_{L_2(\lambda)}\norm{\eta_n^* + \eta_0}_{L_2(\lambda)} \geq \varepsilon/2 \right) \\
        & \qquad \leq \prob_W^* \left(2M\norm{\eta_n^* - \eta_0}_{L_2(\lambda)} \geq \varepsilon/2 \right) + \prob_W^* \left(4M\norm{\eta_n^* - \eta_0}_{L_2(\lambda)} \geq \varepsilon/2 \right) + \prob_W^* \left(\norm{\eta_n^*}_{L_2(\lambda)} \geq M \right). 
    \end{align*}
    Since $\norm{\eta_n^* - \eta_0}_{L_2(\lambda)} \leq \norm{\eta_n^* - \hat\eta_n}_{L_2(\lambda)} + \norm{\hat\eta_n - \eta_0}_{L_2(\lambda)}$ and $\norm{\eta_n^*}_{L_2(\lambda)} \leq \norm{\eta_n^*}_{\infty}$, each term on the right-hand side of the previous display is $o_{\prob_0^*}(1)$, which implies condition~\ref{cond: bootstrap weak consistency}.

    For condition~\ref{cond: bootstrap second order}, we have
    \begin{align*}
        T_1(\eta_n^*, \d{P}_n^*) - T_1(\eta_n, \hat{P}_n) - (\d{P}_n^* - \hat{P}_n)\phi_n^* =& - \int (\eta_n^* - \hat{\eta}_n)^2 + \int (\hat{\eta}_n - \eta_n)^2,
    \end{align*}
    which is $o_{\prob_W^*}(n^{-1/2})$ by assumption. Hence, condition~\ref{cond: bootstrap second order} holds for $\psi_{n, 1}^*$. We next have
    \begin{align*}
        T_2(\eta_n^*, \d{P}_n^*) - T_2(\eta_n, \hat{P}_n) - (\d{P}_n^* - \hat{P}_n)\phi_n^* &= \int \eta_n^{*2} - \int \eta_n^{2} - 2(\d{P}_n^* - \hat{P}_n)\eta_n^* \\
        &= - \int (\eta_n^* - \hat{\eta}_n)^2 + \int (\hat{\eta}_n - \eta_n)^2 +2 \left[ \int \eta_n^{*2} - \d{P}_n^* \eta_n^* \right] \\
        &\qquad- 2\left[ \int \eta_n^2 - \hat{P}_n \eta_n \right],\\
        T_3(\eta_n^*, \d{P}_n^*) - T_3(\eta_n, \hat{P}_n) - (\d{P}_n^* - \hat{P}_n)\phi_n^* &= \d{P}_n^* \eta_n^* - \hat{P}_n \eta_n - 2(\d{P}_n^* - \hat{P}_n)\eta_n^* \\
        &= -\int (\eta_n^* - \hat{\eta}_n)^2 + \int (\hat{\eta}_n - \eta_n)^2 + \left[ \int \eta_n^{*2} - \d{P}_n^* \eta_n^* \right] \\
        &\qquad- \left[ \int \eta_n^2 - \hat{P}_n \eta_n \right],
    \end{align*}
    both of which are $o_{\prob_W^*}(n^{-1/2})$ by assumption. Hence, condition~\ref{cond: bootstrap second order} holds for $\psi_{n, 2}^*$ and $\psi_{n, 3}^*$.

\end{proof}

The proof of \Cref{prop: auto bias correction smooth} relies on technical lemmas that are stated and proved in Appendix~\ref{sec: smooth auto bias lemmas}.
\propautobiassmoothbootstrapaverage*
\begin{proof}[\bfseries{Proof of \Cref{prop: auto bias correction smooth}}]
    Since $\phi_P = 2(\eta_P - \psi_P)$, we then have 
    $S_n^* = (\d{P}_n^* - \hat{P}_n)(\phi_n^* - \phi_0) = 2(\d{P}_n^* - \hat{P}_n)(\eta_n^* - \eta_0)$ and $S_n = 2(\d{P}_n - P_0)(\eta_n - \eta_0)$. By \Cref{thm: auto smooth empirical process}, we have $ S_n^* - S_n= o_{\prob_W^*}(n^{-1/2})$ under the stated assumptions.
    
    We have $R_n^* = \psi_n^* - T(\eta_n, \hat{P}_n) - (\d{P}_n^* - \hat{P}_n)\phi_n^* = \psi_n^* - T(\eta_n, \hat{P}_n) - 2(\d{P}_n^* - \hat{P}_n)\eta_n^*$. For the one-step estimator $\psi_{n, 1}^*$, we have $R_n^* = -\int (\eta_n^* - \eta_n)^2$ and $R_n = -\int(\eta_n - \eta_0)^2 $. Hence
    \[
        R_n^* - R_n = -\int (\eta_n^* - \eta_n)^2 + \int (\eta_n - \eta_0)^2
    \]
    If $nh^{4m} \longrightarrow 0$ and $nh^{2d} \longrightarrow \infty$ hold, then $\int (\eta_n^* - \eta_n)^2 - \int (\eta_n - \eta_0)^2  = o_{\prob_W^*}(n^{-1/2})$ by \Cref{thm: auto smooth mise}. This implies bootstrap percentile confidence intervals based on $\psi_{n, 1}^*$ are asymptotically valid by \Cref{thm: perc bootstrap CI}.

    For the plug-in estimator $\psi_{n, 2}^*$, we have $R_n^* = -\int (\eta_n^* - \eta_n)^2 - \d{P}_n^*\phi_n^*$ and $R_n = -\int(\eta_n - \eta_0)^2 - \d{P}_n \phi_n$.  Hence, 
    \[
        R_n^* - R_n = -\int (\eta_n^* - \eta_n)^2 + \int (\eta_n - \eta_0)^2 - \d{P}_n^*\phi_n^* + \d{P}_n \phi_n.
    \]
    Since $nh^{4m} \longrightarrow 0$ and $nh^{2d} \longrightarrow \infty$ hold, we have $\int (\eta_n^* - \eta_n)^2 - \int (\eta_n - \eta_0)^2 = o_{\prob_W^*}(n^{-1/2})$ by \Cref{thm: auto smooth mise}, $\d{P}_n^*\phi_n^* - \d{P}_n \phi_n = o_{\prob_W^*}(n^{-1/2})$ by \Cref{thm: auto smooth asym bias}. This implies bootstrap percentile confidence intervals based on $\psi_{n, 2}^*$ are asymptotically valid by \Cref{thm: perc bootstrap CI}.
    
   For the empirical mean plug-in estimator $\psi_{n, 3}^*$, we have $R_n^* = -\int (\eta_n^* - \eta_n)^2 - \d{P}_n^*\phi_n^*/2$ and $R_n = -\int(\eta_n - \eta_0)^2 - \d{P}_n \phi_n/2$.  Hence, 
    \[
        R_n^* - R_n = -\int (\eta_n^* - \eta_n)^2 + \int (\eta_n - \eta_0)^2 - \d{P}_n^*\phi_n^* /2+ \d{P}_n \phi_n/2.
    \]
    Since $nh^{4m} \longrightarrow 0$ and $nh^{2d} \longrightarrow \infty$ hold, we have $\int (\eta_n^* - \eta_n)^2 - \int (\eta_n - \eta_0)^2 =  o_{\prob_W^*}(n^{-1/2})$ by \Cref{thm: auto smooth mise}, $\d{P}_n^*\phi_n^* - \d{P}_n \phi_n = o_{\prob_W^*}(n^{-1/2})$ by \Cref{thm: auto smooth asym bias}. This implies bootstrap percentile confidence intervals based on $\psi_{n, 3}^*$ are asymptotically valid by \Cref{thm: perc bootstrap CI}.
\end{proof}

\propclassicalcounterfactual*
\begin{proof}[\bfseries{Proof of \Cref{prop: classical counterfactual}}]
    Condition~\ref{cond: limited complexity} holds by the Donsker assumptions, the boundedness assumptions on $\mu_n$ and $g_n$, and $E(Y^2) < \infty$ together with preservation of the Donsker condition under Lipschitz transformations (e.g., Theorem 2.10.6 of \citealp{van1996weak}).
    
    We next address condition~\ref{cond: weak consistency}. By adding and subtracting terms, we have 
    \begin{align*}
        \phi_n(y,a,z) - \phi_0(y,a,z) &= (1-a)\left[y-\mu_{0}(z)\right] \frac{g_n(z)}{1-g_n(z)} \frac{1}{\pi_n\pi_0} \left[\pi_0 - \pi_n\right] + a\left[\mu_{0}(z) - \psi_0\right]\frac{1}{\pi_n\pi_0} \left[\pi_0 - \pi_n\right] \\
        & \qquad + \frac{(1-a)\left[y-\mu_{0}(z)\right]}{\pi_0\left[1-g_n(z)\right]\left[1-g_0(z)\right]} \left[g_n(z) - g_0(z)\right]\\
        & \qquad + \frac{1}{\pi_n}\left[\frac{a}{g_n(z)} - 1 \right]\frac{g_n(z)}{1-g_n(z)} \left[\mu_{n}(z)- \mu_{0}(z)\right] - \frac{a}{\pi_n} \left[ \psi_n - \psi_0\right].
    \end{align*}
    Therefore, by the assumed bounds and the triangle inequality, $\norm{\phi_n - \phi_0}_{L_2(P_0)}$ is bounded up to a constant by
    \begin{align*}
        |\pi_n - \pi_0| +  \norm{g_n - g_0}_{L_2(P_0)} + \norm{\mu_{n}- \mu_{0}}_{L_2(P_0)} + |\psi_n - \psi_0|.
    \end{align*}
    We note that consistency of $\mu_{n}$ and $g_n$ implies consistency of $\pi_n$ and $\psi_n$. Therefore, condition~\ref{cond: weak consistency} holds.

    For condition~\ref{cond: second order} for the one-step estimator $\psi_{n, 1}=T_1(\eta_n, \d{P}_n)$, we have $T_1(\eta_n, \d{P}_n) - T_1(\eta_0, P_0) -  (\d{P}_n - P_0) \phi_n =  \psi_n - \psi_0 + P_0 \phi_n$. Using the law of total expectation, a straightforward calculation shows that
 \begin{align*}
        P_0 \phi_{\psi, \eta} &=  P_0 \left\{ \frac{(g_0 - g)(\mu - \mu_0)}{\pi (1-g)} \right\} + \frac{\pi_0}{\pi} (\psi_0 - \psi).
    \end{align*}
    Therefore,
    \begin{align*}
       \psi_n - \psi_0 + P_0 \phi_n  &= P_0 \left\{ \frac{(g_0 - g_n)(\mu_{n}- \mu_{0})}{\pi_n (1-g_n)}\right\} + \frac{\pi_n-\pi_0}{\pi_n} (\psi_n - \psi_0). 
    \end{align*}
    These terms are both $o_{\prob_0^*}(n^{-1/2})$ by the assumed rates and boundedness conditions. The result for the one-step estimator follows.

    For the estimating equations-based estimator, we use \Cref{lemma: classical estimating-equations}. We have already established that condition~\ref{cond: limited complexity} holds for $\phi_n = \phi_{\psi_n, \eta_n}$. We have that the population estimating equation is $G_{0,\eta_0}(\psi) = \psi - \psi_0$, so $G_{0,\eta_0}(\psi_0) = 0$ and $\inf_{|\psi - \psi_0| > \delta} | G_{0,\eta_0}(\psi)| = \delta$, and hence $\psi_0$ is a well-separated solution of the population estimating equation. Furthermore, it is straightforward to see that the boundedness conditions and the assumption that $E_0(Y^2) < \infty$ imply that $\psi_{n,2} = O_{\prob_0^*}(1)$. Next, defining $G_{0, \eta}'(\psi_0) = -\pi_0\pi^{-1}$, we then have $G_{0, \eta_0}'(\psi_0) = -1$ and $\lim_{\eta \to \eta_0}G_{0, \eta}'(\psi_0) = \lim_{\eta \to \eta_0} -\pi_0 \pi^{-1} = -1$. In addition, 
    \begin{align*}
        \Gamma_{0,\eta}(\psi) &= G_{0, \eta}(\psi) - G_{0, \eta}(\psi_0) - G_{0, \eta}'(\psi_0)(\psi-\psi_0) \\
        &= P_0 \phi_{\eta, \psi} - P_0 \phi_{\eta, \psi_0} + \frac{\pi_0}{\pi}(\psi-\psi_0)\\
        &= P_0 \left\{ \frac{g_0(Z) - g(Z)}{\pi \left[1-g(Z)\right]}\left[\mu(Z) - \mu_{0}(Z)\right] \right\} + \frac{\pi_0}{\pi} (\psi - \psi_0) \\
        & \qquad- P_0 \left\{ \frac{g_0(Z) - g(Z)}{\pi \left[1-g(Z)\right]}\left[\mu(Z) - \mu_{0}(Z)\right] \right\} - \frac{\pi_0}{\pi} (\psi_0 - \psi_0)\\
        & \qquad  + \frac{\pi_0}{\pi} (\psi_0 - \psi) \\
        & = 0.
    \end{align*}
    Therefore, both conditions about $\Gamma_{0,\eta}$ hold. Finally, 
    \begin{align*}
        P_0 \phi_{\psi_0, \eta_n} &=  P_0 \left\{ \frac{g_0 - g_n}{\pi_n \left[1-g_n\right]}\left[\mu_{n}- \mu_{0}\right] \right\} - \frac{\pi_0}{\pi_n} (\psi_0 - \psi_0)\\
        & = P_0 \left\{ \frac{g_0 - g_n}{\pi_n \left[1-g_n\right]}\left[\mu_{n}- \mu_{0}\right] \right\},
    \end{align*}
    which is $o_{\prob_0^*}(n^{-1/2})$ by assumption. Hence, condition~\ref{cond: second order} holds for the estimating equations-based estimator.
\end{proof}

\propempbootstrapcounterfactual*
\begin{proof}[\bfseries{Proof of \Cref{prop: empirical bootstrap counterfactual}}]
    By~\Cref{lemma: empirical bootstrap}, condition~\ref{cond: bootstrap limited complexity} holds by the Donsker assumptions, the boundedness assumptions on $\mu_n^*$ and $g_n^*$, and $E(Y^2) < \infty$ together with preservation of the Donsker condition under Lipschitz transformations.
    
    As in the proof of~\Cref{prop: classical counterfactual}, we can show that $\left\|\phi_n^* - \phi_0\right\|_{L_2(P_0)}$ is bounded up to a constant by 
    \[ 
        \left| \pi_n^* - \pi_0 \right| + \left\| g_n^* - g_0 \right\|_{L_2(P_0)} + \left\| \mu_n^* - \mu_0 \right\|_{L_2(P_0)} + \left| \psi_n^* - \psi_0\right|.
    \]
    Therefore, condition~\ref{cond: bootstrap weak consistency} holds under the conditional weak $L_2(P_0)$ conditional consistency of $g_n^*$ and $\mu_{n}^*$, which also implies weak conditional consistency of $\psi_n^*$ and $\pi_n^*$.
    
    For condition~\ref{cond: bootstrap second order} for the one-step estimator $\psi_{n, 1}^* = T_1(\eta_n^*, \d{P}_n^*)$, as in equation~\eqref{eq: one step Rnstar}, we have 
    \begin{align*}
        R_n^* &= \left[\psi_n^* - \psi_0 + P_0 \phi_n^*\right] - \left[ \psi_n - \psi_0 + P_0 \phi_n\right] + (\d{P}_n - P_0)( \phi_{n}^* - \phi_n)
    \end{align*}
    By the proof of~\Cref{prop: classical counterfactual}, $\psi_n - \psi_0 + P_0 \phi_n = o_{\prob_0^*}(n^{-1/2})$, and 
    \[ n^{1/2}(\d{P}_n - P_0)( \phi_{n}^* - \phi_n) = \d{G}_n\left( \phi_{n}^* - \phi_0\right) -  \d{G}_n \left(\phi_{n}- \phi_0\right),\]
    which is $o_{\prob_W^*}(n^{-1/2})$ by~\Cref{prop: classical counterfactual} and the above. By the derivations in~\Cref{prop: classical counterfactual}, we also have 
    \begin{align*}
         \psi_n^* - \psi_0 + P_0 \phi_n^*&= P_0 \left\{ \frac{\left(g_0 - g_n^*\right)\left(\mu_{n}^* - \mu_{0}\right)}{\pi_n^* \left(1-g_n^*\right)} \right\} + \frac{\pi_n^*-\pi_0}{\pi_n^*} \left(\psi_n^* - \psi_0\right).
    \end{align*}  
    This is $o_{\prob_W^*}(n^{-1/2})$ by assumption.
    
    For condition~\ref{cond: bootstrap second order} for the estimating equations-based estimator $\psi_{n, 2}^*$, we use~\Cref{lemma: empirical bootstrap estimating-equations}. We note that the conditions of~\Cref{lemma: classical estimating-equations} hold by~\Cref{prop: classical counterfactual} and condition~\ref{cond: bootstrap limited complexity} holds by the above. We have $\psi_{n,2}^* =  O_{\prob_W^*}(1)$ by the boundedness conditions and because $E_0(Y^2) < \infty$. As in the proof of~\Cref{prop: classical counterfactual}, we have $\Gamma_{0,\eta}(\psi) = 0$ for all $\eta \in \s{H}$ and $\psi \in \d{R}$, so that $\sup_{|\psi| \leq M} \left| \Gamma_{0,\eta_n^*}(\psi) - \Gamma_{0,\eta_0}(\psi_0)\right| = 0$ for every $M > 0$. Finally,
    \begin{align*}
        P_0 \phi_{\psi_0, \eta_n^*} & =  P_0 \left\{ \frac{\left(g_0- g_n^*\right)\left(\mu_{n} - \mu_{0}\right)}{\pi_n^* \left(1-g_n^*\right)} \right\}.
    \end{align*}
    which is $o_{\prob_W^*}(n^{-1/2})$ by  assumption.
\end{proof}

\propsmobootstrapcounterfactual*

\begin{proof}[\bfseries{Proof of \Cref{prop: smooth bootstrap counterfactual}}]
    We use \Cref{prop: uniform donsker} to show condition~\ref{cond: bootstrap limited complexity} holds. By the assumed classes of $\mu_n^*$ and $g_n^*$, with conditional probability tending to 1, $\phi_n^*$ falls in a class $\s{F}$ with envelope $F(y, a, z) := c|Y| + d$ for some fixed $c, d \in (0, \infty)$. Thus, condition (i) of \Cref{prop: uniform donsker} holds by the assumption that $\lim_{M \to \infty} \sup_{P \in \s{P}} E_P [ Y^2 I( Y^2 > M)] = 0$. Next, by preservation of the finite uniform entropy integral condition under Lipschitz transformations, (see, e.g., Lemma~9.17 of \citealp{kosorok2008introduction}),  $\s{F}$ posesses finite uniform entropy integral. Hence, condition (ii) of \Cref{prop: uniform donsker} is satisfied by Theorem 2.8.3 of \cite{van1996weak}. Next, as noted following \Cref{prop: uniform donsker}, $\sup_{f, g \in \s{F}} \left| (\hat{P}_n - P_0)(fg)\right| = o_{\prob_0^*}(1)$ implies \eqref{eq: convergence of semimetric}. Let $\mu, \bar\mu \in \s{F}_\mu$ and $g, \bar{g} \in \s{F}_g$, and let $\pi, \bar\pi$ be implied by $g, \bar{g}$ and $\psi, \bar\psi$ be implied by $\mu, g$ and $\bar\mu, \bar{g}$. 
    Using the fact that $a (1-a) = 0$ and the tower property, we then have 
    \begin{align*}
        &(\hat{P}_n - P_0)\left( \phi_{\mu,g}\phi_{\bar\mu, \bar{g}}\right)   \\
        &\qquad =  \int \left[ \frac{(1-a)g\bar{g}}{\pi\bar\pi(1-g) (1-\bar{g})}  (y- \mu) ( y- \bar\mu)+ \frac{a}{\pi\bar\pi} (\mu - \psi)(\bar\mu - \bar\psi)  \right] \sd (\hat{P}_n - P_0)\\
        & \qquad =  \int \left[ \frac{(1-a)g\bar{g}}{\pi\bar\pi(1-g) (1-\bar{g})}  \left\{ y^2 - (\mu +  \bar\mu)y + \mu \bar\mu \right\} \right]  \sd (\hat{P}_n - P_0) \\
        &\qquad \qquad + \int \left[ \frac{a}{\pi\bar\pi} (\mu - \psi)(\bar\mu - \bar\psi) \right] \sd (\hat{P}_n - P_0) \\ 
        &\qquad = \frac{1}{\pi\bar\pi} \int \frac{g\bar{g}}{(1-g) (1-\bar{g})}\left\{  (1-\hat{g}_n) \left[ (\hat\sigma_n^2 + \hat\mu_n^2)- (\mu +  \bar\mu)\hat\mu_n  + \mu \bar\mu \right]  \sd \hat{Q}_n \right. \\
        &\qquad \qquad \qquad \qquad \qquad \qquad\qquad -(1-g_0)\left[  (\sigma_0^2 + \mu_0^2)- (\mu +  \bar\mu)\mu_0  + \mu \bar\mu \right]  \sd Q_0 \Big\}\\
        &\qquad \qquad  + \frac{1}{\pi\bar\pi} \int (\mu - \psi)(\bar\mu - \bar\psi)\left[ \hat{g}_n \sd \hat{Q}_n - g_0 \sd Q_0 \right] \\
        &\qquad = \frac{1}{\pi\bar\pi} \int \frac{g\bar{g}}{(1-g) (1-\bar{g})}\left\{  (1-\hat{g}_n) \left[ (\hat\sigma_n^2 + \hat\mu_n^2)- (\mu +  \bar\mu)\hat\mu_n  + \mu \bar\mu \right]  \right. \\
        &\qquad \qquad \qquad \qquad \qquad \qquad\qquad \left. - (1-g_0)\left[ (\sigma_0^2 + \mu_0^2)- (\mu +  \bar\mu)\mu_0  + \mu \bar\mu \right]  \right\} \sd Q_0 \\
        &\qquad \qquad + \frac{1}{\pi\bar\pi} \int \frac{g\bar{g}  (1-\hat{g}_n)}{(1-g) (1-\bar{g})}\left[(\hat\sigma_n^2 + \hat\mu_n^2)- (\mu +  \bar\mu)\hat\mu_n  + \mu \bar\mu \right] \sd (\hat{Q}_n - Q_0)\\
        &\qquad \qquad + \frac{1}{\pi\bar\pi} \int (\mu - \psi)(\bar\mu - \bar\psi)( \hat{g}_n - g_0)  \sd  Q_0 + \frac{1}{\pi\bar\pi} \int (\mu - \psi)(\bar\mu - \bar\psi)\hat{g}_n   \sd (\hat{Q}_n - Q_0). 
    \end{align*}
    By adding and subtracting terms and using the assumed bounds, we have
    \begin{align*}
        \sup_{\mu, \bar\mu ,g, \bar{g}} \left| (\hat{P}_n - P_0)\left( \phi_{\mu,g}\phi_{\bar\mu, \bar{g}}\right) \right| &\lesssim \| \hat{g}_n - g_0 \|_{L_2(Q_0)} + \| \hat{\mu}_n - \mu_0 \|_{L_2(Q_0)} + \| \hat\sigma_n^2 - \sigma_0^2 \|_{L_2(Q_0)} \\
        &\qquad + \sup_{g, \bar{g}} \left| \int \frac{g\bar{g}(1-\hat{g}_n)}{(1-g) (1-\bar{g})}(\hat\sigma_n^2 + \hat\mu_n^2) \sd (\hat{Q}_n - Q_0) \right| \\
        &\qquad +\sup_{g, \bar{g}, \mu} \left| \int \frac{g\bar{g}(1-\hat{g}_n)}{(1-g) (1-\bar{g})} \mu \hat\mu_n   \sd (\hat{Q}_n - Q_0) \right| \\
        &\qquad + \sup_{g, \bar{g}, \mu, \bar\mu}\left| \int \frac{g\bar{g}(1-\hat{g}_n)}{(1-g) (1-\bar{g})} \mu \bar\mu   \sd (\hat{Q}_n - Q_0) \right| \\
        &\qquad +\sup_{\mu} \left| \int  \mu  \hat{g}_n  \sd (\hat{Q}_n - Q_0) \right| +  \sup_{\mu ,\bar\mu}\left| \int  \mu  \bar\mu \hat{g}_n  \sd (\hat{Q}_n - Q_0) \right|.
    \end{align*}
   Since $\hat\mu_n \in \s{F}_\mu$ with probability tending to one, by assumption, each of these terms is $o_{\prob_0^*}(1)$, which implies \eqref{eq: convergence of semimetric}. Hence, the conditions of \Cref{prop: uniform donsker} hold, which implies condition~\ref{cond: bootstrap limited complexity}.

    We next show condition~\ref{cond: bootstrap weak consistency} holds. As in the proof of~\Cref{prop: classical counterfactual}, we can show that $\left\|\phi_n^* - \phi_0\right\|_{L_2(P_0)}$ is bounded up to a constant by 
    \begin{align*}
        & \left| \pi_n^* - \pi_0 \right| + \left\| g_n^* - g_0 \right\|_{L_2(P_0)} + \left\| \mu_n^* - \mu_0 \right\|_{L_2(P_0)} + \left| \psi_n^* - \psi_0\right|.
    \end{align*}
    Each of these terms is $o_{\prob_W^*}(1)$.


    Finally, we  turn to condition~\ref{cond: bootstrap second order}. By the tower property, we can show that 
    \[  P \phi_{\mu, g, \psi} = Q_P \left[ \frac{( g_P- g) (\mu - \mu_P)}{(1-g) \pi} \right] + \frac{\pi_P}{\pi} \left[ \psi(P) - \psi \right]\]
    for $\pi := \int g \, dQ$. Thus, 
    \[ \psi(\eta) - \psi(P) + P \phi_{\eta} = Q_P \left[ \frac{( g_P- g) (\mu - \mu_P)}{(1-g) \pi(\eta)} \right] + \frac{\pi(\eta) - \pi_P}{\pi(\eta)} \left[ \psi(\eta) - \psi(P)\right].\]
    Hence, for the one-step estimator $\psi_{n, 1}^* = T_1(\eta_n^*, \d{P}_n^*)$, as in equation~\eqref{eq: one step Rnstar}, we have 
    \begin{align*}
        R_n^* &= \left[\psi(\eta_n^*) - \psi(\hat\eta_n) + \hat{P}_n \phi_n^*\right] - \left[ \psi(\eta_n) - \psi(\hat\eta_n) + \hat{P}_n \phi_n\right]  \\
        & = \left[ \hat{Q}_n \left\{ \frac{(\hat{g}_n - g_n^*)(\mu_{n}^*- \hat\mu_{n})}{\pi_n^* (1-g_n^*)}\right\} + \frac{\pi_n^*-\hat\pi_n}{\pi_n^*} (\psi_n^* - \hat\psi_n) \right] \\
        & \qquad \qquad - \left[\hat{Q}_n \left\{ \frac{(\hat{g}_n - g_n)(\mu_{n}- \hat\mu_{n})}{\pi_n (1-g_n)}\right\} + \frac{\pi_n-\hat\pi_n}{\pi_n} (\psi_n - \hat\psi_n)\right]. 
    \end{align*}
    Each term is $o_{\prob_W^*}(n^{-1/2})$ by assumption. The result for the one-step estimator follows.


    For condition~\ref{cond: bootstrap second order} for the estimating equations-based estimator $\psi_{n, 2}^*$, we use \Cref{lemma: smooth bootstrap estimating-equations}. Since the conditions of \Cref{prop: classical counterfactual} hold, the conditions of \Cref{lemma: classical estimating-equations} hold as well. We have $\psi_{n,2}^* =  O_{\prob_W^*}(1)$ by the boundedness conditions on $\mu_n^*$, $g_n^*$, and $\hat{g}_n$, and the moment condition for $\hat{P}_n$. A similar argument shows that  $\hat\psi_n = O_{\prob_0^*}(1)$. As in the proof of~\Cref{prop: classical counterfactual}, we have $\Gamma_{0,\eta}(\psi) = 0$ for any $\eta \in \s{H}$ and $\psi \in \d{R}$, so that the condition $\sup_{|\psi| \leq M} | \Gamma_{0, \hat{\eta}_n}(\psi) - \Gamma_{0, \eta_0}(\psi)| = o_{\prob_0^*}(1)$ holds. Next, defining $\hat{G}_{n, \eta}'(\hat\psi_n) = -\hat\pi_n\pi^{-1}$, we then have $\hat{G}_{n, \hat\eta_n^*}'(\hat\psi_n) + 1 = 1 -\hat\pi_n /\pi_n^* = o_{\prob_W^*}(1)$ because $\pi_n^* - \hat\pi_n = o_{\prob_W^*}(1)$ and $g_n^*$ is bounded away from zero. In addition, 
    \begin{align*}
        \hat\Gamma_{n, \eta}(\psi) &= \hat{G}_{n, \eta}(\psi) - \hat{G}_{n, \eta}(\hat\psi_n) - \hat{G}_{n, \eta}'(\hat\psi_n)(\psi- \hat\psi_n) \\
        &= \hat{P}_n \phi_{\eta, \psi} - \hat{P}_n \phi_{\eta, \hat\psi_n} + \frac{\hat\pi_n}{\pi}(\psi- \hat\psi_n)\\
        &= \hat{P}_n \left\{ \frac{\hat{g}_n - g}{\pi \left[1-g\right]}\left[\mu - \hat\mu_{n}\right] \right\} + \frac{\hat\pi_n}{\pi} (\psi - \hat\psi_n) - \hat{P}_n \left\{ \frac{\hat{g}_n - g}{\pi \left[1-g\right]}\left[\mu - \hat\mu_{n}\right] \right\} - \frac{\hat\pi_n}{\pi} (\hat\psi_n - \hat\psi_n)\\
        & \qquad  + \frac{\hat\pi_n}{\pi} (\hat\psi_n- \psi) \\
        & = 0.
    \end{align*}
    Hence, the conditions about $|\hat\Gamma_{n, \eta_n^*}(\psi) - \hat\Gamma_{n, \hat\eta_n}(\hat\psi)|$ and $|\hat\Gamma_{n, \eta}(\psi)|$ hold. Next, $\phi_{\hat\psi_n, \hat\eta_n}$  and $\phi_{\psi_n^*, \hat\eta_n^*}$ fall in a $P_0$-Glivenko Cantelli class with probability to tending one by the finite entropy integral and boundedness assumptions. Next, we verify that $\|\hat{P}_n - P_0\|_{\s{F}} = o_{\prob_0^*}(1)$ holds under the stated assumptions. 
    \begin{align*}
        \int \phi_{\mu, g}(y,a,w) \sd (\hat{P}_n - P_0) &=  \int \left[ \frac{(1-a)g}{\pi(1-g)}  (y- \mu) + \frac{a}{\pi} (\mu - \psi)  \right] \sd (\hat{P}_n - P_0)\\
        & = \frac{1}{\pi} \int \frac{g}{(1-g)}\left\{   (1-\hat{g}_n)(\hat\mu_n - \mu)  \sd \hat{Q}_n -  (1-g_0) (\mu_0 - \mu)  \sd Q_0 \right\}\\
        & \qquad  + \frac{1}{\pi} \int (\mu - \psi)\left[ \hat{g}_n \sd \hat{Q}_n - g_0 \sd Q_0 \right]\\
        & = \frac{1}{\pi} \int \frac{g}{(1-g)}\left\{  (1-\hat{g}_n)(\hat\mu_n - \mu) -  (1-g_0) (\mu_0 - \mu)   \right\} \sd Q_0 \\
        &\qquad +  \frac{1}{\pi} \int \frac{g}{(1-g)} (1-\hat{g}_n)(\hat\mu_n - \mu) \sd (\hat{Q}_n - Q_0) \\
        &\qquad + \frac{1}{\pi} \int (\mu - \psi)\left( \hat{g}_n - g_0 \right) \sd Q_0 + \frac{1}{\pi} \int (\mu - \psi) \hat{g}_n  \sd (\hat{Q}_n - Q_0).
    \end{align*}
    Therefore, adding and subtracting terms and using assumed bounds,
    \begin{align*}
        \sup_{\mu,g} \abs{\int \phi_{\mu, g} \sd (\hat{P}_n - P_0)} & \lesssim \|\hat{g}_n - g_0\|_{L_2(Q_0)} + \| \hat\mu_n - \mu_0\|_{L_2(Q_0)}\\
        & \qquad + \sup_{\mu, g}\abs{\int \frac{g}{(1-g)} (1-\hat{g}_n)(\hat\mu_n - \mu) \sd (\hat{Q}_n - Q_0)} \\
        & \qquad + \sup_{\mu}\abs{\int (\mu - \psi) \hat{g}_n  \sd (\hat{Q}_n - Q_0)}. 
    \end{align*}
    Each of these terms is $o_{\prob_0^*}(1)$ by assumption. Thus, $\|\hat{P}_n - P_0\|_{\s{F}} = o_{\prob_0^*}(1)$. 
    
    Next, we have
    \begin{align*}
        P_0 \phi_{\psi_0, \hat\eta_n} & =  P_0 \left\{ \frac{\left(g_0- \hat{g}_n\right)\left(\hat\mu_{n} - \mu_{0}\right)}{\hat\pi_n \left(1-\hat{g}_n\right)} \right\}.
    \end{align*}
    which is $o_{\prob_0^*}(1)$ by the assumed boundedness of $\hat{g}_n$ away from zero and the consistency of $\hat\mu_n$ and $\hat{g}_n$. 

    Next, we show that $\psi_n^\circ - \hat\psi_n = o_{\prob_0^*}(n^{-1/2})$. Use the definition of the estimating equations-based estimator construction $T_2$, we can see that
    \begin{align*}
        \psi_n^\circ &:= T_2(\eta_n, \hat{P}_n) =  \hat{Q}_n \left\{ \frac{g_n \left(1- \hat{g}_n\right)}{\hat\pi_n (1-g_n)} \left( \hat\mu_n - \mu_n \right) + \frac{\hat{g}_n}{\hat\pi_n} \mu_n \right\},\\
        \hat\psi_n &:= T_2(\hat\eta_n, \hat{P}_n) = \hat{Q}_n \left\{ \frac{\hat{g}_n}{\hat\pi_n} \hat\mu_n \right\}, \\
        \psi_n^\circ - \hat\psi_n &= \frac{1}{\hat\pi_n} \hat{Q}_n \left\{ \frac{g_n - \hat{g}_n}{1-g_n} \left( \hat\mu_n - \mu_n \right) \right\},
    \end{align*}
    which is $o_{\prob_0^*}(n^{-1/2})$ by assumption. Finally, 
    \begin{align*}
        \hat{P}_n \phi_{\hat\psi_n, \eta_n^*} &=  \hat{Q}_n \left\{ \frac{\hat{g}_n - g_n^*}{\pi_n^* \left[1-g_n^*\right]}\left[\mu_{n}^* - \hat\mu_{n}\right] \right\} - \frac{\hat\pi_n}{\pi_n^*} (\hat\psi_n - \hat\psi_n)\\
        & = \hat{Q}_n \left\{ \frac{\hat{g}_n - g_n^*}{\pi_n^* \left[1-g_n^*\right]}\left[\mu_{n}^*- \hat\mu_{n}\right] \right\},
    \end{align*}
    which is $o_{\prob_0^*}(n^{-1/2})$ by assumption as discussed above. Hence, condition~\ref{cond: bootstrap second order} holds for the estimating equations-based estimator. 

\end{proof}

\section{Lemmas supporting the proof of Proposition~\ref{prop: auto bias correction}}\label{sec: emp auto bias lemmas}

We first present several simple algebraic identities that we will use repeatedly.
\begin{lemma}\label{lemma: square decomp}
    For any $\{Z_{ij} \in \d{R} : i,j = 1, 2, \dotsc, n\}$ such that $Z_{ij} = Z_{ji}$ and $Z_{ii} = \tau_f$ for all $i,j$, we have  
    \begin{align}
        \begin{split}\label{eq: auto identity 1}
            \left[ \sum_{i\neq j} Z_{ij}\right]^2 &= 4 \sum_{i \neq j, j \neq k, k \neq i}Z_{ij} Z_{ik} +2  \sum_{i \neq j}Z_{ij}^2  + \sum_{\substack{i\neq j, i\neq k, i\neq l\\ j \neq k, j \neq l, k \neq l}} Z_{ij} Z_{kl}, \text{ and}
        \end{split}\\
        \begin{split}\label{eq: auto identity 2}
            \sum_{i, j, k} Z_{ij}Z_{ik} &  =  2\tau_f \sum_{i\neq j} Z_{ij} + \sum_{i \neq j} Z_{ij}^2 + n\tau_f^2  +  \sum_{i\neq j, i \neq k, j \neq k} Z_{ij}Z_{ik}.
        \end{split}
    \end{align}
\end{lemma}
\begin{proof}[\bfseries{Proof of \Cref{lemma: square decomp}}]
    We have
    \begin{align*}
        \left[\sum_{i\neq j} Z_{ij}\right]^2 &=  \sum_{\substack{i\neq j, k \neq l}} Z_{ij} Z_{kl} \\
        &=  \sum_{\substack{i\neq j, k \neq l\\i=k, j\neq l}} Z_{ij} Z_{kl} +\sum_{\substack{i\neq j, k \neq l\\i=l, j\neq k}} Z_{ij} Z_{kl} +  \sum_{\substack{i\neq j, k \neq l\\i\neq l, j=k}} Z_{ij} Z_{kl} +  \sum_{\substack{i\neq j, k \neq l\\i\neq k, j= l}} Z_{ij} Z_{kl} \\
        & \qquad +  \sum_{\substack{i\neq j, k \neq l\\i=k, j= l}} Z_{ij} Z_{kl} + \sum_{\substack{i\neq j, k \neq l\\i=l, j= k}} Z_{ij} Z_{kl}  +  \sum_{\substack{i\neq j, i\neq k, i\neq l\\ j \neq k, j \neq l, k \neq l}} Z_{ij} Z_{kl}\\
        &= 4\sum_{i \neq j, j \neq k, k \neq i}Z_{ij} Z_{ik} + 2\sum_{i \neq j}Z_{ij}^2  + \sum_{\substack{i\neq j, i\neq k, i\neq l\\ j \neq k, j \neq l, k \neq l}} Z_{ij} Z_{kl},
    \end{align*}
    which proves~\eqref{eq: auto identity 1}. We also have
    \begin{align*}
        \sum_{i, j, k} Z_{ij}Z_{ik}  &= \sum_{i=j\neq k} Z_{ij}Z_{ik} +  \sum_{i=k \neq j} Z_{ij}Z_{ik} + \sum_{i\neq j=k} Z_{ij}Z_{ik} + \sum_{i=j=k} Z_{ij}Z_{ik}  + \sum_{i\neq j, i \neq k, j \neq k} Z_{ij}Z_{ik} \\
        & =  2\tau_f \sum_{i\neq j} Z_{ij} +\sum_{i \neq j} Z_{ij}^2 + n\tau_f^2 +  \sum_{i\neq j, i \neq k, j \neq k} Z_{ij}Z_{ik},
    \end{align*}
    which proves \eqref{eq: auto identity 2}.
\end{proof}

We next present several non-asymptotic bounds for empirical and bootstrap empirical $V$-processes for a generic function $f: \d{R}^d \times \d{R}^d \mapsto \d{R}$ satisfying $f(x, y) = f(y, x)$ and $f(x, x) = \tau_f \in \d{R}$ for all $x, y \in \d{R}^d$. Here, $f$ may depend on $n$, but is assumed to be deterministic. We use the symbol $\lesssim$ to mean ``less than or equal up to a constant not depending on $n$."

\begin{lemma}\label{lemma: V order 1}
    If $f: \d{R}^d \times \d{R}^d \mapsto \d{R}$ satisfies $f(x, y) = f(y, x)$ and $f(x, x) = \tau_f \in \d{R}$ for all $x, y \in \d{R}^d$, then
    \begin{align*}
        & \E_0  \left\{ \int f \sd [(\d{P}_n - P_0) \times (\d{P}_n - P_0)]  - n^{-1} \tau_f \right\}^2  \lesssim n^{-3} \int \left[\int f(x, y) \sd P_0(x)\right]^2 \sd P_0(y) + n^{-2} \norm{f}_{L_2(P_0 \times P_0)}^2.
    \end{align*}
\end{lemma}
\begin{proof}[\bfseries{Proof of \Cref{lemma: V order 1}}]
    For simplicity, we denote $V_0(x, y) :=  \int f \sd \left[(\delta_{x} - P_0)  \times (\delta_{y} - P_0)\right]$ for any $x, y \in \d{R}^d$. By the law of total expectation, for any $i \neq j$, we note that 
    \begin{align*}
        \E_0 [V_0(X_i, X_j) \mid X_j] & = \E_0 \left[f(X_i, X_j) - \int f(X_i, y) \sd P_0(y)  - \int f(x, X_j) \sd P_0(x) +  \int f \sd (P_0 \times P_0) \mid X_j\right] \\
        &= \int f(x, X_j) \sd P_0(x) - \int f \sd (P_0 \times P_0) - \int f(x, X_j) \sd P_0(x) + \int f \sd (P_0 \times P_0) \\
        & = 0
    \end{align*}
    Hence, $\E_0 [V_0(X_i, X_j)] = 0$ for $i \neq j$ and $\E_0 [V_0(X_i, X_j)V_0(X_i, X_k)] = 0$ for $i \neq j \neq k$.
    
    By definition, we have 
    \begin{align}
    \begin{split}\label{eq: V decomposition 1}
        & \E_0  \left\{ \int f \sd [(\d{P}_n - P_0) \times (\d{P}_n - P_0)] - n^{-1} \tau_f \right\}^2 \\
        & \qquad = \E_0  \left\{  \frac{1}{n^2} \sum_{i = j} V_0(X_i, X_j) + \frac{1}{n^2} \sum_{i \neq j} V_0(X_i, X_j) - n^{-1} \tau_f \right\}^2 \\
        & \qquad \lesssim \E_0  \left\{  \frac{1}{n^2} \sum_{i = j} V_0(X_i, X_j) - n^{-1} \tau_f\right\}^2 + \E_0  \left\{ \frac{1}{n^2} \sum_{i \neq j} V_0(X_i, X_j)\right\}^2.
    \end{split}
    \end{align}
    For the first term on the right-hand side of \eqref{eq: V decomposition 1}, by the symmetry of $f$ and the definition of $\tau_f$,
    \begin{align*}
        \E_0 \left\{ \frac{1}{n^2} \sum_{i = j} V_0(X_i, X_j) - \frac{\tau_f}{n} \right\}^2 & =  \E_0 \left\{- \frac{2}{n} \int f \sd (P_0 \times \d{P}_n )  + \frac{1}{n} \int f \sd (P_0 \times  P_0) \right\}^2\\
        & = \E_0 \left\{ - \frac{2}{n} \int \left[ \int f(x, y) \sd P_0(x) \right] \sd (\d{P}_n - P_0) (y) - \frac{1}{n} \int f \sd (P_0 \times P_0)\right\}^2\\
        & \lesssim \E_0 \left\{ \frac{1}{n} \int \left[ \int f(x, y) \sd P_0(x) \right] \sd (\d{P}_n - P_0) (y)\right\}^2 + \frac{1}{n^2} \norm{f}_{L_1(P_0 \times P_0)}^2\\
        & = \frac{1}{n^3} \mathrm{Var}\left( \left[ \int f(x, X_1) \sd P_0(x) \right] \right)  + \frac{1}{n^2} \norm{f}_{L_1(P_0 \times P_0)}^2\\
        & \leq \frac{1}{n^3}  \int \left[\int f(x, y) \sd P_0(x)\right]^2 \sd P_0(y)  + \frac{1}{n^2} \norm{f}_{L_1(P_0 \times P_0)}^2.
    \end{align*}
     For the second term on the right-hand side of \eqref{eq: V decomposition 1}, by \Cref{lemma: square decomp} and properties of $V_0$,
    \begin{align*}
        &\E_0 \left[ \frac{1}{n^2} \sum_{i \neq j} V_0(X_i, X_j) \right]^2 \\
        &\qquad = \E_0 \left[ \frac{4}{n^4} \sum_{i \neq j, j \neq k, k \neq i}V_0(X_i, X_j) V_0(X_i, X_k) + \frac{2}{n^4} \sum_{i \neq j}V_0(X_i, X_j)^2 \right.\\
        & \qquad\qquad \qquad \left.+ \frac{1}{n^4} \sum_{\substack{i\neq j, i\neq k, i\neq l\\ j \neq k, j \neq l, k \neq l}} V_0(X_i, X_j) V_0(X_k, X_l) \right]\\
        & \qquad= \E_0   \left[ \frac{2}{n^4} \sum_{i \neq j}V_0(X_i, X_j)^2 \right]\\
        & \qquad \leq  \frac{2}{n^2} \E_0  \left[V_0(X_1, X_2)^2 \right]\\
        & \qquad=  \frac{2}{n^2} \E_0  \left[ f(X_1, X_2) - \int f(X_1, y) \sd P_0(y) - \int f(x, X_2) \sd P_0(x) + \int f \sd (P_0\times P_0) \right]^2\\
        & \qquad= \frac{2}{n^2} \E_0  \left[ f(X_1, X_2) - \int f(X_1, y) \sd P_0(y) - \int f(x, X_2) \sd P_0(x)\right]^2 - \frac{2}{n^2} \left[\int f \sd (P_0\times P_0) \right]^2\\
        & \qquad \leq \frac{2}{n^2} \E_0  \left[ f(X_1, X_2) - \int f(X_1, y) \sd P_0(y) - \int f(x, X_2) \sd P_0(x)\right]^2.
    \end{align*}
    For the right hand side of previous display, we note that 
    \begin{align*}
        & \E_0  \left[ f(X_1, X_2) - \int f(X_1, y) \sd P_0(y) - \int f(x, X_2) \sd P_0(x)\right]^2 \\
        & \qquad = \E_0 [f(X_1, X_2)^2] + \E_0 \left[\int f(X_1, y) \sd P_0(y)\right]^2 + \E_0 \left[ \int f(x, X_2) \sd P_0(x)\right]^2 \\
        & \qquad \qquad - 2 \E_0\left[\int f(X_1, y) \sd P_0(y)\right]^2 - 2\E_0 \left[ \int f(x, X_2) \sd P_0(x)\right]^2 \\
        & \qquad \qquad + 2\E_0\left[\int f(X_1, y) \sd P_0(y)\right]\left[ \int f(x, X_2) \sd P_0(x)\right]\\
        & \qquad = \E_0 [f(X_1, X_2)^2] - \E_0\left[\int f(X_1, y) \sd P_0(y)\right]^2 - \E_0\left[ \int f(x, X_2) \sd P_0(x)\right]^2  + 2[(P_0 \times P_0) f]^2\\
        & \qquad \leq \E_0 [f(X_1, X_2)^2],
    \end{align*}
    where the last inequality is because $[(P_0 \times P_0) f]^2 \leq \E_0\left[\int f(X_1, y) \sd P_0(y)\right]^2$ and $[(P_0 \times P_0) f]^2 \leq \E_0\left[\int f(x, X_2) \sd P_0(x)\right]^2$ by the Cauchy–Schwarz inequality. This implies $ \E_0 \left[ n^{-2} \sum_{i \neq j} V_0(X_i, X_j) \right]^2 \lesssim  n^{-2} \norm{f}_{L_2(P_0 \times P_0)}^2$.
\end{proof}

\begin{lemma}\label{lemma: V order 2}
    If $f: \d{R}^d \times \d{R}^d \mapsto \d{R}$ satisfies $f(x, y) = f(y, x)$ and $f(x, x) = \tau_f$ for all $x, y \in \d{R}^d$, then for the empirical bootstrap $\hat{P}_n = \d{P}_n$, $\E_0 \E_W \left[ \int f \sd [(\d{P}_n - P_0) \times (\d{P}_n^* - \d{P}_n)] \right]^2 \lesssim n^{-3} \tau_f^2 + n^{-3} \tau_f \norm{f}_{L_1(P_0\times P_0)} + n^{-2}\norm{f}_{L_2(P_0\times P_0)}^2$.
\end{lemma}
\begin{proof}[\bfseries{Proof of \Cref{lemma: V order 2}}]
    For simplicity, we denote $g_n(y) := \int f(x, y) \sd(\d{P}_n - P_0)(x) $. We then have 
    \begin{align*}
        \int f \sd [(\d{P}_n - P_0) \times (\d{P}_n^* - \d{P}_n)] = (\d{P}_n^* - \d{P}_n) g_n = \frac{1}{n} \sum_{i=1}^n [g_n(X_i^*) - \d{P}_n g_n]
    \end{align*}
    Since $X_1^*, \dotsc, X_n^* \iidsim \d{P}_n$, $\E_W [g_n(X_i^*) - \d{P}_n g_n][g_n(X_j^*) - \d{P}_n g_n] = 0$ for any $i \neq j$, so
    \begin{align*}
        \E_W \left[ \frac{1}{n} \sum_{i=1}^n [g_n(X_i^*) - \d{P}_n g_n]\right]^2 & = \E_W \left[ \frac{1}{n^2} \sum_{i,j}^n [g_n(X_i^*) - \d{P}_n g_n][g_n(X_j^*) - \d{P}_n g_n]\right]\\
        & =  \E_W \left[ \frac{1}{n^2} \sum_{i= j}^n [g_n(X_i^*) - \d{P}_n g_n][g_n(X_j^*) - \d{P}_n g_n]\right] \\
        &= \frac{1}{n} \E_W [g_n(X_i^*) - \d{P}_n g_n]^2 \\
        & = \frac{1}{n} \left[ \d{P}_n g_n^2 - (\d{P}_n g_n)^2 \right]\\
        & \leq \frac{1}{n} \d{P}_n g_n^2.
    \end{align*}
    We then have
    \begin{align}
    \begin{split}\label{eq: auto cross}
        \d{P}_n g_n^2 & = \frac{1}{n^3} \sum_{i,j,k} f(X_i, X_j)f(X_i, X_k) - \frac{2}{n^2} \sum_{i, j} f(X_i, X_j) \left[\int f(x, X_j) \sd P_0(x)\right] \\
        & \qquad + \frac{1}{n}\sum_{i=1} \left[\int f(x, X_j) \sd P_0(x)\right]^2.
    \end{split}
    \end{align}
    By \Cref{lemma: square decomp}, the expectation of the first term on the right-hand side of \eqref{eq: auto cross} is
    \begin{align*}
        & \E_0 \left[ \frac{1}{n^3} \sum_{i, j, k} f(X_i, X_j)f(X_i, X_k) \right] \\
        & \quad =  \E_0 \left[ \frac{2\tau_f}{n^3} \sum_{i\neq j} f(X_i, X_j) + \frac{1}{n^3} \sum_{i \neq j} f(X_i, X_j)^2 + \frac{\tau_f^2}{n^2}  + \frac{1}{n^3} \sum_{i\neq j, i \neq k, j \neq k} f(X_i, X_j)f(X_i, X_k) \right]\\
        & \quad= \frac{2(n-1)\tau_f}{n^2} \E_0\left[ f(X_1, X_2) \right]+ \frac{n-1}{n^2} \E_0 \left[f(X_1, X_2)^2 \right]  + \frac{\tau_f^2}{n^2} + \frac{(n-1)(n-2)}{n^2}  \E_0 [f(X_1, X_2)f(X_1, X_3)].
    \end{align*}
    The expectation of the second term on the right-hand side of \eqref{eq: auto cross} is
    \begin{align*}
        & \E_0 \left\{ \frac{2}{n^2} \sum_{i, j} f(X_i, X_j) \left[ \int f(x, X_j) \sd P_0(x) \right] \right\}\\
        & \qquad = \E_0 \left\{ \frac{2}{n^2} \sum_{i= j} f(X_i, X_j) \left[ \int f(x, X_j) \sd P_0(x) \right] \right\}  + \E_0 \left\{ \frac{2}{n^2} \sum_{i\neq j} f(X_i, X_j) \left[ \int f(x, X_j) \sd P_0(x) \right] \right\} \\
        & \qquad = \frac{2\tau_f}{n}\E_0\left[ f(X_1, X_2) \right]+ \frac{2(n-1)}{n} \E_0 \left[ f(X_1, X_2) f(X_1, X_3) \right].
    \end{align*}
    The expectation of the third term on the right-hand side of \eqref{eq: auto cross} is 
    \begin{align*}
        \E_0 \left\{ \frac{1}{n}\sum_{j=1} \left[\int f(x, X_j) \sd P_0(x) \right]^2\right\} = \E_0 \left[ f(X_1, X_2)f(X_1, X_3) \right].
    \end{align*}
    Combining these calculations, we get
    \begin{align*}
        \abs{\frac{1}{n}\E_0 \left( \d{P}_n g_n^2\right)} &= \abs{-\frac{2\tau_f}{n^3} \E_0 \left[f(X_1, X_2)\right] + \frac{n-1}{n^3} \E_0 \left[f(X_1, X_2)^2\right] + \frac{\tau_f^2}{n^3} - \frac{n-2}{n^3} \E_0\left[f(X_1, X_2)f(X_1, X_3)\right]}\\
        & \leq \frac{2\tau_f}{n^3} \E_0 |f(X_1, X_2)| + \frac{1}{n^2} \E_0 \left[f(X_1, X_2)^2\right] + \frac{\tau_f^2}{n^3} + \frac{1}{n^2} \E_0|f(X_1, X_2)f(X_1, X_3)|\\
        & \leq \frac{2\tau_f}{n^3} \E_0 |f(X_1, X_2)| + \frac{2}{n^2} \E_0 \left[f(X_1, X_2)^2\right] + \frac{\tau_f^2}{n^3},
    \end{align*}
    where the last inequality is due to
    \begin{align*}
        \E_0|f(X_1, X_2)f(X_1, X_3)| &= \E_0\left\{ \E_0[|f(X_1, X_2)| 
        \mid X_1] \E_0[|f(X_1, X_3)|\mid X_1]\right\}\\
        &= \E_0\left\{ \left( \E_0[|f(X_1, X_2)| 
        \mid X_1]\right)^2\right\}\\
        & \leq \E_0\left\{ \E_0\left[f(X_1, X_2)^2 \mid X_1\right]\right\}\\
        & = \E_0\left[f(X_1, X_2)^2\right].
    \end{align*}
    The result follows.
\end{proof}

\begin{lemma}\label{lemma: V order 3}
    If $f: \d{R}^d \times \d{R}^d \mapsto \d{R}$ satisfies $f(x, y) = f(y, x)$ and $f(x, x) = \tau_f \in \d{R}$ for all $x, y \in \d{R}^d$, then for the empirical bootstrap $\hat{P}_n = \d{P}_n$,
    \begin{align*}
        & \E_0 \E_W \left\{ \int f \sd [(\d{P}_n^* - \d{P}_n) \times (\d{P}_n^* - \d{P}_n)] - n^{-1} \tau_f\right\}^2 \\
        & \qquad \lesssim n^{-3}\tau_f^2  +  n^{-3}\tau_f\norm{f}_{L_1(P_0 \times P_0)}  + n^{-2}\norm{f}_{L_2(P_0\times P_0)}^2 + n^{-3} \int   \left[ \int f(x, y) \sd P_0(x) \right]^2 P_0(y).
    \end{align*}
\end{lemma}
\begin{proof}[\bfseries{Proof of \Cref{lemma: V order 3}}]
    We define
    \begin{align*}
        V_n(x, y) &:=  \int f \sd \left[(\delta_{x} - \d{P}_n)  \times (\delta_{y} - \d{P}_n)\right] \\
        &= f(x, y) - \frac{1}{n} \sum_{k=1}^n\left[f(x, X_k) + f(y, X_k)\right] + \frac{1}{n^2} \sum_{k,l = 1}^n f(X_k, X_l).
    \end{align*}
    By the law of total expectation, for any $i \neq j$, 
    \begin{align*}
        \E_W \left[V_n(X_i^*, X_j^*) \mid X_j^*\right] & = \E_W \left[f(X_i^*, X_j^*) - \frac{1}{n} \sum_{k=1}^n\left\{f(X_i^*, X_k) + f(X_j^*, X_k)\right\} + \frac{1}{n^2} \sum_{k,l = 1}^n f(X_k, X_l) \mid X_j^*\right]\\
        & = \frac{1}{n}\sum_{i=1}^n f(X_i, X_j^*) - \frac{1}{n^2} \sum_{i, k=1}^n f(X_i, X_k)  - \frac{1}{n} \sum_{k=1}^n f(X_j^*, X_k) + \frac{1}{n^2} \sum_{k,l = 1}^n f(X_k, X_l) \\
        & = 0.
    \end{align*}
    Hence, $\E_W [V_n(X_i^*, X_j^*)] = 0$ for $i \neq j$ and $\E_W [V_n(X_i^*, X_j^*)V_n(X_i^*, X_k^*)] = 0$ for $i \neq j \neq k$.
    
    By definition, we have 
    \begin{align}
    \begin{split}\label{eq: V decomposition 2}
        & \E_0 \E_W \left\{ \int f \sd [(\d{P}_n^* - \d{P}_n) \times (\d{P}_n^* - \d{P}_n)] - \frac{\tau_f}{n} \right\}^2 \\
        & \qquad = \E_0 \E_W \left\{ \frac{1}{n^2} \sum_{i = j} V_n(X_i^*, X_j^*) + \frac{1}{n^2} \sum_{i \neq j} V_n(X_i^*, X_j^*) - \frac{\tau_f}{n} \right\}^2\\
        & \qquad \lesssim  \E_0 \E_W \left\{ \frac{1}{n^2} \sum_{i = j} V_n(X_i^*, X_j^*) - \frac{\tau_f}{n} \right\}^2 + \E_0 \E_W \left\{ \frac{1}{n^2} \sum_{i \neq j} V_n(X_i^*, X_j^*) \right\}^2 \\
    \end{split}
    \end{align}
    For the first term on the right-hand side of \eqref{eq: V decomposition 2}, 
    \begin{align*}
        & \frac{1}{n^2} \sum_{i = j} V_n(X_i^*, X_j^*) - \frac{\tau_f}{n} \\
        &\qquad = \frac{1}{n^2} \sum_{i = 1}^n \int f \sd [(\delta_{X_i^*} - \d{P}_n) \times (\delta_{X_i^*} - \d{P}_n)] - \frac{\tau_f}{n} \\
        &\qquad=  - \frac{2}{n} \int f \sd (\d{P}_n \times \d{P}_n^*)  + \frac{1}{n} \int f \sd (\d{P}_n \times  \d{P}_n) \\
        &\qquad=  - \frac{2}{n} \int f \sd [\d{P}_n \times  (\d{P}_n^* - \d{P}_n)] - \frac{1}{n} \int f \sd (\d{P}_n \times \d{P}_n) \\
        &\qquad=   - \frac{2}{n} \int f \sd [ (\d{P}_n - P_0) \times (\d{P}_n^* - \d{P}_n)] - \frac{2}{n} \int f \sd [P_0 \times (\d{P}_n^* - \d{P}_n)] - \frac{1}{n} \int f \sd (\d{P}_n \times \d{P}_n).
    \end{align*}
    Hence,  
    \begin{align}
    \begin{split}\label{eq: V decomposition 3}
        &\E_0 \E_W \left\{ \frac{1}{n^2} \sum_{i = j} V_n(X_i^*, X_j^*) - \frac{\tau_f}{n} \right\}^2\\
        & \qquad \lesssim \E_0 \E_W \left\{ \frac{1}{n} \int f \sd [ (\d{P}_n - P_0) \times (\d{P}_n^* - \d{P}_n)]\right\}^2 + \E_0 \E_W \left\{\frac{1}{n} \int f \sd [P_0 \times (\d{P}_n^* - \d{P}_n)] \right\}^2 \\
        & \qquad \qquad + \E_0  \left\{\frac{1}{n} \int f \sd (\d{P}_n \times \d{P}_n)\right\}^2.
    \end{split}
    \end{align}
    By \Cref{lemma: V order 2}, the first term on the right-hand side of \eqref{eq: V decomposition 3} is bounded up to a constant by $n^{-5}\tau_f^2 + n^{-5} \tau_f \norm{f}_{L_1(P_0\times P_0)} + n^{-4}\norm{f}_{L_2(P_0\times P_0)}^2$. For the second term on the right-hand side of \eqref{eq: V decomposition 3},
    \begin{align*}
        \E_0 \E_W \left\{\frac{1}{n} \int f \sd [P_0 \times (\d{P}_n^* - \d{P}_n)] \right\}^2 &= \frac{1}{n^3} \E_0 \left[ \mathrm{Var}_W \left( \int f(x, X_1^*) \sd P_0(x) \right) \right]\\
        & \leq \frac{1}{n^3} \E_0 \left[ \frac{1}{n} \sum_{i=1}^n \left( \int f(x, X_i) \sd P_0(x) \right)^2 \right]\\
        & \leq \frac{1}{n^3} \int   \left( \int f(x, y) \sd P_0(x) \right)^2 \sd P_0(y).
    \end{align*}
    For the last term on the right-hand side of \eqref{eq: V decomposition 3}, we have 
    \begin{align*}
        \E_0 \left\{\frac{1}{n} \int f \sd (\d{P}_n \times \d{P}_n) \right\}^2 & =  \frac{1}{n^2} \E_0 \left\{ \frac{1}{n^2}\sum_{i=j}f(X_i, X_j) + \frac{1}{n^2}\sum_{i\neq j}f(X_i, X_j) \right\}^2\\
        & =  \frac{1}{n^2} \E_0 \left\{ \frac{\tau_f}{n} + \frac{1}{n^2}\sum_{i\neq j}f(X_i, X_j) \right\}^2\\
        & = \frac{1}{n^2} \left\{ \frac{\tau_f^2}{n^2} + \frac{2 \tau_f}{n^3}\E_0 \sum_{i\neq j}f(X_i, X_j) + \frac{1}{n^4} \E_0 \left[ \sum_{i\neq j}f(X_i, X_j)\right]^2 \right\}\\
        & \lesssim \frac{\tau_f^2}{n^4}  + \frac{2 \tau_f}{n^3}\E_0 f(X_1, X_2) + \frac{1}{n^6}\E_0 \left[ \sum_{i\neq j}f(X_i, X_j)\right]^2\\
        & \leq \frac{\tau_f^2}{n^4}  + \frac{2 \tau_f}{n^3}\norm{f}_{L_1(P_0 \times P_0)} + \frac{1}{n^6}\E_0 \left[ \sum_{i\neq j}f(X_i, X_j)\right]^2.
    \end{align*}
    By \Cref{lemma: square decomp}, we note that 
    \begin{align*}
        \E_0 \left[ \sum_{i \neq j} f(X_i, X_j) \right]^2 &= \E_0 \left[ 4 \sum_{i \neq j, j \neq k, k \neq i}f(X_i, X_j) f(X_i, X_k) + 2 \sum_{i \neq j}f(X_i, X_j)^2 \right.\\
        & \qquad \qquad \left.+ \sum_{\substack{i\neq j, i\neq k, i\neq l\\ j \neq k, j \neq l, k \neq l}} f(X_i, X_j) f(X_k, X_l) \right]\\
        & \lesssim n^3 \E_0 f(X_1, X_2)f(X_1, X_3) + n^2 \E_0 f(X_1, X_2)^2 + n^4 [\E_0 f(X_1, X_2)]^2\\
        & \leq n^3 \int   \left( \int f(x, y) \sd P_0(x) \right)^2 \sd P_0(y) + n^2\norm{f}_{L_2(P_0 \times P_0)}^2 + n^4 \norm{f}_{L_1(P_0 \times P_0)}^2.
    \end{align*}
    Combining these calculations, we get
    \begin{align*}
        \E_0 \left\{\frac{1}{n} \int f \sd (\d{P}_n \times \d{P}_n) \right\}^2 &\lesssim \frac{\tau_f^2}{n^4}  + \frac{\tau_f}{n^3}\norm{f}_{L_1(P_0 \times P_0)} + \frac{1}{n^2}\norm{f}_{L_1(P_0 \times P_0)}^2 + \frac{1}{n^4} \norm{f}_{L_2(P_0 \times P_0)}^2\\
        & \qquad \qquad + \frac{1}{n^3} \int   \left( \int f(x, y) \sd P_0(x) \right)^2 \sd P_0(y).
    \end{align*}
    Therefore, 
    \begin{align*}
        \E_0 \E_W \left\{ \frac{1}{n^2} \sum_{i = j} V_n(X_i^*, X_j^*) - \frac{\tau_f}{n} \right\}^2 & \lesssim \tau_f^2n^{-4}  + n^{-3} \tau_f\norm{f}_{L_1(P_0 \times P_0)} +  n^{-2}\norm{f}_{L_1(P_0 \times P_0)}^2   \\
        & \qquad + n^{-4}\norm{f}_{L_2(P_0\times P_0)}^2 + n^{-3} \int   \left( \int f(x, y) \sd P_0(x) \right)^2 \sd P_0(y).
    \end{align*}
    For the second term on the right-hand side of \eqref{eq: V decomposition 2}, by \Cref{lemma: square decomp} and the properties of $V_n$ derived above,
    \begin{align*}
        &\E_W \left[ \frac{1}{n^2} \sum_{i \neq j} V_n(X_i^*, X_j^*) \right]^2 \\
        &\qquad =   \E_W \left[ \frac{4}{n^4} \sum_{i \neq j, j \neq k, k \neq i}V_n(X_i^*, X_j^*) V_n(X_i^*, X_k^*) + \frac{2}{n^4} \sum_{i \neq j}V_n(X_i^*, X_j^*)^2 \right.\\
        & \qquad\qquad \qquad \left.+ \frac{1}{n^4} \sum_{\substack{i\neq j, i\neq k, i\neq l\\ j \neq k, j \neq l, k \neq l}} V_n(X_i^*, X_j^*) V_n(X_k^*, X_l^*) \right]\\
        & \qquad=   \E_W \left[ \frac{2}{n^4} \sum_{i \neq j}V_n(X_i^*, X_j^*)^2 \right]\\
        & \qquad\leq  \frac{2}{n^2}   \E_W \left[V_n(X_1^*, X_2^*)^2 \right]\\
        & \qquad =  \frac{2}{n^2}   \E_W \left[ f(X_1^*, X_2^*) - \int f(X_1^*, y) \sd \d{P}_n(y) - \int f(x, X_2^*) \sd \d{P}_n(x) + (\d{P}_n\times \d{P}_n) f\right]^2\\
        & \qquad=  \frac{2}{n^2}   \E_W \left[ f(X_1^*, X_2^*) - \int f(X_1^*, y) \sd \d{P}_n(y) -  \int f(x, X_2^*) \sd \d{P}_n(x) \right]^2 - \frac{2}{n^2} \left[(\d{P}_n\times \d{P}_n) f\right]^2\\
        & \qquad\leq  \frac{2}{n^2}   \E_W \left[ f(X_1^*, X_2^*) - \int f(X_1^*, y) \sd \d{P}_n(y) -  \int f(x, X_2^*) \sd \d{P}_n(x) \right]^2.
    \end{align*}
    Note that 
    \begin{align*}
        &\E_W \left[ f(X_1^*, X_2^*) - \int f(X_1^*, y) \sd \d{P}_n(y) -  \int f(x, X_2^*) \sd \d{P}_n(x) \right]^2 \\
        & \qquad = \E_W \left\{ f(X_1^*, X_2^*)^2 + \left[\int f(X_1^*, y) \sd \d{P}_n(y) \right]^2 + \left[\int f(x, X_2^*) \sd \d{P}_n(x) \right]^2  \right.\\
        & \qquad \qquad\qquad \qquad   - 2f(X_1^*, X_2^*)\left[ \int f(X_1^*, y) \sd \d{P}_n(y) \right]  - 2f(X_1^*, X_2^*)\left[\int f(x, X_2^*) \sd \d{P}_n(x)\right] \\
        & \qquad \qquad \qquad \qquad \left.+ 2\left[ \int f(X_1^*, y) \sd \d{P}_n(y)\right]\left[\int f(x, X_2^*) \sd \d{P}_n(x)\right]\right\}\\
        & \qquad = \E_W \left\{ f(X_1^*, X_2^*)^2 - \left[\int f(X_1^*, y) \sd \d{P}_n(y)\right]^2 - \left[\int f(x, X_2^*) \sd \d{P}_n (x) \right]^2 + 2 \left[ \int f \sd (\d{P}_n \times \d{P}_n)\right]^2 \right\}\\
        & \qquad \leq \int f^2 \sd (\d{P}_n \times \d{P}_n),
    \end{align*}
    where the last inequality is because  $\left[n^{-2}\sum_{i,j} f(X_i, X_j)\right]^2 \leq n^{-1}\sum_{j} \left[n^{-1}\sum_{i} f(X_i, X_j)\right]^2$ by Jensen's inequality. Hence,
    \begin{align*}
        \E_0 \E_W \left[ \frac{1}{n^2} \sum_{i \neq j} V_n(X_i^*, X_j^*) \right]^2 &\leq \frac{2}{n^2}  \E_0 \int f^2 \sd (\d{P}_n \times \d{P}_n)\\
        & = \frac{2}{n^2}  \E_0 \left[ \frac{1}{n^2}\sum_{i=j}f(X_i, X_j)^2 + \frac{1}{n^2}\sum_{i\neq j}f(X_i, X_j)^2 \right]\\
        & \lesssim \frac{\tau_f^2}{n^3} + \frac{1}{n^2} \int f^2 \sd (P_0 \times P_0),
    \end{align*}
    and the result follows.
\end{proof}

In the next few results, we assume that $K : \d{R}^d \to \d{R}$ is a kernel function satisfying  $K(-x) = K(x)$ for all $x \in \d{R}^d$, $\int K(x) \sd x = 1$, and $\int K(x)^2 \sd x < \infty$. Notably, $K$ need not be non-negative or have compact support. We also let $h = h_n > 0$ be a sequence of bandwidths. For simplicity, we assume the bandwidth matrix is $h I_d$, though the results can be extended to more general bandwidth matrices. For some results, we require that $K$ is an $m$th order kernel function, by which we mean that  $\int z^{\alpha} K(z) \sd z = 0$ for all $\alpha$ such that $1 \leq |\alpha| \leq m-1$ and $\int |z^{\alpha} K(z)| \sd z < \infty$ for all $\alpha$ such that $|\alpha| = m$. 

As in the main text, we define $K_h(x, y) := h^{-d}K\left(h^{-1}(x-y) \right)$ for $x, y \in \d{R}^d$. We  define the functions $f_1$ and $f_2$ from $\d{R}^d$ to $\d{R}$, each depending on $K$ and $h$, as $f_1 : (x,y) \mapsto K_h(x, y)$ and $f_2 : (x,y) \mapsto \int K_h(x, z)K_h(y, z) \sd z$. 

\begin{cor}\label{cor: auto empirical}
   If $P_0$ possesses uniformly bounded Lebesgue density function $\eta_0$, then $\int f \sd [(\d{P}_n - P_0) \times (\d{P}_n^* - \d{P}_n)] = o_{\prob_W^*}(n^{-1/2})$ for each $f\in \{f_1, f_2\}$. If in addition $nh^d \longrightarrow \infty$, then $\int f  \sd [(\d{P}_n - P_0) \times (\d{P}_n - P_0)] = n^{-1}\tau_f + o_{\prob_0^*}(n^{-1/2})$ and $\int f  \sd [(\d{P}_n^* - \d{P}_n) \times (\d{P}_n^* - \d{P}_n)] = n^{-1}\tau_f + o_{\prob_W^*}(n^{-1/2})$.
\end{cor}
\begin{proof}[\bfseries{Proof of \Cref{cor: auto empirical}}]
    We first show that $\norm{f}_{L_1(P_0 \times P_0)} = O(1)$, $\int \left\{\int f(x, y) \sd P_0(x) \right\}^2 \sd P_0(y)  = O(1)$, and $\norm{f}_{L_2(P_0 \times P_0)} = O(h^{-d/2})$ for each $f \in \{f_1, f_2\}$. Note that 
    \begin{align*}
        \norm{f_1}_{L_1(P_0 \times P_0)} &= \iint \frac{1}{h^{d}} \abs{K\left(\frac{x - y}{h}\right)} \sd P_0(x) \sd P_0(y)\\
        & = \iint \abs{K(s)} \eta_0(sh+y) \sd s \sd P_0(y) \\
        & \lesssim  \int \abs{K(s)} \sd s \leq \|K\|_{L_2(\lambda)} < \infty, \text{ and} \\
        \norm{f_2}_{L_1(P_0 \times P_0)} &= \iint \frac{1}{h^{2d}} \abs{\int K\left(\frac{x - z}{h}\right) K\left(\frac{y - z}{h}\right) \sd z} \sd P_0(x) \sd P_0(y)\\
        & \leq \iint   \abs{ K\left(s\right)K\left(t\right)} \left[ \int \eta_0(z + hs)\eta_0(z + ht) \sd z \right] \sd s \sd t   \\
        & \lesssim  \left[ \int   \abs{ K\left(s\right)} \sd s \right]^2 \leq \|K\|_{L_2(\lambda)}^2 < \infty.
    \end{align*}
    Next,
    \begin{align*}
        \int \left[ \int f_1(x, y) \sd P_0(x) \right]^2 \sd P_0(y) &= \int  \left[ \frac{1}{h^d} \int K\left(\frac{x - y}{h}\right) \sd P_0(x)  \right]^2 \sd P_0(y)\\
        & = \int  \left[ \int K\left(s\right) \eta_0(sh+y) \sd s  \right]^2 \sd P_0(y)\\
        & \lesssim \| K\|_{L_2(\lambda)}^2 < \infty, \text{ and} \\
        \int \left[ \int f_2(x, y) \sd P_0(x) \right]^2 \sd P_0(y) &= \int \left[ \frac{1}{h^{2d}} \iint K\left(\frac{x - z}{h}\right) K\left(\frac{y - z}{h}\right) \sd z \sd P_0(x) \right]^2  \sd P_0(y)\\
        & \leq \int \left[ \iint    K\left(s\right)K\left(t\right) \eta_0(y - ht + h s)\sd s \sd t \right]^2 \sd P_0(y)  \\
        & \lesssim \| K\|_{L_2(\lambda)}^2 < \infty.
    \end{align*}
    Similarly,
    \begin{align*}
        \norm{f_1}_{L_2(P_0 \times P_0)}^2 &= \frac{1}{h^{2d}} \iint  K^2\left(\frac{x - y}{h}\right) \sd P_0(x) \sd P_0(y)\\
        & = \frac{1}{h^{d}} \iint  K^2\left(s\right) \eta_0(sh+y) \sd s \sd P_0(y) \\
        & \lesssim \frac{1}{h^{d}}  \| K\|_{L_2(\lambda)}^2, 
    \end{align*}
    and
    \begin{align*}
        \norm{f_2}_{L_2(P_0 \times P_0)}^2 &= \iint \left[ \frac{1}{h^{2d}} \int K \left(\frac{x - z}{h}\right)K \left(\frac{y - z}{h}\right) \sd z \right]^2 \sd P_0(x) \sd P_0(y)\\
        & = \frac{1}{h^{4d}} \iiiint   K \left(\frac{x - z}{h}\right)K \left(\frac{y - z}{h}\right)  K \left(\frac{x - w}{h}\right)K \left(\frac{y - w}{h}\right) \sd z \sd w \sd P_0(x) \sd P_0(y) \\
        & = \frac{1}{h^{2d}} \iiiint   K \left(s\right)K \left(s + \frac{y - x}{h}\right)  K \left(t\right)K \left(t + \frac{y - x}{h}\right) \sd s \sd t \sd P_0(x) \sd P_0(y) \\
        & = \frac{1}{h^{d}} \iiiint   K \left(s\right)K \left(r\right)  K \left(t\right)K \left(t + r - s\right) \eta_0(y - h(r-s))  \sd s \sd t \sd r \sd P_0(y) \\
        &\lesssim  \frac{1}{h^{d}} \iiint   \abs{K \left(s\right)  K \left(r\right)  K \left(t\right) K\left(t + r - s\right)} \sd s \sd t \sd r\\
        & \leq \frac{1}{h^{d}} \iint   \abs{ K \left(s\right) K \left(r\right)  } \left[ \int K^2 \left(t\right) \sd t \right]^{1/2} \left[ \int K^2\left(t + r - s\right) \sd t \right]^{1/2} \sd s \sd r \\
        & = \frac{1}{h^{d}} \left[ \int   \abs{ K \left(s\right)}   \sd s \right]^2 \left[ \int K^2 \left(t\right) \sd t \right]\\
        & \leq \frac{1}{h^{d}} \norm{K}_{L_2(\lambda)}^4,
    \end{align*}
    which implies that $\norm{f}_{L_2(P_0 \times P_0)} = O(h^{-d/2})$ for each $f \in \{f_1, f_2\}$ as claimed.

    We note that $f(x,y) = f(y,x)$ for each $f \in \{f_1, f_2\}$ because $K(-u) = K(u)$ by assumption, and $\tau_{f_1} = h^{-d}K(0)$ and $\tau_{f_2} = h^{-d} \int K^2(u) \sd u$ for all $x \in \d{R}^d$, which are both $O(h^{-d})$. Hence, using the results above, by \Cref{lemma: V order 2}, for each $f \in \{f_1, f_2\}$,
    \begin{align*}
        \E_0 \E_W \left\{\int f \sd [(\d{P}_n - P_0) \times (\d{P}_n^* - \d{P}_n)] \right\}^2 &\lesssim h^{-2d} n^{-3} + h^{-d} n^{-3} \norm{f}_{L_1(P_0\times P_0)} + n^{-2}\norm{f}_{L_2(P_0\times P_0)}^2 \\
        &\lesssim h^{-2d}n^{-3} + h^{-d}n^{-3} + n^{-2}h^{-d} \\
        &\lesssim n^{-1} (nh^d)^{-2} + n^{-1} (nh^d)^{-1},
    \end{align*}
    which is $o(n^{-1})$ if $nh^d \longrightarrow \infty$. This implies that $\int f \sd [(\d{P}_n - P_0) \times (\d{P}_n^* - \d{P}_n)] = o_{\prob_W^*}(n^{-1/2})$ if  $nh^d \longrightarrow \infty$. Since $\int \left[ \int f(x, y) \sd P_0(x) \right]^2 \sd P_0(y) = O(1)$, by \Cref{lemma: V order 1}, we have 
    \begin{align*}
        \E_0 \left\{ \int f \sd [(\d{P}_n - P_0) \times (\d{P}_n - P_0)] - n^{-1} \tau_f \right\}^2 & \lesssim  n^{-3} \int \left[ \int f(x, y) \sd P_0(x) \right]^2 \sd P_0(y)  + n^{-2} \norm{f}_{L_2(P_0 \times P_0)}^2 \\
        & \lesssim n^{-3} + n^{-1} (nh^{d})^{-1},
    \end{align*}
    which is $ o(n^{-1})$ if $nh^d \longrightarrow \infty$. This implies that $\int f \sd [(\d{P}_n - P_0) \times (\d{P}_n - P_0)] = n^{-1} \tau_f + o_{\prob_W^*}(n^{-1/2})$ if $nh^d \longrightarrow \infty$. Since $\int \left[ \int f(x, y) \sd P_0(x) \right]^2 \sd P_0(y) = O(1)$, by \Cref{lemma: V order 3}, we have 
    \begin{align*}
        &\E_0 \E_W \left\{ \int f \sd [(\d{P}_n^* - \d{P}_n) \times (\d{P}_n^* - \d{P}_n)] - n^{-1} \tau_f \right\}^2 \\
        & \qquad \lesssim   n^{-3} \tau_f^2 + n^{-3} \tau_f\norm{f}_{L_1(P_0 \times P_0)} + n^{-2}\norm{f}_{L_2(P_0 \times P_0)}^2 + n^{-3} \int \left[ \int f(x, y) \sd P_0(x) \right]^2 \sd P_0(y)\\
        & \qquad \lesssim  n^{-3}h^{-2d} + n^{-3} h^{-d} + n^{-2}h^{-d} + n^{-3},
    \end{align*}
    which is $o(n^{-1})$ if $nh^d \longrightarrow \infty$. This implies that $\int f \sd [(\d{P}_n^* - \d{P}_n) \times (\d{P}_n^* - \d{P}_n)] = n^{-1} \tau_f + o_{\prob_W^*}(n^{-1/2})$ if $nh^d \longrightarrow \infty$. 
\end{proof}

\begin{lemma}\label{lemma: h^m tool 1}
    Suppose that the density function $\eta_0$ of $P_0$ is $m$-times continuously differentiable with $\int [ (D^{\alpha} \eta_0)( x) ]^2 \sd x < \infty$ for all $\alpha$ such that $|\alpha| = m$, and $K$ is an $m$th order kernel. Then for each $f \in \{f_1, f_2\}$
    \begin{align*}
        \int \left[ \int f(x, y) \sd P_0(y) - \eta_0(x) \right]^2 \sd x = O(h^{2m}).
    \end{align*}
    If in addition $\sup_{x \in \d{R}^d} \left| D^{\alpha} \eta_0(x) \right| < \infty$, then $\sup_{x \in \d{R}^d} \abs{ \int f(x, y) \sd P_0(y) - \eta_0(x)} = O(h^m)$.  
\end{lemma}
\begin{proof}[\bfseries{Proof of \Cref{lemma: h^m tool 1}}]
    Since $\eta_0$ is $m$-times continuously differentiable, for all $u$, a Taylor expansion with the Laplacian representation of the remainder gives
    \begin{align}
        \begin{split}
            &\eta_0(x + u) - \eta_0(x)  \\
            & \qquad = \sum_{1 \leq |\alpha| \leq m-1} \frac{1}{\alpha!}u^{\alpha} (D^{\alpha}\eta_0)(x)  + \sum_{|\alpha|=m} \frac{m}{\alpha!} u^{\alpha} \int_0^1 (1-r)^{m-1} (D^{\alpha} \eta_0)( x + ru ) \sd r.
        \end{split}
    \end{align}
    Hence, for $f=f_1$ and any $x\in \d{R}^d$, we have
    \begin{align*}
        \int f_1(x, y)\sd P_0(y) - \eta_0(x)  &= \int \frac{1}{h^d} K\left( \frac{x-y}{h} \right) \eta_0(y) \sd y - \eta_0(x)\\
        & = \int K( s ) [\eta_0(x + hs) - \eta_0(x)] \sd s\\
        &  =  \sum_{1 \leq |\alpha| \leq m-1} \frac{1}{\alpha!} h^{|\alpha|} \int s^{\alpha} K(s)   \sd s (D^{\alpha}\eta_0)(x) \\
        &  \qquad + \sum_{|\alpha|=m} \frac{m}{\alpha!} h^{|\alpha|} \iint_0^1 s^{\alpha} K(s)    (1-r)^{m-1} (D^{\alpha} \eta_0)( x + rhs ) \sd r  \sd s.
    \end{align*}  
    Since $K$ is an $m$th order kernel function, the first term on the right hand size is zero. Defining $H(x, s, \alpha) := \int_0^1 (1-r)^{m-1} (D^{\alpha} \eta_0)( x + rhs ) \sd r$, we then have 
    \begin{align*}
        &\int \left[ \int f_1(x, y)\sd P_0(y) - \eta_0(x) \right]^2 \sd x \\
        & \qquad = \int \left[ \sum_{|\alpha|=m} \frac{m}{\alpha!} h^{|\alpha|} \int s^{\alpha} K(s)   H(x, s, \alpha)  \sd s \right]^2  \sd x\\
        & \qquad = \sum_{|\alpha|, |\beta|=m} \frac{m^2}{\alpha! \beta!} h^{|\alpha|}h^{|\beta|} \iiint s^{\alpha}t^{\beta} K(s) K(t)  H(x, s, \alpha)H(x, t, \beta)  \sd s \sd t \sd x\\
        & \qquad \leq \sum_{|\alpha|,|\beta|=m} \frac{m^2}{\alpha! \beta!} h^{2m} \iint \abs{s^{\alpha}t^{\beta} K(s) K(t)} \left[\int H(x, s, \alpha)^2 \sd x\right]^{1/2} \left[\int H(x, t, \beta)^2 \sd x \right]^{1/2}  \sd s \sd t.
    \end{align*}
    By the Cauchy–Schwarz inequality, we have
    \begin{align*}
        \int H(x, s, \alpha)^2 \sd x &= \int \left[ \int_0^1 (1-r)^{m-1} (D^{\alpha} \eta_0)( x + rhs ) \sd r \right]^2 \sd x \\
        & \leq \int \left[ \int_0^1 (1-r)^{2(m-1)}  \sd r \right] \left[\int_0^1 (D^{\alpha} \eta_0)( x + rhs )^2 \sd r \right] \sd x\\
        & \leq  \iint_0^1 (D^{\alpha} \eta_0)( x + rhs )^2 \sd r  \sd x\\
        & = \int [(D^{\alpha} \eta_0)( x )]^2  \sd x,
    \end{align*}
    which is finite by assumption. Hence,
    \begin{align*}
        \int \left[ \int f_1(x, y)\sd P_0(y) - \eta_0(x) \right]^2 \sd x & \lesssim h^{2m} \sum_{|\alpha|, |\beta|=m} \iint \abs{s^{\alpha}t^{\beta} K(s) K(t)}  \sd s \sd t \\
        & = h^{2m}  \left[ \sum_{|\alpha|=m} \int \abs{s^{\alpha} K(s)}  \sd s \right]^2  = O(h^{2m}).
    \end{align*}
    If $D^{\alpha} \eta_0$ is uniformly bounded, then
    \begin{align*}
        & \abs{\int f_1(x, y)\sd P_0(y) - \eta_0(x)}\\
        & \qquad = \abs{\sum_{|\alpha|=m} \frac{m}{\alpha!} h^{|\alpha|} \iint_0^1 s^{\alpha} K(s)    (1-r)^{m-1} (D^{\alpha} \eta_0)( x + rhs ) \sd r  \sd s}\\
        & \qquad \leq \left\|D^{\alpha} \eta_0 \right\|_\infty \sum_{|\alpha|=m} \frac{mh^m}{\alpha!} \int   |s^{\alpha} K(s)|  \left[\int_0^1 (1-r)^{m-1} \sd r \right] \sd s\\
        & \qquad =\left\|D^{\alpha} \eta_0 \right\|_\infty \sum_{|\alpha|=m} \frac{h^m}{\alpha!} \int |s^{\alpha} K(s)| \sd s,
    \end{align*}
    which is independent of $x$ and order of $h^m$. 
    
    For $f=f_2$, we have
     \begin{align*}
        & \int f_2(x, y)\sd P_0(y) - \eta_0(x)  \\
        &\qquad = \iint \frac{1}{h^{2d}} K\left( \frac{x-z}{h} \right)K\left( \frac{y-z}{h} \right)  \sd z \sd P_0(y) - \eta_0(x)\\
        &\qquad= \iint \frac{1}{h^{d}} K\left( s \right)K\left( s + \frac{y-x}{h} \right)  \sd s \sd P_0(y) - \eta_0(x)\\
        &\qquad= \iint K( s )K( t ) \eta_0(x + h(t-s)) \sd s \sd t - \eta_0(x)\\
        & \qquad = \iint K( s )K( t ) [\eta_0(x + h(t-s)) - \eta_0(x)] \sd s \sd t\\
        & \qquad =  \sum_{1 \leq |\alpha| \leq m-1} \frac{1}{\alpha!} h^{|\alpha|} \iint (t-s)^{\alpha} K(s)K(t)   \sd s \sd t (D^{\alpha}\eta_0)(x) \\
        & \qquad \qquad + \sum_{|\alpha|=m} \frac{m}{\alpha!} h^{|\alpha|} \iiint_0^1 (t-s)^{\alpha} K(s)K(t)    (1-r)^{m-1} (D^{\alpha} \eta_0)( x + rh(t-s) ) \sd r  \sd s \sd t,
    \end{align*} 
    where  the first term on the right hand size is zero because $K$ is an $m$th order kernel function. Hence, we have 
    \begin{align*}
        &\int \left[ \int f_2(x, y)\sd P_0(y) - \eta_0(x) \right]^2 \sd x \\
        & \qquad = \int \left[ \sum_{|\alpha|=m} \frac{m}{\alpha!} h^{|\alpha|} \iint (t-s)^{\alpha} K(s)K(t) H(x, s-t, \alpha)  \sd s \sd t \right]^2 \sd x\\
        & \qquad = \sum_{|\alpha|, |\beta| = m} \frac{m^2}{\alpha!\beta!} h^{|\alpha|}h^{|\beta|} \iiiint (t-s)^{\alpha}(t'-s')^{\beta} K(s)K(s')K(t)K(t') \\
        & \qquad \qquad\qquad \qquad \times \left[\int H(x, s-t, \alpha)H(x, s'-t', \beta) \sd x\right]  \sd s \sd t \sd s' \sd t'\\
        & \qquad \lesssim \sum_{|\alpha|, |\beta| = m}  h^{2m} \iiiint \abs{(t-s)^{\alpha}(t'-s')^{\beta} K(s)K(s')K(t)K(t')}\\
        & \qquad \qquad\qquad \qquad \times\left[\int H(x, s-t, \alpha)^2\sd x\right]^{1/2} \left[\int H(x, s'-t'
        , \beta)^2 \sd x\right]^{1/2}  \sd s \sd t \sd s' \sd t' \\
        &\qquad \lesssim h^{2m} \left[ \sum_{|\alpha|=m} \iint \abs{(t-s)^{\alpha} K(s)K(t)} \sd s \sd t \right]^2\\
        & \qquad \lesssim h^{2m} \left[ \sum_{|\alpha|=m} \iint  \abs{\prod_{i=1}^d \left(\sum_{k_i=0}^{\alpha_i}t_i^{k_i}s_i^{\alpha_i-k_i}\right) K(s)K(t)} \sd s \sd t \right]^2\\
        & \qquad \leq h^{2m} \left[ \sum_{|\alpha|=m} \iint  \abs{\sum_{\substack{0\leq \beta \leq \alpha \\ 0\leq \gamma \leq \alpha}} t^{\beta}s^{\gamma} K(s)K(t)} \sd s \sd t \right]^2\\
        & \qquad \leq h^{2m} \left[ \sum_{\substack{0\leq |\beta| \leq m \\ 0 \leq |\gamma| \leq m}} \iint  \abs{ t^{\beta}s^{\gamma} K(s)K(t)} \sd s \sd t \right]^2 \\
        & \qquad = h^{2m} \left[ \sum_{\substack{0\leq|\beta| \leq m}} \int  \abs{ t^{\beta} K(t)} \sd t \right]^4,
    \end{align*}
    which is $O(h^{2m})$ because $\sum_{\substack{|\beta| = m}} \int  \abs{ t^{\beta} K(t)} \sd t < \infty$ by assumption, and for any $0 \leq \beta < m$ there exists $\alpha$ with $|\alpha| = m$ such that
    \begin{align*}
        \int  \abs{ t^{\beta} K(t)} \sd t &= \int_{[-1, 1]^d}  \abs{ t^{\beta} K(t)} \sd t + \int_{\d{R}^d \backslash [-1, 1]^d}  \abs{ t^{\beta} K(t)} \sd t \\
        & \leq \int  \abs{K(t)} \sd t + \int  \abs{ t^{\alpha} K(t)} \sd t,
    \end{align*}
    which is finite by assumption.
    
    If $D^{\alpha} \eta_0$ is uniformly bounded, then
    \begin{align*}
        &\abs{\int f_2(x, y)\sd P_0(y) - \eta_0(x)} \\
        & \qquad \leq \left\| D^\alpha \eta_0 \right\|_\infty \sum_{|\alpha|=m} \frac{m}{\alpha!} h^{|\alpha|} \iint \abs{(t-s)^{\alpha} K(s)K(t)} \sd s \sd t \int_0^1    (1-r)^{m-1}  \sd r  \\
        & \qquad = \left\| D^\alpha \eta_0 \right\|_\infty\sum_{|\alpha|=m} \frac{h^m}{\alpha!} \iint \abs{(t-s)^{\alpha} K(s)K(t)} \sd s \sd t.
    \end{align*}
    We showed the integral in the final expression is finite above, so the expression is independent of $x$ and order of $h^m$. 
\end{proof}

We now consider kernel density estimators $\eta_n$ and $\eta_n^*$ with common bandwidth $h$ and kernel $K$, i.e.\ $\eta_n(x) = n^{-1}\sum_{i=1}^n K_h(x, X_i)$ and $\eta_n^*(x) = n^{-1} \sum_{i=1}^n K_h(x, X_i^*)$. 

\begin{lemma}\label{thm: auto empirical process}
    Suppose that the density function $\eta_0$ of $P_0$ is uniformly bounded and $m$-times continuously differentiable with $\int [ (D^{\alpha} \eta_0)( x) ]^2 \sd x < \infty$ for all $\alpha$ such that $|\alpha| = m$ and $K$ is an $m$-th order kernel function. Then $(\d{P}_n - P_0)(\eta_n^* - \eta_n) = o_{\prob_W^*}(n^{-1/2})$. If in addition $nh^d \longrightarrow \infty$ holds, then $(\d{P}_n^* - \d{P}_n)(\eta_n^* - \eta_0) - (\d{P}_n - P_0)(\eta_n - \eta_0) = o_{\prob_W^*}(n^{-1/2})$.
\end{lemma}
\begin{proof}[\bfseries{Proof of \Cref{thm: auto empirical process}}]
    To show the first statement, we note that 
    \begin{align*}
        (\d{P}_n - P_0)(\eta_n^* - \eta_n) & = \int  K_h  \sd [(\d{P}_n^* - \d{P}_n ) \times (\d{P}_n - P_0)],
    \end{align*}
    which is $o_{\prob_W^*}(n^{-1/2})$ by \Cref{cor: auto empirical} with $f = f_1$.

    To show the second statement, by adding and subtracting terms and defining $\eta_{0,h}(x) = \int K_{h}(x, y) \sd P_0(y)$, we have 
    \begin{align*}
        & (\d{P}_n^* - \d{P}_n)(\eta_n^* - \eta_0) - (\d{P}_n - P_0)(\eta_n - \eta_0) \\
        &\qquad = (\d{P}_n^* - \d{P}_n)(\eta_n^* - \eta_n) + (\d{P}_n^* - \d{P}_n)(\eta_n - \eta_{0,h}) + (\d{P}_n^* - \d{P}_n)(\eta_{0,h} - \eta_0) - (\d{P}_n - P_0)(\eta_n - \eta_{0,h})\\
        &\qquad\qquad - (\d{P}_n - P_0)(\eta_{0,h} - \eta_0) \\
        &\qquad = \int K_h  \sd [(\d{P}_n^* - \d{P}_n) \times (\d{P}_n^* - \d{P}_n)] + \int K_h \sd [(\d{P}_n - P_0) \times (\d{P}_n^* - \d{P}_n)]  \\
        &\qquad \qquad +\int [K_h(x, y) - \eta_0(x)] \sd (\d{P}_n^* - \d{P}_n)(x) \sd P_0(y)]  - \int K_h \sd [(\d{P}_n - P_0) \times (\d{P}_n - P_0)] \\
        &\qquad \qquad  -\int [K_h(x, y) - \eta_0(x)] \sd (\d{P}_n - P_0)(x) \sd P_0(y)].
    \end{align*}
    By \Cref{cor: auto empirical} with $f = f_1$, 
    \[ \int K_h \sd [(\d{P}_n - P_0) \times (\d{P}_n^* - \d{P}_n)] = o_{\prob_W^*}(n^{-1/2}),\]
    and if $nh^d \longrightarrow 0$, then 
    \[ 
        \int K_h  \sd [(\d{P}_n^* - \d{P}_n) \times (\d{P}_n^* - \d{P}_n)] - \int K_h \sd [(\d{P}_n - P_0) \times (\d{P}_n - P_0)] = o_{\prob_W^*}(n^{-1/2}). 
    \]
    Finally, we write 
    \begin{align*}
        \int [K_h(x, y) - \eta_0(x)] \sd (\d{P}_n - P_0)(x) \sd P_0(y)] &= (\d{P}_n - P_0)g_h, \text{ and} \\
        \int [K_h(x, y) - \eta_0(x)] \sd (\d{P}_n^* - \d{P}_n)(x) \sd P_0(y)] &= (\d{P}_n^* - \d{P}_n) g_h
    \end{align*}
    for $g_h(x) := \int [K_h(x, y) - \eta_0(x)] \sd P_0(y)$. As in the proofs of \Cref{lemma: V order 2} and \Cref{lemma: V order 3}, we have
    \begin{align*}
        \E_0 \left[ (\d{P}_n - P_0)g_h\right]^2 &\leq n^{-1} \| g_h \|_{L_2(P_0)}^2, \text{ and}\\
        \E_0 \E_W \left[ (\d{P}_n^* - \d{P}_n)g_h\right]^2 &\leq n^{-1} \E_0\| g_h \|_{L_2(\d{P}_n)}^2 = n^{-1} \| g_h \|_{L_2(P_0)}^2.
    \end{align*}
    Since $\int g_h(x)^2 \sd x = O(h^{2m})$ by \Cref{lemma: h^m tool 1}, we have $\| g_h \|_{L_2(P_0)}^2 \longrightarrow 0$ as long as $h \longrightarrow 0$ and $P_0$ possesses uniformly bounded density. This implies that $\int [K_h(x, y) - \eta_0(x)] \sd (\d{P}_n - P_0)(x) \sd P_0(y)] = o_{\prob_0^*}(n^{-1/2})$ and $\int [K_h(x, y) - \eta_0(x)] \sd (\d{P}_n^* - \d{P}_n)(x) \sd P_0(y)] = o_{\prob_W^*}(n^{-1/2})$.
\end{proof}

\begin{lemma}\label{thm: auto mise}
    Suppose that the density function $\eta_0$ of $P_0$ is uniformly bounded and $m$-times continuously differentiable with $\int [ (D^{\alpha} \eta_0)( x) ]^2 \sd x < \infty$ for all $\alpha$ such that $|\alpha| = m$ and $K$ is an $m$-th order kernel function. If $nh^d\longrightarrow \infty$ and $nh^{4m} \longrightarrow 0$, then $2\int[\eta_n(x) - \eta_0(x)]^2 \sd x - \int [\eta_n^*(x) - \eta_0(x)]^2 \sd x = o_{\prob_W^*}(n^{-1/2})$ holds for the empirical bootstrap.
\end{lemma}
\begin{proof}[\bfseries{Proof of \Cref{thm: auto mise}}]
    Recall that $f_1(x, y) := K_h\left(x, y\right)$ and $f_2(x, y) := \int K_h(x, z) K_h(y, z) \sd z$. We denote $\eta_{0, h}(x) :=  \int  K_h\left(x, y\right) \sd P_0(y)$ for simplicity. We then note that
    \begin{align*}
        &\int \eta_n^2 = \int f_2 \sd (\d{P}_n \times \d{P}_n), \  \int \eta_n^{*2} = \int f_2 \sd (\d{P}_n^* \times \d{P}_n^*), \ \int \eta_{0,h}^2 = \int f_2  \sd (P_0 \times P_0) \\
        &\int \eta_n^*\eta_n = \int f_2 \sd (\d{P}_n^* \times \d{P}_n), \ \int \eta_n^*\eta_{0,h} = \int f_2 \sd (\d{P}_n^* \times P_0), \ \int \eta_n\eta_{0,h} = \int f_2 \sd (\d{P}_n \times P_0).
    \end{align*}
    Hence, by adding and subtracting term, we have
    \begin{align*}
        &\int (\eta_n^* - \eta_0)^2  - 2\int(\eta_n - \eta_0)^2 \\
        & \qquad = \int (\eta_n^* - \eta_n + \eta_n - \eta_{0, h} + \eta_{0, h} - \eta_0)^2   - 2\int(\eta_n - \eta_{0, h}+ \eta_{0, h} - \eta_0)^2  \\
        & \qquad = \int \left(\eta_n^* - \eta_n\right)^2 - \int \left(\eta_n - \eta_{0, h}\right)^2 - \int \left(\eta_{0, h} - \eta_0\right)^2+ 2\int \left(\eta_n^* - \eta_n\right)\left(\eta_n - \eta_{0, h}\right)\\
        & \qquad \qquad   + 2\int \left(\eta_n^* - \eta_n\right)\left(\eta_{0, h} - \eta_0\right) - 2\int \left(\eta_n - \eta_{0, h}\right)\left(\eta_{0, h} - \eta_0\right)  \\
        & \qquad = \int f_2 \sd [(\d{P}_n^* - \d{P}_n) \times (\d{P}_n^* - \d{P}_n)] - \int f_2 \sd [(\d{P}_n - P_0) \times (\d{P}_n - P_0)] - \int \left(\eta_{0, h} - \eta_0\right)^2  \\
        & \qquad \qquad  + 2\int f_2 \sd [(\d{P}_n^* - \d{P}_n) \times (\d{P}_n - P_0)] + 2 \int [f_2(x, y) - \eta_{0, h}(x)] \sd (\d{P}_n^* - \d{P}_n)(x) \sd P_0(y) \\
        & \qquad \qquad  - 2\int [f_2(x, y) - \eta_{0, h}(x)]\sd (\d{P}_n - P_0)(x) \sd P_0(y) .
    \end{align*}
    By \Cref{cor: auto empirical}, $2\int f_2 \sd [(\d{P}_n^* - \d{P}_n) \times (\d{P}_n - P_0)] = o_{\prob_W^*}(n^{-1/2})$, and if $nh^d \longrightarrow 0$, then
    \[
        \int f_2 \sd [(\d{P}_n^* - \d{P}_n) \times (\d{P}_n^* - \d{P}_n)] - \int f_2 \sd [(\d{P}_n - P_0) \times (\d{P}_n - P_0)] = o_{\prob_W^*}(n^{-1/2}).
    \]
    By \Cref{lemma: h^m tool 1}, the third term on the right is $O(h^{2m})$, which is $o(n^{-1/2})$ if $nh^{4m} \longrightarrow 0$. Finally, we write 
    \begin{align*}
        \int [f_2(x, y) - \eta_{0, h}(x)] \sd (\d{P}_n - P_0)(x) \sd P_0(y) &= (\d{P}_n - P_0)g_h, \text{ and} \\
        \int [f_2(x, y) - \eta_{0, h}(x)] \sd (\d{P}_n^* - \d{P}_n)(x) \sd P_0(y) &= (\d{P}_n^* - \d{P}_n) g_h
    \end{align*}
    for $g_h(x) := \int [f_2(x, y) - \eta_{0, h}(x)] \sd P_0(y) = \int [f_2(x, y) - f_1(x, y)] \sd P_0(y)$. As in the proofs of \Cref{lemma: V order 2} and \Cref{lemma: V order 3}, we have
    \begin{align*}
        \E_0 \left[ (\d{P}_n - P_0)g_h\right]^2 &\leq n^{-1} \| g_h \|_{L_2(P_0)}^2, \text{ and}\\
        \E_0 \E_W \left[ (\d{P}_n^* - \d{P}_n)g_h\right]^2 &\leq n^{-1} \E_0\| g_h \|_{L_2(\d{P}_n)}^2 = n^{-1} \| g_h \|_{L_2(P_0)}^2.
    \end{align*}
    By \Cref{lemma: h^m tool 1}, we have
    \begin{align*}
        &\int \left\{\int [f_2(x, y) - f_1(x, y)] \sd P_0(y)\right\}^2 \sd x \\
        & \qquad \leq 2 \int \left\{\int f_2(x, y)  \sd P_0(y) - \eta_0(x) \right\}^2 \sd x +  2\int \left\{\int f_1(x, y)  \sd P_0(y) - \eta_0(x) \right\}^2 \sd x = O(h^{2m}).
    \end{align*}
    This implies $\| g_h \|_{L_2(P_0)}^2  = O(h^{2m}) = o(n^{-1/2})$ since $nh^{4m} \longrightarrow 0$ and $P_0$ possesses uniformly bounded density. Hence,  $\int [f_2(x, y) - \eta_{0, h}(x)] \sd (\d{P}_n - P_0)(x) \sd P_0(y) = o_{\prob_0^*}(n^{-3/4})$ and $\int [f_2(x, y) - \eta_{0, h}(x)] \sd (\d{P}_n^* - \d{P}_n)(x) \sd P_0(y) = o_{\prob_W^*}(n^{-3/4})$. 
\end{proof}

\begin{lemma}\label{thm: auto asym bias}
    Suppose that the density function $\eta_0$ of $P_0$ is uniformly bounded and $m$-times continuously differentiable with $\int [ (D^{\alpha} \eta_0)( x) ]^2 \sd x < \infty$ for all $\alpha$ such that $|\alpha| = m$ and $K$ is an $m$-th order kernel function.  If $nh^d \longrightarrow \infty$ and $nh^{2m} \longrightarrow 0$, then $2\d{P}_n\phi_n - \d{P}_n^*\phi_n^* = o_{\prob_W}(n^{-1/2})$ holds for the empirical bootstrap.
\end{lemma}
\begin{proof}[\bfseries{Proof of \Cref{thm: auto asym bias}}]
    Recall that $f_1(x, y) := K_h\left(x, y\right)$ and $f_2(x, y) := \int K_h(x, z) K_h(y, z) \sd z$. We define $f_3 := f_1 - f_2$. We then have 
    \begin{align*}
        \d{P}_n \phi_n = 2\int f_3 \sd (\d{P}_n \times \d{P}_n), \text{  and } \quad \d{P}_n^* \phi_n^* = 2\int f_3 \sd (\d{P}_n^* \times \d{P}_n^*).
    \end{align*}
    Hence, by adding and subtracting terms,
    \begin{align*}
        &\left(\d{P}_n^*\phi_n^* - 2\d{P}_n \phi_n \right) / 2 \\
        & \qquad = \int f_3 \sd [(\d{P}_n^* - \d{P}_n + \d{P}_n - P_0 + P_0) \times (\d{P}_n^* - \d{P}_n + \d{P}_n - P_0 + P_0)] \\
        & \qquad \qquad - 2 \int f_3 \sd [(\d{P}_n - P_0 + P_0) \times (\d{P}_n - P_0 + P_0)]\\
        & \qquad = \int f_3 \sd [(\d{P}_n^* - \d{P}_n) \times (\d{P}_n^* - \d{P}_n)] - \int f_3 \sd [(\d{P}_n - P_0) \times (\d{P}_n - P_0)] - \int f_3 \sd (P_0 \times P_0)\\
        & \qquad \qquad + 2 \int f_3 \sd [(\d{P}_n^* - \d{P}_n) \times (\d{P}_n - P_0)] + 2 \int f_3 \sd [(\d{P}_n^* - \d{P}_n) \times P_0] \\
        & \qquad \qquad - 2 \int f_3 \sd [(\d{P}_n - P_0) \times P_0].
    \end{align*}
    By \Cref{cor: auto empirical}, 
    \[
        \int f_3 \sd [(\d{P}_n^* - \d{P}_n) \times (\d{P}_n - P_0)] = \int f_1  \sd [(\d{P}_n^* - \d{P}_n) \times (\d{P}_n - P_0)] -  \int f_2  \sd [(\d{P}_n^* - \d{P}_n) \times (\d{P}_n - P_0)] = o_{\prob_W^*}(n^{-1/2})
    \]
    and if $nh^d \longrightarrow 0$, then 
    \begin{align*}
        &\int f_3 \sd [(\d{P}_n^* - \d{P}_n) \times (\d{P}_n^* - \d{P}_n)] - \int f_3 \sd [(\d{P}_n - P_0) \times (\d{P}_n - P_0)]\\
        &\qquad= \left\{\int f_1  \sd [(\d{P}_n^* - \d{P}_n) \times (\d{P}_n^* - \d{P}_n)] - \int f_1 \sd [(\d{P}_n - P_0) \times (\d{P}_n - P_0)] \right\} \\
        &\qquad\qquad -\left\{\int f_2  \sd [(\d{P}_n^* - \d{P}_n) \times (\d{P}_n^* - \d{P}_n)] - \int f_2 \sd [(\d{P}_n - P_0) \times (\d{P}_n - P_0)] \right\} \\
        &\qquad = o_{\prob_W^*}(n^{-1/2}).
    \end{align*}
    Finally, we define $g_h(x) := \int f_3(x, y) \sd P_0(y)$. As in the proofs of \Cref{lemma: V order 2} and \Cref{lemma: V order 3}, we have
    \begin{align*}
        \E_0 \left[  (\d{P}_n - P_0) g_h \right]^2 &\leq n^{-1} \norm{g_h}_{L_2(P_0)}^2, \text{ and}\\
        \E_0 \E_W \left[ (\d{P}_n^* - \d{P}_n)g_h\right]^2 &\leq n^{-1} \E_0\| g_h \|_{L_2(\d{P}_n)}^2 = n^{-1} \| g_h \|_{L_2(P_0)}^2.
    \end{align*}
    By \Cref{lemma: h^m tool 1}, we have
    \begin{align}
    \begin{split}\label{eq: auto asym bias last}
        \int g_h(x)^2 \sd x  &= \int \left\{\int [f_1(x, y) - f_2(x, y)] \sd P_0(y)\right\}^2 \sd x \\
        &  \leq 2 \int \left\{\int f_1(x, y)  \sd P_0(y) - \eta_0(x) \right\}^2 \sd x +  2\int \left\{\int f_2(x, y)  \sd P_0(y) - \eta_0(x) \right\}^2 \sd x,
    \end{split}
    \end{align}
    which is $O(h^{2m})$. This implies that $(\d{P}_n - P_0) g_h$ and $(\d{P}_n^* - \d{P}_n)g_h$ are  $o_{\prob_0^*}(n^{-1})$ and  $o_{\prob_W^*}(n^{-1})$, respectively, since $nh^{2m} \longrightarrow 0$ and $P_0$ possesses uniformly bounded density. Finally, \eqref{eq: auto asym bias last} also implies that $\int f_3 \sd (P_0 \times P_0) = O(h^m) = o(n^{-1/2})$ since $nh^{2m} \longrightarrow 0$.
\end{proof}

\section{Lemmas supporting the proof of Proposition~\ref{prop: auto bias correction smooth}}\label{sec: smooth auto bias lemmas}

For any $f: \d{R}^d \times \d{R}^d \mapsto \d{R}$, we define 
\[
    G[f](x, y) := \iint f(s, t) K_h(s, x) K_h(t, y) \sd s \sd t = \iint f(x+hu, y + hv)K(u) K(v) \sd u \sd v.
\]
If $f$ is symmetric, which is the case for $f \in \{f_1, f_2\}$, then $G[f]$ is symmetric as well. If $f(x+c, y+c) = f(x, y)$ holds for any $c \in \d{R}^d$, which is also the case for $f \in \{f_1, f_2\}$, then
\begin{align*}
    G[f](x, x) = \iint f(x+ hu, x + hv) K(u)K(v) \sd u \sd v = \iint f(hu, hv) K(u)K(v) \sd u \sd v,
\end{align*}
which does not depend on $x$. For simplicity, we denote $\tau_{G[f]} = G[f](x, x)$ for any $x \in \d{R}^d$.

\begin{lemma}\label{lemma: norm tool 1}
    Suppose that $P_0$ possesses uniformly bounded density. If $f: \d{R}^d \times \d{R}^d \mapsto \d{R}$ satisfies $f(x+c, y+c) = f(x, y) \in \d{R}$ for all $x, y, c \in \d{R}^d$, and $\int [K(x)]^2 \sd x < \infty$, then $\norm{G[f]}_{L_1(P_0 \times P_0)} \lesssim \norm{f}_{L_1(\lambda \times P_0)}$, $\norm{G[f]}_{L_2(P_0 \times P_0)} \lesssim \norm{f}_{L_2(\lambda \times P_0)}$, and $\int \left[\int G[f](x, y) \sd P_0(x)\right]^2 \sd P_0(y) \lesssim \int \left[\int |f(x, y)| \sd x \right]^2 \sd P_0(y)$, where the constants depend on $K$ and $P_0$.  For $f \in \{f_1, f_2\}$, $\norm{f}_{L_1(\lambda \times P_0)} = O(1)$, $\norm{f}_{L_2(\lambda \times \nu)}^2 = O(\nu(\d{R}^d)h^{-d})$, and $\int \left[\int |f(x, y)| \sd x\right]^2 \sd \nu(y) \lesssim \nu(\d{R}^d)$ for a constant only depending on $K$ and any finite measure $\nu$.
\end{lemma}
\begin{proof}[\bfseries{Proof of \Cref{lemma: norm tool 1}}]
    By the property of $f$, the boundedness of the density of $P_0$, and the assumption that $\int K^2 < \infty$,
    \begin{align*}
        \norm{G[f]}_{L_1(P_0 \times P_0)} &= \iint \abs{G[f](x, y)} \sd P_0(x) \sd P_0(y)\\
        & = \iint \abs{\iint f(x+sh, y + th)K(s) K(t) \sd s \sd t} \sd P_0(x) \sd P_0(y)\\
        & \leq \iint \left[ \iint \abs{f(x+sh, y + th)} \sd P_0(x) \sd P_0(y) \right]  \abs{K(s) K(t)} \sd s \sd t \\
        & = \iint \left[ \iint \abs{f(x+sh-th, y)} \sd P_0(x) \sd P_0(y) \right]  \abs{K(s) K(t)} \sd s \sd t \\
        & \lesssim \iint \left[ \iint \abs{f(x', y)} \sd x' \sd P_0(y) \right]  \abs{K(s) K(t)} \sd s \sd t\\
        & = \norm{f}_{L_1(\lambda \times P_0)} \left[ \int |K(s)| \sd s \right]^2\\
        & \leq \norm{f}_{L_1(\lambda \times P_0)} \int [K(s)]^2 \sd s \\
        &\lesssim \norm{f}_{L_1(\lambda \times P_0)}.
    \end{align*}
    Similarly,
    \begin{align*}
        \norm{G[f]}_{L_2(P_0 \times P_0)}^2 &= \iint \abs{G[f](x, y)}^2 \sd P_0(x) \sd P_0(y)\\
        & = \iint \abs{\iint f(x+sh, y + th)K(s) K(t) \sd s \sd t}^2 \sd P_0(x) \sd P_0(y)\\
        & \leq  \iint \left[\iint \abs{f(x+sh, y + th)K(s) K(t)}^2  \sd P_0(x) \sd P_0(y) \right] \sd s \sd t\\
        & =  \iint \left[\iint \abs{f(x+sh-th, y)}^2  \sd P_0(x) \sd P_0(y) \right] \left[K(s) K(t)\right]^2 \sd s \sd t\\
        & \lesssim \iint \left[\iint \abs{f(x', y)}^2  \sd x' \sd P_0(y) \right] \left[ K(s) K(t) \right]^2 \sd s \sd t\\
        & = \norm{f}_{L_2(\lambda \times P_0)}^2 \left[ \int \{K(s)\}^2 \sd s \right]^2\\
        & \lesssim \norm{f}_{L_2(\lambda \times P_0)}^2.
    \end{align*}
    Finally,
    \begin{align*}
        \int \left[\int G[f](x, y) \sd P_0(x)\right]^2 \sd P_0(y) & = \int \left[\iiint f(x+sh, y + th)K(s) K(t) \sd s \sd t \sd P_0(x)\right]^2 \sd P_0(y)\\
        & \leq \iint  \left\{\int \left[\int f(x+sh, y + th)  \sd P_0(x)\right]^2 \sd P_0(y) \right\} \left[K(s) K(t)\right]^2 \sd s \sd t \\
        & = \iint  \left\{\int \left[\int f(x+sh-th, y)  \sd P_0(x)\right]^2 \sd P_0(y) \right\} \left[K(s) K(t)\right]^2 \sd s \sd t \\
        & \lesssim \iint  \left\{\int \left[\int |f(x+sh-th, y)| \sd x\right]^2 \sd P_0(y) \right\} \left[K(s) K(t) \right]^2 \sd s \sd t\\
        & = \left\{\int \left[\int |f(x', y)|  \sd x'\right]^2 \sd P_0(y) \right\}  \left\{ \iint  \left[K(s)K(t)\right]^2 \sd s\sd t \right\}\\
        & \lesssim \int \left[\int |f(x', y)|  \sd x'\right]^2 \sd P_0(y).
    \end{align*}

    Next, we note that 
    \begin{align*}
        \norm{f_1}_{L_1(\lambda \times P_0)} &= \iint \frac{1}{h^{d}} \abs{K\left(\frac{x - y}{h}\right)} \sd x \sd P_0(y)\\
        & = \iint \abs{K(x')}  \sd x' \sd P_0(y) \\
        & =  \int \abs{K(x')} \sd x' , \text{ and} \\
        \norm{f_2}_{L_1(\lambda \times P_0)} &= \iint \frac{1}{h^{2d}} \abs{\int K\left(\frac{x - z}{h}\right) K\left(\frac{y - z}{h}\right) \sd z} \sd x \sd P_0(y)\\
        & \leq \iint   \abs{ K\left(x'\right)K\left(y'\right)} \left[ \int \eta_0(z + hy') \sd z \right] \sd x' \sd y'   \\
        & =  \left[ \int   \abs{ K\left(x'\right)} \sd x' \right]^2.
    \end{align*}
    Both of these are finite  constants depending only on $K$. Next, 
    \begin{align*}
        \int \left[ \int |f_1(x, y)| \sd x \right]^2 \sd \nu(y) &= \int  \left[ \frac{1}{h^d} \int \abs{K\left(\frac{x - y}{h}\right)} \sd x \right]^2 \sd \nu(y)\\
        & = \int  \left[ \int \abs{K\left(x'\right)}  \sd x'  \right]^2 \sd \nu(y)\\
        & =\left[ \int \abs{K\left(x'\right)}  \sd x'  \right]^2 \nu(\d{R}^d), \text{ and} \\
        \int \left[ \int |f_2(x, y)| \sd x \right]^2 \sd \nu(y) & = \int \left[  \int \abs{ \frac{1}{h^{2d}} \int K\left(\frac{x - z}{h}\right) K\left(\frac{y - z}{h}\right) \sd z} \sd x \right]^2  \sd \nu(y)\\
        & \leq  \int \left[  \frac{1}{h^{2d}}\iint \abs{K\left(\frac{x - z}{h}\right) K\left(\frac{y - z}{h}\right)} \sd z \sd x \right]^2  \sd \nu(y)\\
        & = \int \left[ \iint    \abs{K\left(x'\right)K\left(z'\right)} \sd x' \sd z' \right]^2 \sd \nu(y)  \\
        & = \left[ \int \abs{K\left(x'\right)}  \sd x'  \right]^4 \nu(\d{R}^d).
    \end{align*}
    Similarly,
    \begin{align*}
        \norm{f_1}_{L_2(\lambda \times \nu)}^2 &= \frac{1}{h^{2d}} \iint  \left[K\left(\frac{x - y}{h}\right)\right]^2 \sd x \sd \nu(y)\\
        & = \frac{1}{h^{d}} \iint  \left[K\left(x'\right) \right]^2 \sd x' \sd \nu(y) \\
        & = \frac{1}{h^{d}}  \| K\|_{L_2(\lambda)}^2 \nu(\d{R}^d), 
    \end{align*}
    and
    \begin{align*}
        \norm{f_2}_{L_2(\lambda \times \nu)}^2 &= \iint \left[ \frac{1}{h^{2d}} \int K \left(\frac{x - z}{h}\right)K \left(\frac{y - z}{h}\right) \sd z \right]^2 \sd x \sd \nu(y)\\
        & = \frac{1}{h^{4d}} \iiiint   K \left(\frac{x - z}{h}\right)K \left(\frac{y - z}{h}\right)  K \left(\frac{x - w}{h}\right)K \left(\frac{y - w}{h}\right) \sd z \sd w \sd x \sd \nu(y) \\
        & = \frac{1}{h^{2d}} \iiiint   K \left(z'\right)K \left(z' + \frac{y - x}{h}\right)  K \left(w'\right)K \left(w' + \frac{y - x}{h}\right) \sd z' \sd w' \sd x \sd \nu(y) \\
        & = \frac{1}{h^{d}} \iiiint   K \left(z'\right)K \left(x'\right)  K \left(w'\right)K \left(w' + x' - z'\right) \sd z' \sd w' \sd x' \sd \nu(y) \\
        & = \frac{1}{h^{d}} \nu(\d{R}^d)\iiint   K \left(z'\right)  K \left(x'\right)  K \left(w'\right) K\left(w' + x' - z'\right) \sd z' \sd w' \sd x'\\
        & \leq \frac{1}{h^{d}} \nu(\d{R}^d)\iint   \abs{ K \left(z'\right) K \left(x'\right)  } \left[ \int \left\{K \left(w'\right)\right\}^2 \sd w' \right]^{1/2} \left[ \int \left\{ K\left(w' + x' - z'\right)\right\}^2 \sd w' \right]^{1/2} \sd z' \sd x' \\
        & = \frac{1}{h^{d}} \nu(\d{R}^d)\left[ \int   \abs{ K \left(z'\right)}   \sd z' \right]^2 \left[ \int 
        \left\{K \left(w'\right)\right\}^2 \sd w' \right]\\
        & \leq \frac{1}{h^{d}} \nu(\d{R}^d)\norm{K}_{L_2(\lambda)}^4,
    \end{align*}
    which yields the result.
\end{proof}

\begin{cor}\label{lemma: sV order 1}
    If $\hat{P}_n$ is the distribution corresponding to a kernel density estimator with bandwidth $h$ and kernel $K$ satisfying $\int K^2 < \infty$, $P_0$ possesses uniformly bounded density, and $f: \d{R}^d \times \d{R}^d \mapsto \d{R}$ satisfies $f(x, y) = f(y, x)$, $f(x, x) = \tau_f \in \d{R}$ and $f(x+c, y+c) = f(x, y) \in \d{R}$ for all $x, y, c \in \d{R}^d$, then
    \begin{align*}
        \E_0  \left\{ \int f \sd \left[(\hat{P}_n - P_{0, h}) \times (\hat{P}_n - P_{0, h})\right]  -  n^{-1}\tau_{G[f]} \right\}^2 & \lesssim n^{-3} \int \left[\int |f(x, y)| \sd x \right]^2 \sd P_0(y) \\
        & \qquad  + n^{-2} \norm{f}_{L_2(\lambda \times P_0)}^2
    \end{align*}
    for a constant depending on $K$ and $P_0$.
\end{cor}
\begin{proof}[\bfseries{Proof of \Cref{lemma: sV order 1}}]
    Recall that $G[f](x, y) := \int f(s, t) K_h(s, x) K_h(t, y) \sd s \sd t$. Under the assumptions on $f$, we have $G[f](x, y) = G[f](y, x)$ and $G[f](x, x) = \tau_{G[f]}$ for all  $x, y\in \d{R}^d$. Since $(\hat{p}_n - \eta_{0, h})(x) = \int K_h(s,x) \sd (\d{P}_n - P_0)(s)$, we can write
    \begin{align*}
        \int f \sd \left[(\hat{P}_n - P_{0, h}) \times (\hat{P}_n - P_{0, h})\right] &= \int G[f] \sd [(\d{P}_n - P_{0}) \times (\d{P}_n - P_{0})].
    \end{align*}
    By \Cref{lemma: V order 1}, we then have 
    \begin{align*}
         \E_0  \left\{ \int f \sd \left[(\hat{P}_n - P_{0, h}) \times (\hat{P}_n - P_{0, h})\right]  -  n^{-1}\tau_{G[f]} \right\}^2  &= \E_0 \left\{ \int G[f] \sd \left[(\d{P}_n - P_{0}) \times (\d{P}_n - P_{0})\right] - n^{-1} \tau_{G[f]} \right\}^2 \\
         &  \lesssim n^{-3} \int \left[\int G[f](x, y) \sd P_0(x)\right]^2 \sd P_0(y)  \\
        & \qquad+ n^{-2} \norm{G[f]}_{L_2(P_0 \times P_0)}^2.
    \end{align*}
    The result follows by \Cref{lemma: norm tool 1}. 
\end{proof}

\begin{lemma}\label{lemma: sV order 2}
    If $P_0$ possesses uniformly bounded and $m$-times continuously differentiable density $\eta_0$ with $\int [ (D^{\alpha} \eta_0)( x) ]^2 \sd x < \infty$ for all $\alpha$ such that $|\alpha| = m$, $\hat{P}_n$ is the distribution corresponding to a kernel density estimator with bandwidth $h$ and kernel $K$ satisfying $\int K^2 < \infty$, and $f: \d{R}^d \times \d{R}^d \mapsto \d{R}$ satisfies $f(x, y) = f(y, x)$ and $f(x, x) = \tau_f$ for all $x, y \in \d{R}^d$, then
    \[
        \E_0 \E_W \left\{ \int f \sd \left[(P_{0, h} - P_0) \times (\d{P}_n^* - \hat{P}_n)\right] \right\}^2 \lesssim n^{-1}h^{2m} \norm{f}_{L_2(\lambda \times P_{0,h})}^2.
    \]
\end{lemma}
\begin{proof}[\bfseries{Proof of \Cref{lemma: sV order 2}}]
    We denote $f^\circ(y) := \int f(x, y) \sd(P_{0, h} - P_0)(x) $. We then have 
    \begin{align*}
        \int f \sd [(P_{0, h} - P_0) \times (\d{P}_n^* - \hat{P}_n)] = (\d{P}_n^* - \hat{P}_n) f^\circ = \frac{1}{n} \sum_{i=1}^n \left[f^\circ(X_i^*) - \hat{P}_n f^\circ\right].
    \end{align*}
    Since $X_1^*, \dotsc, X_n^* \iidsim \hat{P}_n$, $\E_W \left[f^\circ(X_i^*) - \hat{P}_n f^\circ\right]\left[f^\circ(X_j^*) - \hat{P}_n f^\circ\right] = 0$ for any $i \neq j$, so
    \begin{align*}
        \E_W \left\{ \frac{1}{n} \sum_{i=1}^n \left[f^\circ(X_i^*) - \hat{P}_n f^\circ\right]\right\}^2 & = \E_W \left\{ \frac{1}{n^2} \sum_{i,j = 1}^n \left[f^\circ(X_i^*) - \hat{P}_n f^\circ\right]\left[f^\circ(X_j^*) - \hat{P}_n f^\circ\right]\right\}\\
        & =  \E_W \left\{ \frac{1}{n^2} \sum_{i = 1}^n \left[f^\circ(X_i^*) - \hat{P}_n f^\circ\right]^2\right\} \\
        &= \frac{1}{n} \E_W \left[f^\circ(X_i^*) - \hat{P}_n f^\circ\right]^2 \\
        & = \frac{1}{n} \left[ \hat{P}_n (f^{\circ})^2 - (\hat{P}_n f^\circ)^2 \right]\\
        & \leq \frac{1}{n} \hat{P}_n (f^{\circ})^2.
    \end{align*}
    By \Cref{lemma: h^m tool 1}, we have 
    \begin{align*}
        \E_0 \left[ n^{-1} \hat{P}_n (f^{\circ})^2 \right] &= n^{-1} P_{0, h} (f^{\circ})^2 \\
        & = n^{-1} \int \left[ \int f(x, y) [\eta_{0,h}(x) - \eta_0(x)] \sd x \right]^2 \sd P_{0, h}(y)\\
        & \leq n^{-1} \int \left[ \int |f(x, y)| \abs{\eta_{0,h}(x) - \eta_0(x)} \sd x \right]^2 \sd P_{0, h}(y)\\
        & \leq n^{-1} \iint |f(x, y)|^2 \sd x \sd P_{0, h}(y)  \int [\eta_{0,h}(x) - \eta_0(x)]^2 \sd x \\
        & \lesssim n^{-1}h^{2m} \norm{f}_{L_2(\lambda \times P_{0, h})}^2.
    \end{align*}

\end{proof}

\begin{lemma}\label{lemma: sV order 3}
    If $\hat{P}_n$ is the distribution corresponding to a kernel density estimator with bandwidth $h$ and uniformly bounded kernel $K$, $P_0$ possesses uniformly bounded density, then for any fixed $f: \d{R}^d \times \d{R}^d \mapsto \d{R}$,  
    \begin{align*}
        & \E_0 \E_W \left\{ \int f \sd [ (\hat{P}_n - P_{0, h}) \times (\d{P}_n^* - \hat{P}_n)]\right\}^2  \lesssim n^{-1}\left[(nh^d)^{-1} + (nh^d)^{-2} \right]\int \left[ \int |f(x, y)| \sd x\right]^2 \sd P_{0, h}(y),
    \end{align*}
    where the constant in the bound depends on $P_0$ and $K$.
\end{lemma}
\begin{proof}[\bfseries{Proof of \Cref{lemma: sV order 3}}]
    We note that 
    \begin{align*}
        \E_W \left\{ \int f \sd [ (\hat{P}_n - P_{0, h}) \times (\d{P}_n^* - \hat{P}_n)]\right\}^2 &= \E_W \left\{ \int \left[ \int f(x, z) \sd (\hat{P}_n - P_{0, h})(x) \right] \sd (\d{P}_n^* - \hat{P}_n)(z) \right\}^2\\
        & = \frac{1}{n} \mathrm{Var}_W  \left( \int f(x, Z) \sd (\hat{P}_n - P_{0, h})(x) \right)\\
        & \leq \frac{1}{n}\E_W \left[  \int f(x, Z) \sd (\hat{P}_n - P_{0, h})(x) \right]^2\\
        & = \frac{1}{n}\int \left[  \int f(x, z) \sd (\hat{P}_n - P_{0, h})(x) \right]^2 \sd \hat{P}_n(z)\\
        & = \frac{1}{n}\int \left[  \int f(x, z) \sd (\hat{P}_n - P_{0, h})(x) \right]^2 \sd  P_{0, h}(z) \\
        &\qquad +  \frac{1}{n}\int \left[  \int f(x, z) \sd (\hat{P}_n - P_{0, h})(x) \right]^2 \sd  (\hat{P}_n -P_{0, h})(z).
    \end{align*}
    We bound the two terms in the final expression separately. We can write
    \begin{align*}
        &\E_0 \left\{ \int \left[  \int f(x, z) \sd (\hat{P}_n - P_{0, h})(x) \right]^2 \sd P_{0, h}(z) \right\}\\
        & \qquad =  \E_0 \left\{ \iiint f(x, z)f(y, z) \sd (\hat{P}_n - P_{0, h})(x) \sd (\hat{P}_n - P_{0, h})(y)  \sd P_{0, h}(z) \right\}\\
        & \qquad =  \E_0 \left\{ \iiint f(x, z)f(y, z)  (\hat{p}_n - \eta_{0, h})(x)  (\hat{p}_n - \eta_{0, h})(y) \sd x \sd y  \sd P_{0, h}(z) \right\}\\
        & \qquad \leq    \iiint |f(x, z)f(y, z)| \left\{\E_0 \left[ \hat{p}_n(x) - \eta_{0, h}(x)\right]^2\right\}^{1/2}  \left\{\E_0 \left[ \hat{p}_n(y) - \eta_{0, h}(y)\right]^2\right\}^{1/2}  \sd x \sd y  \sd P_{0, h}(z) 
    \end{align*}
    Since $\eta_0$ is uniformly bounded and $\int K^2 < \infty$, 
    \begin{align*}
        \E_0 \left[ \hat{p}_n(x) - \eta_{0, h}(x)\right]^2 &= \E_0 \left[ \int K_h(s,x) \sd (\d{P}_n - P_0)(s)\right]^2\\
        & = \frac{1}{n} \mathrm{Var}_0 \left[K_h(X, x)\right]\\
        & \leq \frac{1}{nh^{2d}} \E_0\left[ \left\{ K\left(\frac{X-x}{h}\right) \right\}^2 \right]\\
        & = \frac{1}{nh^{2d}}\int \left\{ K\left(\frac{t-x}{h}\right) \right\}^2 \eta_0(t)\sd t\\
        & = \frac{1}{nh^d} \int \left\{ K\left(u\right) \right\}^2 \eta_0(x + hu)\sd u\\
        & \lesssim \frac{1}{nh^d}
    \end{align*}
    for each $x \in \d{R}^d$. Thus,
    \begin{align*}
        \E_0 \left\{ \int \left[  \int f(x, z) \sd (\hat{P}_n - P_{0, h})(x) \right]^2 \sd P_{0, h}(z) \right\}  &\lesssim \frac{1}{nh^d} \iiint |f(x, z)f(y, z)|  \sd x \sd y  \sd P_{0, h}(z) \\
        &= \frac{1}{nh^d} \int  \left[ \int |f(x, z)|  \sd x\right]^2  \sd P_{0, h}(z).
    \end{align*}
    Next, we note that for each $x,y,z \in \d{R}^d$, since $X_1, \dotsc, X_n$ are IID,
    \begin{align*}
        &\E_0 \left\{ \left[\hat{p}_n(x) - \eta_{0,h}(x) \right]\left[\hat{p}_n(y) - \eta_{0,h}(y) \right]\left[\hat{p}_n(z) - \eta_{0,h}(z) \right] \right\}\\
        &\qquad = \E_0 \left[\int K_h(s,x) \sd (\d{P}_n - P_0)(s)\int K_h(t,y) \sd (\d{P}_n - P_0)(t)\int K_h(u,z) \sd (\d{P}_n - P_0)(u)\right]\\
        &\qquad = \E_0 \left\{ \frac{1}{n^3} \sum_{i,j,k}\left[K_h(X_i, x) - \eta_{0,h}(x)\right]\left[K_h(X_j, y) - \eta_{0,h}(y)\right]\left[K_h(X_k, z) -\eta_{0,h}(z)\right] \right\}\\
        & \qquad = \E_0 \left\{ \frac{1}{n^3} \sum_{i=1}^n\left[K_h(X_i, x) - \eta_{0,h}(x)\right]\left[K_h(X_i, y) - \eta_{0,h}(y)\right]\left[K_h(X_i, z) -\eta_{0,h}(z)\right] \right\}\\
        & \qquad = \frac{1}{n^2} \E_0\left\{ \left[K_h(X, x) - \eta_{0,h}(x)\right]\left[K_h(X, y) - \eta_{0,h}(y)\right]\left[K_h(X, z) -\eta_{0,h}(z)\right] \right\}.
    \end{align*}
    Hence, 
    \begin{align*}
        &\E_0\left\{ \int \left[  \int f(x, z) \sd (\hat{P}_n - P_{0, h})(x) \right]^2 \sd (\hat{P}_n - P_{0, h})(z) \right\}\\
        & \qquad =  \E_0 \left\{ \iiint f(x, z)f(y, z) \left[\hat{p}_n(x) - \eta_{0,h}(x) \right]\left[\hat{p}_n(y) - \eta_{0,h}(y) \right]\left[\hat{p}_n(z) - \eta_{0,h}(z) \right] \sd x \sd y  \sd z \right\}\\
        & \qquad = \frac{1}{n^2} \E_0 \iiint f(x, z)f(y, z)  \left[K_h(X, x) - \eta_{0,h}(x)\right]\left[K_h(X, y) -  \eta_{0,h}(y)\right] \left[K_h(X, z) -  \eta_{0,h}(z)\right] \sd x \sd y \sd z\\
        & \qquad = \frac{1}{n^2} \E_0 \int \left\{ \int f(x, z) \left[K_h(X, x) -  \eta_{0,h}(x)\right] \sd x\right\}^2 \left[K_h(X, z) -  \eta_{0,h}(z)\right]  \sd z.
    \end{align*}
    Since $K$ is uniformly bounded, we have $|K_h(x, y) -\eta_{0,h}(x)| \lesssim h^{-d}$ for any $x, y \in \d{R}^d$. Therefore, 
    \begin{align*}
        & \abs{\frac{1}{n^2} \E_0 \int \left\{ \int f(x, z) \left[K_h(X, x) -  \eta_{0,h}(x)\right] \sd x\right\}^2 \left[K_h(X, z) - \eta_{0,h}(z)\right]  \sd z}\\
        & \qquad \leq \abs{\frac{1}{n^2} \E_0 \int \left\{ \int f(x, z) \left[K_h(X, x) - \eta_{0,h}(x)\right] \sd x\right\}^2 K_h(X, z) \sd z}\\
        & \qquad \qquad + \abs{\frac{1}{n^2} \E_0 \int \left\{ \int f(x, z) \left[K_h(X, x) - \eta_{0,h}(x)\right] \sd x\right\}^2  \eta_{0,h}(z) \sd z} \\
        & \qquad \lesssim \frac{1}{n^2h^{2d}} \abs{\E_0 \int \left\{ \int |f(x, z)| \sd x\right\}^2 K_h(X, z) \sd z} + \frac{1}{n^2h^{2d}} \abs{ \int \left\{ \int |f(x, z)| \sd x\right\}^2   \eta_{0,h}(z) \sd z}\\
        & \qquad =  \frac{2}{n^2h^{2d}}\int \left\{ \int |f(x, z)| \sd x\right\}^2  \sd P_{0, h}(z).
    \end{align*}
    Putting the pieces together completes the result.
\end{proof}

\begin{lemma}\label{lemma: sV order 3.5}
    If $\hat{P}_n$ is the distribution corresponding to a kernel density estimator with bandwidth $h$ and uniformly bounded kernel $K$, $P_0$ possesses a Lebesgue density function, and $f: \d{R}^d \times \d{R}^d \mapsto \d{R}$ satisfies $f(x, y) = f(y, x)$ for all $x, y \in \d{R}^d$ then
    \begin{align*}
        & \E_0 \E_W \left\{ \int f \sd [ (\d{P}_n - P_0) \times (\d{P}_n^* - \hat{P}_n)]\right\}^2  \lesssim \frac{1}{n^2} \norm{f}_{L_2(P_0 \times P_{0, h})}^2 +  \frac{1}{n^3h^{d}} \norm{f}_{L_2(\lambda\times P_0)}^2,
    \end{align*}
    where the constant in the bound depends on $K$ only.    
\end{lemma}
\begin{proof}[\bfseries{Proof of \Cref{lemma: sV order 3.5}}]
    As in the proof of~\Cref{lemma: sV order 3}, we have
    \begin{align*}
        \E_W \left\{ \int f \sd [ (\d{P}_n - P_0) \times (\d{P}_n^* - \hat{P}_n)]\right\}^2
         & \leq \frac{1}{n}\int \left[  \int f(x, z) \sd (\d{P}_n - P_0)(x) \right]^2 \sd  P_{0, h}(z) \\
         &\qquad + \frac{1}{n}\int \left[  \int f(x, z) \sd (\d{P}_n - P_0)(x) \right]^2 \sd (\hat{P}_n-P_{0, h})(z).
    \end{align*}
    We first bound the first term. We have
    \begin{align*}
        \E_0 \left\{ \int \left[  \int f(x, z) \sd (\d{P}_n - P_0)(x) \right]^2 \sd P_{0, h}(z) \right\} & = \int \E_0 \left[  \int f(x, z) \sd (\d{P}_n - P_0)(x) \right]^2 \sd P_{0, h}(z) \\
        & = \frac{1}{n} \int \mathrm{Var}_0 ( f(X, z) ) \sd P_{0, h}(z) \\
        & \leq \frac{1}{n} \int \E_0\left[ f(X, z)\right]^2 \sd P_{0, h}(z) \\
        & = \frac{1}{n} \norm{f}_{L_2(P_0 \times P_{0,h})}^2.
    \end{align*}
    Next, we note that 
     \begin{align*}
        &\E_0\left\{ \int \left[  \int f(x, z) \sd (\d{P}_n - P_0)(x) \right]^2 \sd (\hat{P}_n - P_{0, h})(z) \right\} \\
        & \qquad = \E_0\left\{ \int \int f(x, z) \sd (\d{P}_n - P_0)(x)\int f(y, z) \sd (\d{P}_n - P_0)(y)\int K_h(w, z) \sd (\d{P}_n - P_0)(w) \sd z \right\}\\
        &\qquad = \int \E_0 \left\{ \frac{1}{n^3} \sum_{i,j,k} \left[f(X_i, z) - P_0 f(\cdot, z)\right]\left[f(X_j, z) - P_0 f(\cdot, z)\right]\left[K_h(X_k, z) - \eta_{0,h}(z)\right]\right\} \sd z \\
        & \qquad = \int \E_0 \left\{ \frac{1}{n^3} \sum_{i=1}^n \left[f(X_i, z) - P_0 f(\cdot, z)\right]^2 \left[K_h(X_i, z) - \eta_{0,h}(z)\right]\right\} \sd z\\
        & \qquad = \frac{1}{n^2} \int  \E_0 \left\{ \left[f(X, z) - P_0 f(\cdot, z)\right]^2 \left[K_h(X, z) -\eta_{0,h}(z)\right]\right\} \sd z .
    \end{align*}
    Since $K$ is uniformly bounded, we have $\left| K_h(x, z) - \eta_{0,h}(z) \right| \lesssim h^{-d}$ for all $x, z \in \d{R}^d$. Thus,
    \begin{align*}
    &\left| \E_0\left\{ \int \left[  \int f(x, z) \sd (\d{P}_n - P_0)(x) \right]^2 \sd (\hat{P}_n - P_{0, h})(z) \right\} \right| \\
        & \qquad \lesssim \frac{1}{n^2 h^d}  \int \E_0  \left[f(X, z) - P_0 f(\cdot, z)\right]^2 \sd z\\
        & \qquad \leq \frac{1}{n^2h^d}   \int \E_0\left[ f(X, z)\right]^2 \sd z\\
        & \qquad = \frac{1}{n^2h^d}  \norm{f}_{L_2(\lambda \times P_0)}^2.
    \end{align*}
   Putting the two bounds together yields the result.  
\end{proof}

\begin{lemma}\label{lemma: sV order 4}
    If $\hat{P}_n$ is the distribution corresponding to a kernel density estimator with bandwidth $h$ and uniformly bounded kernel $K$ satisfying $\int K^2 < \infty$, $P_0$ possesses uniformly bounded density, and $f: \d{R}^d \times \d{R}^d \mapsto \d{R}$ satisfies $f(x, y) = f(y, x)$, $f(x, x) = \tau_f \in \d{R}$ and $f(x+c, y+c) = f(x, y)$ for all $x, y, c \in \d{R}^d$, then
    \begin{align*}
        & \E_0 \E_W \left\{ \int f \sd [(\d{P}_n^* - \hat{P}_n) \times (\d{P}_n^* - \hat{P}_n)] - n^{-1}\tau_f\right\}^2 \\
        & \qquad \lesssim \frac{\tau_{G[f]}^2}{n^4} + \frac{ \tau_{G[f^2]}}{n^3} + \frac{\tau_{G[f]}}{n^3}\norm{f}_{L_1(\lambda \times P_0)} + \frac{1}{n^2} \norm{f}_{L_2(\lambda \times P_0)}^2 + \frac{1}{n^3} \int   \left[ \int |f(x, y)| \sd x \right]^2 \sd P_0(y)\\
        & \qquad \qquad   + \frac{1}{n^3} \left[1 + (nh^d)^{-1} + (nh^d)^{-2} \right] \int   \left( \int |f(x, y)| \sd x \right)^2 \sd P_{0, h}(y)
    \end{align*}
\end{lemma}
\begin{proof}[\bfseries{Proof of \Cref{lemma: sV order 4}}]
    We define $V_n(x, y) :=  \int f \sd \left[(\delta_{x} - \hat{P}_n)  \times (\delta_{y} - \hat{P}_n)\right]$.  We then have 
    \begin{align}
        & \E_0 \E_W \left\{ \int f \sd [(\d{P}_n^* - \hat{P}_n) \times (\d{P}_n^* - \hat{P}_n)] - \frac{\tau_f}{n} \right\}^2 \notag \\
        & \qquad = \E_0 \E_W \left\{ \frac{1}{n^2} \sum_{i = 1}^n V_n(X_i^*, X_i^*) + \frac{1}{n^2} \sum_{i \neq j} V_n(X_i^*, X_j^*) - \frac{\tau_f}{n} \right\}^2 \notag \\
        & \qquad \leq  2\E_0 \E_W \left\{ \frac{1}{n^2} \sum_{i = 1}^n V_n(X_i^*, X_i^*) - \frac{\tau_f}{n} \right\}^2 + 2\E_0 \E_W \left\{ \frac{1}{n^2} \sum_{i \neq j} V_n(X_i^*, X_j^*) \right\}^2. \label{eq: sV decomposition 2}
    \end{align}
    For the first term on the right-hand side of \eqref{eq: sV decomposition 2}, by definition of $\tau_f$, symmetry of $f$, and adding and subtracting terms,
    \begin{align*}
        & \frac{1}{n^2} \sum_{i = 1}^n V_n(X_i^*, X_i^*) - \frac{\tau_f}{n} \\
        &\qquad = \frac{1}{n^2} \sum_{i = 1}^n \int f \sd [(\delta_{X_i^*} - \hat{P}_n) \times (\delta_{X_i^*} - \hat{P}_n)] - \frac{\tau_f}{n} \\
        &\qquad=  - \frac{2}{n} \int f \sd (\hat{P}_n \times \d{P}_n^*)  + \frac{1}{n} \int f \sd (\hat{P}_n \times  \hat{P}_n) \\
        &\qquad=  - \frac{2}{n} \int f \sd [\hat{P}_n \times  (\d{P}_n^* - \hat{P}_n)] - \frac{1}{n} \int f \sd (\hat{P}_n \times \hat{P}_n) \\
        &\qquad=   - \frac{2}{n} \int f \sd [ (\hat{P}_n - P_{0, h}) \times (\d{P}_n^* - \hat{P}_n)] - \frac{2}{n} \int f \sd [P_{0, h} \times (\d{P}_n^* - \hat{P}_n)] - \frac{1}{n} \int f \sd (\hat{P}_n \times \hat{P}_n).
    \end{align*}
    Hence,  
    \begin{align}
        &\E_0 \E_W \left\{ \frac{1}{n^2} \sum_{i = j} V_n(X_i^*, X_j^*) - \frac{\tau_f}{n} \right\}^2\notag\\
        & \qquad \lesssim \E_0 \E_W \left\{ \frac{1}{n} \int f \sd [ (\hat{P}_n - P_{0, h}) \times (\d{P}_n^* - \hat{P}_n)]\right\}^2 + \E_0 \E_W \left\{\frac{1}{n} \int f \sd [P_{0, h} \times (\d{P}_n^* - \hat{P}_n)] \right\}^2 \label{eq: sV decomposition 3}\\
        & \qquad \qquad + \E_0  \left\{\frac{1}{n} \int f \sd (\hat{P}_n \times \hat{P}_n)\right\}^2. \notag
    \end{align}
    By \Cref{lemma: sV order 3}, the first term on the right-hand side of \eqref{eq: sV decomposition 3} is bounded up to a constant by 
    \[n^{-3} \left[ (nh^d)^{-1} + (nh^d)^{-2}\right] \int \left[ \int |f(x, y)| \sd x\right]^2 \sd P_{0, h}(y).\] 
    For the second term on the right-hand side of \eqref{eq: sV decomposition 3},
    \begin{align*}
        \E_0 \E_W \left\{\frac{1}{n} \int f \sd [P_{0, h} \times (\d{P}_n^* - \hat{P}_n)] \right\}^2 &= \frac{1}{n^3} \E_0 \left[ \mathrm{Var}_W \left( \int f(x, X_1^*) \sd P_{0, h}(x) \right) \right]\\
        & \leq \frac{1}{n^3} \E_0 \left[ \int \left( \int f(x, y) \sd P_{0, h}(x) \right)^2 \sd \hat{P}_n(y) \right]\\
        & = \frac{1}{n^3} \int   \left( \int f(x, y) \sd P_{0, h}(x) \right)^2 \sd P_{0, h}(y)\\
        & \lesssim \frac{1}{n^3} \int   \left( \int |f(x, y)| \sd x \right)^2 \sd P_{0, h}(y),
    \end{align*}
    where the last inequality is because $K$ is uniformly bounded. For the last term on the right-hand side of \eqref{eq: sV decomposition 3}, we have 
    \begin{align*}
        \E_0 \left\{\frac{1}{n} \int f \sd (\hat{P}_n \times \hat{P}_n) \right\}^2 & =  \E_0 \left\{\frac{1}{n^3} \sum_{i, j} G[f](X_i, X_j) \right\}^2\\
        & = \frac{1}{n^2} \E_0 \left\{ \frac{1}{n^2}\sum_{i=1}^n G[f](X_i, X_i) + \frac{1}{n^2}\sum_{i\neq j}G[f](X_i, X_j) \right\}^2\\
        & =  \frac{1}{n^2} \E_0 \left\{ \frac{\tau_{G[f]}}{n} + \frac{1}{n^2}\sum_{i\neq j}G[f](X_i, X_j) \right\}^2\\
        & = \frac{1}{n^2} \left\{ \frac{\tau_{G[f]}^2}{n^2} + \frac{2 \tau_{G[f]}}{n^3}\sum_{i\neq j} \E_0 \left[G[f](X_i, X_j)\right] + \frac{1}{n^4} \E_0 \left[ \sum_{i\neq j}G[f](X_i, X_j)\right]^2 \right\}\\
        & \lesssim \frac{\tau_{G[f]}^2}{n^4}  + \frac{\tau_{G[f]}}{n^3}\E_0 \left[ G[f](X_1, X_2)\right] + \frac{1}{n^6}\E_0 \left[ \sum_{i\neq j}G[f](X_i, X_j)\right]^2.
    \end{align*}
    By~\Cref{lemma: norm tool 1}, we have $\left|\E_0 \left[ G[f](X_1, X_2)\right] \right| \lesssim \| f\|_{L_1(\lambda \times P_0)}$. By \Cref{lemma: square decomp} and \Cref{lemma: norm tool 1}, we note that 
    \begin{align*}
        \E_0 \left[ \sum_{i \neq j} G[f](X_i, X_j) \right]^2 &= \E_0 \left[ 4 \sum_{i \neq j, j \neq k, k \neq i}G[f](X_i, X_j) G[f](X_i, X_k) + 2 \sum_{i \neq j}\left[G[f](X_i, X_j)\right]^2 \right.\\
        & \qquad \qquad \left.+ \sum_{\substack{i\neq j, i\neq k, i\neq l\\ j \neq k, j \neq l, k \neq l}} G[f](X_i, X_j) G[f](X_k, X_l) \right]\\
        & \lesssim n^3 \E_0 \left[ G[f](X_1, X_2)G[f](X_1, X_3)\right] + n^2 \E_0\left[ G[f](X_1, X_2)\right]^2 + n^4 [\E_0 G[f](X_1, X_2)]^2\\
        & \leq n^3 \int   \left( \int G[f](x, y) \sd P_0(x) \right)^2 \sd P_0(y) + n^2\norm{G[f]}_{L_2(P_0 \times P_0)}^2 + n^4 \norm{G[f]}_{L_1(P_0 \times P_0)}^2\\
        & \lesssim n^3 \int   \left( \int |f(x, y)| \sd x \right)^2 \sd P_0(y) + n^2\norm{f}_{L_2(\lambda \times P_0)}^2 + n^4 \norm{f}_{L_1(\lambda \times P_0)}^2.
    \end{align*}
    Combining these calculations, we get
    \begin{align*}
        \E_0 \left\{\frac{1}{n} \int f \sd (\hat{P}_n \times \hat{P}_n) \right\}^2 &\lesssim \frac{\tau_{G[f]}^2}{n^4}  + \frac{\tau_{G[f]}}{n^3}\norm{f}_{L_1(\lambda \times P_0)} + \frac{1}{n^2}\norm{f}_{L_1(\lambda \times P_0)}^2 + \frac{1}{n^4} \norm{f}_{L_2(\lambda \times P_0)}^2\\
        & \qquad \qquad + \frac{1}{n^3} \int   \left( \int |f(x, y)| \sd x \right)^2 \sd P_0(y).
    \end{align*}
    Therefore, 
    \begin{align*}
        \E_0 \E_W \left\{ \frac{1}{n^2} \sum_{i = j} V_n(X_i^*, X_j^*) - \frac{\tau_f}{n} \right\}^2 & \lesssim \frac{\tau_{G[f]}^2}{n^4}  + \frac{\tau_{G[f]}}{n^3}\norm{f}_{L_1(\lambda \times P_0)} + \frac{1}{n^2}\norm{f}_{L_1(\lambda \times P_0)}^2 + \frac{1}{n^4} \norm{f}_{L_2(\lambda \times P_0)}^2\\
        & \qquad  + \frac{1}{n^3} \int   \left( \int |f(x, y)| \sd x \right)^2 \sd P_0(y)\\
        & \qquad + \frac{1}{n^3} \left[1 + (nh^d)^{-1} + (nh^d)^{-2}\right]\int   \left( \int |f(x, y)| \sd x \right)^2 \sd P_{0, h}(y).
    \end{align*}
    
    For the second term on the right-hand side of \eqref{eq: sV decomposition 2}, we note that for any $i \neq j$, 
    \begin{align*}
        & \E_W \left[V_n(X_i^*, X_j^*) \mid X_j^*\right] \\
        & \qquad = \E_W \left[f(X_i^*, X_j^*) - \int f(X_i^*, y) \sd \hat{P}_n(y) -  \int f(x, X_j^*) \sd \hat{P}_n(x) + (\hat{P}_n \times \hat{P}_n) f \mid X_j^*\right]\\
        & \qquad = \int f(x, X_j^*) \sd \hat{P}_n(x) - (\hat{P}_n \times \hat{P}_n) f  -  \int f(x, X_j^*) \sd \hat{P}_n(x) + (\hat{P}_n \times \hat{P}_n) f = 0
    \end{align*}
    Hence, by  the law of total expectation, $\E_W [V_n(X_i^*, X_j^*)] = 0$ for all $i \neq j$, and by symmetry of $f$, which implies symmetry of $V_n$, $\E_W [V_n(X_i^*, X_j^*)V_n(X_i^*, X_k^*)] = 0$ for all $i \neq j \neq k$. Thus, by \Cref{lemma: square decomp},
    \begin{align*}
        &\E_W \left[ \frac{1}{n^2} \sum_{i \neq j} V_n(X_i^*, X_j^*) \right]^2 \\
        &\qquad =   \E_W \left[ \frac{4}{n^4} \sum_{i \neq j, j \neq k, k \neq i}V_n(X_i^*, X_j^*) V_n(X_i^*, X_k^*) + \frac{2}{n^4} \sum_{i \neq j}\left\{V_n(X_i^*, X_j^*)\right\}^2 \right.\\
        & \qquad\qquad \qquad \left.+ \frac{1}{n^4} \sum_{\substack{i\neq j, i\neq k, i\neq l\\ j \neq k, j \neq l, k \neq l}} V_n(X_i^*, X_j^*) V_n(X_k^*, X_l^*) \right]\\
        & \qquad=   \E_W \left[ \frac{2}{n^4} \sum_{i \neq j} \left\{V_n(X_i^*, X_j^*)\right\}^2 \right]\\
        & \qquad\lesssim  \frac{1}{n^2}   \E_W \left[V_n(X_1^*, X_2^*) \right]^2\\
        & \qquad =  \frac{1}{n^2}   \E_W \left[ f(X_1^*, X_2^*) - \int f(X_1^*, y) \sd \hat{P}_n(y) - \int f(x, X_2^*) \sd \hat{P}_n(x) + (\hat{P}_n\times \hat{P}_n) f\right]^2\\
        & \qquad=  \frac{1}{n^2}   \mathrm{Var}_W \left[ f(X_1^*, X_2^*) - \int f(X_1^*, y) \sd \hat{P}_n(y) -  \int f(x, X_2^*) \sd \hat{P}_n(x) \right] \\
        & \qquad\leq  \frac{1}{n^2}   \E_W \left[ f(X_1^*, X_2^*) - \int f(X_1^*, y) \sd \hat{P}_n(y) -  \int f(x, X_2^*) \sd \hat{P}_n(x) \right]^2 \\
        & \qquad =  \frac{1}{n^2}\E_W \left\{ f(X_1^*, X_2^*)^2 + \left[\int f(X_1^*, y) \sd \hat{P}_n(y) \right]^2 + \left[\int f(x, X_2^*) \sd \hat{P}_n(x) \right]^2  \right.\\
        & \qquad \qquad\qquad \qquad   - 2f(X_1^*, X_2^*)\left[ \int f(X_1^*, y) \sd \hat{P}_n(y) \right]  - 2f(X_1^*, X_2^*)\left[\int f(x, X_2^*) \sd \hat{P}_n(x)\right] \\
        & \qquad \qquad \qquad \qquad \left.+ 2\left[ \int f(X_1^*, y) \sd \hat{P}_n(y)\right]\left[\int f(x, X_2^*) \sd \hat{P}_n(x)\right]\right\}\\
        & \qquad =  \frac{1}{n^2}\E_W \left\{ f(X_1^*, X_2^*)^2 - 2\left[\int f(X_1^*, y) \sd \hat{P}_n(y)\right]^2  + 2 \left[ \int f \sd (\hat{P}_n \times \hat{P}_n)\right]^2 \right\}\\
        & \qquad \leq  \frac{1}{n^2}\int f^2 \sd (\hat{P}_n \times \hat{P}_n),
    \end{align*}
    where the last inequality is because  $[\int f \sd (\hat{P}_n \times \hat{P}_n) ]^2 \leq \int  [\int f \sd \hat{P}_n]^2 \sd \hat{P}_n$ by Jensen's inequality. Hence, by the assumption that $f(x + c, y + c) = f(x,y)$, 
    \begin{align*}
        \E_0 \E_W \left[ \frac{1}{n^2} \sum_{i \neq j} V_n(X_i^*, X_j^*) \right]^2 &\lesssim \frac{1}{n^2}  \E_0 \left[ \int f^2 \sd (\hat{P}_n \times \hat{P}_n) \right]\\
        & = \frac{1}{n^2}  \E_0 \left[ \frac{1}{n^2h^{2d}} \sum_{i, j} \iint \left\{f(s, t)\right\}^2 K \left(\frac{s - X_i}{h} \right) K \left(\frac{t-X_j}{h} \right) \sd s \sd t\right]\\
        & = \frac{1}{n^2}  \E_0 \left[ \frac{1}{n^2} \sum_{i, j} \iint f(X_i + hs', X_j + ht')^2 K (s') K (t') \sd s' \sd t'\right]\\
        & = \frac{1}{n^2}  \E_0 \left[ \frac{1}{n^2} \sum_{i = 1}^n \iint \left\{f(X_i + hs', X_i + ht')\right\}^2 K (s') K (t') \sd s' \sd t'\right] \\
        & \qquad \qquad + \frac{1}{n^2}  \E_0 \left[ \frac{1}{n^2} \sum_{i \neq j} \iint \left\{f(X_i + hs', X_j + ht')\right\}^2 K (s') K (t') \sd s' \sd t'\right]\\
        & = \frac{1}{n^3} \iint \{f(hs', ht')\}^2 K (s') K (t') \sd s' \sd t' \\
        & \qquad \qquad + \frac{n-1}{n^3}  \iint \E_0 \left[ f(X_i + hs' - ht', X_j)^2 \right] K (s') K (t') \sd s' \sd t'\\
        & \lesssim \frac{1}{n^3} \iint \{f(hs', ht')\}^2 K (s') K (t') \sd s' \sd t' + \frac{1}{n^2}  \norm{f}_{L_2(\lambda \times P_0)}^2\\
        & = \frac{1}{n^3} \tau_{G[f^2]} + \frac{1}{n^2}  \norm{f}_{L_2(\lambda \times P_0)}^2,
    \end{align*}
    where the last equality is because $f(x+c, y+c) = f(x, y)$ for all $x, y, c \in \d{R}^d$ and the definition of $\tau_{G[f]}$. The result follows.
\end{proof}

\begin{cor}\label{cor: auto smooth empirical}
   If $\hat{P}_n$ is the distribution corresponding to a kernel density estimator with bandwidth $h$ and uniformly bounded kernel $K$ satisfying $\int K^2 < \infty$, $P_0$ possesses uniformly bounded density, and $nh^d \longrightarrow \infty$, then for each $f\in \{f_1, f_2\}$, we have $\int f \sd [(\hat{P}_n - P_{0, h}) \times (\d{P}_n^* - \hat{P}_n)] = o_{\prob_W^*}(n^{-1/2})$, $\int f  \sd [(\d{P}_n - P_0) \times (\d{P}_n - P_0)] = n^{-1}\tau_f + o_{\prob_0^*}(n^{-1/2})$, and $\int f  \sd [(\d{P}_n^* - \hat{P}_n) \times (\d{P}_n^* - \hat{P}_n)] = n^{-1}\tau_f + o_{\prob_W^*}(n^{-1/2})$.
\end{cor}
\begin{proof}[\bfseries{Proof of \Cref{cor: auto smooth empirical}}]
    We note that  $\int \left\{\int |f(x, y)| \sd x \right\}^2 \sd P_{0,h}(y)  = O(1)$  for each $f \in \{f_1, f_2\}$ by \Cref{lemma: norm tool 1} since $P_{0,h}(\d{R}^d) = 1$. Hence, by \Cref{lemma: sV order 3}, we then have 
    \begin{align*}
        \E_0 \E_W \left\{ \int f \sd [ (\hat{P}_n - P_{0, h}) \times (\d{P}_n^* - \hat{P}_n)]\right\}^2  &\lesssim n^{-1}\left[(nh^d)^{-1} + (nh^d)^{-2} \right]\int \left[ \int |f(x, y)| \sd x\right]^2 \sd P_{0, h}(y)\\
        & \lesssim n^{-1}\left[(nh^d)^{-1} + (nh^d)^{-2} \right],
    \end{align*}
    which is $o(n^{-1})$ if $nh^d \longrightarrow \infty$. This shows the first statement.

    Next, we note that  $\int \left\{\int |f(x, y)| \sd x \right\}^2 \sd P_{0}(y)  = O(1)$  and $\norm{f}_{L_2(\lambda \times P_0)} = O(h^{-d/2})$ for each $f \in \{f_1, f_2\}$ by \Cref{lemma: norm tool 1}. Hence, since $P_0$ possesses uniformly bounded Lebesgue density function, by \Cref{lemma: V order 1}, we have 
    \begin{align*}
        \E_0 \left\{ \int f \sd [(\d{P}_n - P_0) \times (\d{P}_n - P_0)] - n^{-1} \tau_f \right\}^2 & \lesssim  n^{-3} \int \left[ \int f(x, y) \sd P_0(x) \right]^2 \sd P_0(y)  + n^{-2} \norm{f}_{L_2(P_0 \times P_0)}^2 \\
        & \leq  n^{-3} \int \left[ \int |f(x, y)| \sd x \right]^2 \sd P_0(y)  + n^{-2} \norm{f}_{L_2(\lambda \times P_0)}^2 \\
        & \lesssim n^{-3} + n^{-1} (nh^{d})^{-1},
    \end{align*}
    which is $o(n^{-1})$ if $nh^d \longrightarrow \infty$. This implies that $\int f \sd [(\d{P}_n - P_0) \times (\d{P}_n - P_0)] = n^{-1} \tau_f + o_{\prob_W^*}(n^{-1/2})$ if $nh^d \longrightarrow \infty$. 
    
    Finally, since for each $f \in \{ f_1, f_2 \}$, we have $f(x, y) = f(y, x)$, $f(x, x) = \tau_f$ and $f(x+c, y+c) = f(x, y)$ for all $x, y, c \in \d{R}^d$, under the stated conditions, by \Cref{lemma: sV order 4}, 
    \begin{align*}
        & \E_0 \E_W \left\{ \int f \sd [(\d{P}_n^* - \hat{P}_n) \times (\d{P}_n^* - \hat{P}_n)] - n^{-1}\tau_f\right\}^2 \\
        & \qquad \lesssim \frac{\tau_{G[f]}^2}{n^4} + \frac{ \tau_{G[f^2]}}{n^3} + \frac{\tau_{G[f]}}{n^3}\norm{f}_{L_1(\lambda \times P_0)} + \frac{1}{n^2} \norm{f}_{L_2(\lambda \times P_0)}^2 + \frac{1}{n^3} \int   \left[ \int |f(x, y)| \sd x \right]^2 \sd P_0(y)\\
        & \qquad \qquad   + \frac{1}{n^3} \left[1 + (nh^d)^{-1} + (nh^d)^{-2} \right] \int   \left( \int |f(x, y)| \sd x \right)^2 \sd P_{0, h}(y).
    \end{align*}
    By \Cref{lemma: norm tool 1}, we have $\norm{f}_{L_1(\lambda \times P_0)} = O(1)$, $\norm{f}_{L_2(\lambda \times P_0)} = O(h^{-d/2})$, and $\int \left\{\int |f(x, y)| \sd x \right\}^2 \sd P_0(y)  = O(1)$  for each $f \in \{f_1, f_2\}$. We also note that $\tau_{G[f]} :=\iint f(hu, hv)K(u) K(v) \sd u \sd v \lesssim h^{-d}$ and $\tau_{G[f^2]} :=\iint \{f(hu, hv)\}^2 K(u) K(v) \sd u \sd v \lesssim h^{-2d}$ for each $f \in \{f_1, f_2\}$, where the constants only depend on $K$. Thus,
    \begin{align*}
        \E_0 \E_W \left\{ \int f \sd [(\d{P}_n^* - \hat{P}_n) \times (\d{P}_n^* - \hat{P}_n)] - n^{-1}\tau_f\right\}^2 &\lesssim \frac{h^{-2d}}{n^4} +\frac{h^{-2d}}{n^3} + \frac{h^{-d}}{n^3} + \frac{h^{-d}}{n^2} + \frac{1}{n^3} + \frac{h^{-2d}}{n^5} \\
        &= n^{-2} (nh^{d})^{-2} + n^{-1}(nh^{d})^{-2}+  n^{-2} (nh^{d})^{-1} \\
        &\qquad +  n^{-1} (nh^{d})^{-1} + n^{-3} +  n^{-3} (nh^{d})^{-2},
    \end{align*}
    which is again $o(n^{-1})$ if $nh^{d} \longrightarrow \infty$.

\end{proof}

\begin{lemma}\label{thm: auto smooth empirical process}
   If $P_0$ possesses uniformly bounded and continuously differentiable density function $\eta_0$ with $\int [ (D^{\alpha} \eta_0)( x) ]^2 \sd x < \infty$ for all $\alpha$ such that $|\alpha| = 1$, $\hat{P}_n$ is a kernel density estimator with uniformly bounded kernel function $K$ satisfying $\int K^2 < \infty$ and bandwidth $h$ satisfying $nh^d \longrightarrow \infty$ and $h \longrightarrow 0$, $\eta_n$ is the density corresponding to $\hat{P}_n$, and $\eta_n^*$ is a kernel density estimator based on the bootstrap data with the same kernel and bandwidth, then $(\d{P}_n^* - \hat{P}_n)(\eta_n^* - \eta_n) - (\d{P}_n - P_0)(\eta_n - \eta_0) = o_{\prob_W^*}(n^{-1/2})$.
\end{lemma}
\begin{proof}[\bfseries{Proof of \Cref{thm: auto smooth empirical process}}]
    As above, we define $\eta_{0,h}(x) := \int K_{h}(x, y) \sd P_0(y)$ and $\eta_{n,h}(x) := \int K_{h}(x, y) \sd \hat{P}_n(y)$. By adding and subtracting terms, we then have
    \begin{align*}
        & (\d{P}_n^* - \hat{P}_n)(\eta_n^* - \eta_0) - (\d{P}_n - P_0)(\eta_n - \eta_0) \\
        &\qquad = (\d{P}_n^* - \hat{P}_n)(\eta_n^* - \eta_{n, h}) + (\d{P}_n^* - \hat{P}_n)(\eta_{n, h} - \eta_{0, h}) + (\d{P}_n^* - \hat{P}_n)(\eta_{0,h} - \eta_0) \\
        &\qquad\qquad - (\d{P}_n - P_0)(\eta_n - \eta_{0,h}) - (\d{P}_n - P_0)(\eta_{0,h} - \eta_0) \\
        &\qquad = \int K_h  \sd [(\d{P}_n^* - \hat{P}_n) \times (\d{P}_n^* - \hat{P}_n)] + \int K_h \sd [ (\d{P}_n^* - \hat{P}_n) \times (\hat{P}_n - P_{0, h})]  \\
        &\qquad \qquad +\iint [K_h(x, y) - \eta_0(x)] \sd (\d{P}_n^* - \hat{P}_n)(x) \sd P_0(y)]  - \int K_h \sd [(\d{P}_n - P_0) \times (\d{P}_n - P_0)] \\
        &\qquad \qquad  -\iint [K_h(x, y) - \eta_0(x)] \sd (\d{P}_n - P_0)(x) \sd P_0(y)]
    \end{align*}
    By \Cref{cor: auto smooth empirical}, if $nh^d \longrightarrow \infty$, then
    \[
        \int K_h \sd [ (\d{P}_n^* - \hat{P}_n) \times (\hat{P}_n - P_{0, h})] = o_{\prob_W^*}(n^{-1/2})
    \]
    and 
    \begin{align*}
        &\int K_h  \sd [(\d{P}_n^* - \hat{P}_n)  \times (\d{P}_n^* - \hat{P}_n)] - \int K_h \sd [(\d{P}_n - P_0) \times (\d{P}_n - P_0)]\\
        &\qquad= \left[n^{-1} \tau_{f_1} + o_{\prob_W^*}(n^{-1/2})\right] - \left[n^{-1} \tau_{f_1} +  o_{\prob_0^*}(n^{-1/2})\right] \\
        &\qquad = o_{\prob_W^*}(n^{-1/2}).
    \end{align*}
    Finally, we write 
    \begin{align*}
        \iint [K_h(x, y) - \eta_0(x)] \sd (\d{P}_n^* - \hat{P}_n)(x) \sd P_0(y)] &= (\d{P}_n^* - \hat{P}_n) g_h, \text{ and} \\
        \iint [K_h(x, y) - \eta_0(x)] \sd (\d{P}_n - P_0)(x) \sd P_0(y)] &= (\d{P}_n - P_0)g_h
    \end{align*}
    for $g_h(x) := \int [K_h(x, y) - \eta_0(x)] \sd P_0(y)$. As in the proofs of \Cref{lemma: sV order 2} and \Cref{lemma: sV order 4}, we have $\E_0 \left[ (\d{P}_n - P_0)g_h\right]^2 \leq n^{-1} \| g_h \|_{L_2(P_0)}^2$. By \Cref{lemma: h^m tool 1}, $\int g_h(x)^2 \sd x = O(h^{2})$, so $\| g_h \|_{L_2(P_0)}^2 = o(1)$ as long as $h \longrightarrow 0$ since $P_0$ has uniformly bounded density. This implies that $\iint [K_h(x, y) - \eta_0(x)] \sd (\d{P}_n - P_0)(x) \sd P_0(y) = o_{\prob_0^*}(n^{-1/2})$. Furthermore,
    \begin{align*}
        \E_0 \E_W \left[ (\d{P}_n^* - \hat{P}_n)g_h\right]^2 &\leq n^{-1} \E_0\| g_h \|_{L_2(\hat{P}_n)}^2 \\
        &= n^{-1} \int \left[g_h(x)\right]^2 \sd P_{0, h}(x) \\
        &= n^{-1} \iint \left[g_h(x)\right]^2 K_h(x,y) \eta_0(y) \sd y \sd x  \\
        &=  n^{-1} \iint \left[g_h(x)\right]^2 K(u) \eta_0(x + hu) \sd u \sd x \\
        &\lesssim n^{-1} \int \left[g_h(x)\right]^2 \sd x.
    \end{align*}
    By \Cref{lemma: h^m tool 1}, $\int g_h(x)^2 \sd x = O(h^{2}) = o(1)$. Thus, $\int [K_h(x, y) - \eta_0(x)] \sd (\d{P}_n^* - \hat{P}_n)(x) \sd P_0(y)] = o_{\prob_W^*}(n^{-1/2})$.
\end{proof}

\begin{lemma}\label{thm: auto smooth mise}
    If $P_0$ possesses uniformly bounded and $m$-times continuously differentiable density $\eta_0$ with $\int [ (D^{\alpha} \eta_0)( x) ]^2 \sd x < \infty$ for all $\alpha$ such that $|\alpha| = m$, $\hat{P}_n$ is a kernel density estimator with uniformly bounded $m$th order kernel function $K$ satisfying $\int K^2 < \infty$ and bandwidth $h$ satisfying $nh^{2d} \longrightarrow \infty$ and $nh^{4m} \longrightarrow 0$, $\eta_n$ is the density corresponding to $\hat{P}_n$, and $\eta_n^*$ is a kernel density estimator based on the bootstrap data with the same kernel and bandwidth, then $\int[\eta_n(x) - \eta_0(x)]^2 \sd x - \int [\eta_n^*(x) - \eta_n(x)]^2 \sd x = o_{\prob_W^*}(n^{-1/2})$.
\end{lemma}
\begin{proof}[\bfseries{Proof of \Cref{thm: auto smooth mise}}]
    Recall that $f_1(x, y) := K_h\left(x, y\right)$ and $f_2(x, y) := \int K_h(x, z) K_h(y, z) \sd z$. As above, we denote $\eta_{0, h}(x) :=  \int  K_h\left(x, y\right) \sd P_0(y)$ and $\eta_{n, h}(x) = \int K_{h}(x, y) \sd \hat{P}_n(y)$. We then note that
    \begin{align*}
        &\int \eta_n^2 = \int f_2 \sd (\d{P}_n \times \d{P}_n), \  \int (\eta_n^{*})^2 = \int f_2 \sd (\d{P}_n^* \times \d{P}_n^*), \ \int \eta_{0,h}^2 = \int f_2  \sd (P_0 \times P_0) \\
        &\int \eta_n^*\eta_n = \int f_2 \sd (\d{P}_n^* \times \d{P}_n), \ \int \eta_n^*\eta_{0,h} = \int f_2 \sd (\d{P}_n^* \times P_0), \ \int \eta_n\eta_{0,h} = \int f_2 \sd (\d{P}_n \times P_0) \\
        &\int \eta_{n,h}^2 = \int f_2 \sd (\hat{P}_n \times \hat{P}_n), \ \int \eta_n^*\eta_{n,h} = \int f_2 \sd (\d{P}_n^* \times \hat{P}_n), \ \int \eta_n\eta_{n,h} = \int f_2 \sd (\d{P}_n \times \hat{P}_n).
    \end{align*}
    Hence, by adding and subtracting term, we have
    \begin{align*}
        &\int (\eta_n^* - \eta_n)^2  - \int(\eta_n - \eta_0)^2 \\
        & \qquad = \int (\eta_n^* - \eta_{n, h} + \eta_{n, h} - \eta_n)^2   - \int(\eta_n - \eta_{0, h}+ \eta_{0, h} - \eta_0)^2  \\
        & \qquad = \int \left(\eta_n^* - \eta_{n, h}\right)^2 + \int \left(\eta_{n, h} - \eta_n \right)^2 + 2\int \left(\eta_n^* - \eta_{n, h}\right) \left(\eta_{n, h} - \eta_n \right) \\
        & \qquad \qquad - \int \left(\eta_n - \eta_{0, h}\right)^2 - \int \left(\eta_{0, h} - \eta_0\right)^2  - 2\int \left(\eta_n - \eta_{0, h}\right)\left(\eta_{0, h} - \eta_0\right)  \\
        & \qquad = \left\{ \int f_2 \sd [(\d{P}_n^* - \hat{P}_n) \times (\d{P}_n^* - \hat{P}_n)] - \int f_2 \sd [(\d{P}_n - P_0) \times (\d{P}_n - P_0)] \right\}   \\
        & \qquad \qquad + \int \left(\eta_{n, h} - \eta_n \right)^2 - \int \left(\eta_{0, h} - \eta_0\right)^2  \\
        & \qquad \qquad + 2\int f_2\sd [(\d{P}_n^* - \hat{P}_n) \times (\hat{P}_n - \d{P}_n)]  - 2\int (f_2 - f_1)\sd [(\d{P}_n - P_0) \times P_0].
    \end{align*}
    By \Cref{cor: auto smooth empirical}, we have 
    \[
        \int f_2 \sd [ (\d{P}_n^* - \hat{P}_n) \times (\d{P}_n^* - \hat{P}_n)] - \int f_2 \sd [ (\d{P}_n - P_0) \times (\d{P}_n - P_0)] = o_{\prob_W^*}(n^{-1/2}).
    \]
    By \Cref{lemma: h^m tool 1}, $\int \left(\eta_{0, h} - \eta_0\right)^2 = O(h^{2m})$, which is $o(n^{-1/2})$ since $nh^{4m} \longrightarrow 0$. Next, we note that 
    \begin{align*}
        \eta_{n, h}(x) - \eta_n(x)  &= \int K_h(x, z) \sd \hat{P}_n(z) - \int K_h(x, y) \sd \d{P}_n(y)\\
        & = \int K_h(x, z) \eta_n(z) \sd z - \int K_h(x, y) \sd \d{P}_n(y)\\
        & = \int \left[ \int K_h(x, z) K_h(y, z) \sd z \right] \sd \d{P}_n(y) - \int K_h(x, y) \sd \d{P}_n(y)\\
        & = \int \left[ f_2(x, y) - f_1(x, y)\right] \sd \d{P}_n(y)\\
        & = \int f_3(x, y) \sd \d{P}_n(y),
    \end{align*}
    where $f_3 = f_2 - f_1$. Therefore, 
    \begin{align*}
        \E_0 \int (\eta_{n, h} - \eta_n)^2 &= \E_0 \int \left[ \int f_3(x, y) \sd \d{P}_n(y)\right]^2  \sd x\\
        & = \E_0 \int  \left\{ \frac{1}{n^2} \sum_{i,j} f_3(x, X_i)f_3(x, X_j) \right\} \sd x\\
        & = \E_0 \int  \left\{ \frac{1}{n^2} \sum_{i=1}^n f_3(x, X_i)f_3(x, X_i) \right\} \sd x + \E_0 \int  \left\{ \frac{1}{n^2} \sum_{i\neq j} f_3(x, X_i)f_3(x, X_j) \right\} \sd x\\
        & =  \frac{1}{n} \int    \E_0 \left[f_3^2(x, X)\right] \sd x + \frac{n-1}{n}  \int  \left\{ \E_0 \left[f_3(x, X)\right] \right\}^2 \sd x\\
        & =  \frac{1}{n} \norm{f_3}_{L_2(\lambda \times P_0)}^2 + \frac{n-1}{n}  \int  \left\{ \E_0 \left[f_3(x, X) \right] \right\}^2 \sd x.
    \end{align*}
    By \Cref{lemma: norm tool 1}, we have $\norm{f_3}_{L_2(\lambda \times P_0)} \leq \norm{f_1}_{L_2(\lambda \times P_0)} + \norm{f_2}_{L_2(\lambda \times P_0)} = O(h^{-d/2})$. By \Cref{lemma: h^m tool 1}, we have
    \[
        \int  \left\{ \E_0 \left[ f_3(x, X)\right] \right\}^2 \sd x \leq 2\int  \left\{ \E_0 \left[f_1(x, X) \right]  - \eta_0(x) \right\}^2 \sd x + 2\int  \left\{ \E_0 \left[f_2(x, X)\right] - \eta_0(x) \right\}^2 \sd x = O(h^{2m}).
    \]
    Hence, $\int (\eta_{n, h} - \eta_n)^2 = O_{\prob_0^*}( \{nh^d\}^{-1} + h^{2m})$, which is $o_{\prob_0^*}(n^{-1/2})$ under the conditions $nh^{4m} \longrightarrow 0$ and $nh^{2d} \longrightarrow \infty$.

    Next, we have
    \begin{align*}
        &\int f_2\sd [(\d{P}_n^* - \hat{P}_n) \times (\hat{P}_n - \d{P}_n)] \\
        &\qquad = \int f_2\sd [(\d{P}_n^* - \hat{P}_n) \times (\hat{P}_n - P_{0, h})] + \int f_2\sd [(\d{P}_n^* - \hat{P}_n) \times (P_{0, h} - P_0)] \\
        & \qquad \qquad - \int f_2\sd [(\d{P}_n^* - \hat{P}_n) \times (\d{P}_n - P_0)].
    \end{align*}
    By \Cref{lemma: sV order 3}, we have 
    \begin{align*}
        & \E_0 \E_W \left\{ \int f_2 \sd [ (\hat{P}_n - P_{0, h}) \times (\d{P}_n^* - \hat{P}_n)]\right\}^2  \lesssim n^{-1}\left[(nh^d)^{-1} + (nh^d)^{-2} \right]\int \left[ \int |f_2(x, y)| \sd x\right]^2 \sd P_{0, h}(y),
    \end{align*}
    which is $o(n^{-1})$ under the condition $nh^d \longrightarrow \infty$ by \Cref{lemma: norm tool 1}. By \Cref{lemma: sV order 2}, we have 
    \begin{align*}
        \E_0 \E_W \left[ \int f_2 \sd [(P_{0, h} - P_0) \times (\d{P}_n^* - \hat{P}_n)] \right]^2 \lesssim n^{-1}h^{2m} \| f_2 \|_{L_2(\lambda \times P_{0,h})}^2,
    \end{align*}
    which is $O(n^{-1} h^{2m -d})$ by \Cref{lemma: norm tool 1}, which is $o(n^{-1})$ by assumption. By \Cref{lemma: norm tool 1} and \Cref{lemma: sV order 3.5}, we have 
    \begin{align*}
        \E_0 \E_W \left\{ \int f_2 \sd [ (\d{P}_n - P_0) \times (\d{P}_n^* - \hat{P}_n)]\right\}^2  & \lesssim \frac{1}{n^2} \norm{f_2}_{L_2(P_0 \times P_{0, h})}^2 +  \frac{1}{n^3h^{d}} \norm{f_2}_{L_2(\lambda\times P_0)}^2\\
        & \lesssim \frac{1}{n^2} \norm{f_2}_{L_2(P_0 \times P_{0, h})}^2 +  \frac{1}{n^3h^{2d}}\\
        & \lesssim \frac{1}{n^2h^{d}} +  \frac{1}{n^3h^{2d}},
    \end{align*}
    where the last inequality is because 
    \begin{align*}
        \iint f_2^2(x, z) \sd P_0(x) \sd P_{0, h}(z) &= \iint \left[ \int K_h(x, s)K_h(z, s) \sd s \right]^2 \sd P_0(x) \sd P_{0, h}(z)\\
        & = \iint \left[ \iint K_h(x, s)K_h(z, s)K_h(x, t)K_h(z, t) \sd s \sd t\right] \sd P_0(x) \sd P_{0, h}(z)\\
        & = \frac{1}{h^{4d}} \iiiint K\left(\frac{x-s}{h}\right)K\left(\frac{z-s}{h}\right)K\left(\frac{x-t}{h}\right)  \\
        & \qquad \qquad \times K\left(\frac{z-t}{h}\right) \sd s \sd t \sd P_0(x) \sd P_{0, h}(z) \\
        & = \frac{1}{h^{2d}}\left| \iiiint K\left(s'\right)K\left(s' + \frac{z-x}{h}\right)K\left(t'\right) \right. \\
        & \qquad \qquad \left. \times K\left(t' + \frac{z-x}{h}\right) \sd s' \sd t' \sd P_0(x) \sd P_{0, h}(z) \right|\\
        & = \frac{1}{h^{d}}\left| \iiiint K\left(s'\right)K\left(x'\right)K\left(t'\right) \right.\\
        & \qquad \qquad \left. \times K\left(t' + x' - s'\right) \eta_0(z - (x' -s')h ) \sd s' \sd t' \sd x' \sd P_{0, h}(z) \right|\\
        & \lesssim \frac{1}{h^{d}} \iiiint \left|K\left(s'\right)K\left(x'\right)K\left(t'\right) K\left(t' + x' - s'\right)\right| \sd s' \sd t' \sd x' \sd P_{0, h}(z) \\
        & \leq \frac{1}{h^{d}} \left[ \int K^2(x) \sd x\right]^2.
    \end{align*}
    This implies that $\E_0 \E_W \left\{ \int f_2 \sd [ (\d{P}_n - P_0) \times (\d{P}_n^* - \hat{P}_n)]\right\}^2  = o(n^{-1})$  under the assumption $nh^d \longrightarrow \infty$. Hence, $\int f_2\sd [(\d{P}_n^* - \hat{P}_n) \times (\hat{P}_n - \d{P}_n)] = o_{\prob_W^*}(n^{-1/2})$ if $nh^d \longrightarrow \infty$.

    Finally, by \Cref{lemma: h^m tool 1}, we have
    \begin{align*}
        &\E_0 \left[ \int (f_2 - f_1)\sd [(\d{P}_n - P_0) \times P_0] \right]^2 \\
        & \qquad \leq \frac{1}{n} \E_0 \left[ \int \left\{f_2(X, y) - f_1(X, y) \right\} \sd P_0(y) \right]^2\\
        & \qquad \leq \frac{2}{n} \E_0 \left[ \int\left\{ f_1(X, y) - \eta_0(X)\right\} \sd P_0(y) \right]^2 + \frac{2}{n} \E_0 \left[ \int\left\{ f_2(X, y) - \eta_0(X) \right\} \sd P_0(y) \right]^2\\
        & \qquad = O(n^{-1}h^{2m}),
    \end{align*}
    which is $o(n^{-1})$. 
\end{proof}

\begin{lemma}\label{thm: auto smooth asym bias}
     If $P_0$ possesses uniformly bounded and $m$-times continuously differentiable density $\eta_0$ with $\int [ (D^{\alpha} \eta_0)( x) ]^2 \sd x < \infty$ for all $\alpha$ such that $|\alpha| = m$, $\hat{P}_n$ is a kernel density estimator with uniformly bounded $m$th order kernel function $K$ satisfying $\int K^2 < \infty$ and bandwidth $h$ satisfying $nh^{2d} \longrightarrow \infty$ and $nh^{4m} \longrightarrow 0$, $\eta_n$ is the density corresponding to $\hat{P}_n$, and $\eta_n^*$ is a kernel density estimator based on the bootstrap data with the same kernel and bandwidth, then $\d{P}_n\phi_n - \d{P}_n^*\phi_n^* = o_{\prob_W}(n^{-1/2})$.
\end{lemma}
\begin{proof}[\bfseries{Proof of \Cref{thm: auto smooth asym bias}}]
    Recall that $f_1(x, y) := K_h\left(x, y\right)$ and $f_2(x, y) := \int K_h(x, z) K_h(y, z) \sd z$. We define $f_3 := f_1 - f_2$. We then have $\d{P}_n \phi_n = 2\int f_3 \sd (\d{P}_n \times \d{P}_n)$ and $\d{P}_n^* \phi_n^* = 2\int f_3 \sd (\d{P}_n^* \times \d{P}_n^*).$ Hence, by adding and subtracting terms and the symmetry of $f_3$,
    \begin{align*}
        &\left(\d{P}_n^*\phi_n^* - \d{P}_n \phi_n \right) / 2 \\
        & \qquad = \int f_3 \sd [(\d{P}_n^* - \hat{P}_n + \hat{P}_n - P_{0, h} + P_{0, h}) \times (\d{P}_n^* - \hat{P}_n + \hat{P}_n - P_{0, h} + P_{0, h})] \\
        & \qquad \qquad - \int f_3 \sd [(\d{P}_n - P_0 + P_0) \times (\d{P}_n - P_0 + P_0)]\\
        & \qquad = \left\{ \int f_3 \sd [(\d{P}_n^* - \hat{P}_n) \times (\d{P}_n^* - \hat{P}_n)] - \int f_3 \sd [(\d{P}_n - P_0) \times (\d{P}_n - P_0)] \right\}  \\
        & \qquad \qquad + \int f_3 \sd [(\hat{P}_n - P_{0, h}) \times (\hat{P}_n - P_{0, h})] + \left\{ \int f_3 \sd [P_{0, h} \times P_{0, h}] - \int f_3 \sd [ P_0 \times P_0] \right\}   \\
        & \qquad \qquad + 2 \int f_3 \sd [(\d{P}_n^* - \hat{P}_n) \times (\hat{P}_n - P_{0, h})] + 2 \int f_3 \sd [(\d{P}_n^* - \hat{P}_n) \times (P_{0, h} - P_0)] \\
        & \qquad \qquad + 2 \int f_3 \sd [(\d{P}_n^* - \hat{P}_n) \times P_0] + 2 \int f_3 \sd [(\hat{P}_n - P_{0, h}) \times P_{0, h}] - 2 \int f_3 \sd [(\d{P}_n - P_0) \times P_0].
    \end{align*}
    We show each term on the right hand side of previous display is $o_{\prob_W^*}(n^{-1/2})$. First, we can use the same logic as in the proof of \Cref{thm: auto smooth mise} to show that
    \begin{gather*}
        \int f_3 \sd [ (\d{P}_n^* - \hat{P}_n) \times (\d{P}_n^* - \hat{P}_n)] - \int f_3 \sd [ (\d{P}_n - P_0) \times (\d{P}_n - P_0)] = o_{\prob_W^*}(n^{-1/2}),\\
        \int f_3 \sd [(\d{P}_n^* - \hat{P}_n) \times (\hat{P}_n - P_{0, h})] = o_{\prob_W^*}(n^{-1/2}), \\
        \int f_3 \sd [(\d{P}_n^* - \hat{P}_n) \times (P_{0, h} - P_0)] = o_{\prob_W^*}(n^{-1/2}), \text{ and}\\
        \int f_3 \sd [(\d{P}_n - P_0) \times P_0] = o_{\prob_0^*}(n^{-1/2}).
    \end{gather*}
    By \Cref{lemma: norm tool 1} and \Cref{lemma: sV order 1}, we have
    \begin{align*}
        &\E_0  \left\{ \int f \sd [(\hat{P}_n - P_{0, h}) \times (\hat{P}_n - P_{0, h})]  -  n^{-1}\tau_{G[f]} \right\}^2  \\
        &\qquad \lesssim n^{-3} \int \left[\int |f(x, y)| \sd x \right]^2 \sd P_0(y) + n^{-2} \norm{f}_{L_2(\lambda \times P_0)}^2\\
        & \qquad = O(n^{-3}) + O(n^{-2} h^{-d})
    \end{align*}
    for $f \in \{f_1, f_2\}$, which is $o(n^{-1})$ because $nh^d \longrightarrow \infty$. Hence, 
    \begin{align*}
        & \int f_3 \sd [(\hat{P}_n - P_{0, h}) \times (\hat{P}_n - P_{0, h})] \\
        & \qquad = \int f_1 \sd [(\hat{P}_n - P_{0, h}) \times (\hat{P}_n - P_{0, h})] - \int f_2 \sd [(\hat{P}_n - P_{0, h}) \times (\hat{P}_n - P_{0, h})] \\
        & \qquad = n^{-1} \left(\tau_{G[f_1]} + \tau_{G[f_2]}\right)+ o_{\prob_0^*}(n^{-1/2}).
    \end{align*}
    Furthermore, we have $\tau_{G[f]} = O(h^{-d})$ for $f \in \{f_1, f_2\}$ as shown above. Since  $nh^{2d} \longrightarrow \infty$, we then have $n^{-1}(\tau_{G[f_1]} + \tau_{G[f_2]}) = O(\{nh^d\}^{-1})= O( n^{-1/2} \{ nh^{2d}\}^{-1/2}) = o(n^{-1/2})$. 

    Next, we note that 
    \begin{align*}
        &\abs{\int f_3 \sd (P_{0, h} \times P_{0, h}) - \int f_3 \sd ( P_0 \times P_0)} \\
        & \quad = \abs{\iint \left[ \iint f_3(s, t) K_h(x, s) K_h(y, t) \sd s \sd t \right] \sd P_0(x) \sd P_0(y) - \int  f_3 \sd (P_0 \times P_0) }\\
        & \quad =\abs{ \iint \left[ \iint f_3(x+sh, y+th) K(s) K(t) \sd s \sd t \right] \sd P_0(x) \sd P_0(y) - \int  f_3 \sd (P_0 \times P_0) }\\
        & \quad = \abs{\iiiint f_3(x, y+t'h-s'h)  K(s') K(t') \sd P_0(x) \sd P_0(y)  \sd s' \sd t'- \int  f_3 \sd (P_0 \times P_0)}\\
        & \quad =\abs{\iiiint f_3(x, y')  K(s') K(t') \sd P_0(x) \eta_0(y' + s'h - t'h) \sd y'  \sd s' \sd t'- \int  f_3 \sd (P_0 \times P_0)}\\
        & \quad = \abs{\int \left[ \int f_3(x, y) \sd P_0(x)\right] \left[ \iint \left\{\eta_0(y + sh- th) - \eta_0(y)\right\} K(s) K(t) \sd s \sd t \right] \sd y}\\
        & \quad \leq \left\{ \int \left[ \int f_3(x, y) \sd P_0(x)\right]^2  \sd y \right\}^{1/2} \left\{\int \left[ \iint \left\{\eta_0(y + sh- th) - \eta_0(y)\right\} K(s) K(t) \sd s \sd t \right]^2 \sd P_0(y)\right\}^{1/2}.
    \end{align*}
    We can write
    \begin{align*}
        &\int  \left[ \int f_3(x, y) \sd P_0(y) \right]^2  \sd x \\
        & \qquad \leq  2 \int  \left[ \int f_1(x, y)\sd P_0(y) - \eta_0(x) \right]^2  \sd x + 2 \int  \left[ \int f_2(x, y)\sd P_0(y) - \eta_0(x)  \right]^2  \sd x,
    \end{align*}
    which is $O(h^{2m})$ by \Cref{lemma: h^m tool 1}. Since $\eta_0$ is $m$-times continuously differentiable, for all $u$, a Taylor expansion with the Laplacian representation  of the remainder gives
    \begin{align*}
        &\iint \left\{\eta_0(y + sh- th) - \eta_0(y)\right\} K(s) K(t) \sd s \sd t\\
        & \quad = \iint \left\{ \sum_{1 \leq |\alpha| \leq m-1} \frac{1}{\alpha!}(sh- th)^{\alpha} (D^{\alpha}\eta_0)(y) \right\} K(s) K(t) \sd s \sd t \\
        & \quad\qquad + \iint  \left\{ \sum_{|\alpha|=m} \frac{mh^m}{\alpha!} (s- t)^{\alpha} \int_0^1 (1-r)^{m-1} (D^{\alpha} \eta_0)( y + rh[s -t] ) \sd r \right\} K(s) K(t) \sd s \sd t\\
        &\quad = \sum_{1 \leq |\alpha| \leq m-1} \frac{1}{\alpha!} (D^{\alpha}\eta_0)(y) \iint  (sh- th)^{\alpha}  K(s) K(t) \sd s \sd t \\
        & \quad\qquad + \sum_{|\alpha|=m} \frac{mh^m}{\alpha!} \iint  (s- t)^{\alpha} \int_0^1 (1-r)^{m-1} (D^{\alpha} \eta_0)( y + rh[s -t] ) \sd r K(s) K(t)  \sd s \sd t.
    \end{align*}
    Since $K$ is an $m$th order kernel function, $\iint  (sh- th)^{\alpha}  K(s) K(t) \sd s \sd t = 0$ for all $\alpha$ such that  $1 \leq |\alpha| \leq m-1$. Defining $H(y, u, \alpha) := \int_0^1 (1-r)^{m-1} (D^{\alpha} \eta_0)( y + rhu ) \sd r$, we then have 
    \begin{align*}
        & \int \left[ \iint \left\{\eta_0(y + sh- th) - \eta_0(y)\right\} K(s) K(t) \sd s \sd t \right]^2 \sd P_0(y)\\
        &\quad = \int \left[\sum_{|\alpha|=m} \frac{mh^m}{\alpha!} \iint  (s- t)^{\alpha} H(y, s-t, \alpha)  K(s) K(t) \sd s \sd t \right]^2 \sd P_0(y)\\
        &\quad = \int \sum_{|\alpha|, |\beta|=m} \frac{m^2h^{2m}}{\alpha!\beta!} \iiiint    (s- t)^{\alpha}(s'- t')^{\beta} H(y, s-t, \alpha)H(y, s'-t', \beta) \\
        & \qquad \qquad \times K(s) K(t)K(s') K(t')  \sd s \sd t \sd s' \sd t' \sd P_0(y).
    \end{align*}
    The same method as in the proof of \Cref{lemma: h^m tool 1} can be used to show that this is $O(h^
{2m})$.
    Hence, $\int f_3 \sd (P_{0, h} \times P_{0, h}) - \int f_3 \sd ( P_0 \times P_0) = O(h^{2m})$, which is $o(n^{-1/2})$ if $nh^{4m} \longrightarrow 0$.

    We next consider the term $\int f_3 \sd [(\d{P}_n^* - \hat{P}_n) \times P_0]$. As in the proofs of \Cref{lemma: V order 2} and \Cref{lemma: V order 3},
    \begin{align*}
        &\E_0 \E_W \left\{\int \left[ \int f_3(x, y) \sd P_0(y) \right] \sd (\d{P}_n^* - \hat{P}_n)(x) \right\}^2 \\
        &\qquad \leq \E_0 \left\{ \frac{1}{n} \int  \left[ \int f_3(x, y) \sd P_0(y) \right]^2  \sd \hat{P}_n(x) \right\}\\
        & \qquad =  \frac{1}{n} \int  \left[ \int f_3(x, y) \sd P_0(y) \right]^2  \sd P_{0, h}(x)\\
        & \qquad \lesssim  \frac{1}{n} \int  \left[ \int f_3(x, y) \sd P_0(y) \right]^2  \sd x,
    \end{align*}
    where the last inequality is because $P_{0,h}$ possesses uniformly bounded density. The last expression is $O(n^{-1}h^{2m})$ by \Cref{lemma: h^m tool 1} as shown above. Hence, $\int f_3 \sd [(\d{P}_n^* - \hat{P}_n) \times P_{0}] =O_{\prob_W^*}(n^{-1/2}h^m) = o_{\prob_W^*}(n^{-1/2})$.

    Finally, we consider the term $\int f_3 \sd [(\hat{P}_n - P_{0, h}) \times P_{0, h}]$. We note that 
    \begin{align*}
        \int f_3 \sd [(\hat{P}_n - P_{0, h}) \times P_{0, h}] & = \iiiint  f_3(s, t) K_h(x, s) K_h(y, t) \sd s \sd t \sd P_0(y) \sd (\d{P}_n - P_0)(x).
    \end{align*}
    Hence, 
    \begin{align*}
        &\E_0 \left\{ \int f_3 \sd [(\hat{P}_n - P_{0, h}) \times P_{0, h}] \right\}^2\\
        &\qquad \leq \frac{1}{n} \int \left[ \iiint  f_3(s, t) K_h(x, s) K_h(y, t) \sd s \sd t \sd P_0(y) \right]^2 \sd P_0(x) \\
        & \qquad = \frac{1}{n} \int \left[ \iiint  f_3(x+s'h, y+t'h) K(s') K(t') \sd s' \sd t' \sd P_0(y) \right]^2 \sd P_0(x) \\
        & \qquad \leq \frac{1}{n} \iiint  \left[ \int f_3(x+s'h-t'h, y) \sd P_0(y)  \right]^2 K(s') K(t')\sd s' \sd t' \sd P_0(x) \\
        & \qquad \lesssim \frac{1}{n} \int  \left[ \int f_3(x, y) \sd P_0(y)  \right]^2 \sd x.
    \end{align*}
    As above, the last expression is $O(n^{-1} h^{2m})$, so that $ \int f_3 \sd [(\hat{P}_n - P_{0, h}) \times P_{0, h}] = O_{\prob_0^*}(n^{-1/2}h^m) = o_{\prob_0^*}(n^{-1/2})$.
    
\end{proof}

\end{appendices}

\end{document}